\documentclass[10pt]{report}
\usepackage[latin1]{inputenc}
\usepackage[francais]{babel}
\usepackage[T1]{fontenc}
\usepackage{amsmath}
\usepackage{amsfonts}
\usepackage{amssymb}
\usepackage{graphicx}
\usepackage{lmodern}
\usepackage{amsthm,amscd}
\usepackage{fixmath}
\usepackage{enumitem}
\usepackage{listings}
\usepackage[left=2cm,right=2cm,top=2cm,bottom=3cm]{geometry}

\usepackage{nicefrac}
\usepackage{color}
\usepackage{mathabx}
\usepackage{stmaryrd}

\usepackage{tikz-cd}
\usepackage{pst-node} 

\usepackage{mathtools}
\DeclarePairedDelimiter\ceil{\lceil}{\rceil}
\DeclarePairedDelimiter\floor{\lfloor}{\rfloor}
\author{Sybille Rosset}
\date{PhD Thesis defended in July 2020}

\begin{document}
	\newtheorem{proposition}{Proposition}
	\newtheorem*{proposition*}{Proposition}
	\numberwithin{proposition}{subsection}
	\newtheorem{theorem}{Theorem}
	\numberwithin{theorem}{subsection}
	\newtheorem*{theorem*}{Theorem}
	\newtheorem*{theoreme*}{Th\'eor\`eme}
	\newtheorem{corollaire}{Corollary}
	\numberwithin{corollaire}{subsection}
	\newtheorem*{corollaire*}{Corollaire}
	\newtheorem*{corollary*}{Corollary}
	\newtheorem{lemme}{Lemma}
	\numberwithin{lemme}{subsection}
	\newtheorem{Hyp}{Hypothesis}
	\newtheorem{conjecture}{Conjecture}
	\numberwithin{conjecture}{subsection}
	\newtheorem*{conjecture*}{Conjecture}
	\theoremstyle{remark}
	\newtheorem{Rem}{Remark}
	\numberwithin{Rem}{subsection}
	\newtheorem{fact}{Fact}
	\numberwithin{fact}{subsection}
	\newcommand{\Spec}{\mathrm{Spec}}
	\theoremstyle{definition}
	\newtheorem{defn}{Definition}
	\numberwithin{defn}{subsection}
	\newtheorem{d'efinition}{D\'efinition}
	\numberwithin{d'efinition}{subsection}
	\newenvironment{thmbis}
	{\addtocounter{theorem}{-1}%
		\renewcommand{\thetheorem}{\arabic{theorem}bis}%
		\begin{theorem}}
		{\end{theorem}}
	
	 \def\diag{{\rm diag}}
	\def\Im{{\rm Im}}
	\def\dim{{\rm dim}}
	\def\codim{{\rm codim}}
	\def\GL{{\rm GL}}
	\def\SL{{\rm SL}}
	\def\ev{{\rm ev}}
	\def\Gr{{\rm Gr}}
	\def\id{{\rm id}}
	\def\Spec{{\rm Spec}}
	\def\S{{\mathfrak S}}
	\def\d{{\mathbf d}}
	\def\p{{\mathfrak{p}}}
	\def\C{{\mathbb C}}
	\def\P{{\mathbb P}}
	\def\A{{\mathbb A}}
	\def\Z{{\mathbb Z}}
	\def\Q{{\mathbb Q}}
	\def\I{{\mathcal I}}
	\def\G{{\mathcal G}}
	\def\O{{\mathcal O}}
	\def\W{{\mathcal W}}

\title{Quantum $K$-theory for flag varieties}
\maketitle

\chapter*{Summary}
\addcontentsline{toc}{chapter}{Summary}  

We present here an overview of the different results presented in this thesis. Chapters \ref{chap : stabilization correlators} and \ref{chap: Schubert calculus} use some of the results of Chapter \ref{chap : geometry W}, but are otherwise independent and each of them may be read separately. 

\subsection{Chapter \ref{chap : geometry W}: Gromov-Witten varieties of partial flag varieties.}
Let $n \geq 0$, let $J=\{j_{1}, \dots,j_{M}\}$ be a set of $M$ integers satisfying $0 < j_1 < \dots < j_M <n$. Consider the flag variety $X=Fl_J$ parametrizing the flags of vector spaces of type $$V_{j_1} \subset V_{j_2} \subset \dots \subset V_{j_M} \subset  \C^n,$$
where $V_{j_k}$ is a vector subspace of $\C^n$ of dimension $j_k$. For an element $\d$ in the semi-group $E(X) \subset H_2(X, \Z)$ of effective classes of curves, we denote by ${\mathcal{M}_{0,r}}(X, \d)$ the space of maps $(f:\P^1 \to X, \{p_1, \dots, p_r\})$ with $r$ marked points $p_i$ on $\P^1$ satisfying $f_*[\P^1] = \d$. Associating to a map the image of its marked points yields an evaluation morphism $ev_\d : {\mathcal{M}}_{0,r}(X, \d) \rightarrow  X^r$.  
For $r=3$ and $X$ a cominuscule variety, and in particular for $X$ a Grassmannian, Chaput-Perrin proved that the general fiber of the evaluation map is a rational variety \cite{chaput2011rationality}. For $X= \P^2$, Pandharipande computed the genus of the general fiber of the evaluation map $ev_{d} : {{\mathcal{M}}_{0,r}}(\P^2, d) \rightarrow (\P^2)^r$ when this fiber has dimension $1$ \cite{pandharipande1997canonical}, genus which turns out to be positive for $d \geq 3$. For $X=\P^1$, examples of values of $d$ such that the general fiber of $ev_d : {{\mathcal{M}}_{0,r}}(\P^1, d) \rightarrow (\P^1)^r$ is not a rationally connected variety have been obtained using techniques from quantum $K$-theory \cite{iritani2014reconstruction}. One might wonder what happens for a variety of Picard rank greater than one. Chapter \ref{chap : geometry W} provides a step in that direction by studying the geometry of the general fiber of the evaluation map $ev_\d : {\mathcal{M}}_{0,r}(X, \d) \rightarrow  X^r$ for $X$ a flag variety, and more generally the geometry of Gromov-Witten varieties of flag varieties.

  \begin{defn}\label{def: W variety english} Let $X_1$, $\dots$, $X_r$ be Schubert varieties of the flag variety $X$, let $\mathbf{d}$ be an element in $E(X)$. For any element $g=(g_1, \dots, g_r)$ general in $\GL_n^r$, we call \textit{Gromov-Witten variety of degree $\mathbf{d}$ associated with the varieties $X_i$ and with $g$} the subscheme \begin{equation*}
 	\W_{X; X_1, \dots, X_r}^{g, \mathbf{d}} := ev_1^{-1}(g_1 X_1) \cap \dots \cap ev_r^{-1}(g_r X_r)
 	\end{equation*} of $\overline{\mathcal{M}_{0,r}}(X, \mathbf{d})$.
 \end{defn}
Note that the Gromov-Witten variety $\W_{X; X_1, \dots, X_r}^{g, \mathbf{d}}$ parametrizes morphisms $\P^1 \rightarrow X$ representing the class $\d$ whose image has a non empty intersection with the translates $g_i \cdot X_i$ of the Schubert varieties $X_i$.  When these varieties are zero dimensional, their number of points is a Gromov-Witten invariant. More generally, the Euler-Poincar\'e characteristic of their adherence in the space $\overline{\mathcal{M}_{0,r}}(X, \mathbf{d})$ is a correlator in quantum $K$-theory \cite{givental}. 

Let $I=\{i_1, \dots, i_m\}$ be a set of $m$ integers satisfying $0<i_1 < \dots < i_m<n$, such that  $I$ is contained in $J$. We denote by $\pi: X=Fl_J \rightarrow X'=Fl_I$ the forgeftul map. 
Fix a class $\lambda = [l] \in H_2(X, \Z)$ of a curve $l$ whose projection by the forgetful map $\pi$ is a point. We observe in Chapter \ref{chap : geometry W} a relation between the general fiber of the evaluation map $$ev_{[C]+\lambda}: \overline{\mathcal{M}_{0,r}}(X, [C]+\lambda) \rightarrow X^r$$ and the general fiber of the map $ev_{[\pi(C)]}: \overline{\mathcal{M}_{0,r}}(X', [\pi(C)]) \to (X')^r$.  
\newpage
\begin{theorem*}(Theorem \ref{th : RC fibrations bewteen W varieties})
	Suppose the collection $(I,J, [C]+\lambda)$ is a stabilized collection in the sense of Definition \ref{def: stabilized collection}. Then: \begin{itemize}
		\item  For any Schubert varieties $X_1$, $\dots$, $X_r$ of $X$, for $g$ general in $\GL_n^r$, the Gromov-Witten variety $\W_{Fl_J; X_1, \dots, X_r}^{g, [C]+ \lambda}$ of $Fl_J$ is a rationally connected fibration above the Gromov-Witten variety $\W_{Fl_I; \pi(X_1), \dots, \pi(X_r)}^{g, [\pi(C)]}$ of $Fl_I$.
		\item For $x$ general in $(Fl_J)^r$, the fiber $ev_{[C]+ \lambda}^{-1}(x)$ is a tower of unirational fibrations over the fiber $ev_{[\pi(C)]}^{-1}(\pi(x))$.
	\end{itemize} 
\end{theorem*}
We called a variety $W$ a \textit{rationally connected fibration} (respectively unirational fibration) above $W'$  if there exists a dominant map $W \rightarrow W'$ sending each irreducible component of $W$ onto a different irreducible component of $W'$, and whose general fiber is a rationally connected variety (respectively unirational). A variety $W$ is a \textit{tower of unirational fibrations} above a variety $W'$ if there exists a sequence of morphisms $$W_1=W \rightarrow W_2 \rightarrow \dots \rightarrow W_l=W'$$  such that for $1 \leq i <l$, the variety $W_i$  is a unirational fibration above $W_{i+1}$.

\subsection{Chapter \ref{chap : stabilization correlators}: A comparison formula between genus $0$ correlators of flag varieties.}
Fix integers $1 \leq n_1 < \dots < n_m<n$. We consider the flag variety $X$ parametrizing flags of vector spaces \[V_{n_1} \subset \dots \subset V_{n_m} \subset \C^.\]
satisfying $\dim V_{n_i} = n_i$. Let $T$ be a torus in $\GL_n$ acting on $X$ by left multiplication. Let $\mathbf{d}=(d_1, \dots, d_m)$ be an effective class of curve in $H_2(X, \Z)$. We denote by $\overline{\mathcal{M}_{0,r}}(X, \d)$ the coarse moduli space parametrizing genus $0$ stable maps representing the class $\d$. Assocating to a stable map with $r$ marked points the image of its $i$-th marked point induces an evaluation map $ev_i : \overline{\mathcal{M}_{0,r}}(X, \d) \rightarrow X$. Let $E_1$, $\dots$, $E_r$ be $T$-equivariant fiber bundles over $X$. 
In quantum $K$-theory, $T$-equivariant correlators are defined by \begin{equation}\label{eq summary: corr X} \langle \phi_1, \dots , \phi_r \rangle_{T, (d_1, \dots, d_m)}^X := \chi_T \left( ev_1^*E_1 \otimes \dots \otimes ev_r^*E_r \otimes \O_{\overline{\mathcal{M}_{0,r}}(X, \d)} \right),  \end{equation}
where we denote by $\chi_T$ the $T$-equivariant Euler-Poincar\'e characteristic \cite{givental, lee2004quantum}.
For $1 \leq k \leq m$, we denote by $X_{\widehat{k}}$ the flag variety obtained from $X$ by forgetting the vector space $V_{n_k}$, and by  $$\pi_{\widehat{k}}: X \rightarrow X_{\widehat{k}}$$ the forgetful map. In order to facilitate notations, we set $d_0=0=d_{m+1}$. Chapter \ref{chap : stabilization correlators} yields the following equality between correlators-cf. Theorem \ref{th : euler char equal}bis.
\begin{theorem*}
	Suppose: \begin{itemize}
		\item $\forall \: p<k, \: d_k \geq n_k \ceil*{ \frac{d_p-d_{p-1}}{n_p-n_{p-1}}}$,
		\item $d_{k-1} \leq \lfloor \frac{d_{k+1}}{n_{k+1}} \rfloor$,
		\item $d_k \geq r(n_k-n_{k-1}) + d_{k-1} + (n_k-n_{k-1})(\lfloor \frac{d_{k+1}-d_{k-1}}{n_{k+1} - n_{k-1}}\rfloor +1)$.
	\end{itemize} 
	Then for any $T$-equivariant fiber bundles $E_1$, $\dots$, $E_r$ over $X_{\widehat{k}}$, the correlators of $X$ and $X_{\widehat{k}}$ associated with the degree $\mathbf{d}$ and the fiber bundles $E_i$ are equal: \begin{equation*}
	\langle (\pi_{\widehat{k}})^*E_1, \dots, (\pi_{\widehat{k}})^*E_r \rangle_{T, (d_1, \dots, d_m)}^{X} = \langle E_1, \dots, E_r \rangle_{T, (d_1, \dots, d_{k-1}, d_{k+1}, \dots, d_m)}^{X_{\widehat{k}}}.
	\end{equation*}
\end{theorem*}
Note that for $3$-pointed Gromov-Witten invariants, a result of Peterson/Woodward allows to deduce all invariants with $3$ marked points of generalized flag varieties $G/P$ from $3$-pointed Gromov-Witten invariants of $G/B$ \cite{woodward2005d}.

\subsection{Chapter \ref{chap: Schubert calculus}: Schubert calculus for the incidence variety $Fl_{1,n-1}$.} We study in Chapter \ref{chap: Schubert calculus} different variations of modern Schubert calculus for the incidence variety $X=Fl_{1,n-1}$ parametrizing pairs $(p,H)$ where $p$ is a point in $\P^{n-1}$ and $H$ is a hyperplane in $\P^{n-1}$ containing the point $p$. Note that $X$ can be written as $G/P$,  where $G=\GL_n$, $T$ is the set of diagonal matrices, $B$ is the set of upper triangular matrices and $P$ is the parabolic subgroup of $G$ satisfying $B \subset P$ associated with the roots $\{\alpha_2, \dots, \alpha_{n-1}\}$. The smallest length representative $w$ of an element in $W/W_P = \mathfrak{S}_n / \mathfrak{S}_{n-2}$  is a permutation $w_{i,j}$ satisfying $w(1)=i$, $w(n)=j$, and such that for $1<k<n-1$  $w(k) < w(k+1)$. For $1 \leq i,j \leq n$ and $i \neq j$, we call $X(i,j):=X(w_{i,j})$ the Schubert variety of $Fl_{1,n-1}$ associated with the element $w_{i,j}$ in $W^P \simeq W/W_P$.

\textsc{Littlewood-Richardson coefficients in $K(Fl_{1,n-1})$.} For $1 \leq i,j \leq n$, $i \neq j$, the natural embedding $i:X(i,j) \hookrightarrow X$ defines a sheaf $i_*\O_{X(i,j)}$, which is a coherent sheaf of $\O_{Fl_{1,n-1}}$-module. We denote by $\O_{i,j} := [i_*\O_{X(i,j)}]$ the class of $i_*\O_{X(i,j)}$ in the Grothendieck group $K(Fl_{1,n-1})$ of coherent sheaves of $\O_{Fl_{1,n-1}}$-module. Schubert classes $(\O_{i,j})_{1 \leq i,j \leq n, \: i \neq j}$ form a basis of  $K(Fl_{1,n-1})$, and the product $$\O_{i,j} \cdot \O_{k,p} = \sum_{i=0}^{\dim Fl_{1,n-1}} (-1)^i  [Tor_i^{Fl_{1,n-1}} \left(\O_{X(i,j)}, \O_{X(k,p)}\right)]$$ defines a structure of associative and commutative ring $(K(Fl_{1,n-1}), \cdot)$ with identity element $[\O_{Fl_{1,n-1}}]$. We set $\O_{i,j}=0$ if $i<1$ or $j>n$. \\
\begin{proposition*}[Littlewood-Richardson coefficients in $K(Fl_{1,n-1})$-cf. Proposition \ref{prop : produit dans K(X)}]
	Let $1 \leq i,j, k, p \leq n$, where $i \neq j$ and $k \neq p$. Then
	\begin{equation*}
	\left\{ 
	\begin{aligned}
	\mathcal{O}_{k,p} \cdot \mathcal{O}_{i,j} &= \mathcal{O}_{i+k -n, j+p -1} \: &\mathrm{if} \: i+k-n \geq j+p \: \mathrm{or} \: i<j \: \mathrm{or} \: k<p; \\
	\mathcal{O}_{k,p} \cdot \mathcal{O}_{i,j} &= \mathcal{O}_{i+k-n-1, j+p -1} + \mathcal{O}_{i+k-n, j+p} - \mathcal{O}_{i+k-n-1,j+p} \:  &\mathrm{else}.
	\end{aligned}
	\right.
	\end{equation*}
\end{proposition*}

\textsc{Chevalley formula and positive algorithm in $QK_s(Fl_{1,n-1})$.}  We denote by $l_1 := X(2,1)$ et $l_2:= X(1,2)$ the two one-dimensional Schubert varieties of $Fl_{1,n-1}$. Their classes $[l_1]$ and $[l_2]$ in the singular homology of $Fl_{1,n-1}$ generate $H_2(Fl_{1,n-1}, \Z)$. The \textit{small quantum $K$-ring} $QK_s(Fl_{1,n-1}) = (K(Fl_{1,n-1})) \otimes \mathbb{Q}[Q_1,Q_2], \star)$ of $Fl_{1,n-1}$ is a deformation of the ring $K(Fl_{1,n-1})$ by the Novikov variables $Q_1$ and $Q_2$ associated with the classes $[l_1]$ and $[l_2]$ \cite{lee2004quantum}. This deformation depends on the correlators $\langle \O_{i,j}, \O_{k,p}, \O_{s,t} \rangle_{d_1[l_1] + d_2 [l_2]}$ defined by (\ref{eq summary: corr X}). Considering the limit $Q_1, Q_2 \rightarrow 0$ yields the $K$-ring $K(Fl_{1,n-1})$.

We name $h_1 := X(n-1,1)$ and $h_2:=X(n,2)$ the two Schubert varieties of $Fl_{1,n-1}$ of codimension $1$. 
\begin{proposition*}[Chevalley formula in $QK_s(Fl_{1,n-1})$-cf. Proposition \ref{prop: chevalley formula QK(X)}]
	\begin{align*}
	\O_{h_1} \star \O_{k,p} &=\left \{ \begin{array}{lr} Q_1\O_{n-1,n} +Q_1Q_2 \left( [\O_X] - \O_{h_1} \right) & \mathrm{if} \: k=1, p=n\\
	Q_1 \O_{n,p} & \mathrm{if} \:k=1, \: p<n\\
	\O_{1,2} + Q_1 \left( [\O_X] - \O_{h_1} \right)  & \mathrm{if} \: k=2, p=1\\
	\O_{k-1,p} & \mathrm{if} \: k>1, \: k \neq p+1\\
	\O_{p-1,p} + \O_{p,p+1} - \O_{p-1,p+1} & \mathrm{if} \: 1< p<n-1, \: k=p+1\\
	\end{array}\right.
	\end{align*}
	\begin{align*}
	\O_{h_2} \star \O_{k,p} &=\left \{ \begin{array}{lr} Q_2\O_{1,2} +Q_1Q_2 \left( [\O_X] - \O_{h_2} \right) & \mathrm{if} \: k=1, p=n\\
	Q_2 \O_{k,1} & \mathrm{if} \:k>1, \: p=n\\
	\O_{n-1,n} + Q_2 \left( [\O_X] - \O_{h_2} \right)  & \mathrm{if} \: k=n, p=n-1\\
	\O_{k,p+1} & \mathrm{if} \: p<n, \: k \neq p+1\\
	\O_{p,p+1} + \O_{p-1,p} - \O_{p-1,p+1} & \mathrm{if} \: 1< k<n-1, \: k=p+1\\
	\end{array}\right.
	\end{align*}
\end{proposition*}
Note that for $X$ a cominuscule variety, a Chevalley formula in small quantum $K$-theory is given by Buch-Chaput-Mihalcea-Perrin \cite{buch2016chevalley}.

 Since $(\O_w)_{w \in W^P}$ is a basis of $K(Fl_{1,n-1})$, for any elements $u$ and $v$ in $W^P$, there is a unique decomposition \begin{equation*}
 \O_u \star \O_v = \sum_{d_1, \: d_2 \in \mathbb{N}} \sum_{w \in W^P} a_{u, v}^{w; d_1, d_2} Q_1^{d_1} Q_2^{d_2} \O_w.
 \end{equation*} 
 We call the polynomials $\sum_{d_1, \: d_2 \in \mathbb{N}} a_{u, v}^{w; d_1, d_2} Q_1^{d_1} Q_2^{d_2}$ \textit{Littlewood-Richardson coefficients in $QK_s(Fl_{1,n-1})$}.
 We call the ring $QK_s(Fl_{1,n-1})$ \textit{positive} if for any elements $u$ and $v$ in $W^P$, $d_1$, $d_2$ in $\mathbb{N}$ the sign of the coefficient $a_{u, v}^{w; d_1, d_2}$ is given by the following inequality. \begin{equation*}
 (-1)^{\codim X(w) - \codim X(u) - \codim X(v) + \int_{d_1 [l_1] + d_2 [l_2]} c_1(T_X)} a_{u, v}^{w; d_1, d_2} \geq 0.
 \end{equation*}
 Note that for $d_1=0=d_2$, this inequality is satisfied according to M. Brion's result on positivity in the Grothendieck groups of generalized flag varieties \cite{brion2002positivity}.
 Finally, we call an algorithm \textit{positive} if at each iteration the computation performed is a sum of coefficients having the same sign. 
  \begin{proposition*}
  	The algorithm given Subsection \ref{subsec: algo} is a positive algorithm computing Littlewood-Richardson coefficients in $QK_s(Fl_{1,n-1})$.
  \end{proposition*}
\begin{corollary*}(Cf. Proposition \ref{prop: positivity Ou star Ov})
	The small quantum $K$-theory ring $QK_s(Fl_{1,n-1})$ is positive.
\end{corollary*}

\begin{conjecture*}
	The closed formula described in Subsection \ref{subsec: conjecture LR} yields Littlewood-Richardson coefficients in $QK_s(Fl_{1,n-1})$.
\end{conjecture*}

\chapter*{Definitions and notations}
\addcontentsline{toc}{chapter}{Definitions and notations}  

\section{Space of stable maps $\overline{{\mathcal{M}}_{0,r}}(Y, \gamma)$.} \textsc{Definitions.} Let $Y$ be a complex projective smooth variety. Let $C$ be a connecetd reduced nodal curve of genus zero equipped with $r$ distinct non singular marked points.
A morphism $C \rightarrow Y$ from $C$ to the variety $Y$ is called a \textit{stable map} if each irreducible component of $C$ sent onto a point contains at least three points that are nodal or marked. Let $\gamma$ be an element in $H_2(Y, \Z)$ corresponding to an effective class of curve. We say the morphism $\varphi : C \rightarrow Y$ \textit{represents the class $\gamma$} if $\varphi_*[C]= \gamma \in H_2(Y, \Z)$.
Let $S$ be a complex scheme. A family of quasi-stable maps with $r$ marked points above $S$ is given by \begin{itemize} \item A flat and projective morphism $\pi: \mathcal{C} \rightarrow S$ equipped with $r$ sections $p_i: S \rightarrow \mathcal{C}$, such that each geometrical fiber is a reduced nodal projective connected curve of genus zero $C_s$ and such that the marked points $p_i(s)$ are distinct and non singular.
	\item A morphism $\mathcal{C} \rightarrow Y$. 
\end{itemize} The \textit{space $\overline{{\mathcal{M}}_{0,r}}(Y, \gamma)$ of genus $0$ stable maps to $Y$ representing $\gamma$}  is the coarse moduli space parametrizing families of stable maps with $r$ marked points above $S$ from genus $0$ curves to $Y$ representing the class $\gamma$, such that the geometric fiber $\mathcal{C}_s$ above any point$s$ in $S$ defines a stable map $\mathcal{C}_s \rightarrow Y$ \cite{fulton1996notes}. 
For a homogeneous variety $Y$, the space $\overline{{\mathcal{M}}_{0,r}}(Y, \gamma)$ is an irreducible rational projective variety, ofe dimension the expected dimension and with quotient singularities \cite{fulton1996notes, kim2001connectedness}. Furthermore, the subvariety ${{\mathcal{M}}_{0,r}}(Y, \gamma)$ parametrizing maps $\P^1 \rightarrow Y$ is a dense open subset of $\overline{{\mathcal{M}}_{0,r}}(Y, \gamma)$. A general point in $\overline{{\mathcal{M}}_{0,r}}(Y, \gamma)$ thus corresponds to a map $\P^1 \rightarrow Y$ representing the class $\gamma$.

Note that a  general point in $\overline{{\mathcal{M}}_{0,r}}(Y, \gamma)$ is not necessarily associated with an injective map $\P^1 \to Y$. This is immediate if $Y= \P^1$; indeed, a general point in $\overline{{\mathcal{M}}_{0,r}}(\P^1, d)$ is associated with a degree $d$ map $\P^1 \rightarrow \P^1$. This phenomenon may also be observed in higher dimensions. Consider for example the full flag variety $Y=Fl_{1,2}=\GL_3/B$ parametrizing pairs $(p,l)$ where $p$ is a point in $\P^2$ and $l$ is a line in $\P^2$ taht contains $p$. Fix a line $l$ in $\P^2$. Denote by $\gamma$ the class of the curve $C \simeq \P^1 \subset Fl_{1,2}$ containing all pairs $(p,l)$, where $p \subset l$ and by $\pi : Fl_{1,2} \rightarrow \P^2$ the map forgetting the line $l'$ in a pair $(p,l')$. Then the restriction of $\pi$ to the curve $C$ is an isomorphism onto its image, which is the line $l \subset \P^2$. Denote by $\pi': Fl_{1,2} \rightarrow \Gr(2,3)$ the map induced by forgetting the point  $p$ in a pair $(p,l)$. An irreducible curve $C'$ in $Fl_{1,2}$ of class $2\gamma$ satisfies $\pi_*[C']=2$ in $A_*(\P^2)$. Furthermore its image $\pi(C')$ in $\P^2$ is included in the line in $\P^2$ defined by the image $\pi'(C') \in \Gr(2,3)$. Then the projection $\pi$ induces a degree $2$ morphism between the curve $C'$ and its image $\pi(C')$, which contradicts the fact that the restriction of $\pi$ to $C'$ is an injection. There thus is no irreducible curve of class $2 \gamma$ in $Fl_{1,2}$, and a general point in ${{\mathcal{M}}_{0,r}}(Fl_{1,2}, 2\gamma) \hookrightarrow \overline{{\mathcal{M}}_{0,r}}(Fl_{1,2}, 2\gamma)$ corresponds to a degree $2$ map $\P^1 \rightarrow Fl_{1,2}$. 

\textsc{$G$-action on $\overline{{\mathcal{M}}_{0,r}}(Y, \gamma)$.}  Let $G$ be a complex algebraic group, $Y$ be a smooth projective $G$-variety (in this thesis $G$ will be the group $\GL_n$ of invertible matrices and $Y$ will be a flag variety parametrizing flags of vector spaces included in $\C^n$). Let $\gamma$ be an element of $H_2(Y, \Z)$ corresponding to an effective class of curve. We consider the natural transformation of functors associating to a morphism $\phi: S \rightarrow G$ and to a family of quasi-stable maps $(\pi: \mathcal{C} \rightarrow S, \mu: \mathcal{C} \rightarrow Y)$ above $S$ with $r$ marked points representing the class $\gamma$ the following family: \[\begin{tikzcd} \mathcal{C} \arrow{r}{G \cdot \mu} \arrow{d}{\pi} & Y \\ S & \end{tikzcd}\]
where we denote by $G \cdot \mu : \mathcal{C} \rightarrow Y$ the morphism $\mathcal{C} \ni c \rightarrow \Phi \circ \pi(c) \cdot \mu(c) \in Y$ translating curves $\mu(\mathcal{C}_s)$ by the element in $G$ image of $s$ by $\Phi$.
This natural transformation of functors induces an action $\overline{{\mathcal{M}}_{0,r}}(Y, \gamma) \times G \rightarrow \overline{{M}_{0,r}(Y, \gamma)}$, associating to an element $g$ in $G$ and to a map $\mu: C \rightarrow Y$ the map $$g \cdot \mu: C \ni p \rightarrow g \cdot \mu(p) \in Y.$$

\section{Flag varieties}

We recall here some classical definitions and properties of flag varieties. 
Let $n>0$. We consider a collection $I= \{i_1, \dots, i_m\}$ of $m$ non-negative integers $i_k$ satisfying $i_0:=0 < i_1 < \dots < i_m < i_{m+1} :=n$. We denote by $Fl_I$ the variety parametrizing flags of vector spaces $$V_{i_1} \subset V_{i_2} \subset \dots \subset V_{i_m} \subset  \C^n,$$
where $V_{i_k}$ is a vector subspace of $\C^n$ of dimension $i_k$. Note that the variety $Fl_I$ represents the functor $Sch \to Sets$ associating to a scheme $S$ the following set: $$\{\mathcal{F}_1 \subset \dots \subset \mathcal{F}_m \subset \O_S^{\oplus n}\},$$ where the fiber bundles over $S$ $\mathcal{F}_k$ have rank $i_k$ and the quotient sheaves $\mathcal{F}_k/ \mathcal{F}_{k-1}$ are rank $(i_k-i_{k-1})$ fiber bundles. Furthermore, we consider the complex algebraic group $G= \GL_n$, the maximal torus $T$ of $G$ given by the set of diagonal matrices and the Borel subgroup $B \subset G$  given by the set of upper triangular matrices. We denote by $W \simeq \mathfrak{S}_n$ the Weyl group of $G=\GL_n$, and $\Delta={\alpha_1, \dots, \alpha_n}$ the set of simple roots positive with respect to $B$. We denote by $\Delta(I)=\{\alpha_1, \dots \alpha_{i_1-1}, \dots,  \alpha_{i_m-1}\}$ the subset of $\Delta$ associated with $I$, and by $W_I \simeq \prod_{k=1}^{m+1} \mathfrak{S}_{i_k-i_{k-1}}$ the subgroup of $W$ generated by reflexions $s_\alpha$ associated with elements $\alpha$ in $\Delta(I)$. Finally, we denote by $P_I := B W_I B$ the parabolic subgroup associated with $\Delta(I)$. Note that the block diagonal matrix whose $k$-th block is $GL_{i_k - i_{k-1}}$ is a Levi sugroup $L_I$  associated with $P_I$. We have the following identification: $Fl_I \simeq G/P_I$.

Let $1 \leq k <m$. Composing the map forgetting all vector spaces except from the $k$-th one with the Pl\"ucker embedding $\Gr(i_k,n) \hookrightarrow \P^{\binom{n}{i_k}-1}$ induces a map $\varphi: Fl_I \rightarrow \P^{\binom{n}{ i_k}-1}$. Let $l_i$ be a curve on $Fl_I$ sent onto a line in $\P^{\binom{n}{i_k}-1}$ by $\varphi$, and such that the restriction of $\varphi$ to $l_i$ is an isomorphism. The classes $[l_i]$ generate the free abelian group $H_2(Fl_I, \Z)$, and we may thus identify an effective class of curve $\d=\sum_{k=1}^{m} d_k [l_k]$ in $E(Fl_I) \subset H_2(Fl_I, \Z)$ with $m$ positive integers $(d_1, \dots, d_{m})$. 


\section{Algebraic $K$-theory}

\subsection{Equivariant algebraic $K$-theory.} We recall here some classical definitions and properties of the equivariant $K$-theory of a variety; complete proofs can be found in \cite{chriss2009representation}, chapter 5. Let $H$ be a complex linear algebraic group, let $X$ be an irreducible variety with an $H$-action. Let $E$ be a vector bundle on $X$. We call $E$ an \textit{$H$-equivariant vector bundle} if there exists an $H$-action $\Phi: H \times E \rightarrow E$ satisfying \begin{itemize}[label={--}]
	\item The map $E \rightarrow X$ commutes with $H$-actions.
	\item For any $x$ in $X$ and $g$ in $G$, the map $\Phi(g, \bullet) : E_x \rightarrow E_{g \cdot x}$ is a linear map of vector spaces,
\end{itemize} where we denote by $E_x$ the fiber over a point $x$ in $X$. Let $\mathcal{F}$ be a coherent sheaf on $X$. Denote by $a: H \times X \rightarrow X$ the action of $H$ on $X$, by $p: H \times X \rightarrow X$ the projection map and by $p_{2\:3}: H \times H \times  X\rightarrow H \times  X$ the projection along the first factor $H$. We call $\mathcal{F}$ an \textit{$H$-equivariant coherent sheaf} if there exists an $H$-action $\Phi: H \times \mathcal{F} \rightarrow \mathcal{F}$ satisfying \begin{itemize}[label={--}]
	\item There exists an isomorphism $I : a^* \mathcal{F} \rightarrow p^*\mathcal{F}$.
	\item ${p_{2 \:3}}^* I \circ (\id_H \times a)^* I = (m \times \id_X)^* I$ of sheaves on $H \times X$.
	\item There is an identification $I_{1 \times X}=\id: \mathcal{F} = a^* \mathcal{F}_{| 1 \times X} \rightarrow p^* \mathcal{F}_{| 1 \times X} = \mathcal{F}$.
\end{itemize}Denote by $K_H(X)$ the Grothendieck group of $H$-equivariant coherent sheaves on $X$ and by $K^H(X)$ the Grothendieck group of $H$-equivariant vector bundles on $X$. The tensor product of equivariant vector bundles defines a ring structure on $K^H(X)$. Moreover, the map $K^H(X) \otimes K_H(X) \rightarrow K_H(X)$ defined by the tensor product $$[E] \otimes [\mathcal{F}] \rightarrow [E \otimes_{\mathcal{O}_X} \mathcal{F}]$$ makes $K_H(X)$ a $K^H(X)$-module. When $X$ is non singular, the map $[E] \rightarrow [E \otimes \O_X]$ yields an isomorphism $K_H(X) \simeq K^H(X)$. Indeed, any equivariant sheaf then has a finite resolution by equivariant vector bundles. \\
Furthermore, for any equivariant proper morphism of irreducible varieties with an $H$-action $f : X \rightarrow Y$, there is a pushforward map $f_* : K_H(X) \rightarrow K_H(Y)$ defined by: \begin{equation*}
[\mathcal{F}] \rightarrow \sum_i (-1)^i [R^i f_* (\mathcal{F})].
\end{equation*}
Indeed, if $\mathcal{F}$ is an $H$-equivariant sheaf on $X$, the Godement resolution implies that $R^i f_* (\mathcal{F})$ will also be an $H$-equivariant sheaf on $Y$.\\
In particular, if $X$ is an irreducible projective $H$-variety, the equivariant pushforward $K_H(X) \rightarrow K_H(point)$ is called \textit{equivariant Euler characteristic}, and denoted by $\chi_H$.\\
Moreover, an equivariant morphism $f: X \rightarrow Y$ induces a pullback map $f^*: K^H(Y) \rightarrow K^H(X)$: \begin{equation*}
[E] \rightarrow [f^* (E)].
\end{equation*}
Consider an equivariant proper morphism $f: X \rightarrow Y$ of irreducible varieties with an $H$-action. Let $E$ be an equivariant vector bundle on $Y$, and $\mathcal{F}$ be an equivariant coherent sheaf on $X$. 
Note that the following projection formula is satisfied in $K_H(Y)$: \begin{equation*}
f_*(f^*[E] \otimes [\mathcal{F}]) = [E] \otimes f_*[\mathcal{F}].
\end{equation*}
\subsection{Algebraic K-theory of $G/P$.}\label{sec K theory} We recall here some classical definitions and properties of the Grothendieck ring of coherent sheaves on a variety. 
Denote by $K_\circ (X)$ the Grothendieck group of coherent sheaves on $X$ and by $K^\circ(X)$ the Grothendieck group of vector bundles on $X$. Note that for $H=\id$, the equivariant groups $K_H(X)$ and $K^H(X)$ are equal to $K_\circ(X)$ and $K^\circ(X)$. In particular, the tensor product of vector bundles defines a ring structure on $K^\circ(X)$, and for a non singular variety $X$ we can identify $K(X) = K_\circ(X) \simeq K^\circ(X)$. Furthermore, for any coherent sheaves $\mathcal{F}$ and $\mathcal{G}$ on $X$ we have
\[[\mathcal{F}] \cdot [\mathcal{G}] = \sum_{i=0}^{\dim X} (-1)^i Tor_i^X(\mathcal{F}, \mathcal{G}).\]
Furthermore, note that for an irreducible projective variety $X$, the pushforward $K_\circ(X) \to K_\circ(\Spec \C)$ is given by: \begin{align*} \chi: K_\circ(X) \rightarrow \Z, && [\mathcal{F}] \rightarrow \chi(\mathcal{F}) = \sum_{i} (-1)^i h^i(\mathcal{F}),\end{align*}
where $h^i$ denote the dimension of the $i$-th cohomology group of $\mathcal{F}$; i.e. $\chi$ is the {Euler-Poincar\'e characteristic} of $\mathcal{F}$.
A morphism of schemes $f : X \rightarrow Y$ is called \textit{perfect} if there is an $N$ such that 
\[Tor_i^Y(\O_X, \mathcal{F}) =0\]
for all $i > N$ and for all coherent sheaves $\mathcal{F}$ on $Y$. For a perfect morphism $f : X \rightarrow Y$, one can define the pull-back
\[f^*: K_\circ(Y ) \rightarrow K_\circ(X)\] by
\[f^* [\mathcal{F}] = \sum_{i =0}^N [Tor_i^Y(\O_X, \mathcal{F})].\]
Furthermore suppose $f$ is a morphism locally of finite type. Then if $f$ is flat or $Y$ is regular, $f$ is perfect-cf. for example Stacks Project $37.53$. 
We will also be using the following easy results.
\begin{lemme}\label{lem: equality homotetie}
	Let $Y$ be an irreducible projective variety and $\Phi$ be an irreducible projective subvariety of $\mathbb{P}^1 \times Y$. Denote by $\pi: \Phi \rightarrow \mathbb{P}^1$ the natural projection. Let $p_1$, $p_2$ be two points on $\mathbb{P}^1$. Then we have the following equality in $K_\circ(Y)$: \begin{equation*}
	[\mathcal{O}_{\pi^{-1}(p_1)}] = [\mathcal{O}_{\pi^{-1}(p_2)}] 
	\end{equation*}
\end{lemme}
\begin{proof}
	Since the projection map $\pi$ is flat, we have $\pi^*[\O_{p_i}]=[\O_{\pi^{-1}(p_i)}]$ in $K_\circ(\Phi)$. Denote by $\Pi:  \Phi \rightarrow Y$ the natural projection. We obtain $\Pi_*\pi^*[\O_{p_i}]= \Pi_* [\O_{\pi^{-1}(p_i)}]=[\O_{\pi^{-1}(p_i)}]$ in $K_\circ(Y)$. Finally, the equality $[\O_{p_1}]=[\O_{p_2}]$ in $K_\circ(\P^1)$ yields the result. 
\end{proof}
\begin{lemme}\label{lem: equality g.}
	Let $X$ be a projective $\GL_n$-variety. Consider a subvariety $Y$ of $X$. Then for all $g$ in $\GL_n$, we have the following equality in $K_\circ(X)$: $$[\mathcal{O}_{g \cdot Y}] = [\mathcal{O}_Y].$$
\end{lemme}
\begin{proof}
	Let $g$ in $\GL_n$. Since $\GL_n$ is a dense open subset of $\mathcal{M}_{n,n} \simeq \mathbb{A}^{n^2}$, there exits an irreducible rational curve $f: \P^1 \supset U \rightarrow \GL_n$ sending points $p_1$ and $p_2$ on $\P^1$ onto the points $g$ and $\id$. Denote by $\Phi$ the irreducible projective subvariety of $\P^1 \times X$ obtained as the adherence of the image of the map $U \times Y \rightarrow U \times Y \subset \P^1 \times X$ sending $(x,y)$ to $(x,f(x) \cdot y)$. Denote by $\pi: \Phi \rightarrow \P^1$ the first projection. Then according to Lemma~\ref{lem: equality homotetie} $$[\O_{\pi^{-1}(p_1)}] = [\O_{\pi^{-1}(p_2)}] \: \mathrm{hence} \: [\mathcal{O}_{g \cdot Y}] = [\mathcal{O}_Y].$$
\end{proof}

\section{Rational singularities.}\label{sec : rational sing}

We recall here known definitions and results on rational singularities. The treatment in this section is inspired by \cite{brion2002positivity}. An irreducible complex variety $X$ has \textit{rational singularities} if there exists a desingularization $\psi : \widetilde{X} \rightarrow X$ such that: \begin{equation}\label{eq: def RS} \psi_* \O_{\widetilde{X}} = \O_X \: \mathrm{and} \:\: \forall i>0, \: R^i \pi_* \O_{\widetilde{X}} =0. \end{equation}
Note that, according to Zariski's main theorem, the equality $\psi_* \O_{\widetilde{X}} = \O_X$  is equivalent to the normality of $X$.
The relevance of this notion here comes from the following results. 
\begin{theorem}\label{th: Viehweg X//G has RS}(Viehweg \cite{viehweg1977rational})
	The quotient of a smooth complex scheme by a finite group has rational singularities.
\end{theorem}
Note that, since for $X$ convex the moduli spaces $\overline{{\mathcal{M}}_{0,r}}(X, \d)$ have quotient singularities \cite{fulton1996notes}, this implies that for $X$ a generalized flag variety the moduli spaces $\overline{{\mathcal{M}}_{0,r}}(X, \d)$ have rational singularities.

Let $H$ be a  complex linear algebraic group. Recall a variety $X$ is called \textit{unirational} if there exists a dominant map $\P^N \dashrightarrow X$, and \textit{rationally connected} if there exists an irreducible rational curve joining two general points in $X$. Note that since $\P^N$ is rationally connected, a unirational variety is rationally connected.
\begin{theorem}\label{th: BM f_* if fiber rational}(Buch-Mihalcea \cite{buch2011}, Theorem 3.1.)
	Let $f : X \rightarrow Y$ be a surjective equivariant morphism of projective varieties with rational singularities equipped with an $H$-action. Assume that the general fiber of $f$ is rationally connected. Then $f_*[\O_X]=[\O_Y] \in K_H(Y)$.
\end{theorem}

\begin{proof}
	Buch and Mihalcea prove in \cite{buch2011} that this statement is satisfied if a desingularization $\psi : \widetilde{X_y} \rightarrow X_y$ of the fiber of $f$ over $y$ general in $Y$ satisfying $H^i(\widetilde{X_y}, \O_{\widetilde{X_y}} )=0$ for all $i >0$. Since $X_y$ is rationally connected, and $\widetilde{X_y}$ is birational to $X_y$, $\widetilde{X_y}$ is also rationally connected. Hence $\widetilde{X_y}$ is a smooth rationally connected variety, hence according to \cite{debarre2013higher}, $\widetilde{X_y}$ satisfies $H^i(\widetilde{X_y}, \O_{\widetilde{X_y}} )=0$ for all $i >0$.
\end{proof}

Note that for $X$ a full flag variety, Brion gives a nice proof of the rational singularities of Schubert varieties in \cite{brion}.
\begin{theorem}\label{th: schubert var have rational sing} (Ramanathan, \cite{ramanathan})
	Schubert varieties have rational singularities.
\end{theorem}

\chapter{Gromov-Witten varieties of partial flag varieties}\label{chap : geometry W}

\section{Introduction}

Removing subspaces from a partial flag gives another partial flag composed of fewer subspaces. This induces a forgetful map $\pi : X \rightarrow X'$ between the corresponding flag varieties.
The main result of this chapter is that, for a degree large enough, the variety associated with rational curves of a given degree in $X$ having their marked points within Schubert varieties $X_i$ of $X$ is a rationally connected fibration over its image, which is the variety associated with rational curves of a given degree in $X$ having their marked points within Schubert varieties $\pi(X_i)$ of $X'$. These subvarieties of the space of genus $0$ stable maps to a homogeneous space are called Gromov-Witten varieties. The Euler characteristics of the structure sheaf of these varieties are quantum K-theoretical invariants \cite{givental}. When these varieties are zero dimensional, their number of points are Gromov-Witten invariants. Our result in particular implies the equality between (non equivariant) correlators proved in Chapter \ref{chap : stabilization correlators}, providing an alternative proof-cf. chapter \ref{chap : stabilization correlators} Section \ref{sec: geometric proof}.

\subsection{Definitions.} Let $n >0$. Let $I = \{i_1, \dots, i_m\}$ be a collection of integers satisfying $0<i_1<i_2< \dots <i_m<n$. Denote by~$Fl_I$ the flag variety parametrizing $m$-tuples of vector spaces ordered by inclusion \begin{equation}\label{eq: def flag I}  V_{i_1} \subset V_{i_2} \dots  \subset V_{i_{m}} \subset  \mathbb{C}^n\end{equation} such that ${\rm dim} V_{i_k}=i_k$.
Let $B$ be  a Borel subgroup of $\GL_n$. Note that $Fl_I$ is the homogeneous variety $Fl_I \simeq \GL_n/P_I$, where $P_I$ is the parabolic subgroup of $\GL_n$ satisfying $B \subset P_I$ with associated reductive subgroup $L_I \simeq \prod_{p=1}^{m+1} \GL(i_p-i_{p-1})$, where we set $i_0=0$ and $i_{m+1}=n$. Let $J= \{j_{1}, \dots, j_M\}$ be a collection of integers satisfying $0 < j_1 <\dots < j_M <n$ containing $I$, i.e. for each $k$ there exists an element $s$ such that $j_s=i_k$. Consider an integer $p$ such that $1 \leq p \leq M$. We denote by $J_{\geq p}$ the collection $J_{\geq p} := \{i_1, \dots, i_{p'}, j_{p}, \dots, j_{M} \}$, where $p'$ is the largest integer such that $i_{p'} <j_p$. Forgetting all vector spaces included in $V_{j_{p}}$ except those of the form $V_{i_k}$ yields a forgetful morphism $$\pi_{I/J_{\geq p}} : Fl_{J_{\geq p}} \rightarrow Fl_I.$$ In the same way, for $1 \leq p \leq M$, we name $J_{\leq p} := \{j_{1}, \dots, j_{p}, i_{p'+1} \dots, i_m \}$, where $p'$ is the smallest integer such that $j_p < i_{p'}$. We name $\pi_{I/J_{\leq p}} : Fl_{J_{\leq p}} \rightarrow Fl_I$ the  morphism induced by forgetting all vector spaces containing $V_{ j_{p}}$ except those of the form $V_{i_k}$.

Let $W=N_{\GL_n}(T)/T \simeq \S_n$ be the Weyl group of the algebraic group $\GL_n$. We denote by $Fl_J(u) = \overline{BuP_J/P_J}$ the Schubert variety of $Fl_J$ associated with $u \in W$. 
Note that the image of a Schubert variety $Fl_J(u)$ by the projection $\pi_{I/J}$ is the Schubert variety $\pi_{I/J} (Fl_J(u)) = Fl_{I}(u)$.

Let $r>0$. The group  $H_2(Fl_I, \Z)$ is a free abelian group generated by classes of Schubert varieties of dimension $1$. Denote by $E(Fl_I) \subset H_2(Fl_I, \mathbb{Z})$ the semi-group of effective classes of curves. Note that $H_2(\Gr(k,n), \Z) \simeq \Z$, and the class of a curve on the Grassmannian $\Gr(k,n)$ is given by the degree of its image in $\P^{\binom{n}{k}}$ by the Pl\"ucker embedding. The class $\mathbf{d}=(d_1, \dots, d_m) \in E(Fl_I)$ of a curve $C$ in $Fl_I$ is determined by $m$ non-negative integers $d_1$, $d_2$,$\dots$, $d_m$, where $d_k$ is the degree of the projection of $C$ to the Grassmannian ${\rm Gr}( i_k,n)$. 
For a class $\d$ in $E(Fl_I)$, we denote by $\overline{{\mathcal{M}}_{0,r}}(Fl_I, \d)$ the moduli space parametrizing  stable maps $C \rightarrow Fl_I$ from genus $0$ curves $C$ with $r$ marked points $x_i \in C$ to $Fl_I$ representing the class $\d$. We name $ev: \overline{{\mathcal{M}}_{0,r}}(Fl_I, \d) \rightarrow (Fl_I)^r$ the evaluation morphism sending the point $p$ associated with a map $f_p: C\rightarrow Fl_I$ to the image $\prod_{i=1}^r f_p(x_I)$ of the $r$ marked points $x_i \in C$. For any class $\d$ in $E(Fl_J )\simeq \mathbb{N}^{M}$, the morphism $\pi_{I/J}: Fl_J \rightarrow Fl_I$ induces a morphism \begin{equation}\label{eqintro: def Pixev}\Pi_{I/J} \times ev: \overline{{\mathcal{M}}_{0,r}}(Fl_J, \d) \rightarrow \overline{{\mathcal{M}}_{0,r}}(Fl_I, (\pi_{I/J})_* \d) \times_{(Fl_I)^r} (Fl_J)^r.\end{equation}
Let $1 \leq k \leq M$ such that the integer $j_k$ is not in $I$. We denote by $$\Pi_{J_{\geq k}} \times ev: \overline{\mathcal{M}_{0,r}}(Fl_{J_{\geq k}}, (\pi_{I/J_{\geq k}})_* \d) \rightarrow \overline{\mathcal{M}_{0,r}}(Fl_{J_{\geq k+1}}, (\pi_{I/J_{\geq k+1}})_* \d) \times_{(Fl_{J_{\geq k+1}})^r} (Fl_{J_{\geq k}})^r$$ the morphism induced by forgetting the vector space $V_{j_k}$  and by evaluating the $r$ marked points. Let $1 \leq \mu \leq k \leq M$ such that the element $j_k$ is not in $I$. We denote by $J_{ \mu \leq k }$ the collection $\{i_1, \dots, i_p, j_{\mu}, \dots, j_k, i_{p'}, \dots i_m\}$, where $p$ is the largest element such that $i_p < j_\mu$ and $p'$ is the smallest element such that $i_{p'} >j_k$. We denote by $${\Pi}_{J_{\mu \leq k}} \times ev: \overline{\mathcal{M}_{0,r}}(Fl_{J_{\mu \leq k}}, (\pi_{J_{\mu \leq k}/J})_* \d) \rightarrow \overline{\mathcal{M}_{0,r}}(Fl_{J_{\mu \leq k -1}}, (\pi_{J_{\mu \leq k-1}/J})_* \d) \times_{(Fl_{J_{\mu \leq k-1}})^r} (Fl_{J_{\mu \leq k}})^r$$ the morphism induced by forgetting the vector space $V_{j_{k}}$ from ${Fl}_{\mu \leq k}$ and evaluating the $r$ marked points.

\begin{defn}\label{def: stabilized collection}
	Consider a collection $\{I, J, \d\}$ of a set $I=\{i_1, \dots, i_m\}$ of $m$ positive integers satisfying $$0 < i_1 < \dots < i_m <n,$$  a set $J=\{j_{1}, \dots, j_M\}$ of $M$ positive integers satisfying $$0 < j_1 < \dots < j_M < n$$ such that $I$ is contained in $J$, and a set $\d=(d_{1}, \dots, d_{M})$ in $\mathbb{N}^{M} \simeq E(Fl_J)$. We denote by $\d'=(d'_1, \dots, d'_m) \in \mathbb{N}^m \simeq E(Fl_I)$ the pushforward of $\d$ by the forgetful map $\pi_{I/J}$. Let $1 \leq \mu \leq M$ such that $j_\mu$ is contained in $I$. We call the collection $\{I, J, \d\}$ \textit{stabilized with respect to $\mu$} if  the following conditions are satisfied. First, for all elements $1 \leq k < \mu$ such that the integer $j_k$ is not contained in $I$, the two following conditions are satisfied. \begin{enumerate}[label=$\bullet$]
		\item \begin{equation}\label{eq: Geometry surjectivity Pi left} \mathrm{The} \: \mathrm{morphism} \: \Pi_{J_{\geq k}} \times ev \: \mathrm{is} \: \mathrm{surjective.} \end{equation}
		\item  Let $k'$ be the largest integer such that $i_{k'} < j_k$. For any $1 \leq p <k'$ we have \begin{equation}\label{eq: Geometry dk left} d_{k} \geq \ceil*{ \left(d'_{p}-d'_{p-1}\right) /(i_p - i_{p-1})}  j_k. \end{equation}
	\end{enumerate}
	Futhermore, for all $\mu \leq k \leq M$ such that the integer $j_k$ is not contained in $I$ the two following symmetric conditions are satisfied. \begin{enumerate}[label=$\bullet$]
		\item The morphism ${\Pi}_{J_{ \mu \leq k}} \times ev$ is surjective.
		\item Let $k'$ be the smallest integer such that $i_{k'} > j_k$. For $k'<p\leq m$ we have $$d_{k} \geq  \ceil*{\left(d'_{p}-d'_{p+1}\right)/(i_{p+1}-i_p)} (n-j_k).$$
	\end{enumerate}
	A collection $\{I,J,\d\}$ will be called \textit{stabilized} if there exists an integer $\mu$ such that the collection $\{I,J,\d\}$ is stabilized with respect to $\mu$. A collection $\{I,J,\d\}$ will be called \textit{stabilized with respect to $M$} if for all $1 \leq k \leq M$ such that $j_k$ is not contained in $I$ the conditions (\ref{eq: Geometry surjectivity Pi left}) and (\ref{eq: Geometry dk left}) are satisfied.
\end{defn}
Observe that when $\mu=1$, the conditions (\ref{eq: Geometry surjectivity Pi left}) and (\ref{eq: Geometry dk left}) are empty. For simplicity we have set $d'_{0}=d'_{m+1}=0=i_0$ and $i_{m+1}=n$. We have denoted by $\ceil{q}$ the ceiling of $q \in \mathbb{Q}$.  
We will see in Section \ref{sec: conditions surjectivity} some conditions to ensure that a collection is stabilized. The following interesting cases yield examples of stabilized  collections (cf. Subsection \ref{subsec: Geometry ex stabilized collections}).  \begin{enumerate}
	\item Suppose $I=\{i_1, \dots, i_m\}$  is a collection of integers satisfying $1 < i_1 < \dots < i_m <n$ and $J= \{j_{1},  i_1, \dots, i_m\}$ satisfies $0 < j_1  < i_1$.  Note that then $Fl_J \rightarrow Fl_I$ is the morphism forgetting the first vector space. Then if $d_{1} \geq j_{1} \left( r +1+\floor{d_{2}/i_1} \right)$ the collection $(I,J, (d_{1}, d_2, \dots d_{m+1}))$ is stabilized.
	\item Suppose $J=\{k\}$ and $I=\emptyset$; then $Fl_J$ is the Grassmannian $\Gr(k,n)$ and $Fl_I$ is the point $\Spec(\C)$. Then if $d \geq r k$ or $d \geq r(n-k)$ the collection $(I,J, d)$ is stabilized.	
\end{enumerate}
Note that the dependency on the integer $r$ in the examples here above comes from the condition asking the surjectivity of the following morphisms. $${\Pi}_{J_{\mu \leq k}} \times ev: \overline{\mathcal{M}_{0,r}}(Fl_{J_{\mu \leq k}}, (\pi_{J_{\mu \leq k}/J})_* \d) \rightarrow \overline{\mathcal{M}_{0,r}}(Fl_{J_{\mu \leq k -1}}, (\pi_{J_{\mu \leq k-1}/J})_* \d) \times_{(Fl_{J_{\mu \leq k-1}})^r} (Fl_{J_{\mu \leq k}})^r$$ 

\begin{defn}\label{def: W variety} Let $\mathbf{d}$ be an element in $E(Fl_I)$, let $u_1$, $\dots$, $u_r$ be elements in $W$. We call \textit{Gromov-Witten variety of degree $\mathbf{d}$ associated with $u_1$, $\dots$, $u_r$} the subvariety \begin{equation*}
	\W_{Fl_I; u_1, \dots, u_r}^{g, \mathbf{d}} :=  ev_1^{-1}(g_1 \cdot Fl_I(u_1))\cap \dots \cap ev_r^{-1}(g_r \cdot Fl_I(u_r))
	\end{equation*} of $\overline{\mathcal{M}_{0,r}}(Fl_I, \mathbf{d})$. 
\end{defn}

Note that for $\mathbf{d}$ in $E(Fl_I)$, $u_1$, $\dots$, $u_r$ in $W$, $\W_{Fl_I; u_1, \dots, u_r}^{g, \mathbf{d}}$ is the variety associated with degree $\d$ rational curves on $X$ having their $i$-th marked point within $g_i \cdot Fl_I(u_i)$. Denote by $1$ the identity in $W \simeq \mathfrak{S}_n$; note that $Fl_I(1) := P_I/P_I$ a point in $Fl_I$. Furthermore, note that $\W_{Fl_I; 1, \dots, 1}^{g, \mathbf{d}}$ is the variety associated with degree $\d$ rational curves on $X$ whose $i$-th marked point is the point $g_i P_I/P_I \in Fl_I$.

\begin{Rem}\label{rem: reducibility W} According to the proof of Kleiman's transversality theorem \cite{kleiman1974transversality} for a general $g \in \GL_n^r$ the fiber product $overline{\mathcal{M}_{0,r}}(Fl_I, \mathbf{d}) \underset{(Fl_I)^r}\times (g_1 \cdot Fl_I(u_1) \times \dots \times g_r \cdot Fl_I(u_r))$ is a reduced equidimensional scheme. Hence for $g$ general in $\GL_n^r$, the natural map
	$$\overline{\mathcal{M}_{0,r}}(Fl_I, \mathbf{d}) \underset{(Fl_I)^r}\times (g_1 \cdot Fl_I(u_1) \times \dots \times g_r \cdot Fl_I(u_r)) \to \overline{\mathcal{M}_{0,r}}(Fl_I, \mathbf{d}) \supset \W_{Fl_I; u_1, \dots, u_r}^{g, \mathbf{d}}$$ is a closed immersion surjecting into $\W_{Fl_I; u_1, \dots, u_r}^{g, \mathbf{d}}$, and we can identify the two schemes.   \end{Rem}
\subsection{Main results.} 
%

\begin{defn}\label{def: RC fibration} Let $Y$, $Z$ be varieties. We call $Y$ a \textit{rationally connected fibration over $Z$} (respectively \textit{unirational fibration}) if there exists a morphism $Y \rightarrow Z$ such that each irreducible component of $Y$ dominates a different irreducible component of $Z$, $Y$ dominates $Z$, and the general fiber $Y \to Z$ is a rationally connected (respectively unirational) irreducible variety. 
	We call $Y$ a \textit{tower of unirational fibrations over $Z$} if there exists a sequence of morphisms $$Y_1=Y \rightarrow Y_2 \rightarrow \dots \rightarrow Y_n=Z$$  such that, for all $1 \leq i <n$, $Y_i$  is a unirational fibration over $Y_{i+1}$. \end{defn}

\begin{theorem}\label{th : RC fibrations bewteen W varieties}
	Consider a set $J=\{j_1, \dots, j_M\}$ of $M$ integers satisfying $0 < j_1 < \dots < j_M <n$, a set $I= \{i_1, \dots, i_m\}$ of $m$ integers satisfying $0 < i_1 < \dots < i_m <n$ and a degree $\mathbf{d}=(d_{1}, \dots, d_{M}) \in E(Fl_J) \simeq \mathbb{N}^M$ such that the collection $\{I,J,\d\}$ is stabilized in the sense of Definition \ref{def: stabilized collection}.	
	\begin{enumerate}[label= \roman*)]
		\item For $x$ general in $(Fl_J)^r$, the fiber $ev^{-1}(x) \subset \overline{\mathcal{M}_{0,r}}(Fl_J, \d)$ is a tower of unirational fibrations over the fiber $ev^{-1}(\pi_{I/J}(x)) \subset \overline{{\mathcal{M}}_{0,r}}(Fl_I, (\pi_{I/J})_* \mathbf{d})$.
		\item Let $u_1$, $\dots$, $u_r$ be elements in $W$. For $g$ general in $\GL_n^r$ the Gromov-Witten variety $\W_{Fl_J; u_1, \dots, u_r}^{g, \mathbf{d}}$  is a rationally connected fibration over the Gromov-Witten variety $\W_{Fl_I; u_1, \dots, u_r}^{g, (\pi_{I/J})_* \mathbf{d}}$.
	\end{enumerate}
\end{theorem}

Let $J$ be a subsrt of $I$. Denote by $W_{J} := N_{L_J}(T)/T$ the Weyl group associated with $P_J$. Denote by $$p: W/W_{J} \rightarrow W/W_{I}$$ the natural projection. For $u$ in $W/W_{I}$, the inverse image $p^{-1}(u)$ is the set of elements $u \cdot s$, where $s$ is an element in $W_I / W_{J}$. Note that the Schubert variety $FL_I(u) := \overline{BuP_I/P_I}$ is independent of the choice of a representative in $W$ of the element $u$.

\begin{corollaire}
	Let $\d =(d_1, \dots, d_m)$ be a degree in $E(Fl_I) \simeq \mathbb{N}^{m}$ and $u_1$, $\dots$, $u_r$ be elements in $W$ such that the following Gromov-Witten invariant is non-zero: $$n_{u_1, \dots, u_r}^{{I}, \d} := \int_{\overline{{\mathcal{M}}_{0,r}}(Fl_{I}, \mathbf{d})} ev_1^*[Fl_I(u_1)] \cup \dots \cup ev_r^*[Fl_I(u_r)] \neq 0.$$
	Let $J$ and $I$ be sets of increasing positive integers and $\d$ be a degree in $E(Fl_J)$ such that the collection $\{I,J,\d\}$ is stabilized in the sense of Definition \ref{def: stabilized collection}. Then for any elements $v_i \in p^{-1}(u_i)$, for $g$ general in $GL_n^r$ the Gromov-Witten variety $\mathcal{W}_{Fl_J; v_1, \dots, v_r}^{g, \d}$ has $n_{u_1, \dots, u_r}^{{I}, \d}$ connected components, which are irreducible rationally connected varieties.
\end{corollaire}
\subsection{Previous research.} Gromov-Witten varieties of $\P^1$ and $\P^2$ that are not rationally connected varieties have been found in \cite{iritani2014reconstruction}. For $\P^2$, the locus of $\overline{\mathcal{M}_{0,3d-2}}(\mathbb{P}^2,d)$ corresponding to curves of degree $d$ going through $3d-2$ points is a curve whose genus depends on the degree $d$, of positive genus as soon as $d>2$ \cite{pandharipande1997canonical}. A recursive procedure computing the geometric genus of $1$-dimensional Gromov-Witten varieties of $\mathbb{P}^r$ was also given in \cite{ran2001variety}. High degree Gromov-Witten varieties have been shown to be rational varieties for Grassmannians by Buch and Milhacea \cite{buch2011}, and more generally for cominuscule varieties by Chaput and Perrin \cite{chaput2011rationality}. \\

This chapter is organized as follows.
We first prove in Section \ref{subsec : correspondence irreducible components} that, when the general fiber of the morphism $\Pi_{I/J} \times ev$ defined by (\ref{eqintro: def Pixev}) is rationally connected, Gromov-Witten varieties of $Fl_I$ are rationally connected fibrations over Gromov-Witten varieties of $Fl_J$ obtained by projection. We then recall in Section \ref{sec: balanced flags over P^1} results on flags of vector bundles over $\P^1$. Section \ref{sec : geometry general fiber} is dedicated to the study of the general fiber of the map $\Pi_{I/J} \times ev$. We observe that when the flag variety $Fl_I$ is obtained from $Fl_J$ by forgetting one vector space, the general fiber of this map is a unirational variety. This result is central in the derivation of Theorem \ref{th : RC fibrations bewteen W varieties} presented in Section \ref{sec: proof main theorem}. Finally, we study the surjectivity of the morphism $\Pi_{I/J} \times ev$ in Section \ref{sec: conditions surjectivity}.

\section*{Conventions}
Schemes considered in the subsections \ref{sec : stability composition} and \ref{sec: image general subvariety} will be defined over an algebraically closed field. Schemes considered in all other Sections will be defined over $\mathbb{C}$. 

We say a property (P) holds for a \textit{general} point in a variety $X$ if (P) is true for points belonging to a dense open subset of $X$. We say the general fiber of a morphism $f: X \rightarrow Y$ satisfies $(P)$ if $f^{-1}(x)$ satisfies $(P)$ for points $x$ in a dense open subset of $Y$.
Let $G$ be an integral algebraic group scheme, let $Y$ be an integral scheme endowed with a transitive $G$-action, let $Z$ be a subscheme of $Y$, and $M \rightarrow Y$ be a morphism of integral schemes. We say a property (P) holds for a subscheme $M \times_Y g \cdot Z$ of $M$ in \textit{general position} in $M$ if (P) holds for element $g$ of $G$ belonging to a dense open subset of $G$. 

We consider here that a variety is not necessarily irreducible; i.e. a {variety} is a reduced, finite type scheme defined over an algebraically closed field.

\section{Restriction to a general subvariety} \label{subsec : correspondence irreducible components}

Let $X\simeq G/P$ and $X' \simeq G/P'$ be (generalized) flag varieties associated with parabolic subgroups $P \subset P'$ of a semi-simple linear algebraic group $G$. We observe here that some properties of Gromov-Witten varieties of $X$ and $X'$ can be deduced from properties of the general fiber of the morphism $\Pi \times ev: \overline{{\mathcal{M}}_{0,r}}(X, \d) \rightarrow \overline{{\mathcal{M}}_{0,r}}(X', \pi_*\d) \times_{(X')^r} X^r$-cf. Proposition \ref{prop : W is irreducible}. This is one of the key ingredients in the derivation of Theorem \ref{th: ex stability under composition} we provide in Section \ref{sec: proof main theorem}.

More precisely, the property we prove here is the following. Let $G$ be a semi-simple algebraic group, let $Z'$ be a projective variety with a transitive $G$-action, let $Z$ be a projective variety over $Z'$ such that the morphism $Z \rightarrow Z'$ is $G$-equivariant. Let $\Pi : M \rightarrow M'$ be a morphism of projective varieties over $Z'$. Let $M \to Z$ be a morphism of projective varieties, let $X$ be a subvariety of $Z$. Then for $g$ general in $G$ the irreducibility (respectively rational connectedness) of the general fiber of the restriction map ${\Pi}_{| g X}: M_{g \cdot X} \to M'_{g \cdot X}$induced by $\Pi$ is implied by the irreducibility (respectively rational connectedness) of the general fiber of the morphism $M \rightarrow M' \times_{Z'} Z$- cf. Proposition \ref{prop : image is equidimensional}. 
\subsection{Rationally connected, unirational and rational schemes.} We call a scheme $X$ \textit{rationally connected} if there is an irreducible rational curve joining two general points on $X$. Note that  a rationally connected scheme is irreducible. By contradiction suppose a scheme $Y$ has several irreducible components $Y_i$ and is rationally connected. We can then find two points $y_1$ and $y_2$ that each belongs to an irreducible component $Y_i$ and is not contained in any other irreducible component of $Y$ such that there is an irreducible curve $C$ joining the points $y_1$ and $y_2$. Then the irreducible curve $C$ belongs to both $Y_1$ and $Y_2$ and is not contained in the intersection of the two irreducible components, which is a contradiction.

We call a variety $Y$ \textit{rational} if there exists $N>0$ such that $Y$ is birational to $\P^N$. A variety $Y$ is called \textit{unirational} if there exists a dominant rational map $\P^N \dashrightarrow Y$. Note that since $\P^N$ is rationally connected, rational and unirational varieties are rationally connected. 

\subsection{Composition of rationally connected fibrations.}\label{sec : stability composition} 
We call a  a morphism of schemes $f: Y \to S$ a rationally connected fibration if \begin{itemize}
	\item  Each irreducible component of $Y$ dominates a different irreducible component of $S$;
	\item $Y$ dominates $S$;
	\item The general fiber of $f$ is rationally connected.
\end{itemize}
We call a scheme $Y$ a \textit{rationally connected fibration over a scheme $S$} if there exists a rationally connected fibration $Y \to S$. 

\begin{defn}\label{def: stability composition}
	Let $\mathcal{C}$ be a subcategory of the category of schemes $\mathrm{Sch}$. We call a property $(P)$ \textit{stable under composition in $\mathcal{C}$} if for every morphisms $Y \xrightarrow{f} S$ and $S' \xrightarrow{f'} S$ in $\mathcal{C}$ satisfying $(P)$, the composition $f' \circ f: Y \rightarrow S$ also satisfies $(P)$. 
\end{defn}

Types of morphisms of schemes stable under composition are well known; for example flat, smooth or etale morphisms are stable under composition-cf. for example Stacks Project $62.4$. 
We focus here on composing rationally connected fibrations. Grabber-Harris-Starr's Theorem on rational connectedness implies that the composition of rationally connected fibrations is itself a rationally connected fibration-cf. Theorem \ref{th: ex stability under composition}. This property will be a key ingredient in the derivation of Theorem \ref{th : RC fibrations bewteen W varieties} proposed in Section \ref{sec: proof main theorem}.

\begin{theorem}\label{th: ex stability under composition}
	Let $\mathcal{C}$ be the category of irreducible projective schemes.
	\begin{enumerate}[label={(\arabic*)}]
		\item The property \textsc{"the general fiber is Irreducible"}  is stable under composition in $\mathcal{C}$.
		\item  Let $(P)$ be a property such that, for all morphisms $Y \rightarrow S$ in $\mathcal{C}$, the following statement is verified:
		\begin{center}
			if $S$ and the general fiber of the morphism $Y \rightarrow S$ satisfy $(P)$, then $Y$ satisfies $(P)$.\end{center}
		Then the property \textsc{"the general fiber satisfies $(P)$ and is irreducible"} is stable under composition in $\mathcal{C}$.
		\item The property \textsc{"the general fiber is rationally connected"} is stable under composition in the category of irreducible complex projective schemes.
		\item Let $Y$, $Z$ and $S$ be projective complex varieties such that $Y$ is a rationally connected fibration over $Z$ and $Z$ is a rationally connected fibration over $S$. Then $Y$ is a rationally connected fibration over $S$.
	\end{enumerate}
\end{theorem}

\begin{lemme}\label{lem : Intersection dense open is dense}
	Let $f : Y \rightarrow Z$ be a dominant morphism of schemes, of finite presentation, where $Y$ is an equidimensional noetherian scheme. Let $V$ be a dense open subset of $Y$. Then the intersection of the general fiber of $f$ with $V$ is dense in the fiber.
\end{lemme}

\begin{proof}
	First, since the question is local on $Z$, we can suppose $Z$ irreducible. Since $Y$ is a noetherian scheme, $Y \setminus V$ can be written as the union of a finite number of schemes, i.e. $Y \setminus V = \cup_{i=1}^N F_i$, where the $F_i$ are closed subschemes of $Y$, of dimension smaller than the pure dimension of $Y$. Denote by $f_{|F_i}$ the restriction of $f$ to one of the boundary subschemes $F_i$. If $f_{|F_i}$ does not dominate $Z$, then the preimage under $f_{|F_i}$ of a general point in $Z$ is empty. On the other side, if $f_{|F_i}$ dominates $Z$, then the preimage under $f_{|F_i}$ of a general point in $Z$ is a scheme of pure dimension smaller than $\mathrm{dim} Y - \mathrm{dim} Z$. Indeed, since $f$ is a morphism of finite presentation, Chevalley's constructibility theorem implies $f(F_i)$ contains an open subset of $Z$; hence the subset of points $x$ in $f(F_i)$ such that the fiber $f_{|F_i}^{-1}(x)$ is of expected dimension  is an open subset $V$ of $Z$, verifying $\dim f_{|F_i}^{-1}(x) = \dim F_i - \dim Z$ for all points $x$ in $V$. Thus, since $Y$ is equidimensional, the general fiber $Y_0$ of $f$ is a scheme of pure dimension $\mathrm{dim} Y - \mathrm{dim} Z$, whose intersection with $F_i$ is a scheme of dimension smaller than $\mathrm{dim} Y - \mathrm{dim} Z$; this implies the intersection $Y_0 \cap V$ is dense in $Y_0$.
\end{proof}

Let $h : M \to M'$ and $f: M \to Z$ be morphisms of schemes. We denote by $h \times f: M \rightarrow M' \times Z$ the morphism induced by $h$ and $f$. 
\begin{proposition}\label{prop: image general fiber}
	Let $M \xrightarrow{h} M'$ and $M \xrightarrow{f} Z$ be surjective morphisms of finite presentation of integral noetherian schemes. 
	
	Then, for $z$ general in $Z$:\begin{enumerate}[label=\roman*)]
		\item The image by $h$ of any irreducible component of $f^{-1}(z)$ has dimension equal to $\dim (h\times f)(M)- \dim Z$;
		\item Each irreducible component of $f^{-1}(z)$ dominates one of the irreducible components of $h(f^{-1}(z))$.
	\end{enumerate}
\end{proposition}

\begin{proof}
	\begin{enumerate}[label=\roman*)]
		\item Consider a dense open subset $U$ of $(h\times f)(M)$ such that the fiber $(h \times f)^{-1}(u)$ is a scheme of pure dimension $\mathrm{dim} M - \mathrm{dim} (h \times f)(M)$ for all $u$ in $U$. Consider a point $z$ general in $Z$, and an irreducible component $Z_0$ of $f^{-1}(z)$. Denote by $Z_0^\circ$ the variety $Z_0$ minus the intersection of $Z_0$ with other irreducible components of $f^{-1}(z)$.
		We first compute the dimension of the fiber  $h^{-1}(h(\mu)) \cap Z_0^\circ$, where $\mu$ is a point in $(h \times f)^{-1}(U) \cap Z_0^\circ$, before computing $\dim h(Z_0)$.
		
		\textbf{Computation of $\dim {h}^{-1}(h(\mu))\cap Z_0^\circ$.} First, we have: \begin{align*}
		{h}^{-1}(h(\mu))\cap f^{-1}(z) &= (h \times f)^{-1}\left(h(\mu), z\right)\\
		&= (h \times f)^{-1} (h \times f)(\mu).
		\end{align*}
		Since $(h \times f)(\mu)$ lies in $U$, this implies ${h}^{-1}(h(\mu)) \cap f^{-1}(z)$ is a scheme of pure dimension $\mathrm{dim} M - \mathrm{dim} (h\times f)(M)$. Hence the intersection of the irreducible variety $Z_0^\circ$ with ${h}^{-1}(h(\mu))$ is a scheme of pure dimension $\mathrm{dim} M - \mathrm{dim} (h\times f)(M)$. Indeed, the intersection of the open subvariety $Z_0^\circ$ of $f^{-1}(z)$ with $h^{-1}(h(\mu))$ is a non empty open subvariety of $h^{-1}(h(\mu)) \cap Z_0$ disconnected from the other pure dimensional irreducible components of $h^{-1}(h(\mu)) \cap f^{-1}(z)$.  Hence $h^{-1}(h(\mu)) \cap Z_0^\circ$ is a scheme of pure dimension $\mathrm{dim} M - \mathrm{dim} (h\times f)(M)$. 
		
		\textbf{Computation of $\dim h(Z_0)$.} Note that, since $z$ is general in $Z$, according to Lemma \ref{lem : Intersection dense open is dense}, $(h \times f)^{-1}(U)\cap f^{-1}(z)$ is dense in $f^{-1}(z)$; hence $(h \times f)^{-1}(U)\cap Z_0^\circ$ is a dense open subset of $Z_0$. Consider a point $\nu$ general in $h(Z_0)$; since $\nu$ is general and $h\left( (h \times f)^{-1}(U)\cap Z_0^\circ \right)$ is dense in the irreducible variety $h(Z_0)$, $\nu$ belongs to $h\left( (h \times f)^{-1}(U)\cap Z_0^\circ \right)$, and we have: \begin{align*}
		\dim Z_0 - \dim h(Z_0) &= \dim {h}^{-1}(\nu)\cap Z_0^\circ \\
		&= \mathrm{dim} M - \mathrm{dim} (h\times f)(M)
		\end{align*}
		where the first equality is deduced from the general position of $\nu$ in $h(Z_0)$, and the second equality holds since $\nu$ is in $h\left( (h \times f)^{-1}(U)\cap Z_0^\circ \right)$. This implies:
		\begin{align*}
		\mathrm{dim} h(Z_0) &= \mathrm{dim} Z_0 - \mathrm{dim} M + \mathrm{dim} (h \times f)(M)\\
		&= \mathrm{dim} f^{-1}(z) - \mathrm{dim} M + \mathrm{dim} (h \times f)(M),
		\end{align*}
		where the second equality holds since $z$ is general in $Z$ and thus $f^{-1}(z)$ is equidimensional.
		\item By contradiction, if an irreducible component $Z_0$ of $f^{-1}(z)$ has an image strictly contained in $h\left(f^{-1}(z)\right)$, then: $$\dim h(Z_0) < \dim h\left(f^{-1}(z)\right)= \dim (h\times f)(M)- \dim Z,$$ which would contradict $i)$.
	\end{enumerate}
\end{proof}

\begin{proof}[Proof of Theorem \ref{th: ex stability under composition}]
	\begin{enumerate}[label={(\arabic*)}]
		\item  Let $h: Y \rightarrow S$ and $S \rightarrow Z$ be morphisms in $\mathcal{C}$ whose general fiber is irreducible. Note that since we consider properties of general fibers, we can assume that both morphisms are dominant. Since both morphisms are morphisms of projective schemes, they are then surjective.
		
		Let $U$ be a dense open subset of $Z$ such that the fiber $S_z$ over a point $z$ in $U$ is irreducible.  For $z$ general in $Z$, according to Proposition \ref{prop: image general fiber} the image by $h: Y \rightarrow S$ of each irreducible component of $Y_z$ dominates one of the irreducible components of $S_z$, which is irreducible since $z$ is general in $Z$. Since, according to Lemma \ref{lem : Intersection dense open is dense}, for $z$ general in $Z$ the fiber $Y_z$ has a dense intersection with the dense open subset $Y_U$, the general fiber $Y_z$ is irreducible. Indeed, if we denote by $Y_z^1$ and $Y_z^2$ two different irreducible components of $Y$, and by $V$ the dense open subset of $S_z$ such that the fiber $Y_s$ over a point $s$ in $V$ is irreducible and of exepected dimension, for all $v$ in $V$ the intersection $Y_z^1 \cap Y_v$ is equal to $Y_z^2 \cap Y_v$, since both are closed subschemes of maximal dimension of an irreducible scheme. Hence the intersection of $Y_z$ with the dense open subset $Y_z \cap Y_V$ is irreducible, hence $Y_z$ is irreducible.
		\item Let $Y \xrightarrow{f_1} S$ and $S \xrightarrow{f_2} S'$ be morphisms in $\mathcal{C}$ satisfying $(P)$ and whose general fiber is irreducible. Since we consider properties of the general fiber of $Y \rightarrow S'$, we can assume that $Y$ dominates $S$. Let $U$ be a dense open subset of $S$ such that the fiber of $f_1$ over a point in $U$ satisfies $(P)$. For $x$ general in $S'$, the following properties are satisfied:
		\begin{enumerate} \item The fiber $S_x$ satisfies $(P)$;
			\item The fiber $S_x$ has a dense open intersection with $U$;
			\item The morphism $Y_x \rightarrow S_x$ is in $\mathcal{C}$;
			\item The general fiber of the morphism $Y_x \rightarrow S_x$ satisfies $(P)$;
			\item $Y_x$ satisfies $(P)$.
		\end{enumerate}		
		Note that proving property $(e)$ for $x$ general in $S'$ is our goal here. Property $(a)$ holds since $x$ is general in $S'$. According to Lemma \ref{lem : Intersection dense open is dense} and since $x$ is general in $S'$, property $(b)$ holds. Property $(c)$ holds since $x$ is general in $S'$ and according to $(1)$ the general fiber of $Y \to S'$ is irreducible. $(d)$ is a direct consequence of $(b)$, the fact that the fiber of $Y_x \rightarrow S_x$ over a point $y$ in $S_x$ is isomorphic to the fiber $Y_y = Y \times_{S'} \{y\}$, and the fact that $Y_y$ satisfies $(P)$ for $y$ in $U$.  Finally $(e)$ is implied by $(a)$, $(c)$ and $(d)$.
		\item Note that, according to Grabber-Harris-Starr's theorem on rational connectedness \cite{graber}, if the general fiber of a morphism $Y \rightarrow S$ of irreducible complex projective varieties is rationally connected, and the basis $S$ is rationally connected, then $Y$ is rationally connected. Hence, according to $2.$, the property "the general fiber is irreducible and rationally connected" is stable under composition in $\mathcal{C}$.
		\item Let $Y$, $Z$ and $S$ be projective complex schemes, where $Y$ is a rationally connected fibration over $Z$, and $Z$ is a rationally connected fibration over $S$. 
		Denote by $f: Y \rightarrow Z$ and $g: Z \rightarrow S$ the associated morphisms. First, note that since $Y$ dominates $Z$ and $Z$ dominates $S$, $Y$ dominates $S$. Furthermore, each irreducible component $S_i$ of $S$ is dominated by exactly one irreducible component $Z_i$ of $Z$, which itself is dominated by exactly one irreducible component $Y_i$ of $Y$; hence there is a one-to-one correspondence between irreducible components of $Y$ and $S$. Finally, since by hypothesis the general fibers of $Y_i \to Z_i$ and $Z_i \to S_i$ are rationally connected, according to $(3)$ the general fiber of $Y_i \to S_i$ is a rationally connected variety. Hence $Y$ is a rationally connected fibration over $S$.			
	\end{enumerate}
\end{proof}
\subsection{Restriction to a general subvariety.}\label{sec: image general subvariety} We study here the restriction of a morphism to a general subvariety. 
A typical application of the proposition below will be the study of the image of a Gromov-Witten variety. More precisely, consider $\pi : X \rightarrow X'$ the projection of a flag variety onto another flag variety, and denote by $\Pi : \overline{\mathcal{M}_{0,r}}(X, \mathbf{d}) \rightarrow \overline{\mathcal{M}_{0,r}}(X', (\pi)_* \mathbf{d})$ the morphism induced by $\pi$ between moduli spaces of genus zero stable maps, where $\mathbf{d}$ is the class of a curve in $X$. We thus obtain the commutative diagram:  \[ \begin{tikzcd}
\overline{\mathcal{M}_{0,r}}(X, \mathbf{d}) \arrow{d}{\Pi} \arrow{r} & X^r \arrow{d}{\pi^r} \\
\overline{\mathcal{M}_{0,r}}(X', (\pi)_* \mathbf{d}) \arrow{r} & (X')^r
\end{tikzcd} \]
to which we will apply the following proposition.

\begin{proposition}\label{prop : image is equidimensional}
	Let $G$ be an integral group scheme, let $f: M \rightarrow Z$, $f' : M' \rightarrow Z'$, $\pi : Z \rightarrow Z'$ be projective morphisms, of integral noetherian schemes, where $Z$ and $Z'$ admit a transitive $G$-action, and the morphism $\pi$ is $G$-equivariant. Let $Y$ be an integral subscheme of $Z$.
	Suppose there exists a proper morphism $h: M \rightarrow M'$ such that the following diagram commutes : \[ \begin{tikzcd}
	M \arrow{d}{h} \arrow{r}{f} & Z \arrow{d}{\pi} \\
	M' \arrow{r}{f'} & Z'
	\end{tikzcd} \]
	{Furthermore, suppose the morphism $h \times f: M \rightarrow M' \times_{Z'} Z$ induced by the morphisms $M \xrightarrow{h} M'$ and $M \xrightarrow{{f}} Z$ is surjective.} 
	
	Consider a point $g$ general in $G$. We call  $\mathcal{W}_g := f^{-1}(g Y)$, $\mathcal{W}_g' :=( h\times f)(\mathcal{W}_g)$, and $\mathcal{W}_g'' := h(\mathcal{W}_g)$. We denote by $p_1: M' \times_{Z'} Z \rightarrow M'$ the first projection. Then: \begin{enumerate}[label=\roman*)]
		\item $\mathcal{W}_g' = M' \times_{Z'} gY$ and $\mathcal{W}_g'' = (f')^{-1}(g \pi(Y))$;
		\item $\mathcal{W}_g$, $\mathcal{W}_g'$ and $\mathcal{W}_g''$ are pure dimensional schemes;
		\item The restriction of $h\times f$ to any irreducible component of $\mathcal{W}_g$ is dominant onto an irreducible component of $\mathcal{W}_g'$;
		\item The restriction of $p_1$ to any irreducible component of $\mathcal{W}_g'$ is dominant onto an irreducible component of $\mathcal{W}_g''$.
	\end{enumerate}
\end{proposition}


Let us introduce the following notation. Let $f: Y \rightarrow Z$ and $\varphi: Y' \rightarrow Z$ be morphisms of schemes, where an integral group scheme $G$ acts on $Z$. We denote by $Y \times_{Z} (gY')$ the fiber product induced by the morphism $\varphi_g: Y' \rightarrow Z$ given by $y \rightarrow g \cdot \varphi(y)$, i.e. $Y \times_{Z} (gY')$ is the fiber product associated with the following commutative diagram: \[ \begin{tikzcd}
Y \times_{Z} (gY') \arrow{d} \arrow{r} & Y' \arrow{d}{\varphi_g} \\Y \arrow{r}{f} & Z
\end{tikzcd} \]
Notice that, if $G$ acts on $Y'$ and $\varphi$ is $G$-equivariant, the fiber product $Y \times_Z (gY')$ will simply be the usual fiber product associated with the following commutative diagram: \[ \begin{tikzcd}
Y \times_{Z} (gY') \arrow{d} \arrow{r} & gY' \arrow{d}{\varphi} \\Y \arrow{r}{f} & Z
\end{tikzcd} \]

\begin{lemme}\label{lem : dense intersection with open for homogeneous spaces}
	Let $Y \rightarrow Z$ be a morphism of schemes, of finite presentation, where $Y$ is an equidimensional noetherian scheme; let $\varphi : Y' \rightarrow Z$ be a morphism of integral schemes. Let $V$ be a dense open subset of $Y$. Suppose there exists an integral group scheme $G$ such that $G$ acts transitively on $Z$. 
	
	Then, for a point $g$ general in $G$: \begin{enumerate}[label=\roman*] 
		\item[i)] The scheme $V \times_Z (g Y')$ is a dense subset of $Y \times_Z (g Y')$;
		\item[ii)] The scheme $(g V) \times_Z Y'$ is a dense subset of $(g Y) \times_{Z} Y'$.	
	\end{enumerate}
\end{lemme}

\begin{proof}\begin{enumerate}[label=\roman*] 
		\item[i)] Since $Y$ is a noetherian scheme, $Y \setminus V$ can be written as the union of a finite number of irreducible schemes, i.e. $Y \setminus V = \cup_{i=1}^N F_i$, where the $F_i$ are irreducible closed subschemes of $Y$, of dimension smaller than the pure dimension of $Y$. Consider an element $g$ general in $G$. Suppose the scheme $Y \times_{Z} (g Y')$ is not empty. Then, according to Kleiman's transversality theorem \cite{kleiman1974transversality}, it has pure dimension given by $\mathrm{dim} Y + \mathrm{dim} Y' - \mathrm{dim} Z'$. Furthermore, for $1 \leq i \leq N$, the scheme $F_i \times_{Z'} (g Y')$ is either empty or of pure dimension smaller than $\mathrm{dim} Y + \mathrm{dim} Y' - \mathrm{dim} Z'$. Hence $V \times_{Z'} (g Y') = (Y \setminus \cup_{i=1}^N F_i) \times_{Z'} (g Y')$ is dense in $Y \times_Z (g Y')$.	
		\item[ii)] Consider a point $g$ general in $G$. Then, according to part $i).$, the scheme $V \times_{Z} (g^{-1} Y')$ is a dense subset of $Y \times_Z (g^{-1} Y')$. Furthermore, notice that we have an isomorphism $(g Y) \times_Z  Y' \rightarrow Y \times_Z (g^{-1} Y')$. Indeed, if we denote by $g^{-1} : Z \rightarrow Z$ the isomorphism of $Z$ induced by the action of $g^{-1}$, and by $g : Z \rightarrow Z$ the isomorphism of $Z$ induced by the action of $g$, the following commutative diagram \[ \begin{tikzcd}[row sep=3em,column sep=3em] & Y' \arrow{d}{\varphi} \arrow{ddrr}[pos=0.3]{\varphi_{g^{-1}}} &&\\
		Y \arrow{r}{f_g} \arrow{drrr}[swap]{f} & Z \arrow[drr, shift left, pos=0.2, shorten >= 1em, "g^{-1}"]  && \\ &&& \arrow[llu, shift left, pos=0.8, shorten <= 1em, "g"] Z \end{tikzcd} \] implies that any morphism $T \rightarrow (g Y) \times_Z Y'$ factors through $Y \times_Z (g^{-1} Y')$, and reciprocally any morphism $T \rightarrow Y \times_Z (g^{-1} Y')$ factors through $(g Y) \times_Z Y'$.\\
		Finally, in the same tautological way, notice we have an isomorphism $(g V) \times_Z Y' \simeq V \times_Z (g^{-1} Y')$.
	\end{enumerate}
\end{proof}

\begin{proof}[Proof of Proposition \ref{prop : image is equidimensional}]
	\begin{enumerate}[label=\roman*]
		\item[i)] First, since, by assumption, $h \times f: M \rightarrow M' \times_{Z'} Z$ is surjective, $\mathcal{W}_g'=(h \times f)\left(f^{-1}(g Y)\right)=(h \times f)(M) \cap (M' \times_{Z'} (gY))$ is equal to $M' \times_{Z'} (gY)$. Furthermore, $h(f^{-1}(g Y))$ is equal to the image of $(h \times f)(f^{-1}(g Y))$ by the first projection $p_1$, where $p_1$ is defined by the following commutative diagram: \[\begin{tikzcd} M \arrow{dr}[swap]{h} \arrow{r}{h \times f} & (h \times f)(M)= M' \times_{Z'} Z \arrow{d}{p_1} \arrow{r} & Z \arrow{d}{\pi} \\ & M' \arrow{r}{f'} & Z' \end{tikzcd} \] Meanwhile, $(f')^{-1}(\pi(g Y))$ is equal to the image of  $M' \times_{Z'} (gY)$ by the first projection $p_1$, hence the equality between $h(f^{-1}(g Y))$ and  $(f')^{-1}(\pi(g Y))$. \\
		Finally, since $\pi$ is $G$-equivariant, this yields: $\mathcal{W}_g'' = (f')^{-1}(g \pi(Y))$.
		\item[ii)] According to Kleiman's transversality theorem, since $g$ is general in $G$, the schemes $\mathcal{W}_g = f^{-1}(g Y)$, $\mathcal{W}_g'=M' \times_{Z'} gY$ and $\mathcal{W}_g'' = (f')^{-1}(g \pi(Y))$ are either empty or pure dimensional schemes.
		\item[iii)] We consider the following commutative diagram:  \[ \begin{tikzcd}
		M\arrow{d}{h \times f} \arrow{r}{f} & Z \ar[equal]{d} \\
		M' \times_{Z'} Z  \arrow{r} & Z
		\end{tikzcd} \]
		
		Denote by $U$ a dense open subset of $(h \times f)(M)$ such that, for all $z$ in $U$, $(h \times f)^{-1}(z)$ has pure dimension equal to $\mathrm{dim} M - \mathrm{dim} (h \times f)(M)$. Fix an irreducible component $\mathcal{W}_0$ of $\mathcal{W}_g$. Denote by $\mathcal{W}_0^\circ$ the variety $\W_0$ minus the intersection of $\W_0$ with other irreducible components o $\mathcal{W}_g$. We first compute the dimension of $(h \times f)^{-1}(x) \cap \mathcal{W}_0^\circ$, where $x$ is a point in $(h \times f)(\mathcal{W}_0^\circ)\cap U$, before computing the dimension of $(h \times f) (\mathcal{W}_0)$.
		
		\textbf{Computation of $\dim (h \times f))^{-1}(x) \cap \mathcal{W}_0^\circ$.} First, the fiber of the restriction $(h \times f)_{| \mathcal{W}_g}$ of the morphism $h \times f$ to $\mathcal{W}_g$ is isomorphic to the fiber of $h \times f$. Indeed, for $(m,z)$ in $(h \times f)(\mathcal{W}_g)= M' \times_{Z'} gY$, the fiber ${(h \times f)_{| \mathcal{W}_g}}^{-1}(m,z)$ is equal to $(h \times f)^{-1}(m,z) \cap \mathcal{W}_g = h^{-1}(m) \cap f^{-1}(z) \cap \mathcal{W}_g$, where $z$ lies in $gY$. Meanwhile, we observe $f^{-1}(z) \cap \mathcal{W}_g = f^{-1}(z) \cap f^{-1}(gY) = f^{-1}(z)$, hence ${(h \times f)_{| \mathcal{W}_g}}^{-1}(m)$ is equal to $(h \times f)^{-1}(m)$.
		
		Since $x$ lies in $U$, this implies ${(h \times f)}^{-1}(x) \cap \mathcal{W}_g$ is a scheme of pure dimension $\mathrm{dim} M - \mathrm{dim} (h\times f)(M)$. Hence the intersection of the irreducible variety $\mathcal{W}_0^\circ$ with ${(h \times f)}^{-1}(x)$ is a scheme of pure dimension $\mathrm{dim} M - \mathrm{dim} (h\times f)(M)$. Indeed, the intersection of the open subvariety $\mathcal{W}_0^\circ$ of $\mathcal{W}_g$ with $(h \times f)^{-1}(x)$ is a non empty open subvariety of $(h \times f)^{-1}(x) \cap \mathcal{W}_0$ disconnected from the other pure dimensional irreducible components of $(h \times f)^{-1}(x) \cap \mathcal{W}_g$.  Hence $(h \times f)^{-1}(x) \cap \mathcal{W}_0^\circ$ is a scheme of pure dimension $\mathrm{dim} M - \mathrm{dim} (h\times f)(M)$. 
		
		\textbf{Computation of $\dim (h \times f)(\mathcal{W}_0)$.} Note that, since $g$ is general in $G$, according to Lemma \ref{lem : dense intersection with open for homogeneous spaces}, $(h \times f)^{-1}(U)\cap \mathcal{W}_g$ is dense in $\mathcal{W}_g$; hence $(h \times f)^{-1}(U)\cap \mathcal{W}_0^\circ$ is a dense open subset of $\mathcal{W}_0$. Consider a general point $\nu$ in $(h \times f)(\mathcal{W}_0)$; since $\nu$ is general and $(h \times f)\left( (h \times f)^{-1}(U)\cap \mathcal{W}_0^\circ \right)$ is dense in the irreducible variety $(h \times f)(\mathcal{W}_0)$, $\nu$ belongs to $(h \times f)\left( (h \times f)^{-1}(U)\cap \mathcal{W}_0^\circ \right)$, and we have: \begin{align*}
		\dim \mathcal{W}_0 - \dim (h \times f)(\mathcal{W}_0) &= \dim {(h \times f)}^{-1}(\nu)\cap \mathcal{W}_0^\circ \\
		&= \mathrm{dim} M - \mathrm{dim} (h\times f)(M)
		\end{align*}
		where the first equality is deduced from the general position of $\nu$ in $(h \times f)(\mathcal{W}_0)$, and the second equality holds since $\nu$ is in $(h \times f)\left( (h \times f)^{-1}(U)\cap \mathcal{W}_0^\circ \right)$. This implies:
		\begin{align*}
		\mathrm{dim} (h \times f)(\mathcal{W}_0) &= \mathrm{dim} \mathcal{W}_0 - \mathrm{dim} M + \mathrm{dim} (h \times f)(M)\\
		&= \mathrm{dim} \mathcal{W}_g - \mathrm{dim} M + \mathrm{dim} (h \times f)(M),
		\end{align*}
		where the second equality holds since $\mathcal{W}_g$ is equidimensional. Since $(h \times f)(\mathcal{W}_0)$ is obviously contained in $\mathcal{W}_g'= (h \times f)(\mathcal{W}_g)$, and this dimension does not depend on the irreducible component of $\mathcal{W}_g$ considered, this implies the irreducible variety $(h \times f)(\mathcal{W}_0)$ is an irreducible component of $\mathcal{W}_g'$, which has pure dimension equal to $ \mathrm{dim} \mathcal{W}_g - \mathrm{dim} M + \mathrm{dim} (h \times f)(M)$.
		
		\item[iv)]
		First, denote by $\pi_{| Y}: Y \rightarrow \pi(Y)$ the restriction of the morphism $\pi$ to $Y$. Fix a dense open subset $U$ of $\pi(Y)$ such that, for all $x$ in $U$, the fiber  $\pi_{| Y} ^{-1}(x)$ is a scheme of pure dimension equal to $\mathrm{dim} Y - \mathrm{dim} \pi(Y)$. 
		
		We consider the following commutative diagram: \[ \begin{tikzcd}
		M' \times_{Z'} Z \arrow{d}{p_1} \arrow{r}{p_2} & Z \arrow{d}{\pi} \\
		M' \arrow{r}{f'} & Z'
		\end{tikzcd} \]
		where $p_1$ and $p_2$ are respectively the first and second projection morphisms. Consider a general element $g$ in $G$. We denote by $p_1'$ the restriction of the morphism $p_1$ to $\mathcal{W}_g'=M' \times_{Z'} g Y$. We want to prove here that the image of any irreducible component of $\mathcal{W}_g'$ by $p_1'$ is an irreducible component of $\mathcal{W}_g'' = (f')^{-1}(g \pi(Y))$. In order to do this, we study the dimension of the fiber $(p_1')^{-1}(m)$ of $p_1'$ at a point $m$ in $\mathcal{W}_g''$.
		
		We will use the following observations. \begin{lemme}\label{lem: aux}
			Let $f': M' \rightarrow Z'$ and $\pi: Z \rightarrow Z'$ be projective morphisms of integral noetherian schemes. Let $Y$ be an integral subscheme of $Z$. Suppose there exists an integral group scheme $G$ such that $G$ acts transitively on $Z$, and $\pi$ is $G$-equivariant. Consider the following commutative diagram:  \[ \begin{tikzcd}
			M' \times_{Z'} Z \arrow{d}{p_1} \arrow{r}{p_2} & Z \arrow{d}{\pi} \\
			M' \arrow{r}{f'} & Z'
			\end{tikzcd} \]		
			where $p_1$ and $p_2$ are respectively the first and second projection morphisms.	
			\begin{enumerate} 
				\item Consider an element $g$ in $G$. The fiber $p_1^{-1}(m) \cap {p_2}^{-1}(gY)$ of the restriction of the morphism $p_1$ to $(p_2)^{-1}(gY)$ is isomorphic to  ${\pi}_{| Y}^{-1}(g^{-1} f'(m))$, where we denote by $\pi_{| Y}: Y \rightarrow \pi(Y)$ the restriction of the morphism $\pi$ to $Y$;
				\item Consider a dense open subset $U$ of $Y$. Then, for an element $g$ general in $G$, ${p_2}^{-1}(gU)$ is dense in ${p_2}^{-1}(gY)$.
			\end{enumerate}
		\end{lemme}
		
		\begin{proof}[Proof of Lemma \ref{lem: aux}]\begin{enumerate}
				\item  Denote by $h'$ the restriction of the morphism $p_1$ to ${p_2}^{-1}(gY)$. The fiber $(h')^{-1}(m)$ is equal to $\{m\} \times_{Z'} (gY)$, which is isomorphic to $\pi^{-1}(f'(m)) \cap (gY)$. Furthermore, since $\pi$ is $G$-equivariant, the morphism $y \rightarrow g^{-1} y$ induces an isomorphism between $\pi^{-1}(f'(m)) \cap (gY)$ and $\pi^{-1}(g^{-1} f'(m)) \cap (Y)$. Finally, notice $\pi^{-1}(g^{-1} f'(m)) \cap (Y)$ is equal to ${\pi}_{| Y}^{-1}(g^{-1} f'(m))$.
				
				\item According to Lemma \ref{lem : dense intersection with open for homogeneous spaces} $ii)$, since $g$ is general in $G$, the scheme $M' \times_{Z'} (g U)$ is dense in $M' \times_{Z'} (g Y)$. Since the morphism $\pi$ is $G$-equivariant, we have here the following equalities: ${p_2}^{-1}(gY) = M' \times_{Z'} (gY)$ and ${p_2}^{-1}(gU) = M' \times_{Z'} (gU)$. Hence the scheme ${p_2}^{-1}(gU)$ is dense in ${p_2}^{-1}(gY)$.
			\end{enumerate}
		\end{proof}
		
		Fix an irreducible component $\mathcal{W}_0'$ of $p_2^{-1}(gY) = M' \times_{Z'} gY$. For $\mu$ general in $\mathcal{W}_0'$, we have: \begin{align*}
		\dim \mathcal{W}_0' - \dim p_1(\mathcal{W}_0') &= \dim p_1^{-1}(p_1(\mu)) \cap \mathcal{W}_0'\\
		&= \dim \pi_{| Y}^{-1}\left(g^{-1} \cdot (f' \circ p_1)(\mu)\right)\\
		&= \dim Y - \dim \pi(Y),
		\end{align*}
		where the first equality holds since the image of a dense open subset of $\mathcal{W}_0'$ by $p_1$ is a dense open subset of $p_1(\mathcal{W}_0')$, the second equality is given by Lemma~\ref{lem: aux} $(a)$, and the last equality holds since $(f' \circ p_1)(\mu)$ is in $g \cdot \pi(U)$ for $\mu$ general in $\mathcal{W}_0'$. Indeed, according to Lemma \ref{lem: aux} $(b)$, $p_2^{-1}(g\pi^{-1}(U))$ is dense in $p_2^{-1} (gY)$, hence $p_2^{-1}(g\pi^{-1}(U)) \cap \mathcal{W}_0'$ is dense in $\mathcal{W}_0'$, hence a point in $(p_2 \circ \pi)^{-1}(gU) = (f' \circ p_1)^{-1}(gU)$ is general in $\mathcal{W}_0'$.	
		
		Finally, notice that, since $g$ is general in $G$, any irreducible component $\mathcal{W}_0'$ of $\mathcal{W}_g'$ will be of dimension equal to $\mathrm{dim} M' + \mathrm{dim} Y - \mathrm{dim} Z'$, hence :  \begin{equation*}
		\mathrm{dim} p_1(\mathcal{W}_0') = \mathrm{dim} M' + \mathrm{dim} \pi(Y) - \mathrm{dim} Z',
		\end{equation*}
		which is equal to the pure dimension of $\mathcal{W}_g'' = (f')^{-1}(g \pi(Y))$ since $g$ is general in $G$. Since $p_1(\mathcal{W}_0')$ is obviously contained in $\mathcal{W}_g''$, this implies that $\mathcal{W}_0'$ dominates an irreducible component of $\mathcal{W}_g''$.		
	\end{enumerate}
\end{proof}

%

\subsection{Application to Gromov-Witten varieties.}\label{sec: image W variety} Consider a semi-simple algebraic group $G$, a maximal torus $T$ of $G$, a Borel subgroup $B$ and parabolic subgroups $P$ and $P'$ satisfying $T \subset B \subset P \subset P'$. Denote by $X:= G/P$ and $X':=G/P'$ the corresponding homogeneous varieties, and by $$\pi: X=G/P \rightarrow X'=G/P'$$ the natural projection. For an element $u$ in the Weyl group $W:=N_G(T)/T$, denote by $X(u) := \overline{BuP}/P$ the associated Schubert variety. Note that $\pi(X(u))=X'(u)$.
Let $r \geq 0$, $u_1$, $\dots$, $u_r$ be elements in $W$, and $\mathbf{d}$ be a degree in the semi-group $E(G/P)$ of effective classes of curves in $H_2(G/P, \mathbb{Z})$. Note that the morphism $\pi : G/P \rightarrow G/P'$ induces a morphism: $$\Pi \times ev: \overline{{\mathcal{M}}_{0,r}}(G/P, \mathbf{d}) \rightarrow \overline{{\mathcal{M}}_{0,r}}(G/P', \pi_* \mathbf{d}) \times_{(G/P')^r} (G/P)^r,$$ where we denote by $\Pi: \overline{{\mathcal{M}}_{0,r}}(G/P, \mathbf{d}) \rightarrow \overline{{\mathcal{M}}_{0,r}}(G/P', \pi_* \mathbf{d})$ the morphism induced by $\pi$. For an element $g=(g_1, \dots, g_r)$ in $G^r$, we will consider the following commutative diagram:  \[ \begin{tikzcd}
\overline{\mathcal{M}_{0,r}}(G/P, \mathbf{d}) \arrow{d}{\Pi} \arrow{r}{\ev} & (G/P)^r \arrow{d}{\pi} \arrow[hookleftarrow]{r} & \prod_i g_i X(u_i)  \\
\overline{\mathcal{M}_{0,r}}(G/P',\pi_*\mathbf{d}) \arrow{r} & (G/P')^r \arrow[hookleftarrow]{r} &  \prod_i g_i X'(u_i)
\end{tikzcd} \]
This commutative diagram allows us to define the following fiber products: \begin{align*} \mathcal{W}_g&:= \overline{\mathcal{M}_{0,r}}(X, \mathbf{d}) \times_{X^r} \prod_i g_i X(u_i) &&\hookrightarrow \overline{{\mathcal{M}}_{0,r}}(X, \mathbf{d})\\ \mathcal{W}_g'&:= \overline{\mathcal{M}_{0,r}}(X',\pi_{*} \mathbf{d}) \times_{(X')^r} \prod_i g_i X(u_i) &&\hookrightarrow \overline{{\mathcal{M}}_{0,r}}(X', \pi_* \mathbf{d}) \times X^r\\
\mathcal{W}_g''&:= \overline{\mathcal{M}_{0,r}}(X',\pi_{*}\mathbf{d}) \times_{(X')^r} \prod_i g_i X'(u_i) &&\hookrightarrow \overline{{\mathcal{M}}_{0,r}}(X', \pi_* \mathbf{d})\end{align*} Note that for $g=(g_1, \dots, g_r)$ general in $G^r$, the scheme $\mathcal{W}_g = \W_{u_1, \dots, u_r}^{g,\mathbf{d}}$ is the variety parametrizing genus zero, degree $\mathbf{d}$, stable maps to $X$ sending their $i$-th marked point within $g_i \cdot X(u_i)$, and $\mathcal{W}_g''= \W_{u_1, \dots, u_r}^{g,\pi_* \mathbf{d}}$ is the variety parametrizing genus zero, degree $(\pi_{*})\mathbf{d}$ stable maps sending their $i$-th marked point within $g_i \cdot X'(u_i)$. Indeed, according to the proof of Kleiman's transversality theorem \cite{kleiman1974transversality}, for a point $g=(g_1, \dots, g_r)$ general in $G^r$, the natural maps $\W_g \to \overline{{\mathcal{M}}_{0,r}}(G/P,\d)$ and $\W_g'' \to \overline{{\mathcal{M}}_{0,r}}(G/P', \d)$ are both closed immersions. This allows us to identify the schemes $\W_g$ and $\W_g''$ with their images in the space of stable maps.

\begin{proposition}\label{prop : W is irreducible}	
	Suppose the morphism $\Pi \times \ev$ is surjective, and suppose its general fiber is connected.
	\begin{enumerate}
		\item For $g$ general in $G^r$, $(\Pi \times ev)(\mathcal{W}_g) = \mathcal{W}_g'$ and $\Pi(\mathcal{W}_g) = \mathcal{W}_g''$;
		\item  For $g$ general in $G^r$, each irreducible component of the Gromov-Witten variety $\mathcal{W}_g$ surjects into a different irreducible component of $\Pi(\mathcal{W}_g)$, i.e. there is a $1$-to-$1$ correspondence between irreducible components of $\mathcal{W}_g$ and irreducible components of $\mathcal{W}_g''$;
		\item Suppose the general fiber of $\Pi \times ev$ is rationally connected. Then for $g$ general in $G^r$, the Gromov-Witten variety $\mathcal{W}_g$ is a rationally connected fibration over both $\W_g'$ and $\mathcal{W}_g''$, in the sense of Definition~\ref{def: RC fibration}.	
	\end{enumerate}
\end{proposition}

\begin{lemme}\label{lem: surjection to irr components if fiber connected}
	Let $S$ be a normal scheme, and $S \xrightarrow{f} Y$ be a surjective morphism. 	
	If the general fiber of $f$ is connected, then each irreducible component of $f(S)$ is dominated by exactly one irreducible component of $S$.
\end{lemme}

\begin{proof}
	First, observe that, since the general fiber of $f$ is connected, two connected components of $S$ cannot dominate the same irreducible component of $f(S)$. Indeed, if we consider a dense open subset $U$ of $f(S)$ such that the fiber $f^{-1}(s)$ is connected for all $s$ in $U$, two connected components of $S$ dominating the same irreducible component of $f(S)$ would have a point in common in their image and in $U$; these two connected components hence would be connected within $S$. 	
	Moreover, since $S$ is normal, its connected components are irreducible. This implies all irreducible components of $f(S)$ are dominated by exactly one irreducible component of $S$.
\end{proof}

For $u$ in $W$, we denote by $x(u):=uP/P$  the $T$-fixed point in $X(u)=\overline{BuP/P}$ associated with $u$.
\begin{lemme}\label{lem: general fiber between Sch var irr}
	Let $u \in W$. 
	\begin{enumerate}[label= \roman*)]
		\item Let $f: Z \rightarrow X(u)$ be a $B$-equivariant morphism. Then there exists a surjective map $f^{-1}(BuP/P) \rightarrow f^{-1}(x(u))$.
		\item Denote by $\pi_{|u}$ the restriction of the projection $\pi$ to $X(u)$. The general fiber of the surjective morphism $\pi_{|u} : X(u) \rightarrow X'(u)$ is an irreducible unirational variety.
	\end{enumerate}
\end{lemme}

\begin{proof}
	This is implied by \cite{buch}, Proposition 2.3. We include here a proof for completeness. We denote by $U=BuP/P$ the open $B$-orbit of the $T$-fixed point $x(u)$.
	\begin{enumerate}[label= \roman*)]  \item Denote by $s : U \rightarrow B$ a section of $B \rightarrow BuP/P$, i.e. a morphism such that for all elements $x$ in $U$, $s(x) \cdot x(u)=x$. Such a section exists according to \cite{buch}, Proposition 2.2.	
		Note that for $y$ in $f^{-1}(U)$, the element $\left(s(f(y)) \right)^{-1} \cdot y$ lies in $f^{-1}(x(u))$. Indeed, since $f$ is $B$-equivariant, we have $f(\left(s(f(y)) \right)^{-1} \cdot y) = \left(s(f(y)) \right)^{-1} \cdot f(y) = x(u)$. We obtain a morphism: \begin{align*} f^{-1}(U) &\rightarrow f^{-1}(x(u)) \times U \\
		y &\rightarrow (\left(s(f(y)) \right)^{-1} \cdot y; f(y)) \end{align*}
		which is an isomorphism; indeed, the morphism $(p,u) \rightarrow s(u) \cdot p$ defines its inverse.  Composition with the first projection induces a surjective morphism  $f^{-1}(U) \rightarrow f^{-1}(x(u))$.
		
		\item According to $i)$, there exists a surjective map $\pi_{|u}^{-1}(U) \rightarrow \pi_{|u}^{-1}(x(u))$. Hence, since $\pi_{|u}^{-1}(U)$ is a dense open subvariety of the irreducible rational variety $X(u)$, the fiber $\pi_{|u}^{-1}(x(u))$ is an irreducible unirational variety. Furthermore, since $\pi_{|u}$ is $B$-equivariant, the fiber $\pi_{| u}^{-1} (x)$ at a point $x$ in $U$ is isomorphic to $\pi_{|u}^{-1}(x(u))$. Indeed, if $x =b \cdot x(u)$, where $b$ lies in $B$, the morphism $y \rightarrow b^{-1}\cdot y$ defines an isomorphism between $\pi^{-1}(x) \cap X(u)$ and $\pi^{-1}(x(u)) \cap X(u)$, since $X(u)$ is $B$-stable. Hence the general fiber of $\pi_{| u}$ is isomorphic to $\pi_{| u}^{-1}(x(u))$, which is an irreducible unirational variety.
	\end{enumerate}
\end{proof}

We name ${(\Pi \times ev)}_{|\mathcal{W}_g} : \mathcal{W}_g \rightarrow \mathcal{W}_g'$ the restriction of the morphism $\Pi \times ev$ to the Gromov-Witten variety $\mathcal{W}_g$. 
\begin{lemme}\label{lem: properties general fiber of pi x ev|W}
	If the general fiber of $\Pi \times ev$ satisfies a property $(P)$, then for $g$ general in $G^r$ the general fiber of $(\Pi \times ev)_{| \mathcal{W}_g}$ also satisfies $(P)$.
\end{lemme}

\begin{proof}
	Name $M:=\overline{{\mathcal{M}}_{0,r}}(X, \d)$. Let $U$ be a dense open subset of $(\Pi \times ev)(M)$ such that the fiber of $\Pi \times ev$ over a point in $U$ satisfies $(P)$. Let $(m,x)$ be a point in $U$, where $m$ lies in $M':=\overline{{\mathcal{M}}_{0,r}}(X', \pi_*\d)$ and $x$ lies in $X^r$. Note that, if $(m,x)$ is in the image $(\Pi \times ev)(\W_g)$, then the fiber over $(m,x)$ of $(\Pi \times ev)_{|\mathcal{W}_g}$ satisfies $(P)$. Indeed, we have: \begin{align*}
	(\Pi \times ev)_{|\mathcal{W}_g}^{-1}(m,x) &= \left( M_{m} \times_{X^r} \{x\} \right) \times_{X^r} (g\prod_i X(u_i))\\
	&= M_m \times_{X^r} \left( \{x\} \times_{X^r} (g\prod_i X(u_i)) \right) \\
	&= M_m \times_{X^r} \{x\}\\
	&= (\Pi \times ev)^{-1}(m,x),
	\end{align*}
	where the third equality holds since $(m,x)$ is in $(\Pi \times ev)(\W_g)$, and hence $x$ lies in $ev(\W_g) \subset g\prod_i X(u_i)$. Since $(m,x)$ lies in $U$, the fiber $(\Pi \times ev)^{-1}(m,x)$ satisfies $(P)$; hence the fiber $(\Pi \times ev)_{|\mathcal{W}_g}^{-1}(m,x)$ satisfies $(P)$ for all $(m,x)$ in $(\Pi \times ev)(\W_g) \cap U$
	
	Finally, note that for $g$ general in $G^r$, according to Lemma \ref{lem : dense intersection with open for homogeneous spaces} $\W_g \cap (\Pi \times ev)^{-1}(U)$ is dense in $\W_g$. Hence $(\Pi \times ev)(\W_g) \cap U$ is dense in $(\Pi \times ev)(\W_g)$. 
\end{proof}

\begin{proof}[Proof of Proposition \ref{prop : W is irreducible}]
	\begin{enumerate}
		\item Since the morphism $\Pi \times ev$ is surjective, this follows directly from Proposition \ref{prop : image is equidimensional} $i)$, applied to $h= \Pi : M = \overline{\mathcal{M}_{0,r}}(X, \mathbf{d}) \rightarrow M' = \overline{\mathcal{M}_{0,r}}(X',\pi_{*}\mathbf{d})$, $f = \ev : \overline{\mathcal{M}_{0,r}}(X, \mathbf{d}) \rightarrow Z = X^r$, $f' = \ev : \overline{\mathcal{M}_{0,r}}(X', \pi_{*}\mathbf{d}) \rightarrow Z' = (X')^r$, and from considering the integral group $G^r$.
		\item We proceed in two steps; we first prove that each irreducible component of $\mathcal{W}_g$ surjects into a different irreducible component of $\mathcal{W}_g'$, before proving the correspondence with irreducible components of $\mathcal{W}_g''$.
		
		\textbf{Correspondence between irreducible components of $\mathcal{W}_g$ and $\mathcal{W}_g'$.} First, note that, since the space $\overline{{\mathcal{M}}_{0,r}}(X, \mathbf{d})$ and Schubert varieties are normal \cite{fulton1996notes, ramanathan}, according to Kleiman's transversality theorem $\mathcal{W}_g= ev^{-1}(g_1 X(u_1) \times \dots g_r X(u_r))$ is normal for $(g_1, \dots, g_r)$ general in $G^r$. Furthermore, since the general fiber of $\Pi \times ev$ si connected, according to Lemma \ref{lem: properties general fiber of pi x ev|W} for $g$ general in $G^r$ the general fiber of the restriction ${(\Pi \times ev)}_{| \mathcal{W}_g}$ is connected. Hence according to Lemma \ref{lem: surjection to irr components if fiber connected} all irreducible components of $(\Pi \times ev)(\mathcal{W}_g)$ are dominated by exactly one irreducible component of $\mathcal{W}_g$.
		
		We then only have to ensure that there is no irreducible component of $\mathcal{W}_g$ whose image is a strict closed subset of an irreducible component of $(\Pi \times ev)(\mathcal{W}_g)=\mathcal{W}_g'$. Proposition \ref{prop : image is equidimensional} $iii)$, tells us that, since $\mathcal{W}_g$ is in general position, the image by $\Pi$ of any of its irreducible components is an irreducible component of $\mathcal{W}_g'$, hence the $1$-to-$1$ correspondence.
		
		\textbf{Correspondence between irreducible components of $\mathcal{W}_g'$ and $\mathcal{W}_g''$.} 	
		According to Lemma \ref{lem: general fiber between Sch var irr} the general fiber of the restriction $\pi_u: \prod_i X(u_i) \rightarrow \prod_i X'(u_i)$ of $\pi^r$ to $X(u_1) \times \dots \times X(u_r)$ is irreducible, hence there exists an open subvariety $U$ of $X'(u_1) \times \dots \times X'(u_r)$ such that the fiber of $\pi$ at $x$ in $U$ is irreducible. Consider the projections $p_1$ and $p_2$ defined by the following Cartesian diagram:  \[ \begin{tikzcd}
		\mathcal{W}_g' \arrow{d}{p_1} \arrow{r}{p_2} & \prod_i g_i X(u_i) \arrow{d}{\pi_u} \\
		\mathcal{W}_g'' \arrow{r}{ev} & \prod_i g_i X'(u_i)
		\end{tikzcd} \]
		Note that the fiber $p_1^{-1}(m)$ at a point $m$ in $p_2^{-1}(g \pi_u^{-1}(U))$ is irreducible. Indeed, we have: \begin{align*} 
		p_1^{-1}(m) &= \{m\} \times_{(X')^r} \left(g_1 X(u_1) \times \dots g_r X(u_r)\right)\\
		& \simeq \pi_u^{-1}(ev(m)) \\
		&\simeq g \cdot \pi_u^{-1}(g^{-1} ev(m)),
		\end{align*}
		which is irreducible since $ev(m)$ lies in $g \pi_u^{-1}(U)$. Furthermore, according to Lemma \ref{lem: aux} $b)$, for $g$ general in $G^r$, $\mathcal{W}_g' \cap p_2^{-1}(g\pi_u^{-1}(U))$ is a dense open subvariety of $\mathcal{W}_g'=p_2^{-1}(\prod_i g_i X(u_i))$. Hence the general fiber of $p_1$ is irreducible. Furthermore, note that, since $\overline{\mathcal{M}_{0,r}}(X',\pi_{*}\mathbf{d})$ and Schubert varieties are normal \cite{fulton1996notes, ramanathan}, according to Kleiman's transversality theorem $$\mathcal{W}_g'=\overline{\mathcal{M}_{0,r}}(X',\pi_{*}\mathbf{d}) \times_{(X')^r} \left(g_1 X(u_1) \times \dots g_r X(u_r)\right)$$ is normal for $g=(g_1, \dots, g_r)$ general in $G^r$. According to Lemma \ref{lem: surjection to irr components if fiber connected}, all irreducible components of $p_1(\mathcal{W}_g') = \mathcal{W}_g''$ are thus dominated by exactly one irreducible component of $\mathcal{W}_g'$. 
		
		Finally, Proposition \ref{prop : image is equidimensional} $iv)$, tells us that, since $\mathcal{W}_g'$ is in general position, the image by $p_1$ of any of its irreducible components is an irreducible component of $\mathcal{W}_g''$, hence the $1$-to-$1$ correspondence.
		
		\textbf{To sum up,} each irreducible component of $\mathcal{W}_g$ surjects into a different irreducible component of $\mathcal{W}_g'$, and each irreducible component of $\mathcal{W}_g'$ surjects into a different irreducible component of $\mathcal{W}_g''$. Furthermore, $\mathcal{W}_g$ surjects onto $\mathcal{W}_g''$. There thus is a $1$-to-$1$ correspondence between irreducible components of $\mathcal{W}_g$ and irreducible components of $\mathcal{W}_g''$.
		\item The only point left to prove is that the general fiber of $\W_g \rightarrow \W_g''$ is an irreducible rationally connected variety. According to Lemma \ref{lem: general fiber between Sch var irr} there exists a dense open subset $U$ of $\prod_i X'(u_i)$ such that the fiber of $\pi_{|u}: \prod_i X(u_i) \rightarrow \prod_i X'(u_i)$ over a point in $U$ is an irreducible rationally connected variety. For $g$ general in $G^r$, the following properties are satisfied: \begin{enumerate}
			\item $\left(\overline{{\mathcal{M}}_{0,r}}(X', \pi_* \d)\right)_{g \cdot U}$ is dense in $\W_g''= \left(\overline{{\mathcal{M}}_{0,r}}(X', \pi_*\d)\right)_{g \cdot \prod_i X'(u_i)}$;
			\item The general fiber of $\W_g' \rightarrow \W_g''$ is rationally connected;
			\item The general fiber of $\W_g \rightarrow \W_g'$ is rationally connected;
			\item The general fiber of $\W_g \rightarrow \W_g''$ is rationally connected.
		\end{enumerate}
		Note that proving $(d)$ is our goal here. 
		$(a)$ holds according to Lemma \ref{lem : dense intersection with open for homogeneous spaces}. Note that the fiber of the projection $p_1: \W_g' \rightarrow \W_g''$ over a point $m$ in $\left(\overline{{\mathcal{M}}_{0,r}}(X', \pi_* \d)\right)_{g \cdot U}$ is rationally connected; indeed, it is isomorphic to $g \cdot (\pi_{|u})^{-1}(g^{-1} \cdot ev(m))$, where $g^{-1} \cdot ev(m)$ lies in $U$. Hence, since according to $(a)$ $\left(\overline{{\mathcal{M}}_{0,r}}(X', \pi_* \d)\right)_{g \cdot U}$ is dense in $\W_g''=\left(\overline{{\mathcal{M}}_{0,r}}(X', \pi_*\d)\right)_{g \cdot \prod_i X'(u_i)}$, $(b)$ holds. $(c)$ holds according to Lemma \ref{lem: properties general fiber of pi x ev|W}. Finally, $(d)$ is deduced from $(b)$, $(c)$ and Theorem \ref{th: ex stability under composition} $(3)$.
	\end{enumerate}
\end{proof}

\section{Balanced flags of vector bundles}\label{sec: balanced flags over P^1}

Let $X$ be a flag variety. We describe here the flag of vector bundles associated with a general point in $\overline{{\mathcal{M}}_{0,r}}(X, \mathbf{d})$ for a degree $\d$ high enough. All definitions and results presented in this part are contained in I. Coskun's description of rational curves on a flag variety \cite{coskun2010quantum}, or easily deduced from it. 

\subsection{Notations.} Fix two positive integers $m<n$ . Let $I=\{i_1, \dots, i_m\}$ be a set of integers satisfying $0 < i_1 < \dots < i_m <n$.  Recall  $Fl_I$ is the $m$-step flag variety parametrizing $m$-tuples $(V_{i_1}, \dots, V_{i_m})$ of vector subspaces of $\mathbb{C}^n$,  where the subspaces are ordered by inclusion $0 \subset V_{i_1} \subset \dots \subset V_{i_m} \subset \mathbb{C}^n$, and $V_{i_k}$ is of dimension $i_k$. 
The class $\mathbf{d}=(d_1, \dots, d_m)$ of a curve $C$ in $X$ is determined by non-negative integers $d_1$, $d_2$,$\dots$, $d_m$, where $d_i$ is the Plücker degree of the projection of $C$ to the Grassmannian $\Gr(n_i,n)$. A point $p$ in $\mathcal{M}_{0,r}(X, \mathbf{d})$ is associated with a morphism $f_p: \mathbb{P}^1 \rightarrow X$. By the functorial definition of flag varieties, the morphism $f_p: \P^1 \rightarrow X$ is uniquely determined by the data of a flag of vector bundles over $\mathbb{P}^1$: $$\mathcal{F}_p = E_{1} \subset \dots \subset E_m \subset \O_{\P^1}^{\oplus n},$$ where $\mathrm{rank}E_k=i_k$ and $\mathrm{deg} E_k = -d_k$. Note that we have $f_p^* S_i \simeq E_i$, where $S_i$ is the $i$-th tautological vector bundle over $X$.
\subsection{Admissible sets of sequences.}
Recall $I=\{i_1, \dots, i_m\}$ is a set of $m$ integers $0 < i_1 < \dots <i_m <n$.

\begin{defn}\label{def : admissible set of sequence}
	Consider a set $\mathbf{d} =(d_1, \dots, d_m)$ of $m$ non negative integers. 	
	A \textit{$(I, \mathbf{d})$-admissible set of sequences} is a set \begin{equation*} A_\bullet = \{(a_{1, j})_{j=1}^{i_1}, (a_{2,j})_{j=1}^{i_2}, \dots ,(a_{m,j})_{j=1}^{i_m} \}
	\end{equation*}
	of $m$ sequences of non-negative integers of lengths $i_1$, $\dots$,$i_m$, respectively, such that: \begin{enumerate}  
		\item $0 \leq a_{k,j} \leq a_{k,j+1}$ for every $k$ and every $1\leq j \leq i_k-1$;
		\item $a_{k+1,j} \leq a_{k,j}$ for every $1 \leq k \leq m-1$ and $j \leq i_k$;
		\item $\sum_{j=1}^{i_k} a_{k,j}=d_k$ for every $k$.
	\end{enumerate}
\end{defn}

\begin{defn}\label{def : minimal admissible set}
	A $(I,\mathbf{d})$-admissible set of sequences $A_\bullet= \{(a_{k,j})\}$ is balanced if
	$A_\bullet$ 
	minimizes the function \begin{equation*} \sum_{1 \leq k \leq m} \sum_{1 \leq l<p \leq n_i} (a_{k,p}-a_{k,l})
	\end{equation*}
	among the $(I, \mathbf{d})$-admissible sets of sequences.
\end{defn}

Note that, according to \cite{coskun2010quantum} (see the proof of Lemma $2.1$), for any set $\d = (d_1, \dots, d_m)$ in $\mathbb{N}^m$ there exists a unique balanced $(I, \d)$-admissible set of sequences $(a_{k,j})$, which can be constructed by iteration on $k$ in the following way. \begin{enumerate}[label=$\bullet$]
	\item First, set $a_{1,j} = \ceil{d_1/n_1} - \epsilon_{1,j}$ where $\epsilon_{1,j}$ is in $\{0,1\}$ and is chosen in such a way that $\sum_{j=1}^{n_1} a_{1,j} =d_1$ and $a_{1,j} \leq a_{1,j+1}$ and $\lceil q \rceil$ denote the ceiling of $q \in \mathbb{Q}$;
	\item  Let $1 < k \leq m$. Suppose all $a_{k-1,j}$ have been constructed. Denote by $r_k$ the largest index such that the element $a_{k-1,r_{k}}$ satisfies $a_{k-1, r_{k}} \leq d_k/n_k$. Set: \begin{equation}\label{eq : def a_k,j} a_{k,j} = a_{k-1,j} \: \mathrm{for} \:1 \leq j \leq r_k. \end{equation} Denote by $R_k$ the largest index $>r_k$ such that the element $a_{k-1,R_k}$ satisfies $a_{k-1, R_k} \leq (d_k - \sum_{j=1}^{r_k} a_{k-1,j})/(i_k - r_k)$. Repeat the process replacing $r_k$ with $R_k$ until all elements $a_{k-1,j}$ of index larger than $r_k$ are larger than $(d_k - \sum_{j=1}^{r_k} a_{k-1,j})/(i_k - r_k)$.  Then set: \begin{equation}\label{eq: def a_i,j iteration} \forall r_k < j \leq i_k, \:a_{k,j} = \ceil*{\frac{(d_k- \sum_{j=1}^{r_k}a_{k-1,j})}{(a_k- r_k)}} - \epsilon_{k,j},\end{equation} where $\epsilon_{k,j}$ is in $\{0,1\}$ and is chosen in such a way that $\sum_{j=1}^{i_k} a_{k,j} =d_k$ and $a_{k,j} \leq a_{k,j+1}$.
\end{enumerate}

For simplicity, we set $d_0 :=0$ and $n_0 :=0$.
\begin{proposition}\label{prop: description balanced sets}
	Let $\d=(d_1, \dots, d_m)$ be a set of nonnegative integers. Let $1 < k \leq m$ such that: \begin{equation}\label{eq: condition splittng} \forall \: p<k, \: d_k \geq i_k \ceil*{ \frac{d_p-d_{p-1}}{i_p-i_{p-1}}}.\end{equation}
	Then the balanced $(I, \d)$-admissible set of sequences $\{(a_{1,j})_{j=1}^{i_1}, \dots, (a_{m,j})_{j=1}^{i_m}\}$ satisfies $a_{k,j}=a_{k-1,j}$ for $1 \leq j \leq i_{k-1}$.
\end{proposition}

\begin{proof}
	By iteration on $s<k$, we prove that for $1\leq s <k$ all elements $a_{s,j}$ satisfy $a_{s,j} \leq d_k/n_k$. According to (\ref{eq : def a_k,j}) this will yield the expected result. First, for $s=1$, since $d_k \geq n_k \ceil{d_1 /i_1}$, $a_{1,j} \leq d_k/i_k$. Now let $1 \leq s <k$, and suppose $a_{s-1,j} \leq d_k/i_k$ for all $1 \leq j \leq n_{s-1}$. Note that, if $r_s < i_{s-1}$, then there exists elements $a_{s-1,j}$ larger than $(d_s - \sum_{j=1}^{r_s} a_{s-1,j})/(i_s - r_s)$, where we use the notations of (\ref{eq: def a_i,j iteration}); hence all elements $a_{s,j}$ are smaller than an element $a_{s-1,s'}$ for some $s'$. Hence all elements $a_{s,j}$ then satisfy $a_{s,j} \leq d_k/n_k$. On the other hand, if $r_s = i_{s-1}$, then according to (\ref{eq: def a_i,j iteration}): \begin{align*}
	\forall i_{s-1} < j \leq i_s, \:\:\:\: a_{s,j} &\leq \ceil*{\frac{\left(d_s- \sum_{j=1}^{i_{s-1}}a_{s-1,j} \right)}{(i_s - i_{s-1})}} \\
	& \leq \ceil*{\frac{(d_s- d_{s-1})}{(i_s - i_{s-1})}} \:\:\:\: \leq \frac{d_k}{n_k}.
	\end{align*}
\end{proof}

\subsection{Balanced flag bundles.}\label{sec: balanced bundles} The partial flag variety $Fl_I$ comes equipped with $m$ tautological bundles $S_1$, $\dots$, $S_m$. We describe here the pullback $f_p^*S_i$ by a morphism $f_p : \mathbb{P}^1 \rightarrow X$ associated with a point $p$ general in $\overline{\mathcal{M}_{0,r}(X, \mathbf{d})}$.

Let $\mathbf{d} = (d_1, \dots, d_m)$ be an effective class of curve in $H_2(Fl_I, \mathbb{Z})$, $p$ be a point in ${{\mathcal{M}}_{0,r}}(X, \mathbf{d})$, and $f_p : \mathbb{P}^1 \rightarrow X$ be its associated morphism. The pullback $f_p^*S_s$ of the $s$-th tautological bundle $S_s$ by $f_p$ is a vector bundle of degree $d_s$ and rank $i_s$. According to Grothendieck's decomposition of vector bundles, $f_p^*S_s$ is isomorphic to a direct sum of line bundles $$f_p^*S_s \simeq \mathcal{O}_{\mathbb{P}^1}(-a_{s,1}) \oplus \dots \oplus \mathcal{O}_{\mathbb{P}^1}(-a_{s,i_s}),$$ where $\sum_{j=1}^{i_s-k_s} a_{s,j}=d_s$. 
Suppose, for a given $s$, the $a_{s,j}$ are ordered in increasing order. Then, since, for all $1 < s \leq m$, $f_p^*S_{s-1}$ is a subbundle of $f_p^*S_s$, we have: $a_{s-1,j} \geq a_{s,j}$. Hence the sequence $(a_{s,j})$ forms a $(I, \mathbf{d})$-admissible set of sequence in the sense of Definition \ref{def : admissible set of sequence}. We call $a_{s,j}$ the \textit{admissible set of sequences associated with $p$}.
\begin{proposition}\label{lem : description balanced bundles}
	\begin{enumerate}[label=\roman*)]
		\item The admissible set of sequences associated with a general point in $\mathcal{M}_{0,r}(Fl_I, \mathbf{d})$ is balanced in the sense of Definition \ref{def : minimal admissible set};
		\item Suppose $\forall \: p<k, \: d_k \geq n_k \ceil*{ \frac{d_p-d_{p-1}}{i_p-i_{p-1}}}$. Denote by $\mathcal{F}_\p = E_1 \hookrightarrow \dots \hookrightarrow E_m \hookrightarrow \mathcal{O}_{\mathbb{P}^1}^{\oplus n}$ the flag of vector bundles associated with a general element $\p$ in $\overline{{\mathcal{M}}_{0,r}}(Fl_I,\mathbf{d})$. Then $E_k$  splits into $E_k \simeq E_{k-1} \oplus_{\sum_{s=n_{k-1}+1}^{i_k}} \O_{\P^1}(-a_{k,s})$.
	\end{enumerate}
\end{proposition}

\begin{proof}
	\begin{enumerate}[label=\roman*)]
		\item 		Cf. proof of \cite{coskun2010quantum}, Proposition $2.3.$
		\item Let $\mathcal{F}_\p = E_1 \subset \dots \subset E_m \subset \mathcal{O}_{\mathbb{P}^1}^{\oplus n}$ be the flag of vector bundles  associated with a general element $\p$ in $\overline{{\mathcal{M}}_{0,r}}(X,\mathbf{d})$. Then, according to $i)$, the admissible set of sequences associated with $\mathcal{F}_\p$ is balanced. Since $d_{k}$ satisfies $\forall \: p<k, \: d_k \geq i_k \ceil*{ \frac{d_p-d_{p-1}}{i_p-i_{p-1}}}$, we have according to Proposition \ref{prop: description balanced sets} the following sequence: \begin{align*} 0 \rightarrow & E_{k-1} \simeq \mathcal{O}_{\mathbb{P}^1}(-a_{k-1,1}) \oplus  \dots \oplus \mathcal{O}_{\mathbb{P}^1}(-a_{k-1,n_{k-1}}) \rightarrow \\
		& \:\:\:\:\:\:\:\:\:\:\:\: E_k \simeq \mathcal{O}_{\mathbb{P}^1}(-a_{k-1,1}) \oplus \dots \oplus \mathcal{O}_{\mathbb{P}^1}(-a_{k-1,i_{k-1}}) \oplus \mathcal{O}_{\mathbb{P}^1}(-a_{k,i_{k-1}+1}) \oplus \dots \oplus \mathcal{O}_{\mathbb{P}^1}(-a_{k,n_{k}}) \end{align*}
		which splits. Indeed, for all $j$, $l$ verifying $j \leq l$ and $j \leq n_{k-1} <l$, we have: $a_{k,j} = a_{k-1,j}$ and $a_{k,j} \leq a_{k,l}$, hence $a_{k-1,j} \leq a_{k,l}$ for any $l >i_{k-1}$. 
		Hence according to Lemma \ref{lem: map O(-d) -> O(-d) splits} the morphisms $\mathcal{O}_{\mathbb{P}^1}(-a_{k-1,j}) \rightarrow E_k$ factor through $\oplus_{b=1}^{n_{k-1}} \mathcal{O}_{\mathbb{P}^1}(-a_{k,b}) \hookrightarrow E_k$. Finally, since the morphism $E_{k-1} \rightarrow E_k$ is injective, $E_{k-1}$ is a rank $i_{k-1}$ subvector bundle of $\oplus_{b=1}^{i_{k-1}} \mathcal{O}_{\mathbb{P}^1}(-a_{k,b})$; hence $E_{k-1} \simeq \oplus_{b=1}^{n_{k-1}} \mathcal{O}_{\mathbb{P}^1}(-a_{k,b})$.
	\end{enumerate}
\end{proof}

\begin{lemme}\label{lem: map O(-d) -> O(-d) splits}
	Let $d_1 \leq  \dots \leq d_r$ and $d_1'=d_1 \leq \dots \leq d_r'=d_r \leq  d_{r+1}' \leq \dots \leq d_{r+p}'$ be sequences of non negative integers. Then the following sequence splits. \begin{align*} 0 \rightarrow & E= \mathcal{O}_{\mathbb{P}^1}(-d_1) \oplus  \dots \oplus \mathcal{O}_{\mathbb{P}^1}(-d_r) \rightarrow \\
	& \:\:\:\:\:\:\:\:\:\:\:\:  F=\mathcal{O}_{\mathbb{P}^1}(-d_1) \oplus \dots \oplus \mathcal{O}_{\mathbb{P}^1}(-d_r) \oplus \mathcal{O}_{\mathbb{P}^1}(-d_{r+1}') \oplus \dots \oplus \mathcal{O}_{\mathbb{P}^1}(-d_{r+p}') \end{align*}
\end{lemme}
\begin{proof}
	Let $k$ be the largest index such that $d_k< d_r$. Note that for $i >k$ all the integers $d_i'$ are greater than $d_{k+1}$ or equal to $d_{k+1}$ and $d_{k+1}=d_r>d_k$. Hence all integers $d_i'$ are greater than $d_k$ for $i>k$. Hence the induced morphism $E_1:=\mathcal{O}_{\mathbb{P}^1}(-d_1) \oplus  \dots \oplus \mathcal{O}_{\mathbb{P}^1}(-d_k) \rightarrow  F$ factors through a closed immersion $\mathcal{O}_{\mathbb{P}^1}(-d_1) \oplus  \dots \oplus \mathcal{O}_{\mathbb{P}^1}(-d_k) \hookrightarrow F$. Denote by $r'$ the largest index such that $d_{r'}'=d_r$. Note that $r' \geq r$. Let us now consider the induced exact sequence  \begin{align*}0 \rightarrow &E/E_1 \simeq \mathcal{O}_{\mathbb{P}^1}(-d_{k+1}) \oplus  \dots \oplus \mathcal{O}_{\mathbb{P}^1}(-d_r) \simeq \O_{\P^1}(-d_k)^{\oplus r-k} \\ & \:\:\:\: \rightarrow F/E_1 \simeq \mathcal{O}_{\mathbb{P}^1}(-d_{k+1}') \oplus \dots \oplus  \mathcal{O}_{\mathbb{P}^1}(-d_{r+p}') \simeq  \O_{\P^1}(-d_r)^{\oplus (r'-k)}  \oplus \O_{\P^1}(-d_{r'+1}) \oplus \dots \oplus \O_{\P^1}(-d_{r+p}').\end{align*}
	Since for any $i>r'$ the integer $d_i'$ is greater than $d_k$, the injective morphism $E/E_1 \rightarrow F/E_1$ factors through \[E/E_1 \simeq \O_{\P^1}(-d_k)^{\oplus r-k} \hookrightarrow  \O_{\P^1}(-d_r)^{\oplus (r'-k)} \hookrightarrow F/E_1 \]
	Finally observe that the morphism $E/E_1 \simeq \O_{\P^1}(-d_k)^{\oplus r-k} \hookrightarrow  \O_{\P^1}(-d_r)^{\oplus (r'-k)}$ factors through $E/E_1 \simeq \O_{\P^1}(-d_k)^{\oplus r-k} \hookrightarrow  \O_{\P^1}(-d_k)^{\oplus r-k} \oplus \O_{\P^1}(-d_k)^{\oplus r'-r} \hookrightarrow F/E_1$. 
\end{proof}
\section{Forgetting one vector space} \label{sec : geometry general fiber}

We denote by $\pi_{\widehat{k}}: X \rightarrow X_{\widehat{k}}$ the morphism forgetting the $k$-th vector space. Let $\d=(d_1, \dots, d_m)$ be a degree in $E(X)$. In this section, our goal is to prove the following results.
\begin{proposition}\label{prop : geometry general fiber pr1W}
	Suppose $\forall \: p<k, \: d_k \geq n_k \ceil*{ \frac{d_p-d_{p-1}}{n_p-n_{p-1}}}$.
	Then the general fiber of $$\Pi_{\widehat{k}}\times ev: \overline{{\mathcal{M}}_{0,r}}(X, \d) \rightarrow \left( \Pi_{\widehat{k}}\times ev \right)(\overline{{\mathcal{M}}_{0,r}}(X, \d)) \subset  \overline{{\mathcal{M}}_{0,r}}(X_{\widehat{k}}, (\pi_{\widehat{k}})_* \d) \times_{(X_{\widehat{k}})^r} X^r$$ is a unirational variety.
\end{proposition}

\begin{corollaire}\label{cor : geometry general fiber pr1 ev-1(x)}
	Suppose $\forall \: p<k, \: d_k \geq n_k \ceil*{ \frac{d_p-d_{p-1}}{n_p-n_{p-1}}}$ and the map $\Pi_{\widehat{k}}\times ev$ is surjective. Then for $x$ general in $X^r$, the general fiber of the map $\overline{{\mathcal{M}}_{0,r}}(X, \d) \ni ev^{-1}(x) \to ev^{-1}(\pi_{\widehat{k}}(x)) \in \overline{{\mathcal{M}}_{0,r}}(X_{\widehat{k}}, (\pi_{\widehat{k}})_* \d)$ is a unirational variety. 
\end{corollaire}
We will assume in this section $\forall \: p<k, \: d_k \geq n_k \ceil*{ \frac{d_p-d_{p-1}}{n_p-n_{p-1}}}$. Let $(p,x)$ be a point in $(\Pi_{\widehat{k}}\times \ev)({{\mathcal{M}}_{0,r}}(X, \mathbf{d}))$, where $p$ is in ${{\mathcal{M}}_{0,r}}(X_{\widehat{k}}, (\pi_{\widehat{k}})* \mathbf{d})$, and $x=(x_1, \dots, x_r)$ is in $X^r$. Denote by $\Pi_{p, x}$ the fiber of $\Pi_{\widehat{k}} \times \ev$ at $(p,x)$, i.e. : \begin{equation*}
\Pi_{p, x} := \Pi_{\widehat{k}}^{-1}(p) \times_{X^r} x.
\end{equation*}
In order to prove the irreducibility and unirationality of $\Pi_{p, x}$, we define an auxilliary variety $\mathcal{Z}_{p,x}$, which dominates $\Pi_{p,x}$ when $(p,x)$ is general in the image $(\Pi_{\widehat{k}}\times ev)({{\mathcal{M}}_{0,r}}(X, \mathbf{d}))$. 
\subsection{Notations.}  
The class $\mathbf{d}=(d_1, \dots, d_m)$ of a curve $C$ in $X$ is determined by non-negative integers $d_1$, $d_2$,$\dots$, $d_m$, where $d_i$ is the Plücker degree of the projection of $C$ to the Grassmannian $\Gr(n_i,n)$. The morphism $\pi_{\widehat{k}} : X \rightarrow X_{\widehat{k}}$ induces a morphism $\Pi_{\widehat{k}} :  \overline{{\mathcal{M}}_{0,r}}(X,\mathbf{d}) \rightarrow \overline{{\mathcal{M}}_{0,r}}(X_{\widehat{k}}, (\pi_{\widehat{k}})_* \mathbf{d})$.\\
Denote by $\pi_{k} : X \rightarrow \Gr(n_k,n)$ the morphism sending a $m$-step flag $\{0\} \subset V_1 \subset \dots \subset V_k \subset \dots \subset V_m$ to the $k$-th linear subspace $V_k$. This morphism induces a morphism $\overline{{\mathcal{M}}_{0,r}}(X, \mathbf{d}) \rightarrow \overline{{\mathcal{M}}_{0,r}} (\Gr(n_k,n), d_k) \rightarrow \overline{{\mathcal{M}}_{0,3}}(\Gr(n_k,n),d_k)$ into the space of genus zero stable maps with three marked points to the Grassmannian $\Gr(n_k,n)$. 
We obtain the following commutative diagram:  \[ \begin{tikzcd}
\overline{{\mathcal{M}}_{0,r}}(X,\mathbf{d}) \arrow[swap]{dr}{\ev^r} \arrow{dd}{\Pi_{\widehat{k}}} \arrow{r}   & \overline{{\mathcal{M}}_{0,3}}(X,\mathbf{d}) \arrow{r}  & \overline{{\mathcal{M}}_{0,3}}(\Gr(n_k,n),d_k)  \\
& X^r \arrow{d}{\pi_{\widehat{k}} \times \dots \times \pi_{\widehat{k}}} \arrow{r}{\pi_k \times \dots \times \pi_{k}} & \Gr(n_k,n)^r \\
\overline{{\mathcal{M}}_{0,r}}(X_{\widehat{k}}, (\pi_{\widehat{k}})_* \mathbf{d}) \arrow{r}  & {X_{\widehat{k}}}^r & 
\end{tikzcd} \]

Furthermore, for $p$ in $\mathcal{M}_{0,r}(X, \mathbf{d})$, we will denote by: $$f_p: \mathbb{P}^1 \rightarrow X$$ the associated morphism. By the functorial definition of flag varieties, $f_p$ is uniquely determined by the data of a flag of vector bundles over $\mathbb{P}^1$: $$\mathcal{F}_p = E_{1} \subset \dots \subset E_m \subset \O_{\P^1}^{\oplus n},$$ where $\mathrm{rank}E_i=n_i$ and $\mathrm{deg} E_i = -d_i$. Note that we have $f_p^* S_i \simeq E_i$, where $S_i$ is the $i$-th tautological vector bundle over $X$. Finally, for a point $\p$ in $\P^1$ we will denote by $(E_i)_\p$ the fiber at $\p$ of the vector bundle $E_i$, which has a natural inclusion $(E_i)_\p \subset \C^n$ induced by the morphism $E_i \hookrightarrow \O_{\P^1}^{\oplus n}$.
\subsection{A preliminary result.}
Denote by $\mathcal{F}_p = E_1 \subset \dots \subset E_{k-1} \subset E_{k+1} \subset \dots \subset E_m \subset \mathcal{O}_{\mathbb{P}^1}^{\oplus n}$ the flag of vector bundles over $\mathbb{P}^1$ associated with the point $p$ in $\mathcal{M}_{0,r}(X_{\widehat{k}}, (\pi_{\widehat{k}})_* \mathbf{d})$.
\begin{lemme}\label{lem: decomposition into a_i}
	\begin{enumerate}[label= \roman*)]
		\item Consider a dense open subset $U$ of $\overline{{\mathcal{M}}_{0,r}}(X, \mathbf{d})$. Then, for $(p,x)$ general in $(\Pi_{\widehat{k}} \times ev)(\overline{{\mathcal{M}}_{0,r}}(X, \mathbf{d}))$, $\Pi_{p, x} \cap U$ is a dense open subset of $\Pi_{p, x}$;
		\item For $(p,x)$ general in $(\Pi_{\widehat{k}} \times ev)(\overline{{\mathcal{M}}_{0,r}}(X, \mathbf{d}))$, there exists a dense open subset $U$ of $\Pi_{p,x}$ such that, for all $q$ in $U$, the flag of vector bundles $\mathcal{F}_q$ associated with $q$ is balanced. 
		Hence there exists an increasing set of integers $a_{n_{k-1}+1} \leq \dots \leq a_{n_k}$ such that for all $q$ in $U$, the flag $\mathcal{F}_q$ is of the form: \begin{equation*}
		\mathcal{F}_q = E_1 \subset \dots \subset E_{k-1} \subset E_k \simeq E_{k-1} \oplus_{i=n_{k-1}+1}^{n_k} \mathcal{O}_{\mathbb{P}^1}(-a_i) \subset E_{k+1} \subset \dots \subset E_m \subset \mathcal{O}_{\mathbb{P}^1}^{\oplus n}.
		\end{equation*}
	\end{enumerate}
\end{lemme}

\begin{proof}
	\begin{enumerate}[label= \roman*)]
		\item 	Since $(p,x)$ is general in $(\Pi_{\widehat{k}} \times ev)(\overline{{\mathcal{M}}_{0,r}}(X,\mathbf{d}))$, according to Lemma \ref{lem : Intersection dense open is dense} the intersection of $\Pi_{p,x}$ with the dense open subset $U$ is dense in $\Pi_{p,x}$.
		\item According to Proposition \ref{lem : description balanced bundles} $i)$, there exists a dense open subset $U$ of $\overline{{\mathcal{M}}_{0,r}}(X, \mathbf{d})$ such that a point in $U$ corresponds to a balanced flag of vector bundles over $\mathbb{P}^1$. According to $i)$, $U \cap \Pi_{p,x}$ is dense open subset of $\Pi_{p,x}$. 
		Finally, notice that according to Proposition \ref{lem : description balanced bundles} $ii)$, since $\forall \: p<k, \: d_k \geq n_k \ceil*{ \frac{d_p-d_{p-1}}{n_p-n_{p-1}}}$, the flag of vector bundles associated with a point in $\Pi_{p,x} \cap U$ splits at $E_{k-1}$.
	\end{enumerate}
\end{proof}

\subsection{Definition of $\mathcal{Z}_{p,x}$.} Our goal here is to construct a rational variety $\mathcal{Z}_{p,x}$ dominating $\Pi_{p,x}$. First, consider the integers $a_{n_{k-1}+1}$, $\dots$, $a_{n_k}$ defined in Lemma \ref{lem: decomposition into a_i} $ii)$ and the following vector space  $$V:=Hom(\oplus_{i=n_{k-1}+1}^{n_k} \mathcal{O}_{\mathbb{P}^1}(-a_i), \mathcal{O}_{\mathbb{P}^1}^{\oplus n}).$$ 
The flag variety $X$ is naturally embedded in $X_{\widehat{k}} \times \Gr(n_k,n)$ by Plücker-embedding $X_{\widehat{k}} \times \Gr(n_k,n)$ in a product of projective spaces, and cutting out $X$ from this product by considering the relations given by the inclusions of flags. Note that, if we denote by $[f_p : \mathbb{P}^1 \rightarrow X_{\widehat{k}}, \{p_1, \dots, p_r\}]$ the stable map associated with $p$, points $q$ in $\Pi_{p, x} \cap \mathcal{M}_{0,r}(X, \mathbf{d})$ correspond to stable maps $[f_q: \mathbb{P}^1 \rightarrow X, \{p_1, \dots, p_r\}]$ satisfying $f_q(p_i)=x_i$ and such that the following diagram is commutative: \[ \begin{tikzcd}
\mathbb{P}^1 \arrow{r}{f_q} \arrow{rrd}{f_p} & X \arrow{r} & X_{\widehat{k}} \times \Gr(n_k,n) \arrow{d} \\
& & X_{\widehat{k}}
\end{tikzcd} \] Furthermore, by the functorial definition of the Grassmannian, elements $u$ in $V$ corresponding to an injective morphism $$\Phi_u: E_{k-1} \oplus_i \mathcal{O}_{\mathbb{P}^1}(-a_i) \rightarrow \mathcal{O}_{\mathbb{P}^1}^{\oplus n}$$
define a morphism $f_u: \mathbb{P}^1 \rightarrow \Gr(n_k,n)$, where we denote by $\Phi_u$ the morphism  induced by the inclusion $E_{k-1} \subset \O_{\P^1}^{\oplus n}$ and the morphism $\mathcal{O}_{\mathbb{P}^1}(-a_i) \rightarrow \mathcal{O}_{\mathbb{P}^1}^{\oplus n}$ associated with a point $u$ in $V$. Conversally, any injective morphism $ E_{k-1} \oplus_i \mathcal{O}_{\mathbb{P}^1}(-a_i) \rightarrow \mathcal{O}_{\mathbb{P}^1}^{\oplus n}$ can, up to automorphisms of $E_{k-1}$, be obtained this way.  Points in $\Pi_{p,x}$ are thus associated with elements $u$ in $V$ such that: \begin{enumerate} 
	\item $\Phi_u: E_{k-1} \oplus_{i=n_{k-1}+1}^{n_k} \mathcal{O}_{\mathbb{P}^1}(-a_i) \rightarrow \mathcal{O}_{\mathbb{P}^1}^{\oplus n}$ is an injective morphism of vector bundles;
	\item For all $1 \leq i \leq r$, $({F_u})_{p_i} = \pi_k(x_i)$;
	\item $f_p \times f_u: \mathbb{P}^1 \rightarrow X_{\widehat{k}} \times \Gr(n_k,n)$ factors through $X \rightarrow  X_{\widehat{k}} \times \Gr(n_k,n)$,
\end{enumerate}
where we denote by $({F_u})_{p_i}$ the fiber of the vector bundle $F_u := \Phi_u(E_{k-1} \oplus_i \mathcal{O}_{\mathbb{P}^1}(-a_i)) \subset \mathcal{O}_{\mathbb{P}^1}^{\oplus n}$ at the point $p_i \in \P^1$, and by $f_u$ the morphism associated with the flag of vector bundles $F_u \subset \mathcal{O}_{\mathbb{P}^1}^{\oplus n}$ by functorial definition of the Grassmannian.
Note that for $u$ in $V$ the morphism $\Phi_u: E_{k-1} \oplus_i \mathcal{O}_{\mathbb{P}^1}(-a_i) \rightarrow \mathcal{O}_{\mathbb{P}^1}^{\oplus n}$ might not be an injective morphism, and the fiber $(F_u)_{p_i}$ might be of dimension less than $n_k$. In order to define a closed subset of $V$, we define $\mathcal{Z}_{p,x}$ as the subset of elements $u$ in $V$ such that: \begin{align} &\forall \: i \in \: \{1, \dots, r\}, \: (F_u)_{p_i} \subset \pi_{{k}}(x_i), \: \mathrm{where} \: \pi_{{k}}(x_i) \in \Gr(n_k,n); \label{eq: mu(x)=xi}\\
&\forall \mathfrak{p} \in \mathbb{P}^1, \: (F_u)_\mathfrak{p} \subset ({E_{k+1}})_{\mathfrak{p}}, \label{eq: muin X}
\end{align}
where we denote by $\pi_k(x_i)$ the vector subspace of $\mathbb{C}^n$ of dimension $n_k$ associated with the point $\pi_k(x_i)$ in $\Gr(n_k,n)$, and by $({E_{k+1}})_{\mathfrak{p}}$ the fiber of the vector bundle $E_{k+1}$ at $\mathfrak{p}$. Note that, when $\Phi_u$ defines a morphism $\mathbb{P}^1 \rightarrow \Gr(n_k,n)$, Equations (\ref{eq: mu(x)=xi}) and (\ref{eq: muin X}) are equivalent to Equations $2.$ and $3.$ here above. Indeed, by construction of $\Phi_u$, $(F_u)_\mathfrak{p}$ satisfies $(E_{k-1})_{\mathfrak{p}} \subset (F_u)_\p$ for all $\p$ in $\mathbb{P}^1$, hence if $u$ satisfies Equation (\ref{eq: muin X}), $u$ also verifies: \begin{equation*}
\forall \p \in \mathbb{P}^1, \: (E_1)_\p \subset \dots \subset (E_{k-1})_\p \subset (F_u)_\p \subset (E_{k+1})_\p \subset \dots \subset (E_m)_\p \subset \mathbb{C}^n,
\end{equation*}
hence $\Phi_u$ defines a morphism to the flag variety $X$.\\

\subsection{Irreducibility and rationality of $\mathcal{Z}_{p,x}$.}
We denote by $F_\mu$ the subsheaf of $\O_{\P^1}^{\oplus r}$ associated with an element $\mu$ in $V \simeq Hom(E_{k-1} \oplus_i \O_{\P^1}(-a_i), \O_{\P^1}^{\oplus r})$, and by $(F_\mu)_\p$ the fiber over a point $\p \in \P^1$ of $F_\mu$. We consider the usual construction of $\mathbb{P}^1$ by gluing two affine schemes together. Let $U_1 := \Spec (\mathbb{C}[s])$, $U_2 := \Spec (\mathbb{C}[t])$, $U_{12} := \Spec (\C[s,s^{-1}])=U_1 \setminus \{0\}$, and $U_{21} := \Spec (\C[t,t^{-1}])=U_2 \setminus \{0\}$. Then $\mathbb{P}^1$ can be obtained by gluing $U_1$ and $U_2$ together along $U_{12} \simeq U_{21}$, where we consider the following isomorphism:\begin{align*}
\C[s, s^{-1}] \simeq \C[t,t^{-1}], \:\:\: s \rightarrow t^{-1}.
\end{align*}
For $d \in \mathbb{N}$, we view $\O_{\P^1}(-d)$ as the line bundle satisfying $\O_{\P^1}(-d)_{| U_i} \simeq U_i \times \mathbb{A}^1$ for $i=1,2$, obtained by identifying $U_{12} \times \A^1 = \Spec[s,s^{-1}] \times \A^1$ and $U_{21}= \Spec[t,t^{-1}] \times \A^1$ via the following isomorphism: \begin{equation*}
(s,v) \rightarrow (s^{-1}, s^d v).
\end{equation*}
Equivalently we endow $\O_{\P^1}(-d)$ with two trivializations on $U_1$ and $U_2$ respectively. We consider here the natural bijection $V \simeq \oplus_j (\C_{a_j}[x])^n$, associating to a section in $(\C_{a_j}[x])^n$ the morphism $\O_{\P^1}(-a_j) \rightarrow \O_{\P^1}^{\oplus n}$ described by: \begin{equation*} \begin{aligned} &\Spec \C[s] \times \A^1 \rightarrow \O_{\P^1}^{\oplus n} &&  (s,v) \rightarrow (s, P_{j}(s) v)  \\ & \Spec \C[s^{-1}] \times \A^1 \rightarrow \O_{\P^1}^{\oplus n} &&  (s^{-1},v) \rightarrow (s^{-1}, s^{-a_j} P_{j}(s) v) \end{aligned} \end{equation*}
We may hence view a point $\mu$ in $\mathcal{Z}_{p,x}$ as a set of $n_k - n_{k-1}$ vectors $P_j$ with coefficients in $\C_{a_j}[x]$. In order to simplify notations, we keep the dependence of the polynomials $P_j$ on the point $\mu$ in $V$ associated with them implicit.
\begin{lemme}\label{lem: Z is unirational}
	\begin{enumerate}
		\item There exist $N \in \mathbb{N}$, $p_1$, $\dots$, $p_N \in \mathbb{P}^1$, such that $\mathcal{Z}_{p,x}$ is the subspace of elements $\mu$ in $V$ satisfying the following equations: \begin{align*} &\forall i \in \{1, \dots, r\}, \: (F_\mu)_{p_i} \subset \pi_{{k}}(x_i), \: \mathrm{where} \: \pi_{{k}}(x_i) \in \Gr(n_k,n); \\
		& \forall \iota \in \{1, \dots, N\}, \:  (F_\mu)_{p_\iota} \subset (E_{k+1})_{p_\iota};
		\end{align*}
		\item Consider a point $\mathfrak{p}$ in $\mathbb{P}^1$. We denote by $V((E_{k+1})_{\mathfrak{p}})$ the subset of elements $\mu$ in $V$ satisfying $(E_{k+1})_{\mathfrak{p}} \supset (F_\mu)_{\p}$. 
		Then $V((E_{k+1})_\p)$ is a linear subspace of $V$;
		\item $\mathcal{Z}_{p,x}$ is a linear subspace of $V$.	  
	\end{enumerate}
\end{lemme}

\begin{proof} 
	\begin{enumerate}
		\item  By definition, the vectors $P_{j}$ with coefficients in $\C_{a_j}[x]$ associated with $\mu \in V$ satisfy:
		$$\forall \p \in \P^1, \: (F_\mu)_\p \simeq \langle (E_{k-1})_\p, (P_{n_{k-1}+1}(\p), \dots, P_{n_k}(\p)) \rangle,$$
		where we denote by $F_\mu$ the subvector bundle of $\O_{\P^1}^{\oplus r}$ associated with $\mu$. Denote by $Q_j$ the $n_{k+1}$ homogeneous polynomials associated with the inclusion of vector bundles $E_{k+1} \subset \O_{\P^1}^{\oplus r}$.  Then there exist $N >0$, and points $p_1$, $\dots$, $p_N$ in $\P^1$ such that the following statements are equivalent: \begin{align}
		&\forall \: \p \in \mathbb{P}^1, \: (F_\mu)_\p \subset ({E_{k+1}})_{\p} \nonumber\\ \Leftrightarrow & 	\forall \: \p \in \mathbb{P}^1, \: \langle P_{n_{k-1}+1}(\p), \dots, P_{n_k}(\p) \rangle \subset \langle Q_1(\p), \dots, Q_{n_{k+1}}(\p) \rangle \label{eq: Pi c Qi}  \\
		\Leftrightarrow & \forall \iota \in \{1, \dots, N\}, \: \langle P_{n_{k-1}+1}(p_\iota), \dots, P_{n_k}(p_\iota) \rangle \subset \langle Q_1(p_\iota), \dots, Q_{n_{k+1}}(p_\iota) \rangle  \nonumber
		\end{align} 
		Indeed, $\mu$ satisfies the homogeneous equations (\ref{eq: Pi c Qi}) iff the polynomials $P_j$ associated with $\mu$ satisfy (\ref{eq: Pi c Qi}) at $N$ different values of $p_\iota=[x_\iota,y_\iota]$ for $N$ large enough. Notice $N$ only depends on the degrees $d_k$ and $d_{k+1}$, which are fixed here. 
		
		From now on, we fix such an $N$, and fix points $(p_{r+1}, \dots , p_N)$ general in $\mathbb{P}^1$. Equation (\ref{eq: muin X}) is thus equivalent to :
		\begin{equation}\label{eq: condition mu(pi)} \forall \iota \in \{1, \dots, N\}, \:  (F_\mu)_{p_\iota} \subset ({E_{k+1}})_{p_i}  \end{equation} 	
		
		\item We set $\p:= [x :y]$.  We will denote by $\langle E, F \rangle$ the vector subspace of $\C^n$ generated by any two vector subspaces $E$ and $F$ of $\C^n$.
		The vector subspace $(E_{k+1})_{\p}$ is described by linear equations. Hence:  \begin{align*} V((E_{k+1})_{\p}) &= \left\{ \mu=(P_{i,j}) \in V \mid  \langle (E_{k-1})_{\p}, (F_\mu)_{p} \rangle \subset (E_{k+1})_{\p} \right\} \\ &= \left\{ \mu=(P_{i,j}) \in V \mid  \langle (E_{k-1})_{\p},  P_{n_{k-1} +1}(\p), \dots, P_{n_k}(\p) \rangle  \subset (E_{k+1})_{\p} \right\}\end{align*} is the intersection of linear equations on the coefficients of $P_j$, $n_{k-1}+1 \leq j \leq n_k$.
		\item Fix $i$ in $\{1 \dots, N\}$. Denote by $\pi_{\widehat{k}}(x_i)$ the image of $x_i$ in the projected flag variety $X_{\widehat{k}}$. 
		The subset $V_{p_i}$ of elements $\mu$ in $V$ satisfying $(F_\mu)_{p_i} \subset \pi_{{k}}(x_i)$ can be described by:  \begin{align*}  V_{p_i}:&=
		\left\{ \mu=(P_{i,j}) \in V \mid  \langle  (F_\mu)_{p_i} \rangle \subset  \pi_k(x_i) \right\}\\
		& =  \left\{ \mu=(P_{i,j}) \in V \mid   \langle P_{n_{k-1} +1}(p_i), \dots, P_{n_k}(p_i) \rangle \subset  \pi_k(x_i) \right\}.\end{align*}
		Since the vector space  $\pi_k(x_i) $ is defined by linear equations, $V_{p_i}$ is defined by linear equations on the coefficients of the polynomials $(P_{i,j})$; $V_{p_i}$ is a linear subspace of $V$.
		
		Finally, according to $1.$ the variety $\mathcal{Z}_{p,x}$ is the intersection of the linear spaces $V_{p_i}$, for $1 \leq i \leq r$, with the linear spaces $V((E_{k+1})_{p_\iota})$, for $r < \iota \leq N$.
		Hence $\mathcal{Z}_{p,x}$ is a linear subspace of $V$.
	\end{enumerate}
\end{proof}
According to Lemma \ref{lem: Z is unirational}, $\mathcal{Z}_{p,x}$ is a linear subspace of $V$, and hence is an irreducible rational variety.
\subsection{Image of $\mathcal{Z}_{p,x} \dashrightarrow \mathcal{M}_{0,r}(X, \mathbf{d})$.}
First, denote by $V^{inj}$ the subset of elements of $V$ corresponding to injective morphisms $E_{k-1} \oplus_i \O_{\P^1}(-a_i) \hookrightarrow \O_{\P^1}^{\oplus r}$. The complementary of $V^{inj}$ is described by a set of polynomial equations. Indeed, for $u=(P_j)_{n_{k-1}+1\leq j \leq n_k}$ in $V$ the associated morphism is injective iff $\langle (E_{k-1})_\p, P_{n_{k-1}+1}(\p), \dots, P_{n_k}(\p) \rangle$ is a vector space of dimension $n_k$ for all $\p \in \P^1$. Hence $V^{inj}$ is an open subset of $V$. Call $$\mathcal{Z}_{p,x}^{inj} := \mathcal{Z}_{p,x} \cap V^{inj}.$$ 
By the functorial definition of flag varieties, each element $u$ in $\mathcal{Z}_{p,x}^{inj}$ defines a degree $\d$ morphism $f_u: \P^1 \rightarrow X$, which by construction of $\mathcal{Z}_{p,x}$ is associated to a point in $\Pi_{p,x}$. Furthermore, by construction of $\mathcal{Z}_{p,x}$, any point $q$ in $\Pi_{p,x}$ whose associated flag of vector bundles is balanced is the image $(F_u)$ of an element $u$ in $\mathcal{Z}_{p,x}^{inj}$.
Hence, according to Lemma \ref{lem: decomposition into a_i} $ii)$, for $(p,x)$ general in $(\Pi_{\widehat{k}} \times ev)(\overline{{\mathcal{M}}_{0,r}}(X, \mathbf{d}))$, the variety $\mathcal{Z}_{p,x}^{inj}$ dominates $\Pi_{p,x}$.

\subsection{Algebraicity of $\mathcal{Z}_{p,x}^{inj} \rightarrow \Pi_{p,x}$.}
Denote by $\psi: \mathcal{Z}_{p,x}^{inj} \times \mathbb{P}^1 \rightarrow X$ the function sending an element $(z,\p)$ in $\mathcal{Z}_{p,x}^{inj} \times \mathbb{P}^1$ to the point of $X$ associated with $(F_z)_\p$. According to Lemma \ref{lem : morphisme Y x P dans M(X)} (Appendix A), $\psi$ defines a morphism $\Psi : \mathcal{Z}_{p,x}^{inj} \rightarrow \overline{{\mathcal{M}}_{0,r}}(X, \mathbf{d})$, which dominates $\Pi_{p,x}$ according to here above.

\subsection{Proof of Proposition \ref{prop : geometry general fiber pr1W}.} Let $(p,x)$ be a general point in  $(\Pi_{\widehat{k}} \times ev)(\overline{{\mathcal{M}}_{0,r}}(X, \mathbf{d}))$. We have obtained a dominating morphism $\Psi : \mathcal{Z}_{p,x}^{inj} \rightarrow \Pi_{p,x}$. Since $\mathcal{Z}_{p,x}^{inj}$ is a dense open subset of the irreducible rational variety $\mathcal{Z}_{p,x}$, $\mathcal{Z}_{p,x}^{inj}$ is an irreducible rational variety.
Hence $\Pi_{p,x}$ is an irreducible unirational variety.

\subsection{Proof of Corollary \ref{cor : geometry general fiber pr1 ev-1(x)}.} For an element $x$ in $X^r$, we denote by $\Pi_x$ the restriction of the forgetful map $\overline{{\mathcal{M}}_{0,r}}(X, \d) \to \overline{{\mathcal{M}}_{0,r}}(X_{\widehat{k}}, (\pi_{\widehat{k}})_* \d)$ to $ev^{-1}(x)$. First, note that since the map $\Pi_{\widehat{k}} \times ev$ is surjective, according to Proposition \ref{prop : W is irreducible} $1.$ for $x$ general in $X^r$ we have $\Pi(ev^{-1}(x)) = ev^{-1}(\pi_{\widehat{k}}(x))$. Hence for $x$ general in $X^r$, the map $\Pi_x : ev^{-1}(x) \to ev^{-1}(\pi_{\widehat{k}}(x))$ is surjective. 

According to Proposition \ref{prop : geometry general fiber pr1W}, there exists a dense open subset $U$ of $(\Pi_{\widehat{k}} \times ev)(\overline{{\mathcal{M}}_{0,r}}(X, \mathbf{d})) = \overline{{\mathcal{M}}_{0,r}}(X_{\widehat{k}}, (\pi_{\widehat{k}})_* \d) \times_{{X_{\widehat{k}}}^r} X^r$ such that for all elements $(m,x)$ in $U$, the fiber $(\Pi \times ev)^{-1}(m,x)=\Pi^{-1}(m) \cap ev^{-1}(x)$ is a unirational variety. Denote by $p_1: \overline{{\mathcal{M}}_{0,r}}(X_{\widehat{k}}, (\pi_{\widehat{k}})_* \d) \times_{{X_{\widehat{k}}}^r} X^r \to \overline{{\mathcal{M}}_{0,r}}(X_{\widehat{k}}, (\pi_{\widehat{k}})_* \d)$ the first projection and by $p_2: \overline{{\mathcal{M}}_{0,r}}(X_{\widehat{k}}, (\pi_{\widehat{k}})_* \d) \times_{{X_{\widehat{k}}}^r} X^r \to X^r$ the second projection. Note that for all $m$ in $p_1(p_2^{-1}(x) \cap U)$ the fiber $\Pi_x^{-1}(m) = \Pi^{-1}(m) \cap ev^{-1}(x) = (\Pi \times ev)^{-1}(m,x)$ is a unirational variety. Furthermore, note that for $x$ general in $X^r$, $p_2^{-1}(x) \cap U$ is a dense open subset of $p_2^{-1}(x)$, and hence $p_1(p_2^{-1}(x) \cap U)$ is a dense open subset of $p_1(p_2^{-1}(x))=\Pi(ev^{-1}(x))$. Hence for $x$ general in $X^r$ the general fiber of $\Pi_x$ is a unirational variety.
\section{Proof of Theorem \ref{th : RC fibrations bewteen W varieties}.}\label{sec: proof main theorem}
Let $J= \{j_{1}, \dots, j_{M}\}$ be a set of $M$ integers satisfying $0 < j_1 < \dots < j_M <n$ such that $I$ is contained in $J$, let $\d=(d_{1}, \dots,  d_{M})$ be an element in $\mathbb{N}^{M}$ and $1 \leq \mu \leq M$ be an integer such that the collection $\{I, J, \d\}$ is stabilized with respect to $\mu$ in the sense of Definition \ref{def: stabilized collection}. 
Note that according to Proposition \ref{prop : W is irreducible} it is enough to prove that the general fiber of \begin{align}\label{eq: def Pixev} \Pi_{I/J} \times ev : && \overline{{\mathcal{M}}_{0,r}}(Fl_J, \d) \rightarrow \overline{{\mathcal{M}}_{0,r}}(Fl_I, (\pi_{I/J})_* \d) \times_{(Fl_I)^r} (Fl_J)^r\end{align} is rationally connected. We prove this in two times. We first prove that the general fiber of the morphism $\Pi_{J_{\geq \mu}/J} \times ev$  is rationally connected, where $\Pi_{J_{\geq \mu}/J}$ is the morphism induced by forgetting vector spaces with indices left of $\mu$ not associated with elements in $I$. We then prove that the general fiber of the morphism $\Pi_{I/J_{\geq \mu}} \times ev$  is rationally connected. We finally deduce Theorem \ref{th : RC fibrations bewteen W varieties} from these two facts.

\subsection{Preliminary results on the base change of $\Pi_{I/J} \times ev$.}\label{subsec: base change Pixev} Let $K=\{k_1, \dots, k_{m'}\}$ be a set of integers satisfying $0 < k_1 < \dots < k_{m'}$ such that $K$ is contained in $J$ and contains $I$. Our goal here is to prove that the base change of the morphism $\Pi_{I/J} \times ev$ defined by (\ref{eq: def Pixev}) induced by $Fl_K \rightarrow Fl_J$ is a morphism of equidimensional projective varieties, whose general fiber satisfies the same properties as the general fiber of $\Pi_{I/J} \times ev$. Since the projection morphism $Fl_K \rightarrow Fl_I$ factors through $Fl_J \rightarrow Fl_I$, we are considering the following morphism: \[\begin{tikzcd} \overline{{\mathcal{M}}_{0,r}}(Fl_J, \d) \times_{(Fl_J)^r} (Fl_K)^r  \arrow{r} & \left(\overline{{\mathcal{M}}_{0,r}}(Fl_I, (\pi_{I/J})_* \d) \times_{(Fl_I)^r} (Fl_J)^r \right)   \times_{(Fl_J)^r} (Fl_K)^r \arrow{d}[below, rotate=90]{\huge \sim} \\ & \overline{{\mathcal{M}}_{0,r}}(Fl_I, (\pi_{I/J})_* \d) \times_{(Fl_I)^r} (Fl_K)^r
\end{tikzcd}\]
First, note that $\overline{{\mathcal{M}}_{0,r}}(Fl_J, \d) \times_{(Fl_J)^r} (Fl_K)^r$ is a reduced scheme of pure dimension the expected dimension. Indeed, according to Kleiman's transversality theorem for $g$ general in $(\SL_n)^r$ the projective scheme $\overline{{\mathcal{M}}_{0,r}}(Fl_J, \d) \times_{(Fl_J)^r} g \cdot (Fl_K)^r$ is an equidimensional reduced scheme. Furthermore, since the morphism $(Fl_J)^r \rightarrow (Fl_I)^r$ is $G:=(SL_n)^r$-equivariant, the following diagram is commutative \[\begin{tikzcd} \overline{{\mathcal{M}}_{0,r}}(Fl_J, \d) \times_{(Fl_I)^r} \left( G \times (Fl_J)^r \right) \arrow{d}{p} \arrow{r}{\huge \sim} & G \times  \left(\overline{{\mathcal{M}}_{0,r}}(Fl_J, \d) \times_{(Fl_I)^r} (Fl_J)^r \right) \arrow{d}{\pi}  \\
G  \arrow{r}{\huge \sim} & G \end{tikzcd}\]
where $p$ and $\pi$ are the obvious projections. This induces an isomorphism \begin{align*}p^{-1}(g) = \overline{{\mathcal{M}}_{0,r}}(Fl_J, \d) \times_{(Fl_I)^r} \left(g \cdot (Fl_J)^r \right) & \simeq \pi^{-1}(g) \\& \simeq g \cdot  \left(\overline{{\mathcal{M}}_{0,r}}(Fl_J, \d) \times_{(Fl_I)^r} (Fl_J)^r \right) \\ & \simeq  \overline{{\mathcal{M}}_{0,r}}(Fl_J, \d) \times_{(Fl_I)^r} (Fl_J)^r  \end{align*}
which allows us to deduce the equidimensionality and reducibility of $\overline{{\mathcal{M}}_{0,r}}(Fl_J, \d) \times_{(Fl_J)^r} (Fl_K)^r$ from the equidimensionality and reducibility of $\overline{{\mathcal{M}}_{0,r}}(Fl_J, \d) \times_{(Fl_J)^r} g \cdot (Fl_K)^r$.

Finally, Lemma \ref{lem: flat base change perserves (P)} here under ensures that the general fiber of the flat base change of $\Pi_{I/J} \times ev$ satisfies the same properties as the general fiber of $\Pi_{I/J} \times ev$.
\begin{lemme}\label{lem: flat base change perserves (P)}
	Let $S$ be a scheme, let $S' \rightarrow S$, $M \rightarrow S$ and $M' \rightarrow S$ be schemes over $S$, and $M \rightarrow M'$ be a morphism of schemes over $S$. 
	Suppose the general fiber of $M \rightarrow M'$ satisfies a property $(P)$, and $S' \rightarrow S$ is flat. Then the general fiber of $M_{S'} \rightarrow M'_{S'}$ satisfies $(P)$.
\end{lemme}

\begin{proof}
	If the morphism $M \rightarrow M'$ is not surjective, it factors through $M \rightarrow \mathrm{Im} M$ and then $\mathrm{Im} (M_{S'}) \simeq (\mathrm{Im} M)_{S'}$; hence we can assume $M \rightarrow M'$ surjective. Denote by $U$ the dense open subset of $M'$ such that the fiber of $M \rightarrow M'$ over $\{m\}$ in $U$ satisfies $(P)$. Note that, since $S' \rightarrow S$ is a flat morphism, $U \times_S S'$ is a dense subset of $M' \times _S S'$. Indeed, if we denote by $p_1 : M' \times_S S' \rightarrow M'$ the first projection, acccording to \cite{EGAIV2} Proposition 2.3.4, all irreducible components of $p_1^{-1}(M')$ dominate $M'$; hence the image of each irreducible component of $M'$ by $p_1$ has a non empty intersection with $U$, hence each irreducible component of $M'_{S'}$ has a non empty intersection with $U_{S'}$. Then for $(m,s)$ in $U \times_S S' \subset M' \times_S S'$, we have: \begin{align*}
	(M \times_{S} S')_{(m,s)} &= M \times_{S} \{s\} \times_{M'} \{m\}\\
	&= (M \times_{M'} \{m\}) \times_{S} \{s\} \simeq M_m,
	\end{align*}
	where $M_m$ satisfies $(P)$ since $m$ lies in $U$. Hence the general fiber of $M_{S'} \rightarrow M'_{S'}$ satisfies $(P)$.
\end{proof}

\subsection{Forgetting vector spaces left of $V_{j_{\mu}}$.}\label{subsec: forgetting left V_i} Recall a point in $Fl_J$ corresponds to a flag of vector spaces: \begin{equation*}
\{0\} \subset V_{j_{1}} \subset \dots \subset V_{j_{M}} \subset \C^n.
\end{equation*}
Denote by $\mu'$ the integer such that $i_{\mu'} = j_\mu$. A point in $Fl_{J_{\geq \mu}}$ corresponds to a flag of vector spaces: \begin{equation*}
\{0\} \subset V_{i_1} \subset \dots \subset V_{i_{\mu'}}=V_{j_\mu} \subset V_{j_{\mu +1}} \subset \dots \subset V_{j_M} \subset \C^n.
\end{equation*}
For $1 \leq k < \mu$, we name $X_{k} := (Fl_{J_{\geq k}})^r$, $\d^{k} := (\pi_{J_{\geq k}/J})_* \d$ and $M_{k}:=\overline{{\mathcal{M}}_{0,r}}(X_{k}, \d^{k})$. Forgetting all vector spaces left of $V_{j_\mu}$ that are not of the type $V_{i_k}$ induces a morphism $\Pi_{J_{\geq \mu}/J} \times ev: \overline{{\mathcal{M}}_{0,r}}(Fl_J,\d) \rightarrow M_{\mu} \times_{X_{\mu}} (Fl_J)^r$. We consider here the general fiber of this map.
\begin{proposition}
	Let $\{I,J, \d\}$ be a stabilized collection. \begin{enumerate}[label=\roman*)]
		\item The general fiber of the map $\Pi_{J_{\geq \mu}/J} \times ev$ is a unirational variety.
		\item  For $x$ general in $(Fl_J)^r$, the fiber $ev^{-1}(x) \in \overline{{\mathcal{M}}_{0,r}}(Fl_J,\d)$ is a tower of unirational fibrations over the fiber $ev^{-1}(\pi_{J_{\geq \mu}}(x)) \in M_\mu$.
	\end{enumerate} 
\end{proposition}
\begin{proof} We proceed by iteration on $k$.
	\begin{enumerate}[label=\roman*)]
		\item  Let  $1 \leq  k  < \mu$ such that the general fiber of $\Pi_{J_{\geq k/J}} \times ev: M_J \rightarrow M_{k} \times_{X_{k}} (Fl_J)^r$ is rationally connected. Suppose $j_{k+1}$ is not contained in $I$. Recall $$\d=(d_{1}, \dots, d_{M}).$$ 
		We denote by $$(\pi_{I/J})_* \d =(d'_1, \dots, d'_m).$$ Denote by $k'$ the largest integers such that $i_{k'} < j_k$. Note that \begin{itemize}
			\item  The degree $\d^{k}$ is given by the set  $$\d^{k}=(d_{1}', \dots, d_{k'}', d_{k''}^k, \dots,  d_{M'}^k),$$ obtained from $\d$ by removing all integers $d_{s}$ such that $s < \mu$ and there exists no integer $s'$ satisfying $i_{s'}= j_{s}$. 
			\item The degree $\d^{k +1}$ is given by the sequence  $$\d^{k+1}=(d_{1}', \dots, d_{k'}',d_{k''+1}^k \dots d_{M'}^k),$$ which is obtained from the set $\d^{k}$ by forgetting the integer $d_{k''}^k$.
		\end{itemize}
		Note that according to Definition \ref{def: stabilized collection} for all $p<k$, $$d^{k}_{k''} \geq \ceil{(d'_p - d'_{p-1})/(i_p-i_{p-1})} j_k.$$
		Then according to Proposition \ref{prop : geometry general fiber pr1W}  the general fiber of $$\Pi_{J_{\geq k+1}/J_{\geq k}} \times ev: M_{k} \rightarrow (\Pi_{J_{\geq k+1}/J_{\geq k}} \times ev) (M_{k}) = M_{k+1} \times_{X_{k+1}} X_{k}$$ is unirational, and hence rationally connected. Furthermore, according to Definition \ref{def: stabilized collection} the morphism surjects onto $M_{k+1} \times_{X_{k+1}} X_{k}$. Then according to Subsection \ref{subsec: base change Pixev} the general fiber of the following morphism 
		\begin{align*}M_{k} \times_{X_{k}} (Fl_J)^r \rightarrow M_{k+1} \times_{X_{k+1}} (Fl_J)^r \end{align*} is also rationally connected. 
		Hence according to Theorem \ref{th: ex stability under composition} by composition the general fiber of $\Pi_{J_{\geq k+1}/J} \times ev: M_J \rightarrow M_{k+1} \times_{X_{k+1}} (Fl_J)^r$ is rationally connected. In the same way, if $j_{k+1}$ is contained in $I$, the general fiber of $\Pi_{J_{\geq k+2}/J} \times ev: M_J \rightarrow M_{k+2} \times_{X_{k+2}} (Fl_J)^r$ is rationally connected.
		
		\item Let  $1 \leq  k  < \mu$ such that for $x$ general in $(Fl_J)^r$, the fiber $ev^{-1}(x) \in \overline{{\mathcal{M}}_{0,r}}(Fl_J,\d)$ is a tower of unirational fibrations over the fiber $ev^{-1}(\pi_{J_{\geq k}/J}(x)) \in M_k$ is unirational. Suppose $j_{k+1}$ is not contained in $I$. Corollary \ref{cor : geometry general fiber pr1 ev-1(x)} and Proposition \ref{prop : W is irreducible} imply that for $x$ general in $(Fl_J)^r$, the fiber $ev^{-1}(x) \in \overline{{\mathcal{M}}_{0,r}}(Fl_J,\d)$ is a tower of unirational fibrations over the fiber $ev^{-1}((\pi_{J_{\geq k+1}/J}(x)) \in M_{k+1}$.
	\end{enumerate}
\end{proof}

\subsection{Forgetting vector spaces right of $V_{j_{\mu}}$.}\label{subsec: forgetting right Vi} We deduce here the rational connectedness of the general fiber of the morphism $\Pi_{I/ J_{\geq \mu}} \times ev$ from Subsection \ref{subsec: forgetting left V_i} by duality. For a set $K= \{k_1, \dots, k_m\}$ of $m$ integers satisfying $0 < k_1 < \dots < k_m <n$, we call \textit{dual partition} of $K$ the set $$K^*:= \{n-k_m, \dots, n-k_1\}.$$ Consider the natural transformation of functors sending a flag $\mathcal{F}_S:=  E_1 \subset \dots \subset  E_m \subset \O_S^{\oplus n}$ of vector bundles over a scheme $S$ to the dual flag $\mathcal{F}_S^*:= (\O_S^{\oplus n}/ E_m)^* \subset \dots \subset (\O_S^{\oplus n}/E_1)^* \subset (\O_S^{\oplus n})^* \simeq \O_S^{\oplus n}$, where we have chosen an isomorphism $(\C^n)^* \simeq \C^n$. This natural transformation of functors induces an isomorphism $Fl_K \simeq Fl_{K^*}$, and we will call $Fl_{K^*}$ the flag variety \textit{dual to} $Fl_K$. Note that this induces here an isomorphism between the flag variety $Fl_{J \geq \mu}$ parametrizing flags of vector spaces \begin{equation*}
V_{i_1} \subset \dots \subset V_{i_{\mu'}} \subset V_{ j_{\mu+1}} \subset \dots \subset V_{j_M} \subset \C^n.
\end{equation*}
and the dual flag variety $Fl_{(J_{\geq \mu})^*}$ parametrizing flags of vector spaces \begin{equation*}
V_{n- j_{M}} \subset \dots \subset V_{n-j_{\mu +1}}  \subset V_{n-i_{\mu'}} \subset \dots \subset V_{n-i_1} \subset \C^n.
\end{equation*}
This isomorphism sends a curve of class $\d=(d_1, d_2 \dots, d_m)$ in $E(Fl_K) \simeq \mathbb{N}^m$ onto a curve of class $\d^*=(d_m, \dots, d_2, d_1)$ in $E(Fl_{K^*}) \simeq \mathbb{N}^m$. Furthermore, note that, since $\{I,J , \d\}$ is stabilized with respect to $\mu$, according to Lemma \ref{lem: Geometry dual stabilized collection} the collection $\left\{I, (J_{\geq \mu})^*, \left(\left(\pi_{I/J_{\geq \mu}}\right)_* \d \right)^* \right\}$ is stabilized with respect to $n$ in the sense of Definition \ref{def: stabilized collection}. Hence, according to Subsection \ref{subsec: forgetting left V_i} here above, the general fiber of the morphism $\Pi_{I^*/(J_{\geq \mu})^*} \times ev$  is rationally connected. Finally, the following commutative diagram allows us to deduce the rational connectedness of the general fiber of $\Pi_{I/J_{\geq \mu}} \times ev$ from the rational connectedness of the general fiber of the morphism $\Pi_{I^*/(J_{\geq \mu})^*} \times ev$. \[  \begin{tikzcd}[column sep =huge]
\overline{{\mathcal{M}}_{0,r}}(Fl_{J_{\geq \mu}}, (\pi_{{J_{\geq \mu}}/J})_* \d)  \arrow{r}{\Pi_{I/J_{\geq \mu}} \times ev} \arrow{d} & \overline{{\mathcal{M}}_{0,r}}(Fl_I, (\pi_{I/J})_* \d) \underset{(Fl_{I})^r} \times (Fl_{J_{\geq \mu}})^r \arrow{d} \\
\overline{{\mathcal{M}}_{0,r}}(Fl_{(J_{\geq \mu})^*}, ((\pi_{{J_{\geq \mu}}/J})_* \d)^*)  \arrow{r}{\Pi_{I^*/(J_{\geq \mu})^*} \times ev} \arrow{u} & \overline{{\mathcal{M}}_{0,r}}(Fl_{I^*}, ((\pi_{I/J})_* \d)^*) \underset{(Fl_{I^*})^r} \times (Fl_{(J_{\geq \mu})^*})^r \arrow{u}
\end{tikzcd} \]
Note that the isomorphism $\overline{{\mathcal{M}}_{0,r}}(Fl_{J_{\geq \mu}}, (\pi_{{J_{\geq \mu}}/J})_* \d)  \simeq \overline{{\mathcal{M}}_{0,r}}(Fl_{(J_{\geq \mu})^*}, ((\pi_{{J_{\geq \mu}}/J})_* \d)^*) $ is induced by the isomorphism $Fl_{J \geq \mu} \simeq Fl_{(J \geq \mu)^*}$ and the isomorphism $$\overline{{\mathcal{M}}_{0,r}}(Fl_I, (\pi_{I/J})_* \d) \underset{(Fl_{I})^r} \times (Fl_{J_{\geq \mu}})^r \simeq \overline{{\mathcal{M}}_{0,r}}(Fl_{I^*}, ((\pi_{I/J})_* \d)^*) \underset{(Fl_{I^*})^r} \times (Fl_{(J_{\geq \mu})^*})^r$$ is induced by the isomorphisms $Fl_I \simeq Fl_{I^*}$ and $Fl_{J \geq \mu} \simeq Fl_{(J \geq \mu)^*}$.
\begin{lemme}\label{lem: Geometry dual stabilized collection}
	Let $J= \{j_1, \dots, j_M\}$ be a set of $M$ integers satisfying $0 < j_1 < \dots < j_M <n$ and let $I=\{i_1, \dots, i_m\}$ be a subset of $J$. Let $\d=(d_1, \dots, d_M)$ be an element in $\mathbb{N}^M$ and let $1 \leq \mu \leq M$ such that the collection $\{I,J,\d\}$ is stabilized with respect to $\mu$. Then the collection $\left\{I^*, (J_{\geq \mu})^*, \left((\pi_{J\geq \mu/J})_* \d\right)^*\right\}$ is stabilized with respect to $n$. 
\end{lemme}

\begin{proof}
	We number by $I^*:= \{i_1^*, \dots, i_m^*\}$ and $J^*=\{j_1^*, \dots, j_M^*\}$ the dual partitions of $I$ and $J$, where for any $1 \leq k \leq m$, $i_k^*=n-i_{m-k+1}$ and for any $1 \leq k \leq M$, $j_k^*=n-j_{M-k+1}$. Recall $Fl_{I^*}$ is the flag variety parametrizing flags of vector spaces $$V_{n-i_m} \subset \dots \subset V_{n-i_1} \subset \C^n$$ and $Fl_{(J \geq \mu)^*}$ is the flag variety parametrizing flags of vector spaces $$V_{n- j_{M}} \subset \dots \subset V_{n-j_{\mu +1}}  \subset V_{n-i_{\mu'}} \subset \dots \subset V_{n-i_1} \subset \C^n.$$ We call $\d':=(\pi_{I/J})_* \d=(d'_1, \dots, d'_m)$. We call $\d^*=(d_1^*, \dots, d_M^*)$ and ${\d'}^*=({{d'}^*}_1, \dots, {{d'}^*}_m)$ the dual classes, where for $1 \leq k \leq M$, $d_k^*=d_{M-k+1}$ and for $1 \leq k \leq m$, ${d'}^*_k = d'_{m-k+1}$. 
	Let $\mu \leq k \leq M$ such that $j_{M-k+1}^*=n-j_k$ is not contained in $I^*$, i.e. let $\mu < k \leq M$ such that $j_k$ is not contained in the set $I^*=I$. Since the collection $\{I,J,\d\}$ is stabilized with respect to $\mu$, the following properties are satisfied. \begin{itemize}
		\item Let $k'$ be the smallest integer such that $i_{k'} > j_k$. Note that $m-k'+1$ is the largest integer such that $i^*_{m-k'+1}=n-i_{k'} < j^*_{M-k+1}= n-j_k$. Then for any $k' < p \leq m$, i.e. for any $1 \leq m-p+1 \leq m-k'+1$, we have $$d^*_{m-k+1}=d_k \geq \ceil{(d'_p-d'_{p+1})/(i_{p+1}-i_p)} (n-j_k)= \ceil{({d'}^*_{m-p+1}-{d'}^*_{m-p})/(i^*_{m-p+1}-i^*_{m-p})} j^*_{M-k+1}.$$
		\item Since the morphism $\Pi_{J_{\mu \leq k}} \times ev$ is surjective, the morphism $\Pi_{{(J_{ \geq \mu})^*}_{\geq M-k+1}} \times ev$ is surjective. 
	\end{itemize}
	Note that  the surjectivity of $\Pi_{{(J_{ \geq \mu})^*}_{\geq M-k+1}} \times ev$ is implied by the following commutative diagram \[ \small \begin{tikzcd}
	\overline{{\mathcal{M}}_{0,r}}(Fl_{J_{\mu \leq k}}, (\pi_{{J_{\mu \leq k}}/J})_* \d)  \arrow{r}{\Pi_{J_{\mu \leq k}} \times ev} \arrow{d} & \overline{{\mathcal{M}}_{0,r}}(Fl_{J_{\mu \leq k-1}}, (\pi_{{J_{\mu \leq k-1}}/J})_* \d) \underset{(Fl_{J_{\mu \leq k-1}})^r} \times (Fl_{J_{\mu \leq k-1}})^r \arrow{d} \\
	\overline{\mathcal{M}_{0,r}}(Fl_{{(J_{ \geq \mu})^*}_{\geq M-k+1}}, ((\pi_{J_{\mu \leq k}/J})_*\d)^*) \arrow{u} \arrow{r }[swap, yshift=-2ex]{\Pi_{{(J_{ \geq \mu})^*}_{\geq M-k+1}} \times ev}
	& \overline{\mathcal{M}_{0,r}}(Fl_{{(J_{ \geq \mu})^*}_{\geq M-k+2}}, ((\pi_{J_{\mu \leq k-1}/J})_*\d)^*) \underset{(Fl_{{(J_{ \geq \mu})^*}_{\geq M-k+2}})^r}\times (Fl_{{(J_{ \geq \mu})^*}_{\geq M-k+1}})^r \arrow{u}
	\end{tikzcd} \]
	where the isomorphism $\overline{{\mathcal{M}}_{0,r}}(Fl_{J_{\mu \leq k}}, (\pi_{{J_{\mu \leq k}}/J})_* \d) \simeq \overline{\mathcal{M}_{0,r}}(Fl_{{(J_{ \geq \mu})^*}_{\geq M-k+1}}, ((\pi_{J_{\mu \leq k}/J})_*\d)^*)$ is induced by the isomorphism $Fl_{J_{\mu \leq k}} \simeq (Fl_{J_{\mu \leq k}})^*=Fl_{{(J_{ \geq \mu})^*}_{\geq M-k+1}}$. 
	Hence the collection $\left\{I^*, (J_{\geq \mu})^*, \left((\pi_{J\geq \mu/J})_* \d\right)^*\right\}$ is stabilized with respect to $n$. 
\end{proof}

\subsection{Proof of Theorem \ref{th : RC fibrations bewteen W varieties}.} 
Let $J_1 := J_{\geq (\mu; n_\mu)}$. Consider the morphism \[\begin{tikzcd} \overline{{\mathcal{M}}_{0,r}}(Fl_{J_1}, (\pi_{J_1/J})_* \d) \times_{(Fl_{J_1})^r} (Fl_J)^r \arrow{r}& \left( \overline{{\mathcal{M}}_{0,r}}(Fl_I, (\pi_{I/J_1})_* (\pi_{J_1/J})_* \d) \times_{(Fl_{I})^r} (Fl_{J_1})^r\right) \times_{(Fl_{J_1})^r} (Fl_J)^r \arrow{d}[below, rotate=90]{\Huge \sim}\\  &\overline{{\mathcal{M}}_{0,r}}(Fl_I, (\pi_{I/J})_* \d) \times_{(Fl_{I})^r} (Fl_J)^r \end{tikzcd}\] where the first arrow is the morphism obtained by flat base change of $\Pi_{I/J_1} \times ev$ over $(Fl_{J})^r \rightarrow (Fl_{J_1})^r$. According to Subsection \ref{subsec: forgetting right Vi} the general fiber of $\Pi_{I/J_1} \times ev$ is rationally connected, hence according to Lemma \ref{lem: flat base change perserves (P)} the general fiber of $\overline{{\mathcal{M}}_{0,r}}(Fl_{J_1}, (\pi_{J_1/J})_* \d) \times_{(Fl_{J_1})^r} (Fl_J)^r \rightarrow \overline{{\mathcal{M}}_{0,r}}(Fl_I, (\pi_{I/J})_* \d) \times_{(Fl_{I})^r} (Fl_J)^r$ is rationally connected.
Furthermore, according to Subsection \ref{subsec: forgetting left V_i} the general fiber of $\overline{{\mathcal{M}}_{0,r}}(Fl_J, \d) \rightarrow \overline{{\mathcal{M}}_{0,r}}(Fl_{J_1}, (\pi_{J_1/J})_* \d) \times_{(Fl_{J_1})^r} (Fl_J)^r$ is also rationally connected. Theorem \ref{th : RC fibrations bewteen W varieties} $ii)$ then follows from the following commutative diagram: \[\begin{tikzcd}[column sep=huge] \overline{{\mathcal{M}}_{0,r}}(Fl_J, \d) \arrow{r}{\Pi_{J_1/J} \times ev} \arrow{dr}[below, rotate=-13]{\Pi_{I/J}\times ev} & \overline{{\mathcal{M}}_{0,r}}(Fl_{J_1}, (\pi_{J_1/J})_* \d) \times_{(Fl_{J_1})^r} (Fl_J)^r \arrow{d}\\ & \overline{{\mathcal{M}}_{0,r}}(Fl_I, (\pi_{I/J})_* \d) \times_{(Fl_{I})^r} (Fl_J)^r \end{tikzcd}\]
Indeed, according to Theorem \ref{th: ex stability under composition} the composition of morphisms of projective complex schemes whose general fiber is rationally connected is a morphism whose general fiber is also rationally connected.

In the same way, since the composition of unirational fibrations is a unirational fibration, Theorem \ref{th : RC fibrations bewteen W varieties} $ii)$ is easily deduced from Subsections  \ref{subsec: forgetting right Vi} and \ref{subsec: forgetting left V_i}.
\section{On the surjectivity of $\Pi_{I/J}$.}\label{sec: conditions surjectivity}
Let us recall our notations. 
We consider the flag variety $X$ parametrizing flags $V_{n_1} \subset \dots \subset V_{n_m} \subset \mathbb{C}^n$, where $V_{n_i}$ is a vector subspace of $\mathbb{C}^n$ of dimension $n_i$. For $1 \leq k \leq m$, we denote by $X_{\widehat{k}}$ the flag variety obtained from $X$ by forgetting the $k$-th vector space, and call $\pi_{\widehat{k}} : X \rightarrow X_{\widehat{k}}$ the forgetful morphism. 
Let $\mathbf{d}=(d_1, \dots, d_m)$ be an element in $E(X) \simeq \mathbb{N}^m$. We denote by $\pi_{\widehat{k}*} \mathbf{d} = (d_1, \dots, d_{k-1}, d_{k+1}, \dots, d_m)$ its pushforward to $E(X_{\widehat{k}})$. In order to ease notations, we set $d_0=0=d_{m+1}$. Finally, we denote by $\Pi_{\widehat{k}} : \overline{{\mathcal{M}}_{0,r}}(X, \mathbf{d}) \rightarrow \overline{{\mathcal{M}}_{0,r}}(X_{\widehat{k}}, \pi_{\widehat{k}*} \mathbf{d})$ the morphism induced by $\pi_{\widehat{k}} : X \rightarrow X_{\widehat{k}}$.

Our goal here is to give a condition on the degree $\mathbf{d}$, under which the morphism $\overline{{\mathcal{M}}_{0,r}}(X, \mathbf{d}) \rightarrow \overline{{\mathcal{M}}_{0,r}}(X_{\widehat{k}}, \pi_{\widehat{k}*} \mathbf{d}) \times_{{X_{\widehat{k}}}^r} X^r$ is surjective. We deduce from it examples of stabilized collections in Subsection \ref{subsec: Geometry ex stabilized collections}. 


\begin{lemme}\label{lem: min degree surjectivity M0,3}
	Consider an element $p$ in general position in $\overline{{\mathcal{M}}_{0,r}}(X_{\widehat{k}}, \pi_{\widehat{k}*} \mathbf{d})$. Denote by $\mathcal{F}_p = E_1 \subset \dots \subset E_{k-1} \subset E_{k+1} \subset \dots \subset E_m \subset \mathbb{C}^n$ the flag of vector bundles associated with $p$. Suppose for all $p<k$, $d_{k+1} \geq n_{k+1} \lceil \frac{d_{p}-d_{p-1}}{n_{p}-n_{p-1}} \rceil$. 
	Then: \begin{enumerate}[label=\roman*)] 
		\item $E_{k+1} = E_{k-1} \oplus E_{k+1}/E_{k-1}$;
		\item $E_{k+1}/E_{k-1} \simeq \mathcal{O}_{\mathbb{P}^1}(- \lfloor \frac{d_{k+1}-d_{k-1}}{n_{k+1} - n_{k-1}} \rfloor)^{\oplus n-r_2} \oplus \mathcal{O}_{\mathbb{P}^1}(- \lfloor \frac{d_{k+1}-d_{k-1}}{n_{k+1} - n_{k-1}} \rfloor +1)^{\oplus r_2} )$, where $r_2$ is the integer defined by $r_2 = (d_{k+1}-d_{k-1}) - (n_{k+1} - n_{k-1}) \lfloor \frac{d_{k+1}-d_{k-1}}{n_{k+1} - n_{k-1}} \rfloor$;
		\item There exists a $\mathbb{P}^1$-vector bundle $F$ of degree $-d$ where $d \leq d_{k-1} + (n_k-n_{k-1})( \left \lfloor \frac{d_{k+1}-d_{k-1}}{n_{k+1} - n_{k-1}} \right \rfloor +1)$ and of rank $n_k$ such that: $$E_{k-1} \subset F \subset E_{k+1} \subset \mathcal{O}_{\mathbb{P}^1}^{\oplus n}.$$
	\end{enumerate}
\end{lemme}

\begin{proof}
	\begin{enumerate}[label=\roman*)] 
		\item Cf. Proposition \ref{lem : description balanced bundles}.
		\item Cf. the description of balanced flags of vector bundles in Section \ref{sec: balanced bundles}.
		\item Call $d:= d_{k+1}-d_{k-1}$ and $n:=n_{k+1} - n_{k-1}$. Consider two non negative integers $p_1$ and $p_2$ such that $p_1 < n-r_2$, $p_2 < r_2$, and $p_1+p_2=n_k-n_{k-1}$. \\
		Denote by $F_0$ the rank $n_k-n_{k-1}$ subbundle of $E_{k+1}$ associated with $\mathcal{O}_{\mathbb{P}^1}(- \lfloor \frac{d}{n} \rfloor)^{\oplus p_1} \oplus \mathcal{O}_{\mathbb{P}^1}(- \lfloor \frac{d}{n} \rfloor +1)^{\oplus p_2} )$ by the isomorphism $E_{k+1}/E_{k-1} \simeq \mathcal{O}_{\mathbb{P}^1}(- \lfloor \frac{d}{n} \rfloor)^{\oplus n-r_2} \oplus \mathcal{O}_{\mathbb{P}^1}(- \lfloor \frac{d}{n} \rfloor +1)^{\oplus r_2} )$ of $ii)$. \\
		Then $F:=E_{k-1}\oplus F_0$ is a vector bundle of rank $n_k$, of degree the opposite of $$d_{k-1} + p_1(\lfloor \frac{d}{n} \rfloor) +p_2(\lfloor \frac{d}{n} \rfloor +1) \leq d_{k-1} + (n_k-n_{k-1}) (\lfloor \frac{d}{n} \rfloor +1).$$ Furthermore, $F$ satisfies $E_{k-1} \subset F \subset E_{k+1} \subset \mathcal{O}_{\mathbb{P}^1}^{\oplus n}.$
	\end{enumerate}	
\end{proof}

\begin{proposition}\label{prop : MX in MxX^r surjective}
	\begin{enumerate}[label=\roman*)] 
		\item Suppose $d_{k-1} \leq \lfloor \frac{d_{k+1}}{n_{k+1}} \rfloor$ and $d_k \geq d_{k-1} + (n_k-n_{k-1})(\lfloor \frac{d_{k+1}-d_{k-1}}{n_{k+1} - n_{k-1}}\rfloor +1)$. \\
		Then the morphism $$\Pi_{\widehat{k}} : \overline{{\mathcal{M}}_{0,3}}(X, \mathbf{d}) \rightarrow \overline{{\mathcal{M}}_{0,3}}(X_{\widehat{k}}, \pi_{\widehat{k}*} \mathbf{d})$$ is surjective;
		\item Consider an integer $\tau$ such that, for all $d_k \geq \tau$, the morphism $\Pi_{\widehat{k}} : \overline{{\mathcal{M}}_{0,3}}(X, \mathbf{d}) \rightarrow \overline{{\mathcal{M}}_{0,3}}(X_{\widehat{k}}, \pi_{\widehat{k}*} \mathbf{d})$ is surjective.\\
		Suppose $d_k \geq {\tau}+r(n_k-n_{k-1})$. 
		Then the morphism $$\Pi_{\widehat{k}} \times \ev^r: \overline{{\mathcal{M}}_{0,r}}(X, \mathbf{d}) \rightarrow \overline{{\mathcal{M}}_{0,r}}(X_{\widehat{k}}, \pi_{\widehat{k}*} \mathbf{d}) \times_{{X_{\widehat{k}}}^r} X^r$$ is surjective.
	\end{enumerate}
\end{proposition}

\begin{proof}
	\begin{enumerate}[label=\roman*)] 
		\item Call $\mathbf{d}_{\widehat{k}}:=\pi_{\widehat{k}*}\mathbf{d}$. Consider an element $p$ in general position in $\overline{{\mathcal{M}}_{0,3}}(X_{\widehat{k}}, \mathbf{d}_{\widehat{k}})$. Since $X$ is convex, $\mathcal{M}_{0,3}(X_{\widehat{k}}, \mathbf{d}_{\widehat{k}})$ is a dense open subset in $\overline{{\mathcal{M}}_{0,3}}(X_{\widehat{k}}, \mathbf{d}_{\widehat{k}})$; hence, according to Lemma \ref{lem : Intersection dense open is dense}, $p$ is in $\mathcal{M}_{0,3}(X_{\widehat{k}}, \mathbf{d}_{\widehat{k}})$. We will construct an antecedent of $p$ by $\Pi_{\widehat{k}}$ as a concatenation of two morphisms $f: \mathbb{P}^1 \rightarrow X$ and $g : \mathbb{P}^1 \rightarrow X$. \\ 
		\textbf{Construction of $f$.} Denote by $\mathcal{F}_p := E_1 \subset \dots \subset E_{k-1} \subset E_{k+1} \subset \dots \subset E_m \subset \mathcal{O}_{\mathbb{P}^1}^{\oplus n}$ the flag of vector bundles associated with $p$. According to Lemma \ref{lem: min degree surjectivity M0,3} $iii)$, there exists a $\mathbb{P}^1$-vector bundle $F$ of degree $d \leq d_{k-1} + (n_k-n_{k-1})(\lfloor \frac{d_{k+1}-d_{k-1}}{n_{k+1} - n_{k-1}}\rfloor +1)$ and of rank $n_k$ defining the following flag of vector bundles: $$\mathcal{F} := E_1 \subset \dots \subset E_{k-1} \subset F \subset E_{k+1} \subset \dots \subset E_m \subset \mathcal{O}_{\mathbb{P}^1}^{\oplus n}.$$ 
		By the functorial definition of flag varieties, $\mathcal{F}$ defines a degree $(d_1, \dots, d_{k-1}, d, d_{k+1}, \dots, d_m)$ morphism $f: \mathbb{P}^1 \rightarrow X$, such that $f_p : \mathbb{P}^1 \rightarrow X_{\widehat{k}}$, factors through $f$. We have denoted by $f_p$ the morphism associated with $p$. \\
		\textbf{Construction of $g$.} Denote by $g: \mathbb{P}^1 \rightarrow X$ a rational curve of class $(d_k-d) \sigma_k$, where $\sigma_k$ is the degree one Schubert class whose pushforward to $\Gr(n_k,n)$ is not $0$. Up to multiplication by an element in $\SL_n$, we can assume $g(\mathbb{P}^1)$ intersects $f(\mathbb{P}^1)$ on a point that is not marked. Notice the image of $\pi_{\widehat{k}} \circ g : \mathbb{P}^1 \rightarrow X_{\widehat{k}}$ is a point.\\
		\textbf{Finally,} the stable map $\mathbb{P}^1 \cup \mathbb{P}^1 \rightarrow X$ induced by $f$ and $g$ defines an element in $\overline{{\mathcal{M}}_{0,3}}(X, \mathbf{d})$, whose projection by $\Pi_{\widehat{k}}$ is $p$.
		\item Call $\mathcal{M}' := \overline{{\mathcal{M}}_{0,r}}(X_{\widehat{k}}, \pi_{\widehat{k}*} \mathbf{d})$. Fix a genus zero stable map $p := (g_p : \mathbb{P}^1 \rightarrow \mathcal{M}', \{p_1, \dots, p_r\})$ corresponding to a general point in $\mathcal{M}'$. According to Lemma \ref{lem : dense intersection with open for homogeneous spaces}, for any dense open subset $U$ of $\mathcal{M}'$, the equidimensional variety $X^r \times_{(X_{\widehat{k}})^r} U$ is a dense open subset of $X^r \times_{(X_{\widehat{k}})^r} \mathcal{M}'$; hence it is enough to prove that $\Pi_{\widehat{k}}^{-1}(\nu)$ dominates $X^r \times_{(X_{\widehat{k}})^r} \{p\}$ for a general point $p$ in  $\mathcal{M}'$.\\
		Fix an element $(x, p)$ in $X^r \times_{(X_{\widehat{k}})^r} \{p\}$, i.e. fix $r$ points $x_1$, $\dots$, $x_r$ on $X$ such that each point $x_i$ gets projected by $\pi_{\widehat{k}} : X \rightarrow X_{\widehat{k}}$ to $g_p(p_i)$. Our goal is to construct an antecedent of $(x, p)$ by $\Pi_{\widehat{k}} : \overline{{\mathcal{M}}_{0,r}}(X, \mathbf{d}) \rightarrow  \mathcal{M}'$.\\
		We proceed in three steps. First, notice there is a degree $(\mathbf{d} - r(n_k-n_{k-1})\sigma_k)$ morphism $f: \mathbb{P}^1 \rightarrow X$ getting projected to $p$, i.e. such that $g_p = \pi_{\widehat{k}} \circ f$. Furthermore, we construct rational curves $C_i \rightarrow X$ of degree $n_k-n_{k-1}$ joining the points $f(p_i)$ and $x_i$. Finally, by concatenation, this gives a degree $\mathbf{d}$ stable map $[\mathbb{P}^1 \cup_i C_i \rightarrow X, \{p_1, \dots, p_r\}]$ which is an antecedent of $(x, g)$, as researched.\\
		\textbf{Existence of $f$.} We call: \begin{equation*} \boldsymbol{\delta} := (d_1, \dots, d_{k-1}, d_k- r(n_k-n_{k-1}), d_{k+1}, \dots, d_m) \in A_1(X).\end{equation*}
		Notice that, since $g$ is in general position in $\mathcal{M}_{0,3}(X, \mathbf{d})$ and the morphism $\Pi_{\widehat{k}} : \overline{{\mathcal{M}}_{0,3}}(X, \boldsymbol\delta) \rightarrow \overline{{\mathcal{M}}_{0,r}}(X_{\widehat{k}}, \pi_{\widehat{k}*} \mathbf{d})$ is surjective, its reciprocal image $\Pi_{\widehat{k}}^{-1}(p)$ has a non empty intersection with the dense open subset ${{\mathcal{M}}_{0,3}}(X, \boldsymbol\delta)$ of $\overline{{\mathcal{M}}_{0,r}}(X, \boldsymbol \delta)$. We can thus find an element $[f : \mathbb{P}^1 \rightarrow X, \{p_1, \dots, p_r\}]$ in $\Pi_{\widehat{k}}^{-1}(\nu) \cap \mathcal{M}_{0,r}(X,\boldsymbol\delta)$.\\
		\textbf{Construction of $C_i \rightarrow X$.} Fix an integer $i$ in $\{1, \dots, r\}$. Denote by $\mathcal{F}_{p_i} = \{\ W_1^i \dots \subset W_m^i \subset \mathbb{C}^n\}$ the germ at $p_i$ of the flag of $\mathbb{P}^1$-vector bundles associated with $f$.  Since $f: \mathbb{P}^1 \rightarrow X$ projects to $g_p$, the flag of vector spaces associated with $g_p(p_i)$ is: \begin{equation*} g_p(p_i)= \{ W_1^i \subset \dots \subset W_{k-1}^i \subset  W_{k+1}^i \subset \dots \subset W_m^i \subset \mathbb{C}^n \}.\end{equation*} 
		Denote by $<u_{n_{k-1}+1}, \dots, u_{n_k}>$ a basis of $W_k^i/W_{k-1}^i$. \\
		Recall $x_i$ is a point in $X$ getting projected to $g_p(p_i)$, i.e. corresponds to a flag of vector spaces $$W_1^i \subset \dots \subset W_{k-1}^i \subset V \subset  W_{k+1}^i \subset \dots \subset W_m^i \subset \mathbb{C}^n.$$
		Since $\ev :\overline{\mathcal{M}_{0,2}}(\Gr(n_k-n_{k-1}, n_{k+1}-n_{k-1}), d) \rightarrow \Gr(n_k-n_{k-1}, n_{k+1}-n_{k-1})^2$ is surjective for $d \geq (n_k-n_{k-1})$ (cf. for example \cite{buch2011} Corollary 2.2), there exists a degree $(n_k-n_{k-1})$ genus zero stable map $$h_i : C_i = \bigcup_{j=1}^N \mathbb{P}^1 \rightarrow \Gr(n_k-n_{k-1}, n_{k+1}-n_{k-1})$$ joining the rank $(n_k-n_{k-1})$ vector spaces $V/W_{k-1}^i$ and $W_k^i/W_{k-1}^i$ within $W_{k+1}^i/W_{k-1}^i \simeq \C^{n_{k+1}-n_{k-1}}$. \\
		By the functorial definition of Grassmannians, each morphism ${h_i}_{| \P^1} : \mathbb{P}^1 \rightarrow \Gr(n_k-n_{k-1}, n_{k+1}-n_{k-1})$ defines a flag of $\mathbb{P}^1$-vector bundles $$E_j^i \subset  \O_{\P^1}^{\oplus n_{k+1}-n_{k-1}},$$ which has one fiber given by the flag of vector spaces $V_j/W_{k-1}^i  \subset W_{k+1}^i/W_{k-1}^i \simeq \C^{n_{k+1}-n_{k-1}}$, and another fiber given by $V_{j+1}/W_{k-1}^i \subset W_{k+1}^i/W_{k-1}^i \simeq \C^{n_{k+1}-n_{k-1}}$, where $V_1=V$ and $V_N=W_k$. Now consider the following flag of $\mathbb{P}^1$-vector bundles: $$\mathcal{F}_{i,j} := W_1^i \subset \dots \subset W_{k-1}^i \subset W_{k-1}^i \oplus E_j^i \subset  W_{k+1}^i \subset \dots \subset W_m^i \subset \mathcal{O}_{\mathbb{P}^1}^{\oplus n}.$$
		$\mathcal{F}_{i,j}$ defines a morphism $f_{i,j} : \mathbb{P}^1 \rightarrow X$ which is projected to $g_p(p_i)$. Furthermore, the morphisms $f_{i,j}$ define a degree $(n_k-n_{k-1})$ morphism from the tree of $\mathbb{P}^1$'s $C_i$ to $X$ which joins by construction $x_i$ and $g_p(p_i)$. We will denote by $f_i : C_i \rightarrow X$ the morphism thus obtained. \\
		\textbf{Construction of an antecedent of $p$.} Finally, consider the stable map $q=[\mathbb{P}^1 \cup_i C_i \rightarrow X, \{p_1, \dots, p_r\}]$ associated with $[f : \mathbb{P}^1 \rightarrow X]$ and $f_i : C_i \rightarrow X$, where we place each of the marked points $p_i$ on an antecedent $f_i^{-1}(x_i)$ on $C_i$. Since each of the map $f_i: C_i \rightarrow X$ projects to $g_p(p_i)$ in $X_{\widehat{k}}$, the morphism $\overline{{\mathcal{M}}_{0,r}}(X, \mathbf{d}) \rightarrow \mathcal{M}'$ collapses each component $C_i$ and sends $q$ to $p$. 
	\end{enumerate}
\end{proof}

\subsection{Examples of stabilized collections.}\label{subsec: Geometry ex stabilized collections} \begin{enumerate}
	\item \textsc{Forgetting the first vector space.} Suppose $I=\{i_1, \dots, i_m\}$  is a collection of integers satisfying $1 < i_1 < \dots < i_m <n$ and $J= \{j_{1},  i_1, \dots, i_m\}$ satisfies $0 < j_1  < i_1$.  Note that then $Fl_J \rightarrow Fl_I$ is the morphism forgetting the first vector space. By definition, if the following condition is satisfied, the collection $\{I,J,\d\}$ is stabilized with respect to $\mu=2$. \begin{itemize}
		\item The morphism $\Pi_{J \geq 1} \times ev: \overline{{\mathcal{M}}_{0,r}}(Fl_J, \d) \rightarrow \overline{{\mathcal{M}}_{0,r}}(Fl_I, (\pi_{I/J})_* \d) \times_{(Fl_I)^r} (Fl_J)^r$ is surjective.
	\end{itemize}
	Note that since there is no integer $i_{k'}$ such that $i_{k'} < j_k$, the collection $\{I,J,\d\}$ does not have to satisfy any other condition. Furthermore, according to Proposition \ref{prop : MX in MxX^r surjective} $i)$, if $d_0=0 \leq \floor{\frac{d_2}{i_1}}$ and $d_1 \geq d_0 + (j_1-0)(\floor{\frac{d_2-d_0}{i_1-0}}+1)=j_1 (\floor{\frac{d_2}{i_1}}+1)$, then the morphism $\overline{{\mathcal{M}}_{0,3}}(Fl_J, \d) \rightarrow \overline{{\mathcal{M}}_{0,3}}(Fl_I, (\pi_{I/J})_* \d)$ is surjective. Note that the first condition is always satisfied. Hence according to Proposition \ref{prop : MX in MxX^r surjective} $ii)$, if $$d_1 \geq j_1 (\floor{\frac{d_2}{i_1}}+1) + r(j_1-0) = j_1(\floor{\frac{d_2}{i_1}}+r+1)$$ then the morphism $\Pi_{J \geq 1} \times ev$ is surjective.
	
	To conclude, we observe the following property: \begin{center} if $d_{1} \geq j_{1} \left( r +1+\floor{d_{2}/i_1} \right)$ the collection $(I,J, (d_{1}, d_2, \dots d_{m+1}))$ is stabilized. \end{center}
	
	\item \textsc{Projection from a Grassmannian to the point.} Suppose $J=\{k\}$ and $I=\emptyset$. Then $Fl_J$ is the Grassmannian $\Gr(k,n)$ and $Fl_J = \Gr(k,n) \rightarrow Fl_I =\Spec \C$ is the projection to the point. By definition, the collection $\{I,J,d\}$ is stabilized iff the morphism $$\Pi_{J \geq 1} \times ev: \overline{{\mathcal{M}}_{0,r}}(\Gr(k,n),d) \rightarrow \overline{{M}_{0,r}} \times (\Gr(k,n))^r$$ is surjective. Since the projection $\overline{{\mathcal{M}}_{0,3}}(\Gr(k,n),d) \rightarrow \overline{{M}_{0,3}}= \Spec \C$ is always surjective, according to Proposition \ref{prop : MX in MxX^r surjective} $ii)$, if $$d \geq r(k-0)=rk,$$ then $\Pi_{J \geq 1} \times ev$ is surjective. Hence: \begin{center} if $d \geq rk$ the collection $\{I,J,d \}$ is stabilized. \end{center}
	Finally, since the collection $\{I,J,d \}$  is stabilized, according to Lemma \ref{lem: Geometry dual stabilized collection}, the dual collection $\{I^*,J^*,d^*=d \}$ is also stabilized. Note that $Fl_{J^*}=\Gr(n-k,n)$. Hence: \begin{center} if $d \geq r(n-k)$ the collection $\{I,J,d \}$ is also stabilized. \end{center}
\end{enumerate}


\newpage

\section*{Appendix A}

Let $Y$ be an irreducible projective complex variety, let $\gamma \in H_2(Y, \mathbb{Z})$, and $r \geq 0$. We first recall the usual definitions associated with the space of stable maps $\overline{{\mathcal{M}}_{0,r}}(Y, \gamma)$. 
The data of a morphism $\varphi : C \rightarrow Y$ from a genus zero projective connected reduced and nodal curve $C$ to $Y$ along with $r$ distinct non singular points $\{c_1, \dots, c_r\}$ on $C$ is called a stable map if any irreducible component of $C$ getting sent to a point contains at least three points which are nodal or marked. 

Let $S$ be a complex scheme. A family of $r$-pointed quasi-stable maps over $S$ consists of the data of morphisms $\mathcal{C} \rightarrow Y$ and $\pi: \mathcal{C} \rightarrow S$ along with $r$ sections $p_i: S \rightarrow \mathcal{C}$, such that each geometric fiber is a genus zero projective connected reduced and nodal curve $C_s$ and the marked points $p_i(s)$ are distinct and non singular. By construction, $\overline{{\mathcal{M}}_{0,r}}(Y, \gamma)$ is the coarse moduli space parametrizing isomorphism classes of families over $S$ of $r$-pointed, quasi-stable maps from genus zero curves to $Y$ representing the class $\gamma$, and such that each geometric fiber $\mathcal{C}_s$ over a point in $S$ defines a stable map $\mathcal{C}_s \rightarrow Y$ \cite{fulton1996notes}.

\begin{lemme}
	\label{lem : morphisme Y x P dans M(X)}
	Fix distinct points $p_1$, $\dots$, $p_r$ in $\P^1$. 
	Let $S$ be a complex scheme, and let $\mu : S \times \mathbb{P}^1 \rightarrow Y$ be a morphism such that, for all closed points $s \in S$, the restriction $\mu_s$ of $\mu$ to $s \times \mathbb{P}^1$ verifies $\mu_{s*} [\mathbb{P}^1] = \gamma \in A_1(Y)$. Let $pr_1$ be the natural projection from $S \times \mathbb{P}^1$ to $S$, and $p_i : S \rightarrow S \times \mathbb{P}^1$, $1 \leq i \leq r$, be the morphisms respectively defined by $p_i : s \rightarrow (s, p_i)$.
	Then the family \[ \begin{tikzcd}
	S \times \mathbb{P}^1 \arrow{r}{\mu}   \arrow{d}{pr_1}  & Y \\%
	S & 
	\end{tikzcd}
	\] together with the sections $p_i : S \rightarrow S \times \mathbb{P}^1$ is a stable family on $S$ of maps from $r$-pointed genus $0$ curves to $Y$ representing the class $\gamma$.
	
	Furthermore, the morphism $f : S \rightarrow \overline{\mathcal{M}_{0,r}(Y, \gamma)}$ associated with this family sends a point $s \in S$ into the point of $\mathcal{M}_{0,r}(Y, \gamma)$ associated with $[\mu_s : \mathbb{P}^1 \rightarrow Y, \{p_1, \dots, p_r\}]$.
\end{lemme}

\begin{proof}
	Let $s$ be a closed point of $S$. Then the geometric fiber in $s$ verifies $pr_1^{-1}(s) = (S \times \mathbb{P}^1) \times_S \Spec (\mathbb{C}(s)) \simeq \Spec(\mathbb{C}) \times \mathbb{P}^1$, i.e. is an irreducible rational projective curve. Hence $pr_1^{-1}(s)$ is a projective, connected, smooth, reduced, rational curve with $r$ distinct non singular points.
	
	Since $\mathrm{char} (\mathbb{C})=0$, $\mathbb{P}^1$ is flat on $\Spec \mathbb{Z}$, and $pr_1 : S \times \mathbb{P}^1 \rightarrow S$ is a flat morphism. Furthermore, $pr_1$ may be written down as $S \times \mathbb{P}^1 \simeq \mathbb{P}_{\mathbb{Z}}^1 \times \Spec \mathbb{C} \times S \rightarrow \Spec \mathbb{C} \times S \simeq S$, hence is a projective morphism. The family $(pr_1 : S \times \mathbb{P}^1 \rightarrow S, \{p_1, \dots , p_r\})$ is thus a family of $r$-pointed genus $0$ quasi stable curves on $S$.
	
	For $s$ in $S$, denote by $\mu_s : C_s \rightarrow Y$ the restriction of $\mu$ to the geometric fiber $C_s := pr_1^{-1}(s)$. Note that $\mu_s$ is a map  from the $r$-pointed quasi-stable map $(C_s, \{p_1, \dots, p_r\})$ to $Y$. The map $\mu_s$ is stable, and by hypothesis represents the class $\gamma$. The family $(pr_1: S \times \mathbb{P}^1 \rightarrow S, \{p_1, \dots, p_r\}, \mu : S \times \mathbb{P}^1 \rightarrow Y)$ hence is a stable family on $S$ of maps from $r$-pointed genus $0$ curves to $Y$ representing the class $\gamma$, i.e. $o=(pr_1: S \times \mathbb{P}^1 \rightarrow S, \{p_1, \dots, p_r\}, \mu : S \times \mathbb{P}^1 \rightarrow Y)/ \sim$ is an element of $\overline{\mathcal{M}_{0,r}}(Y, \gamma)(S)$.
	
	Let $s : \mathrm{Spec}(\mathbb{C}) \rightarrow S$ be a point of $S$. Since $\overline{\mathcal{M}_{0,r}}(Y, \gamma)$ is a coarse moduli space, the functor $\overline{\mathcal{M}_{0,r}}(Y, \gamma)$ associates to $s$ the map from $\overline{\mathcal{M}_{0,r}}(Y, \gamma)(S) \rightarrow \overline{\mathcal{M}_{0,r}}(Y, \gamma)(\mathrm{Spec}(\mathbb{C}))$ associating to a stable family $(\pi : \mathcal{C} \rightarrow S, \{p_1, \dots, p_r\}, \mu : \mathcal{C} \rightarrow Y) / \sim$ corresponding to an element in $\overline{\mathcal{M}}_{0,r}(Y, \gamma)(S)$ the isomorphism class of: 	\[ \begin{tikzcd}
	\Spec (\mathbb{C}) \times_S \mathcal{C}  \arrow{r}   \arrow{d}  & \mathcal{C} \arrow{r}{\mu} \arrow{d} & Y\\%
	\Spec(\mathbb{C}) \arrow{r}{s} &  S
	\end{tikzcd} \]
	We observe $\overline{\mathcal{M}}_{0,r}(Y, \gamma)(s)$ sends $o$ to the element in $\overline{\mathcal{M}_{0,r}}(Y, \gamma)(\mathrm{Spec}(\mathbb{C}))$ associated with $(C_s, \{p_1, \dots, p_r\}, \mu_s : C_s \rightarrow Y)$.
	
	Since $\overline{{M}_{0,r}}(Y, \gamma)$ is a coarse moduli space, there exists a natural transformation of functors :
	\begin{equation*}
	\Phi : \overline{\mathcal{M}_{0,r}}(Y, \gamma) \rightarrow \mathcal{H}om_{Sch}(S, \overline{\mathcal{M}_{0,r} (Y, \gamma)})
	\end{equation*}
	Let $f : S \rightarrow \overline{{M}_{0,r}(Y, \gamma)}$ be the morphism associated by $\Phi$ with $o=(pr_1: S \times \mathbb{P}^1 \rightarrow S, \{p_1, \dots, p_r\}, \mu : S \times \mathbb{P}^1 \rightarrow Y)/ \sim$, i.e. $f = \Phi(o)$. We obtain :
	\begin{equation*}
	\Phi(S)(o) \circ s = \Phi(\mathrm{Spec}(\mathbb{C}))(\overline{\mathcal{M}_{0,r}}(Y, \gamma)(s)),
	\end{equation*}
	i.e. $f$ sends $s$ to the point of $\overline{\mathcal{M}_{0,r}(Y, \gamma)}$ associated with $(C_s, \{p_1, \dots, p_r\}, \mu_s)$.
\end{proof}

\chapter{A comparison formula between genus $0$ correlators of partial flag varieties}\label{chap : stabilization correlators}

\section{Introduction}

Removing a subspace from a partial flag gives another partial flag composed of fewer subspaces. This induces a forgetful map between the corresponding flag varieties. 
The main result of this chapter is that quantum $K$-theoretical genus~0 correlators of these two flag varieties are equal to each other, provided their degree is high enough.

\subsection{Genus zero correlators of $G/P$.}  Quantum $K$-theory takes its simplest form for homogeneous varieties, and genus zero quantum $K$-theory actually first introduced and defined in this setting by Givental some twenty years ago \cite{givental}.

Let $X$ be a homogeneous variety $G/P$, where $G$ is an algebraic group, and $P$ is a parabolic subgroup of $G$. Fix a degree $d$ in the semi-group $E(X)$ of effective classes of $1$-cycles in $A_1(X) \simeq H_2(X,\mathbb{Z})$. Denote by $\overline{\mathcal{M}_{0,n}}(X, d)$ the moduli space parametrizing stable maps $(\mu: C \rightarrow X, \{p_1, \dots, p_n\})$ from a genus zero curve $C$ to $X$ verifying $\mu_* [C]={d}$, with $n$ marked points $p_1$, $\dots$, $p_n$ on $C$. If $X$ is a homogeneous variety, $\overline{\mathcal{M}_{0,n}}(X, d)$ is an irreducible rational variety \cite{kim2001connectedness, thomsen1998irreducibility}, and the subvariety ${\mathcal{M}_{0,n}}(X, d)$ parametrizing stable maps $(\mu: \mathbb{P}^1 \rightarrow X, \{p_1, \dots, p_n\})$ from the projective line to $X$ is a dense open subset of $\overline{\mathcal{M}_{0,n}}(X, d)$ \cite{fulton1996notes}.
The \textit{evaluation morphism} $\ev_i : \overline{\mathcal{M}_{0,n}}(X, d) \rightarrow X$ assigns to a map the image of its $i$-th marked point. If $X$ is a homogeneous variety, the evaluation morphism is flat.

Consider an algebraic group $H$ acting on $X=G/P$. Let $\phi_1$, $\dots$, $\phi_r$ be elements in the Grothendieck ring $K_H(X)$ of $H$-equivariant coherent sheaves on $X$. The \textit{equivariant correlator} $\langle \phi_1, \dots , \phi_r \rangle_{H, d}^X$ of elements $\phi_i$ in $K_\circ(X)$  is defined as:
\begin{equation*}
\langle \phi_1, \dots , \phi_r \rangle_{H, d}^X := \chi_H \left(\ev_1^*(\phi_1) \cdot \: \dots \: \cdot \ev_r^*(\phi_r) \cdot [\mathcal{O}_{\overline{\mathcal{M}_{0,r}}(X, d)}]\right),
\end{equation*}
where $\chi_H$ denotes the $H$-equivariant sheaf Euler characteristic. 

\subsection{$H$-equivariant decomposition of the diagonal.} Let $H$ be a complex linear algebraic group, let $X$ be an irreducible variety with an $H$-action. Let $R(H)$ denote the representation ring of $H$, and $K_H(X)$ denote the Grothendieck ring of $H$-equivariant coherent sheaves on $X$. We say $X$ \textit{admits a diagonal decomposition in $K_H(X)$} if there exists a basis $(\alpha)_{i \in I}$ of the $R(H)$-module $K_H(X)$, and a dual basis $(\alpha^\vee)_{i \in I}$ of $K^H(X)$ such that $$\Delta_*[\O_X] = \sum_{\alpha \in I} \alpha \boxtimes (\alpha^\vee \otimes [\O_X]) \in K_H(X \times X),$$
where $\Delta: X \rightarrow X \times X$ is the diagonal embedding, and $(\alpha^\vee)_{i \in I}$ is the dual basis of $(\alpha)_{i \in I}$ with respect to the pairing given by the $H$-equivariant sheaf Euler characteristic, i.e. $\chi_H(\alpha^\vee \otimes \beta)=\delta_{\alpha, \beta}$ for all elements $\alpha$, $\beta$ in $K_H(X)$.

We study here the following setting. Consider a complex linear algebraic subgroup $H$ of $\GL_n$ acting on flag varieties via left multiplication such that flag varieties $\GL_n/P$ admit a diagonal decomposition in $K_H(X)$. Note that this is in particular true when considering the non equivariant setting or, according to \cite{graham2008positivity}, when considering the action of a maximal torus $T$ of $\GL_n$.

\subsection{Main result.} Fix two positive integers $m<n$. Let $0 < n_1 < \dots < n_m <n$ be a strictly increasing collection of non negative integers. Denote by~$X$ the $m$-step flag variety parametrizing $m$-tuples of vector spaces ordered by inclusion $$0 \subset V_{n_1} \subset \dots \subset V_{n_m} \subset \mathbb{C}^n$$ such that ${\rm dim} V_{n_i}=n_i$.
Fix a positive integer $k$, $1 \leq k \leq m$. Denote by $X_{\widehat{k}}$ the $(m-1)$-step flag variety obtained from $X$ by forgetting the $k$-th vector subspace. Denote by $\pi_{\widehat{k}} : X \rightarrow X_{\widehat{k}}$ the projection. 
The class $\mathbf{d}=(d_1, \dots, d_m)$ of a curve $C$ in $X$ is determined by non-negative integers $d_1$, $d_2$,$\dots$, $d_m$, where $d_i$ is the Plücker degree of the projection of $C$ to the Grassmannian ${\rm Gr}(n_i,n)$. Note that $\pi_{\widehat{k}}$ and the evaluation morphism $\overline{{\mathcal{M}}_{0,r}}(X, \d) \rightarrow X^r$ induce a morphism $\overline{{\mathcal{M}}_{0,r}}(X, \mathbf{d}) \rightarrow \overline{{\mathcal{M}}_{0,r}}(X_{\widehat{k}}, (\pi_{\widehat{k}})_* \mathbf{d}) \times_{{X_{\widehat{k}}}^r} X^r$. Denote by $H$ an algebraic subgroup of $G$ acting via left multiplication on $X$ and $X_{\widehat{k}}$, such that $X_{\widehat{k}}$ admits a diagonal decomposition in $K_H(X_{\widehat{k}})$. 

\begin{theorem}\label{th : euler char equal}
	Let $r >0$, $k \in \{1, \dots, m\}$, $\mathbf{d}=(d_1, \dots, d_m)$ be a nef class of $1$-cycles such that $\forall \: p<k, \: d_k \geq n_k \ceil*{ \frac{d_p-d_{p-1}}{n_p-n_{p-1}}}$ and such that the morphism $$\overline{{\mathcal{M}}_{0,r}}(X, \mathbf{d}) \rightarrow \overline{{\mathcal{M}}_{0,r}}(X_{\widehat{k}}, (\pi_{\widehat{k}})_* \mathbf{d}) \underset{{X_{\widehat{k}}}^r}\times X^r$$ is surjective. 
	Then for all elements $\phi_1$, $\dots$, $\phi_r$ in $K_H(X)$, the correlators of $X$ and $X_{\widehat{k}}$ associated with $H$, $\mathbf{d}$, and the elements $\phi_i$ are equal: \begin{equation*}
	\langle \phi_1, \dots, \phi_r \rangle_{H, \mathbf{d}}^{X} = \langle (\pi_{\widehat{k}})_*(\phi_1), \dots, (\pi_{\widehat{k}})_*(\phi_r) \rangle_{H, (\pi_{\widehat{k}})_* \mathbf{d}}^{X_{\widehat{k}}}.
	\end{equation*}
\end{theorem}

According to Chapter \ref{chap : geometry W}  Proposition \ref{prop : MX in MxX^r surjective}, Theorem \ref{th : euler char equal} implies the following result.
\begin{thmbis}
	Let $r >0$, $k \in \{1, \dots, m\}$, $\mathbf{d}=(d_1, \dots, d_m)$ be a nef class of $1$-cycles such that\begin{itemize}[label={--}]
		\item $\forall \: p<k, \: d_k \geq n_k \ceil*{ \frac{d_p-d_{p-1}}{n_p-n_{p-1}}}$, 
		\item  $d_{k-1} \leq \lfloor \frac{d_{k+1}}{n_{k+1}} \rfloor$,
		\item $d_k \geq r(n_k-n_{k-1}) + d_{k-1} + (n_k-n_{k-1})(\lfloor \frac{d_{k+1}-d_{k-1}}{n_{k+1} - n_{k-1}}\rfloor +1)$.\end{itemize}
	Then for any elements $\phi_1$, $\dots$, $\phi_r$ in $K_H(X)$, the correlators of $X$ and $X_{\widehat{k}}$ associated with $H$, $\mathbf{d}$, and the elements $\phi_i$ are equal: \begin{equation*}
	\langle \phi_1, \dots, \phi_r \rangle_{H, \mathbf{d}}^{X} = \langle (\pi_{\widehat{k}})_*(\phi_1), \dots, (\pi_{\widehat{k}})_*(\phi_r) \rangle_{H, (\pi_{\widehat{k}})_* \mathbf{d}}^{X_{\widehat{k}}}.
	\end{equation*}
\end{thmbis}
This Chapter is organized as follows. Sections \ref{sec: equivariant  ktheory} is devoted to recalling some standard facts about equivariant algebraic $K$-theory, and fixing the notations we use in the rest of the chapter. A proof of Theorem \ref{th : euler char equal} is presented in Section \ref{sec: proof main th}, using results of Chapter \ref{chap : geometry W}. In the non equivariant setting a more geometric proof of Theorem \ref{th : euler char equal} is provided in part \ref{sec: geometric proof}, relying heavily on results from Chapter \ref{chap : geometry W}.

\section{Preliminaries on equivariant algebraic $K$-theory}\label{sec: equivariant  ktheory}
\subsection{Diagonal decomposition of $K_{T}(X)$.}\label{subsec: diag decompositon}
Recall $G$ is a semi-simple linear complex algebraic group, $T$ is a maximal torus of $G$, $B$ is a Borel subgroup of $G$ and $P$ is a parabolic subgroup of $G$ such that $T \subset B \subset P$. 
For any element $u$ in the Weyl group $W$ recall $X(u):= \overline{BuP/P}$ is a Schubert variety of $X=G/P$, and $Y(u):= w_0 \cdot \overline{Bw_0uP/P}$ is the Schubert variety opposite to it. Furthermore, denote by $\partial Y(u) :=Y(u) \setminus (w_0 Bw_0 u P/P)$ the complementary of the dense open Schubert cell in $Y(u)$. 
According to \cite{graham2008positivity} we have the following decomposition in $K_T(X)$.
$$\Delta_*[\O_X] = \sum_{u \in W^P} [\O_{X(u)}] \boxtimes [\O_{Y(u)}(- \partial Y(u))] \in K_T(X \times X),$$
where $\Delta : X \rightarrow X \times X$ denotes the diagonal embedding. Note that this was first observed in the non equivariant setting by Brion \cite{brion2002positivity}.
\subsection{Equivariant flat base change.} We describe here the class of an equivariant flat base change. The main application of this result will be the following setting. Let $H$ be a linear subgroup of $\GL_n$, let $X$ be the flag variety parametrizing flags of vector spaces $V_{n_1} \subset \dots \subset V_{n_m} \subset \C^n$, let $X_{\widehat{k}}$ be the flag variety obtained by forgetting the vector $V_{n_k}$, and denote by $\pi_{\widehat{k}}:X \rightarrow X_{\widehat{k}}$ the associated morphism. For a class $\beta$ of curve on $X_{\widehat{k}}$, we consider the fiber product described by the following $H$-equivariant Cartesian diagram : \[ \begin{tikzcd}   X^r \underset{(X_{\widehat{k}})^r}\times \overline{{\mathcal{M}}_{0,r}}(X_{\widehat{k}}, \beta) \arrow{r} \arrow{d}  & \overline{{\mathcal{M}}_{0,r}}(X_{\widehat{k}}, \beta) \arrow{d}{ev} \\ X^r \arrow{r}{(\pi_{\widehat{k}})^r} & (X_{\widehat{k}})^r \end{tikzcd}\]
Note that, since the action of $(\GL_n)^r$ on $(X_{\widehat{k}})^r$ is transitive, the morphism $(\pi_{\widehat{k}})^r : X^r \rightarrow (X_{\widehat{k}})^r$ is flat. 

Consider two morphisms of schemes $f_1 : M_1 \rightarrow X_1$ and $f_2: M_2 \rightarrow X_2$. We denote by $f_1 \times f_2 : M_1 \times M_2 \rightarrow X_1 \times X_2$ the morphism induced by $f_1$ and $f_2$, i.e. the morphism induced by the following commutative diagram: \[ \begin{tikzcd} M_1 \times M_2 \arrow{r}{p_2} \arrow{d}{p_1} \arrow{rd}{f_1 \times f_2} & M_2 \arrow{rd}{f_2} & \\ M_1 \arrow{dr}{f_1} & X_1 \times X_2 \arrow{d}{\pi_1} \arrow{r}{\pi_2} & X_2 \\ & X_1 & \end{tikzcd}\]
where $p_i : M_1 \times M_2 \rightarrow M_i$ and $\pi_i : X_1 \times X_2 \rightarrow X_i$ are the projection morphisms.  Finally, let $H$ be a complex linear algebraic acting on $M_1$ and $M_2$. We denote by \begin{align*} \boxtimes : \:\:\: K^H(M_1) \times K^H(M_2) &\rightarrow K^H(M_1 \times M_2) \\
(\phi, \psi) \:\:\: &\rightarrow (p_1^* \phi) \otimes (p_2^* \psi) \end{align*}
the external product.

\begin{proposition}\label{prop: image fiber product}
	Let $H$ be a linear algebraic group. Consider two equivariant morphisms of projective varieties with an $H$-action $f_1 : M_1 \rightarrow Y$ and $f_2 : M_2 \rightarrow Y$, such that $M_1 \rightarrow Y$ is flat. Denote by $\Delta: Y \rightarrow Y \times Y$ the diagonal morphism, $p_i : M_1 \times_Y M_2 \rightarrow M_i$ the projection morphisms, and $i: M_1 \times_Y M_2 \rightarrow M_1 \times M_2$ the natural morphism; i.e. we consider the following commutative diagram:  \[ \begin{tikzcd} M_1 \times_Y M_2 \arrow{rr}{i} \arrow{dd}{p_1} \arrow{rd}{p_2} & & M_1 \times M_2 \arrow{dd}{f_1 \times f_2} \\ & M_2 \arrow{d}{f_2}  \\ M_1 \arrow{r}{f_1} & Y \arrow{r}{\Delta} & Y \times Y  \end{tikzcd}\]
	Suppose $Y$ is non singular. Let $(\phi_e)_{e \in I}$ be a basis of $K_H(Y) \simeq K^H(Y)$. Suppose there exists a decomposition of the diagonal in $K_H(Y)$, i.e. suppose there exists a basis $(\phi_e^\vee)_{e \in I}$ of $K_H(Y)$ such that \[\Delta_*[\O_Y] = \sum_{e \in I} \phi_e \boxtimes \phi_e^\vee \in K_H(Y \times Y)\] where we consider the diagonal action of $H$ on $Y \times Y$. We then obtain:  \begin{equation*}
	i_* [\mathcal{O}_{M_1 \times_Y M_2}] = \sum_{e, f \in I} g^{ef} \big(f_1^* \phi_e \boxtimes f_2^*\phi_f \big)  \in K_H(M_1 \times M_2),
	\end{equation*}
	where we denote by $g^{ef} := g^{-1}_{ef}$ the pairing dual to $g_{ef}$, and consider the diagonal action of $H$ on $M_1 \times M_2$ and $M_1 \times_Y M_2$.
\end{proposition}

\begin{lemme}[Flat base change in equivariant $K$-theory] \label{lem: flat base change} 
	Let $H$ be a linear algebraic group. Consider a Cartesian diagram of proper, equivariant morphisms of varieties with an $H$-action:  \[ \begin{tikzcd} X' \arrow{r}{g'} \arrow{d}{f'} & X \arrow{d}{f} \\  S \arrow{r}{g} & S'  \end{tikzcd}\]
	where $S \rightarrow S'$ is a flat morphism. Then, for all $\alpha$ in $K_H(X)$: \begin{equation*}
	f'_*(g'^*\alpha) = g^*(f_* \alpha) \in K_H(S).
	\end{equation*}
\end{lemme}

\begin{proof}
	Cf. \cite{chriss2009representation} Proposition 5.3.15.
\end{proof}

\begin{proof}[Proof of Proposition \ref{prop: image fiber product}]
	We call $f := f_1 \circ p_1 = f_2 \circ p_2$. First, notice that, since $f_1$ and $p_1$ are flat, $f$ is a flat $H$-equivariant morphism. We hence have: $[\mathcal{O}_{M_1 \times_Y M_2}] = f^*([\mathcal{O}_Y])$. 
	Now consider the diagonal action of $H$ on $M_1 \times M_2$ and $Y \times Y$. Since $\Delta: Y \rightarrow Y \times Y$ and $i : M_1 \times_Y M_2 \rightarrow M_1 \times M_2$ are proper $H$-equivariant morphism, $f_1 \times f_2$ is an $H$-equivariant morphism, and the morphism $f_1$ is flat, we can apply Lemma \ref{lem: flat base change} to the following commutative diagram of $H$-equivariant maps.  \[ \begin{tikzcd} M_1 \times_Y M_2 \arrow{r}{p_2} \arrow{d}{p_1} & M_2 \arrow{d}{f_2} \\  M_1 \arrow{r}{f_1} & Y \end{tikzcd}\]
	We obtain: \begin{equation*}
	i_* [\mathcal{O}_{M_1 \times_Y M_2}] = i_* f^* [\mathcal{O}_Y] = (f_1 \times f_2)^* \Delta_* [\mathcal{O}_Y] \in K_H(Y)
	\end{equation*}
	Furthermore, we have supposed: $\Delta_* [\mathcal{O}_Y] = \sum_{e, f \in I} g^{ef} \phi_e \boxtimes \phi_f \in K_H(Y) \simeq K^H(Y)$. 
	We thus obtain:  \begin{equation*}
	i_* [\mathcal{O}_{M_1 \times_Y M_2}] = (f_1 \times f_2)^* \big( \sum_{e, f \in I} g^{ef} \phi_e \boxtimes \phi_f \big) = \sum_{e, f \in I} g^{ef} \big(f_1^* \phi_e \boxtimes f_2^*\phi_f \big) \in K_H(M_1 \times M_2).
	\end{equation*}
\end{proof}



\section{Proof of Theorem \ref{th : euler char equal}.}\label{sec: proof main th}

\subsection{Properties of  $\overline{{\mathcal{M}}_{0,r}}(X_{\widehat{k}}, (\pi_{\widehat{k}})_* \mathbf{d}) \times_ {(X_{\widehat{k}})^r} X^r$.} We use here the same notations as in the preceding parts. Let $X$ be the flag variety parametrizing flags of vector spaces $V_{n_1} \subset \dots \subset V_{n_m} \subset \mathbb{C}^n$. For $1 \leq k \leq m$, we denote by $X_{\widehat{k}}$ the flag variety obtained from $X$ by forgetting the $k$-th vector space, and call $\pi = \pi_{\widehat{k}} : X \rightarrow X_{\widehat{k}}$ the forgetful morphism. For a class $\beta$ of curve on $X_{\widehat{k}}$, we denote by $M':=\overline{{\mathcal{M}}_{0,r}}(X_{\widehat{k}}, \beta)$ the moduli space parametrizing stable maps from genus $0$ $r$-pointed curves to $X_{\widehat{k}}$. We observe here that the scheme $M' \times_{{X_{\widehat{k}}}^r} X^r$ associated with the following Cartesian diagram is a projective variety with rational singularities, obtained via flat base change of $M'$ over $(X_{\widehat{k}})^r$. \[ \begin{tikzcd} M' \times_{{X_{\widehat{k}}}^r} X^r \arrow{d} \arrow{r} & X^r \arrow{d}{(\pi_{\widehat{k}})^r} \\
M' \arrow{r}{ev} & (X_{\widehat{k}})^r\end{tikzcd}\]

\begin{lemme}\label{lem: geometry M'xX^r}
	$\overline{{\mathcal{M}}_{0,r}}(X_{\widehat{k}}, (\pi_{\widehat{k}})_* \mathbf{d}) \times_{(X_{\widehat{k}})^r} X^r$ is a projective variety with rational singularities.
\end{lemme}
\begin{proof}
	For simplicity set $M':= \overline{{\mathcal{M}}_{0,r}}(X_{\widehat{k}}, (\pi_{\widehat{k}})_* \mathbf{d})$. First, notice that since $M' \times_{(X_{\widehat{k}})^r} X^r$ is a projective base change, it is projective over the projective scheme $X^r$ and hence over $\Spec \C$. 
	Furthermore, according to Kleiman's transversality theorem, for $g$ general in $\GL_n^r$ the scheme $M' \times_{(X_{\widehat{k}})^r} g \cdot X^r$ is reduced. Since the morphism $X^r \rightarrow (X_{\widehat{k}})^r$ is $\GL_n^r$-equivariant, and the action of $\GL_n^r$ on $(X_{\widehat{k}})^r$ is transitive, the action of $g^{-1}$ induces an isomorphism $M' \times_{(X_{\widehat{k}})^r} g \cdot X^r \simeq M' \times_{(X_{\widehat{k}})^r} X^r$. Hence $M' \times_{(X_{\widehat{k}})^r} X^r$ is reduced. Note that according to Theorem \ref{th: Viehweg X//G has RS}, $M'$ has rational singularities.
	Hence, according to \cite{buch}, Theorem 2.5, since all varieties considered here have rational singularities, for $g$ general in $\GL_n^r$ the scheme $M' \times_{(X_{\widehat{k}})^r} g \cdot X^r$ has rational singularities. Hence $M' \times_{(X_{\widehat{k}})^r} X^r$ has rational singularities.
\end{proof}
Finally, note that since $\pi_{\widehat{k}}: X \rightarrow X_{\widehat{k}}$ is a fibration in Grassmanians, it is flat. 
\subsection{Proof of Theorem \ref{th : euler char equal}.}
We denote by $i : \overline{{\mathcal{M}}_{0,r}}(X_{\widehat{k}}, (\pi_{\widehat{k}})_* \mathbf{d}) \underset{(X_{\widehat{k}})^r}\times X^r \hookrightarrow \overline{{\mathcal{M}}_{0,r}}(X_{\widehat{k}}, (\pi_{\widehat{k}})_* \mathbf{d}) \times X^r$ the natural morphism.

Let $(\phi_e)_{e \in I}$ be a basis of $K_H(X) \simeq K^H(X)$. By hypothesis there exists a basis $(\phi_e^\vee)_{e \in I}$ such that $\phi_e \boxtimes \phi_f$ is a diagonal decomposition of $X$. Note that $X^r$ admits a diagonal decomposition in $K_H(X^r)$, which is given by $\Delta_* [\O_{X^r}] = \sum_{e_1, \dots, e_r \: \i I}  (\boxtimes_{1 \leq i \leq r} \phi_{e_i})_{e_1, \dots, e_r \in I} (\boxtimes_{1 \leq i \leq r} \phi_{e_i}^\vee)_{e_1, \dots, e_r \in I} \in K_H(X^r \times X^r)$.

\begin{lemme}\label{lem: decomposing pxev_*}
	Let $\d=(d_1, \dots, d_m)$ be a degree in $E(X)$ such that $\forall \: p<k, \: d_k \geq n_k \ceil*{ \frac{d_p-d_{p-1}}{n_p-n_{p-1}}}$. Suppose the morphism $\Pi \times ev : \overline{{\mathcal{M}}_{0,r}}(X, \mathbf{d}) \rightarrow \overline{{\mathcal{M}}_{0,r}}(X_{\widehat{k}}, (\pi_{\widehat{k}})_* \mathbf{d}) \underset{(X_{\widehat{k}})^r}\times X^r$ is surjective. Then \begin{align*} i_*(\Pi \times ev)_* \left[ \O_{\overline{{\mathcal{M}}_{0,r}}(X, \d)} \right] 
	=  \sum_{e_1, \dots, e_r \in I}  \left((ev_1^* \phi_{e_1}) \otimes \dots \otimes (ev_r^* \phi_{e_r})\right) \boxtimes \left(\pi^*(\phi_{e_1}^{\vee}) \boxtimes \dots \boxtimes \pi^*(\phi_{e_r}^{\vee}) \right) 
	\end{align*}
\end{lemme}
\begin{proof}
	For simplicity, set $M':= \overline{{\mathcal{M}}_{0,r}}(X_{\widehat{k}}, (\pi_{\widehat{k}})_* \mathbf{d})$, $S := X^r$, and $S' := (X_{\widehat{k}})^r$. According to Lemma \ref{lem: geometry M'xX^r}, $M' \times_{S'} S$ is a projective variety with rational singularities. Furthermore, since $\Pi \times ev$ is surjective, $M' \times_{S'} S$ is irreducible. Finally, according to Chapter \ref{chap : geometry W} Proposition \ref{prop : geometry general fiber pr1W} since $\forall \: p<k, \: d_k \geq n_k \ceil*{ \frac{d_p-d_{p-1}}{n_p-n_{p-1}}}$ the general fiber of $\Pi \times ev$ is a rationally connected variety. Hence $\Pi \times ev$ is an $H$-equivariant surjective morphism of irreducible projective varieties with rational singularities with rationally connected general fiber, hence according to Theorem \ref{th: BM f_* if fiber rational} we have in $K_H(M' \times S')$: \begin{align*}
	i_* (\Pi \times ev)_* \left[ \O_{\overline{{\mathcal{M}}_{0,r}}(X, \d)} \right] 
	&= i_* \left[ \O_{(\Pi \times ev) \overline{{\mathcal{M}}_{0,r}}(X, \d)} \right] \\
	&= i_* \left[ \O_{M' \times_{S'} S}\right] \\
	&= \sum_{e_1, \dots, e_r \in I} ev^*(\phi_{e_1} \boxtimes \dots \phi_{e_r}) \boxtimes ((\pi)^r)^*\left(\phi_{e_1}^\vee \boxtimes \dots \boxtimes \phi_{e_r}^\vee \right)\\
	&= \sum_{e_1, \dots, e_r \in I}  \left((ev_1^* \phi_{e_1}) \otimes \dots \otimes (ev_r^* \phi_{e_r})\right) \boxtimes \left(\pi^*(\phi_{e_1}^{\vee}) \boxtimes \dots \boxtimes \pi^*(\phi_{e_r}^{\vee}) \right) 
	\end{align*}
	where the third equality holds since the morphism $(\pi)^r$ is a flat $H$-equivariant morphism-cf. Lemma \ref{lem: flat base change}.
\end{proof}

\begin{lemme}
	Let $Y$ and $Z$ be projective varieties with an $H$-action. Suppose $Z$ non singular and suppose both $Y$ and $Z$ are irreducible. Then for $\alpha \in K_H(Y)$, $\beta \in K_H(Z)\simeq K^H(Z)$, we have $\chi_H(\alpha \boxtimes \beta) = \chi_H(\alpha) \chi_H(\beta)$.
\end{lemme}
\begin{proof}
	We consider the following Cartesian diagram \[\begin{tikzcd} Y \times Z \arrow{r}{p_2} \arrow{d}{p_1} & Z \arrow{d}{\pi_1} \\
	Y \arrow{r}{\pi_2} & \Spec(\C)
	\end{tikzcd}\] where we have denoted by $p_1: Y \times Z \rightarrow Y$ and $p_2 : Y \times Z \rightarrow Z$ the $H$-equivariant projections. Since $Y$ is irreducible the morphism $Y \rightarrow \Spec(\C)$ is flat. Since $\chi_H$ is the $H$-equivariant pushforward to the point, we have: \begin{align*}
	\chi_H(\alpha \boxtimes \beta) &= \chi_H(p_1^* \alpha \otimes p_2^*\beta)\\
	&=\chi_H(\alpha \otimes (p_1)_*p_2^* \beta)\\
	&=(\pi_1)_* \left(\alpha \otimes \pi_1^* ((\pi_2)_*\beta)\right) \\
	&= \chi_H(\alpha) \chi_H(\beta),
	\end{align*}
	where the third equality holds according to Lemma \ref{lem: flat base change} and the last equality holds according to the projection formula.
\end{proof}

\begin{proof}[Proof of Theorem \ref{th : euler char equal}] Let $\alpha_1$, $\dots$, $\alpha_r$ be elements in $K^H(X)$. 
	Note that the following diagram is a commutative diagram of proper morphisms: \[\begin{tikzcd}
	\overline{{\mathcal{M}}_{0,r}}(X, \d) \arrow{drr}{ev} \arrow{d}{\Pi \times ev} && \\
	\overline{{\mathcal{M}}_{0,r}}(X_{\widehat{k}}, \pi_* \d) \times_{(X_{\widehat{k}})^r} X^r \arrow{r}{i} & 	\overline{{\mathcal{M}}_{0,r}}(X_{\widehat{k}}, \pi_* \d) \times X^r \arrow{r}[swap]{p_2} & X^r \\
	\end{tikzcd}\]
	
	We obtain in $K_H(\overline{{\mathcal{M}}_{0,r}}(X_{\widehat{k}}, \pi_* \d) \times X^r)$: \begin{align*}
	& i_* (\Pi \times ev)_* \left(ev_1^* \alpha_1 \otimes  \dots \otimes ev_r^* \alpha_r \otimes [\O_{\overline{{\mathcal{M}}_{0,r}}(X, \d)}] \right)  \\
	&=	 i_* (\Pi \times ev)_* \left( ev^*(\alpha_1 \boxtimes \dots \boxtimes \alpha_r) \otimes [\O_{\overline{{\mathcal{M}}_{0,r}}(X, \d)}] \right)\\
	&=	 i_* (\Pi \times ev)_* \left( (p_2 \circ i \circ (\Pi \times ev))^*(\alpha_1 \boxtimes \dots \boxtimes \alpha_r) \otimes [\O_{\overline{{\mathcal{M}}_{0,r}}(X, \d)}] \right)\\
	&= \left((p_2)^*(\alpha_1 \boxtimes \dots \boxtimes \alpha_r)  \right) \otimes i_* (\Pi \times ev)_*[\O_{\overline{{\mathcal{M}}_{0,r}}(X, \d)}]\\
	&=  \left( [\O_{\overline{{\mathcal{M}}_{0,r}}(X_{\widehat{k}}, \pi_* \d)}]  \boxtimes (\alpha_1 \boxtimes \dots \boxtimes \alpha_r) \right) \otimes i_* (\Pi \times ev)_*[\O_{\overline{{\mathcal{M}}_{0,r}}(X, \d)}] \\
	&= \left( [\O_{\overline{{\mathcal{M}}_{0,r}}(X_{\widehat{k}}, \pi_* \d)}]  \boxtimes (\alpha_1 \boxtimes \dots \boxtimes \alpha_r) \right) \otimes  \sum_{e_1, \dots, e_r \in I}  \left((ev_1^* \phi_{e_1}) \otimes \dots \otimes (ev_r^* \phi_{e_r})\right) \boxtimes \left(\pi^*(\phi_{e_1}^{\vee}) \boxtimes \dots \boxtimes \pi^*(\phi_{e_r}^{\vee}) \right)\\
	&= \sum_{e_1, \dots, e_r \in I}  \left((ev_1^* \phi_{e_1}) \otimes \dots \otimes (ev_r^* \phi_{e_r})\right) \boxtimes \left((\pi^*(\phi_{e_1}^{\vee}) \otimes \alpha_1) \boxtimes \dots \boxtimes (\pi^*(\phi_{e_r}^{\vee}) \otimes \alpha_r)) \right),
	\end{align*}
	where the third equality holds according to the projection formula, and the last equality holds according to Lemma \ref{lem: decomposing pxev_*}. Finally, decompose each element $\pi_* \alpha_i$ in the basis $(\phi_e)_{e \in I}$ of $K_H(X_{\widehat{k}}) \simeq K^H(X_{\widehat{k}})$: $\pi_* \alpha_i = \sum_{e \in I} a_{i,e} \phi_e$. Note that we have $\chi_H(\pi_* \alpha_i \otimes \phi_f^\vee) = a_{i,f}$. Denote by $p_2: \overline{{\mathcal{M}}_{0,r}}(X_{\widehat{k}}, \pi_* \d) \times X^r \rightarrow X^r$ the second projection, and by $p: (X_{\widehat{k}})^r \rightarrow \Spec \C$ the $H$-equivariant projection to the point. We obtain: \begin{align*}
	\chi_H &\left(ev_1^* \alpha_1\otimes  \dots \otimes ev_r^* \alpha_r \otimes [\O_{\overline{{\mathcal{M}}_{0,r}}(X, \d)}] \right) \\
	&= p_* (\pi^r)_*(p_2)_* i_* (\Pi \times ev)_* \left(ev_1^* \alpha_1 \otimes  \dots \otimes ev_r^* \alpha_r \otimes [\O_{\overline{{\mathcal{M}}_{0,r}}(X, \d)}] \right) \\
	&=  p_* (\pi^r)_*(p_2)_* \left(  \sum_{e_1, \dots, e_r \in I}  \left((ev_1^* \phi_{e_1}) \otimes \dots \otimes (ev_r^* \phi_{e_r})\right) \boxtimes \left((\pi^*(\phi_{e_1}^{\vee}) \otimes \alpha_1) \boxtimes \dots \boxtimes (\pi^*(\phi_{e_r}^{\vee}) \otimes \alpha_r) \right) \right)\\
	&=  \sum_{e_1, \dots, e_r \in I} \chi_H\left((ev_1^* \phi_{e_1}) \otimes \dots \otimes (ev_r^* \phi_{e_r})\right) p_* (\pi^r)_* \left((\pi^*(\phi_{e_1}^{\vee}) \otimes \alpha_1) \boxtimes \dots \boxtimes (\pi^*(\phi_{e_r}^{\vee}) \otimes \alpha_r) \right)\\
	&= \sum_{e_1, \dots, e_r \in I} \chi_H\left((ev_1^* \phi_{e_1}) \otimes \dots \otimes (ev_r^* \phi_{e_r})\right) \prod_{1 \leq i \leq r} \chi_H(\phi_{e_i}^\vee \otimes \alpha_i)\\
	&= \sum_{e_1, \dots, e_r \in I} \chi_H\left((ev_1^* \phi_{e_1}) \otimes \dots \otimes (ev_r^* \phi_{e_r})\right) \prod_{1 \leq i \leq r} a_{i,e_i}
	\end{align*}
	where the before last equality is obtained by applying the projection formula to the flat morphism $(\pi)^r$. Finally, we obtain:  \begin{align*}
	\chi_H &\left(ev_1^* \alpha_1\otimes  \dots \otimes ev_r^* \alpha_r \otimes [\O_{\overline{{\mathcal{M}}_{0,r}}(X, \d)}] \right) &= \sum_{e_1, \dots, e_r \in I}  \chi_H\left(a_{1,e_1}(ev_1^* \phi_{e_1}) \otimes \dots \otimes a_{r,e_r}(ev_r^* \phi_{e_r})\right)
	\end{align*}
	which yields the expected result since $\sum_{e \in I} a_{i,e} \phi_e = \pi_* \alpha_i$.
\end{proof}

\section{Geometric interpretation in the non equivariant setting}\label{sec: geometric proof}

Consider the case when $H$ is the identity element of $G$. We provide here an alternative proof of the equality between correlators provided by Theorem \ref{th : euler char equal}, using results of Chapter \ref{chap : geometry W}.

We use here the same notations as in Chapter \ref{chap : geometry W}. As usual, we say a property (P) holds for a general point in a variety $X$ if (P) is true for points belonging to a dense open subset of $X$. Recall a flag variety $X$ can be written as $G/P$, where $G=\GL_n$ and $P$ is a parabolic subgroup of $X$ containing a Borel subgroup $B \subset P \subset G$. For an element $u$ in the Weyl group $W$ of $G= \GL_n$ we denote by $X(u) := \overline{BuP/P}$ the associated Schubert variety, and by $\O_u := [\O_{X(u)}]$ the associated class in the Grothendieck ring $K(X)$ of coherent sheaves on $X$. Let $r \geq 0$. Let $u_1$, $\dots$, $u_r$ be elements in $W$. Let $\d=(d_1, \dots, d_m)$ be a degree in $E(X)$. Recall for an element $(g_1, \dots, g_r)$ in $G^r$ we call \textit{Gromov-Witten variety associated with the $u_i$} the variety \[{\W_{g_1, \dots, g_r}^X := \overline{\mathcal{M}}_{0,r}}(X, \d) \times_{X^r} \prod_{i=1}^r g_i X(u_i)\] parametrizing genus zero degree $\d$ stable maps into $X$ sending their $i$-th marked point within the variety $g_i X(u_i)$. Note that for an element $(g_1, \dots, g_r)$ general in $G^r$ the correlators $\langle \O_{u_1}, \dots, \O_{u_r} \rangle_{\d}$ satisfy \[\langle \O_{u_1}, \dots, \O_{u_r} \rangle_{\d}^X = \chi (\W_{g_1, \dots, g_r}^X). \]
The flag variety $X$ parametrizes flags of vector spaces $V_{n_1} \subset \dots \subset V_{n_m} \subset \C^n$, where $0 <n_1 < \dots < n_m < n$. Let $X'$ be the flag variety obtained from $X$ by forgetting the $k$-th vector space. Denote by $\pi: X \rightarrow X'$ the forgetful map. Note that for all $i$ we have $\pi_* \O_{u_i} = [\O_{X'(u_i)}] = \O_{u_i} \in K(X')$. We obtain \[\langle \O_{u_1}, \dots, \O_{u_r} \rangle_{\pi_*\d}^{X'} = \chi (\W_{g_1, \dots, g_r}^{X'}).\]
Now suppose the morphism $\Pi \times ev: \overline{\mathcal{M}_{0,r}}(X, \d) \rightarrow \overline{\mathcal{M}_{0,r}}(X', \pi_* \d) \times_{(X')^r} X^r$ is surjective. Furthermore suppose $\forall \: p<k, \: d_k \geq n_k \ceil*{ \frac{d_p-d_{p-1}}{n_p-n_{p-1}}}$. Then according to Chapter  \ref{chap : geometry W} Theorem \ref{th : RC fibrations bewteen W varieties} for an element $(g_1, \dots, g_r)$ general in $G^r$ each irreducible component of the variety $\W_{g_1, \dots, g_r}^X$ surjects into a different irreducible component of the variety $\W_{g_1, \dots, g_r}^{X'}$, and the general fiber of the projection $\W_{g_1, \dots, g_r}^X \rightarrow \W_{g_1, \dots, g_r}^{X'}$ is a rationnally connected variety. Furthermore according to \cite{buch} Theorem $2.5$ for an element $(g_1, \dots, g_r)$ general in $G^r$ the variety $\W_{g_1, \dots, g_r}$ has rational singularities. Theorem \ref{th: BM f_* if fiber rational} then yields \[\langle \O_{u_1}, \dots, \O_{u_r} \rangle_{\d}^X  = \langle \O_{u_1}, \dots, \O_{u_r} \rangle_{\d'}^{X'}, \]
which is the equality described by Theorem \ref{th : euler char equal}.

\chapter{Schubert calculus for $Fl_{1,n-1}$}\label{chap: Schubert calculus}

\section{Introduction}

Let $X=Fl_{1,n-1}$ be the incidence variety parametrizing lines contained in hyperplanes of $\C^n$. We study here the product of Schubert classes in the small quantum $K$-ring of $X$. 
Part \ref{sec : incidence var} is devoted to recalling some standard facts about incidence varieties, and fixing the notations we use in the rest of the chapter. The aim of parts  \ref{sec : produit d'intersection X} and \ref{sec : main proof} is to prove a closed formula for Littlewood-Richardson coefficients in $K(X)$. Part  \ref{sec : produit d'intersection X} describes one parameter degenerations of Richardson varieties, from which the expression of Littlewood-Richardson coefficients in $K(X)$ is derived in Part \ref{sec : main proof}. The aim of Part \ref{sec: computing 3pts correlators} is to compute all three points quantum $K$-theoretical correlators of $X$. Note that in the context of small quantum cohomology the product of two Schubert classes is given directly by three points Gromov-Witten invariants. In small quantum $K$-theory, the product of two Schubert classes is an alternate sum involving three points correlators and  two points correlators. Part \ref{sec: degree} is dedicated to finding a boundary on the degree of the product of two classes in $QK_s(X)$. A Chevalley formula in $QK_s(X)$ is proven in Part \ref{sec: Chevalley formula}. Finally we study Littlewood-Richardson coefficients in $QK_s(X)$ in Part \ref{sec: LR coefficients in QK(X)}. We provide an algorithm computing these coefficients, and conjecture a closed formula for them. We additionally prove that they satisfy a "positivity rule".


\section{Incidence varieties.}\label{sec : incidence var}
Let us begin by fixing down some of the usual notations associated with a generalized flag variety. Let $G:= \GL_n$, let $T$ be a maximal torus of $G$ and $B$ be a Borel subgroup of $G$ containing $T$. A generalized flag variety can be written as $X:= \GL_n(\mathbb{C})/P$, where $P$ is a parabolic subgroup of $\GL_n$  and satisfying \[T \subset B \subset P \subset G.\] Let us name $W=N_G(T)/T \simeq \mathfrak{S}_n$ the Weyl group of $G$, $W_P:=N_L(T)/T$ the subgroup of $W$ associated with $P$, and $W^P := W/W_P$. For $w \in W^P$, note that the variety $\overline{Bw'P/P}$ does not depend on the choice of an element $w'$ in $W$ representing $w$; we obtain a \textit{Schubert variety} $X(w) := \overline{BwP/P}$. The homology classes $[X(w)]$ from a basis of the integral homology of $X$, where $w$ runs over $W^P$. The classes $\O_{w}:=[\O_{X(w)}]$ form an additive basis of the Grothendieck ring $K_\circ(X)$ of coherent sheaves on $X$. Furthermore, we denote by $(\I_w)_{w \in W^P}$ the dual basis of $(\O_{w})$ for the Euler characteristic. If we denote by $\delta_{u,v}$ the symbol equal to $1$ if $u=v$, and else equal to $0$, we have for all $u,v \in W^P$: \[\chi(\O_u, \I_v) = \delta_{u,v}.\]
For an element $w$ in $W^P$, we denote by $\ell(w)$ the length of a minimal length representant of $w$. Note that $\ell(w)= \dim X(w)$. 

We consider here the incidence variety $X=Fl_{1,n-1}$ parametrizing lines included in hyperplanes of $\C^n$. Forgetting a line or a hyperplane yields a forgetful map $X \rightarrow \P^{n-1}$ which is a fibration in $\P^{n-2}$'s. Fix a basis $<e_1, \dots, e_n>$ of $\mathbb{C}^n$. In this basis, consider the maximal torus in $G=\GL_n$ consisting of diagonal matrices, the Borel subgroup $B$ consisting of upper triangular matrices of $\GL_n$, and the parabolic subgroup $P \supset B$ associated with the reductive group $\GL_1 \times \GL_{n-2} \times \GL_1$. Note that $W=N_G(T)/T \simeq \mathfrak{S}_n$. We denote by $w_0$ the element of $W$ of greatest length. Note that $w_0$ is the permutation of $1$ $\dots$ $n$ defined by $w_0(i)=n-i+1$.  Furthermore, an element in $W^P \simeq \mathfrak{S}_n/\mathfrak{S}_{n-2}$ is uniquely determined by the choice of a permutation sending $1$ to $i$ and $n$ to $j$, where $1\leq i,j \leq n$ and $i \neq j$. We name $w_{i,j}$ the element in $W^P$ thus defined. Schubert varieties of $X \simeq \GL_n(\mathbb{C})/P$ are described by: \begin{align}
X(i,j) :&= X(w_{i,j})= \overline{Bw_{i,j} P/P} \nonumber \\ &=\left\{ (L,H) \in X \bigr\rvert L\subset <e_1, \cdots, e_i>,\: <e_1, \cdots, e_{j-1}> \subset H \right\}.\label{eq : Schubert var}
\end{align}
Schubert varieties can be embedded in $\mathbb{P}^{n-1} \times \mathbb{P}^{n-1}$ as:\begin{equation}\label{eq : X in PnxPn}
X(i,j) =\left\{ \left( [x_1: \dots :x_i:0: \dots :0], [0: \dots 0: y_j : \dots : y_n] \right) \in \mathbb{P}^{n-1} \times \mathbb{P}^{n-1} \biggr\rvert \sum x_k y_k =0 \right\} .
\end{equation}
Note that for $i<j$ the condition $\sum x_k y_k=0$ is redundant, and $X(i,j)$ can simply be described as:\begin{equation}\label{eq : Schubert var embedded i<j}
X(i,j) =\left\{ \left( [x_1: \dots :x_i:0: \dots :0], [0: \dots 0: y_j : \dots : y_n] \right) \in \mathbb{P}^{n-1} \times \mathbb{P}^{n-1} \right\} .
\end{equation}
For $1 \leq i,j \leq n$, $i \neq j$, we name $$\O_{i,j} := [\O_{X(w_{i,j})}] \in K_\circ(X)$$ the class of the pushforward of the structure sheaf $\O_{X(w_{i,j})}$ in the Grothendieck group $K_\circ(X)$, and by $\I_{i,j} := \I_{w_{i,j}}$ its dual with respect to the pairing $\chi: \alpha, \beta \rightarrow \chi(\alpha \cdot \beta)$. 
For any $u \in W^P$, the Schubert variety \textit{opposite} to the Schubert variety $X(u)$ is $$X^u := w_0 X(w_0 u).$$ Applying $w_0$ to (\ref{eq : Schubert var}) yields the following description of opposite Schubert varieties of $X$: \begin{equation}\label{eq : op Schubert var}
X^{w_0 w_{k,l}} = w_0 X(w_{k,l})=\left\{ (L,H) \in X \bigr\rvert L\subset <e_{n-k+1}, \cdots, e_n>,\: <e_{n-l+2}, \cdots, e_{n}> \subset H \right\}.
\end{equation}
Given $v$, $w$ in $W^P$, the corresponding \textit{Richardson variety} is $X_w^v := X(w) \cap w_0 X(w_0 v)$. Consider $1 \leq i,j,k,l \leq n$, $i \neq j$ and $k \neq l$. Richardson varieties of $X$ are described by: \begin{align} X_{i,j}^{k,l} &:=  X_{w_{i,j}}^{w_0 w_{k,l}} = X(w_{i,j}) \cap w_0 X(w_{k,l}) \nonumber \\
&=\{ (L,H) \in X | L \subset <e_{n-k+1}, \dots , e_i>; \: <e_1, \dots e_{j-1},e_{n-l+2}, \dots , e_n> \subset H \},\label{eq: Richardson variety}
\end{align}
if $n-k+1 \leq i$ and $j-1 \leq n-l+2$, and $ X_{i,j}^{k,l} = \emptyset$ else. The second equality comes from considering the description of Schubert varieties given by (\ref{eq : Schubert var}) and of opposite Schubert varieties given by (\ref{eq : op Schubert var}).

\section{Degenerating Richardson varieties of $Fl_{1,n-1}$.}\label{sec : produit d'intersection X}
We describe here one parameter families of deformations of Richardson varieties. More precisely, we describe irreducible projective varieties $\Phi$ of $\mathbb{P}^1 \times X$ such that one fiber of the projection morphism $\Phi \rightarrow \mathbb{P}^1$ is a Richardson variety, while another fiber is a union of Schubert varieties, up to multiplication by elements of $\GL_n$.\\

To facilitate the proof, let us first introduce an auxilliary variety. Set $e_{n+1} = 0 = e_0$. For $1 \leq r \leq h \leq n$, $0 \leq p \leq n-r-1$, we will consider the following subvariety of $X$: \begin{equation}\label{eq: def Y} Y_{r,p}^h:= \{(L,H)\in X | L \subset <e_1, \dots e_h>, \: <e_1, \dots e_r, e_{n-p+1}, \dots , e_n> \subset H \}.\end{equation}
The Plücker embedding $X \hookrightarrow Gr(1,n) \times Gr(n-1,n) \simeq \mathbb{P}^{n-1} \times \mathbb{P}^{n-1}$ embeds $Y_{r,p}^h$ as: \begin{equation}\label{eq: def Y embedded} Y_{r,p}^h:= \{[x_1, \dots, x_h, 0, \dots, 0]; [0: \dots : 0 : y_{r+1}: \dots : y_{n-p}: 0: \dots : 0] \in \mathbb{P}^{n-1} \times \mathbb{P}^{n-1} | \sum_a x_a y_a=0 \}.\end{equation}

\begin{lemme}\label{lem: finding g}
	\begin{enumerate}
		\item Let $1 \leq i,j,k,l \leq n$, where $i \neq j$ and $k \neq l$. Suppose $n-k+1 \leq i$ and $j-1 \leq n-l+2$.\\
		Denote by $r$ the number of shared vectors between $(e_{n-k+1}, \dots , e_i)$ and $(e_1, \dots e_{j-1},e_{n-l+2}, \dots , e_n)$. Call $h:=i+k-n$, and $p := j+l-2-r$.  \\
		There exists $g$ in $\GL_n$ such that $g \cdot X_{i,j}^{k,l} = Y_{r,p}^h$.
		\item Fix $1 \leq r \leq n$, $0 \leq p \leq n-r-1$.Then there exists $g$ in $\GL_n$ such that $Y_{r,p}^r = g \cdot X(r, r+p+1)$.
		\item Fix $1 \leq r \leq h \leq n$. Then $Y_{r,0}^h = X(h, r+1)$.
	\end{enumerate}
\end{lemme}

\begin{proof}
	\begin{enumerate}
		\item Recall we have defined $r$ as the number of shared vectors between $(e_{n-k+1}, \dots , e_i)$ and $(e_1, \dots e_{j-1},e_{n-l+2}, \dots , e_n)$, and $p=j+l-2-r$. Notice: \begin{align*} p=j+l-2-r &=j-1+n-(n-l+1)-r \\ &=\#(e_1, \dots e_{j-1},e_{n-l+2}, \dots , e_n)-r, \end{align*} hence $p$ is the number of vectors in $(e_1, \dots e_{j-1}, e_{n-l+2}, \dots , e_n)$ not belonging to $(e_{n-k+1}, \dots , e_i)$. Now call $(\epsilon_1, \dots, \epsilon_p)$ these $p$ vectors, and $(\epsilon_{p+1}, \dots, \epsilon_{p+r})$ the $r$ vectors in $(e_1, \dots e_{j-1},e_{n-l+2}, \dots , e_n)$ belonging to $(e_{n-k+1}, \dots , e_i)$. Finally, call $(\eta_{r+1}, \dots, \eta_h)$ the $\#(e_{n-k+1}, \dots , e_i)-r =h-r$ vectors of $(e_{n-k+1}, \dots , e_i)$ that do not belong to $(e_1, \dots e_{j-1},e_{n-l+2}, \dots , e_n)$, and $(\eta_{h+1}, \dots, \eta_{n-p})$ the remaining vectors. Denote by $g$ the following permutation of the basis vectors: \begin{equation*}
		(e_1, \dots, e_n) \rightarrow (\epsilon_{p+1}, \dots, \epsilon_{p+r}, \eta_{r+1}, \dots, \eta_h, \eta_{h+1}, \dots, \eta_{n-p}, \epsilon_1, \dots, \epsilon_p).
		\end{equation*}
		Then, according to (\ref{eq : Schubert var}), we have: \begin{align*}
		g \cdot X_{i,j}^{k,l} &= \{(L,H)\in X | L \subset <e_{1}, \dots, e_h>, \: <e_{1}, \dots e_{r}, e_{n-p+1}, \dots , e_n> \subset H \}\\
		&=Y_{r,p}^h
		\end{align*}
		where the first equality follows directly from the definition of $g$ and the definition (\ref{eq: Richardson variety}) of $X_{i,j}^{k,l}$: $$X_{i,j}^{k,l} = \{ (L,H) \in X | L \subset <e_{n-k+1}, \dots , e_i>; \: <e_1, \dots e_{j-1},e_{n-l+2}, \dots , e_n> \subset H \}.$$
		\item By definition, we have: \begin{equation*}
		Y_{r,p}^r:= \{(L,H)\in X | L \subset <e_1, \dots e_r>, \: <e_1, \dots e_r, e_{n-p+1}, \dots , e_n> \subset H \}.
		\end{equation*}
		If $r=n-p$, set $g:=id$. Else, denote by $g$ the permutation of the basis vectors sending: $$(e_1, \dots, e_r,e_{r+1}, \dots, e_{n-p}, e_{n-p+1}, \dots, e_n) \rightarrow (e_1, \dots, e_r, e_{n-p+1}, \dots, e_n, e_{r+1}, \dots, e_{n-p}).$$ Then, according to (\ref{eq : Schubert var}), we have:
		\begin{align*}
		g \cdot Y_{r,p}^r &= \{(L,H)\in X | L \subset <e_{1}, \dots, e_r>, \: <e_{1}, \dots e_{r}, e_{r+1}, \dots , e_{r+p}> \subset H \}\\
		&= X(r,r+p+1)
		\end{align*}
		\item We have: \begin{align*}
		Y_{r,0}^h:&= \{(L,H)\in X | L \subset <e_1, \dots e_h>, \: <e_1, \dots e_r> \subset H \}\\
		&=X(h,r+1)
		\end{align*}
		where the first equality is the definition (\ref{eq: def Y}) of $Y_{r,0}^h$ and the third equality is the definition (\ref{eq : Schubert var}) of Schubert varieties.
	\end{enumerate}
\end{proof}

\begin{proposition}\label{prop: degenerating}
	Let $1 \leq i,j,k,l \leq n$, where $i \neq j$ and $k \neq l$. 
	\begin{enumerate}
		\item If $n+1 > i+k$ or $j+l \geq n+2$, then $X_{i,j}^{k,l} = \emptyset$;
		\item Suppose $i+k \geq n+1$ and $j+l \leq n+1$ and ($i<j$ or $k<l$). Then there exists $g$ in $\GL_n$ such that: \begin{equation*}
		X_{i,j}^{k,l} = g \cdot X(i+k-n, j+l-1);
		\end{equation*}
		\item Suppose $i+k- n \geq j+l$ and $j+l \leq n+1$. Then there exists $g$ in $\GL_n$, and an irreducible projective subvariety $\Phi$ of $\mathbb{P}^1 \times X$ such that: \begin{align*}\pi^{-1}([1 : 0]) = \{[1:0]\} \times g \cdot X_{i,j}^{k,l} && \mathrm{and} &&\pi^{-1}([0: 1])=\{[0:1]\} \times X(i+k-n,j+p-1); \end{align*}
		where $\pi : \Phi \subset \mathbb{P}^1 \times X \rightarrow \mathbb{P}^1$ is the natural projection;
		\item Suppose $1 \leq i+k -n \leq j+l-1 \leq n$ and $i>j$ and $k>l$. Then there exists $g$ in $\GL_n$, and an irreducible projective subvariety $\Phi$ of $\mathbb{P}^1 \times X$ such that: \begin{align*}\pi^{-1}([1 : 0]) &= \{[1:0]\} \times g \cdot X_{i,j}^{k,l} \\ \pi^{-1}([0: 1]) &=\{ [0:1]\} \times Y_{i+k-n-1,j+l-1-i-k+n}^{i+k-n-1} \cup \{ [0:1]\} \times Y_{i+k-n,j+l-1-i-k+n}^{i+k-n} \end{align*}
		where $\pi : \Phi \subset \mathbb{P}^1 \times X \rightarrow \mathbb{P}^1$ is the natural projection.
	\end{enumerate}
\end{proposition}

\begin{proof}
	\begin{enumerate}
		\item By definition of $X_{i,j}^{k,l}$.
		\item Notice that then all vectors of $E_1=(e_{n-k+1}, \dots , e_i)$ belong to $E_2=(e_1, \dots e_{j-1},e_{n-l+2}, \dots , e_n)$, hence the number $r$ of shared vectors between $E_1$ and $E_2$ is equal to $h=\# E_1=i+k-n$. Hence, according to Lemma \ref{lem: finding g} $(i)$, there exists $g$ in $\GL_n$ such that $$g \cdot X_{i,j}^{k,l} = Y_{h,p}^h,$$ where $h=i+k-n$ and $p=j+l-2-h$.\\ Furthermore, according to Lemma \ref{lem: finding g} $(ii)$, there exists $g'$ in $\GL_n$ such that $$Y_{h,p}^h = g' \cdot X(i+k-n, j+l-1).$$ Hence: \begin{equation*}
		X_{i,j}^{k,l} = g^{-1} \cdot Y_{h,p}^h = g^{-1} g' \cdot X(i+k-n, j+l-1).
		\end{equation*}
		\item 	We use here the same notations as in Lemma \ref{lem: finding g}. Denote by $r$ the number of shared vectors between $(e_{n-k+1}, \dots , e_i)$ and $(e_1, \dots e_{j-1},e_{n-l+2}, \dots , e_n)$. Call $h:=i+k-n$, and $p := j+l-2-r$. \\
		Let $\Phi$ be the subvariety of $X \times \mathbb{P}^1$ defined by :
		\begin{multline*}
		\Phi = \{ ([u:v]; [x_1 : \dots : x_h : 0 : \dots :0]; [0: \dots :0: y_{r+1}: \dots : y_{n}]) \in \mathbb{P}^1 \times \mathbb{P}^{n-1}\times \mathbb{P}^{n-1} \\
		\left| \sum_{a=r+1}^{h} x_a y_a  =0 \: \mathrm{and} \: \forall \: r< a \leq r+p, \: vy_a = u y_{n-p+a-r}\} \right.
		\end{multline*}
		As usual, we denote by $D(v)$ (respectively $D(u)$) the subvariety of elements $([u:v]; [x_1 : \dots : x_n]; [y_1 : \dots : y_n])$ in $ \mathbb{P}^1 \times \mathbb{P}^{n-1}\times \mathbb{P}^{n-1}$ satisfying $v \neq 0$ (resp. $u \neq 0$). On $D(v)$, we have \begin{align*} \Phi_{|D(v)} \simeq \{(t; [x_1, \dots, x_h];[0: \dots :0: t y_{n-p+1}: \dots : t y_{n}: y_{r+p+1}: \dots :y_{n}] \in  \mathbb{C} \times \mathbb{P}^{n-1} \times \mathbb{P}^{n-1} \\
		| t \sum_{a=r+1}^{r+p} x_a y_{n-p+a-r} + \sum_{a=r+p+1}^h x_a y_a=0 \}.\end{align*} Since $r+p=j+l-2 < i+k-n-1 =h-1<h$, $\Phi_{|D(v)}$ is irreducible. In the same way, $\Phi_{|D(u)}$ is also irreducible. Hence $\Phi$ is an irreducible projective subvariety of $X \times \mathbb{P}^1$. \\
		
		Denote by $\pi : \Phi \rightarrow \mathbb{P}^1$ the natural projection. $\pi$ verifies :
		\begin{align*}
		&\pi^{-1} (\{ [1:0]\}) \\
		&\:\:\:\:\:\:\:\:\: = \{ ([1:0]; [x_1 : \dots : x_h : 0 : \dots :0]; [0: \dots :0: y_{r+1}: \dots : y_{n-p}, 0, \dots, 0]) \in \mathbb{P}^1 \times \mathbb{P}^{n-1}\times \mathbb{P}^{n-1} 
		| \\ &\:\:\:\:\:\:\:\:\:\:\:\:\:\:\:\:\:\:\:\:\:\:\:\:\:\:\:\sum_{i=r+1}^{h} x_i y_i  =0 \}\\
		&\:\:\:\:\:\:\:\:\: = \{ [1:0]\} \times Y_{r,p}^h\\
		&\:\:\:\:\:\:\:\:\:= \{ [1:0]\} \times g \cdot X_{i,j}^{k,l}  
		\end{align*}
		where the existence of the element $g$ is ensured by Lemma \ref{lem: finding g} $(ii)$, and the second equality is the definition (\ref{eq: def Y embedded}) of $Y_{r,p}^h$.
		
		Finally, $\pi$ verifies: \begin{align*}
		\pi^{-1}([0:1]) &= \{ ([0:1]; [x_1 : \dots : x_h : 0 : \dots :0]; [0: \dots :0: y_{r+p+1}: \dots : y_{n}]) \in \mathbb{P}^{n-1}\times \mathbb{P}^{n-1} | \sum_{a=r+1}^{h} x_a y_a  =0 \}\\
		&=  \{[0:1]\} \times \{(L,H)\in X | L \subset <e_1, \dots e_{h}>, \: <e_1, \dots e_{p+r}> \subset H \}\\
		&= \{[0:1]\} \times Y_{j+l-2,0}^{i+k-n}\\
		&= \{[0:1]\} \times X(i+k-n,j+l-1)
		\end{align*}
		where the second equality is the definition (\ref{eq: def Y embedded}) and the last equality holds according to Lemma \ref{lem: finding g} $(iii)$. 
		\item Again, we use the same notations as in Lemma \ref{lem: finding g}. Denote by $r$ the number of shared vectors between $(e_{n-k+1}, \dots , e_i)$ and $(e_1, \dots e_{j-1},e_{n-l+2}, \dots , e_n)$. Call $h:=i+k-n$, and $p := j+l-2-r$. \\
		Let $\Phi$ be the subvariety of $X \times \mathbb{P}^1$ defined by:
		\begin{multline*}
		\Phi = \bigg\{ ([u:v]; [x_1 : \dots : x_h : 0 : \dots :0];   
		[0 : \dots :0 : y_{r+1}: \dots : y_{n-p-r+h-1}: 0 : \dots :0]) \in \mathbb{P}^1 \times \mathbb{P}^{n-1} \times \mathbb{P}^{n-1} \\
		\left| \forall \: 1 \leq a \leq h-1-r, \: v y_{r+a} = u y_{n-p+a}, \: \mathrm{and} \: \sum_{a= r+1}^{h} x_a y_a =0 \right\}  .
		\end{multline*}
		On $D(v)$, we have $\Phi_{|D(u)} \simeq \mathrm{Spec}\mathbb{C}[t] \times \mathbb{P}^{h-1} \times Y$, where $Y= \{[0: \dots :0: y_{r+1}: \dots : y_{n}] \in  \mathbb{P}^{n-1} | \}$. \begin{align*} \Phi_{|D(v)} \simeq \{(t; [x_1, \dots, x_h:0 : \dots :0];[0: \dots :0: t y_{n-p+r+1}: \dots : t y_{n-p+r+h-1}: y_{h}: \dots :y_{n-p+r-h-1}: 0 \dots:0]\\
		\in  \mathbb{C} \times \mathbb{P}^{n-1} \times \mathbb{P}^{n-1} 
		\left| t \sum_{a=r+1}^{h-1} x_a y_{n-p+a} + x_h y_h=0 \right\}.\end{align*}
		Now notice that, since $j<n-l+2$, the number $r$ of shared vectors between $(e_{n-k+1}, \dots , e_i)$ and $(e_1, \dots e_{j-1},e_{n-l+2}, \dots , e_n)$ satisfies $r<h= \# (e_{n-k+1}, \dots , e_i)$. Hence $\Phi_{|D(u)}$ is irreducible. In the same way, $\Phi_{|D(u)}$ is also irreducible. Hence $\Phi$ is an irreducible projective subvariety of $X \times \mathbb{P}^1$. \\
		
		Denote by $\pi : \Phi \rightarrow \mathbb{P}^1$ the natural projection. $\pi$ satisfies:
		\begin{align*}
		&\pi^{-1} ([1:0]) \\
		&\:\:\:\:\:\:\:\:\: = \{ ([1:0]; [x_1 : \dots : x_h : 0 : \dots :0]; [0: \dots :0: y_{r+1}: \dots : y_{n-p}, 0, \dots, 0]) \in \mathbb{P}^1 \times \mathbb{P}^{n-1}\times \mathbb{P}^{n-1} \\
		&	\:\:\:\:\:\:\:\:\: \:\:\:\:\:\:\:\:\: \:\:\:\:\:\:\:\:\: | \sum_{a=r+1}^{h} x_a y_a  =0 \}\\
		&\:\:\:\:\:\:\:\:\: = \{ [1:0]\} \times Y_{r,p}^h\\
		&\:\:\:\:\:\:\:\:\:= \{ [1:0]\} \times g \cdot X_{i,j}^{k,l}  
		\end{align*}
		where the existence of the element $g$ is ensured by Lemma \ref{lem: finding g} $(ii)$, and the second equality is the definition (\ref{eq: def Y embedded}) of $Y_{r,p}^h$.\\
		Finally, we observe:
		\begin{align*}
		&\pi^{-1} ([0:1]) \\
		&\:\:\:\:\:\:\:\:\: = \{ ([0:1]; [x_1 : \dots : x_h : 0 : \dots :0]; [0: \dots :0: y_{h}: \dots : y_{n-p-r+h-1}, 0, \dots, 0]) \in \mathbb{P}^1 \times \mathbb{P}^{n-1}\times \mathbb{P}^{n-1} 
		| x_h y_h =0 \}\\
		&\:\:\:\:\:\:\:\:\: = \{ ([0:1]; [x_1 : \dots : x_{h-1} : 0 : \dots :0]; [0: \dots :0: y_{h}: \dots : y_{n-p-r+h-1}, 0, \dots, 0]) \in \mathbb{P}^1 \times \mathbb{P}^{n-1}\times \mathbb{P}^{n-1} \} \\
		&\:\:\:\:\:\:\:\:\:\:\:\:\:\:\:\:\:\:\cup \{ ([0:1]; [x_1 : \dots : x_{h} : 0 : \dots :0]; [0: \dots :0: y_{h+1}: \dots : y_{n-p-r+h-1}, 0, \dots, 0]) \in \mathbb{P}^1 \times \mathbb{P}^{n-1}\times \mathbb{P}^{n-1} \}\\
		&\:\:\:\:\:\:\:\:\: = \{ [0:1]\} \times Y_{h-1,p+r-h+1}^{h-1} \cup \{ [0:1]\} \times Y_{h,p+r-h+1}^h\\
		&\:\:\:\:\:\:\:\:\: = \{ [0:1]\} \times Y_{i+k-n-1,j+l-1-i-k+n}^{i+k-n-1} \cup \{ [0:1]\} \times Y_{i+k-n,j+l-1-i-k+n}^{i+k-n}
		\end{align*}
		where the third equality is the definition (\ref{eq: def Y embedded}) and the last equality is deduced from $h=i+k-n$ and $r+p=j+l-2$.
	\end{enumerate}
\end{proof}
%

\section{Littlewood-Richardson coefficients in $K(Fl_{1,n-1})$.}\label{sec : main proof}
\subsection{Computing Littlewood-Richardson coefficients.} We deduce here Littlewood-Richardson coefficients in $K(Fl_{1,n-1})$ from the results of Part \ref{sec : produit d'intersection X}. For simplicity, we set $\O_{i,j} =0$ if $i<1$ or $j>n$. 
\begin{proposition}
	\label{prop : produit dans K(X)}
	Let $1 \leq i,j, k, p \leq n$, where $i \neq j$ and $k \neq p$. Then 
	\begin{equation*}
	\left\{ 
	\begin{aligned}
	\mathcal{O}_{k,p} \cdot \mathcal{O}_{i,j} &= \mathcal{O}_{i+k -n, j+p -1} \: &\mathrm{if} \: i+k-n \geq j+p \: \mathrm{or} \: i<j \: \mathrm{or} \: k<p; \\
	\mathcal{O}_{k,p} \cdot \mathcal{O}_{i,j} &= \mathcal{O}_{i+k-n-1, j+p -1} + \mathcal{O}_{i+k-n, j+p} - \mathcal{O}_{i+k-n-1,j+p} \:  &\mathrm{otherwise}.
	\end{aligned}
	\right.
	\end{equation*}
\end{proposition}
where, for all $1\leq k, p \leq n$, $k \neq p$, we call $\mathcal{O}_{k,p}:=[\mathcal{O}_{X(w_{k,p})}]$ the class in $K_\circ(X)$ associated with the longest length element $w_{k,p}$ in the $n$-th symetric group sending $1$ to $k$ and $n$ to $p$.

\begin{lemme}\label{lem: equality cup}
	Let $1 \leq h \leq n$, $0 \leq p \leq n-h-1$. \\
	Then we have the following equality in $K(Fl_{1,n-1})$:
	\begin{equation*}
	[\mathcal{O}_{Y_{h-1,p}^{h-1} \cup Y_{h,p}^h}] = [\mathcal{O}_{Y_{h-1,p}^{h-1}}] + [\mathcal{O}_{Y_{h,p}^h}] - [\mathcal{O}_{Y_{h+1,p}^{h-1}}] 
	\end{equation*}
\end{lemme}

\begin{proof}
	Denote by $I_1$ and $I_2$ the bihomogeneous ideals defining $Y_{h,p}^h$ and $Y_{h-1,p}^{h-1}$ embedded in $\mathrm{Proj} \mathbb{C}[x_1, \dots, x_n] \times \mathrm{Proj} \mathbb{C}[y_1, \dots, y_n]$.\\
	According to (\ref{eq: def Y embedded}) we have: \begin{align*}
	I_1 = <x_{h+1}, \dots, x_n; y_1, \dots, y_{h}, y_{n-p+1}, \dots, y_n> &&\mathrm{and} && I_2 =  <x_{h}, \dots, x_n; y_1, \dots, y_{h-1}, y_{n-p+1}, \dots, y_n>.
	\end{align*}		
	Call $S=\mathbb{C}[x_1, \dots, x_n]$ and $T=\mathbb{C}[y_1, \dots, y_n]$ the graded rings.  The following is an exact sequence of $S \otimes T$ bigraded-modules: \begin{align*}
	0 \rightarrow    S \otimes T_{/<I_1 \cap I_2>} \rightarrow S \otimes T_{/I_1} \oplus S \otimes T_{/I_2} \rightarrow S \otimes T_{/<I_1, I_2>} \rightarrow 0.
	\end{align*}
	This induces the following exact sequence of $\mathcal{O}_{\mathbb{P}^{n-1} \times \mathbb{P}^{n-1}}$-modules: \begin{equation*}
	0 \rightarrow \mathcal{O}_{Y_{h-1,p}^{h-1} \cup Y_{h,p}^h} \rightarrow  \mathcal{O}_{Y_{h-1,p}^{h-1}} \oplus \mathcal{O}_{Y_{h,p}^h}  \rightarrow \widetilde{S \otimes T}_{/<I_1, I_2>} \rightarrow 0,
	\end{equation*}
	where $\widetilde{S \otimes T}_{/<I_1,I_2>}$ is the $\mathcal{O}_{\mathbb{P}^{n-1} \times \mathbb{P}^{n-1}}$-module associated to the $S \otimes T$ bigraded-module $S \otimes T_{/<x_{h}, \dots, x_n; y_1, \dots, y_{h}, y_{n-p+1}, \dots, y_n>}$. This implies the following equalities in $K(Fl_{1,n-1})$:
	\begin{align*}
	[\mathcal{O}_{X(i,j) \cup X(i-1,j-1)}] &= [\mathcal{O}_{Y_{h,p}^h} \oplus \mathcal{O}_{Y_{h-1,p}^{h-1}}] - [\widetilde{S \otimes T}_{/<I_1, I_2}]\\
	&= [\mathcal{O}_{Y_{h,p}^h}] + [\mathcal{O}_{Y_{h-1,p}^{h-1}}] - [\mathcal{O}_{Y_{h+1,p}^{h-1}}].
	\end{align*}
	The second equality is deduced from $\widetilde{S \otimes T}_{/<I_1,I_2>} = \mathcal{O}_{Y_{h+1,p}^{h-1}}$, which is a direct consequence of (\ref{eq: def Y embedded}).
\end{proof}

Finally, Proposition \ref{prop : produit dans K(X)} follows directly from Lemmas \ref{lem: equality g.}, \ref{lem: equality homotetie}, \ref{lem: equality cup} and from Proposition \ref{prop: degenerating}.
Indeed, notice for any $1 \leq i,j,k,l \leq n$, such that $i \neq j$ and $k \neq l$, we have: \begin{equation*}
[\mathcal{O}_{X(i,j)}] \cdot [\mathcal{O}_{X(k,l)}] = [\mathcal{O}_{X_{i,j}^{k,l}}],
\end{equation*}
since the intersection between a Schubert varieties and its opposite variety is transverse, and Schubert varieties have rational singularities-cf. for example \cite{brion}, Theorem $4.2.1$ $(i)$. Furthermore, according to Lemmas \ref{lem: equality homotetie}, \ref{lem: equality cup}, translations by $g$ and one parameter deformations do not change the sheaf class. Finally, Lemma \ref{lem: equality cup} and Proposition \ref{prop: degenerating} give the result.

\subsection{Link with cohomological Littlewood-Richardson coefficients.}\label{subsec: LR coeff H_*(X)}
Let $i$, $k$, $j$ and $l$ be integers between $1$ and $n$ such that $i \neq j$ and $k \neq l$. Since the morphism $A_*(X) \rightarrow Gr K_\circ(X)$ is a morphism of graded abelian groups, we can deduce the intersection product between two classes $[X(i,j)]$ and $[X(k,l)]$ in $A_*(X)$ from the product betweeen $[\mathcal{O}_{i,j}]$ and $[\mathcal{O}_{k,l}]$. 
Proposition \ref{prop : produit dans K(X)} describes Littlewood-Richardson coefficients in $A_*(X) \simeq H^*(X)$:
\begin{equation*}
\left\{ 
\begin{aligned}
\: [X(k,l)] \cup [X(i,j)] &= 0 \: &&\mathrm{if} \: i+k \leq n \: \mathrm{or} \: j+l \geq n+2; \\
[X(k,l)] \cup [X(i,j)] &= [X(i+k -n -1, j+l -1)] + &&[X(i+k -n, j+l)]  \: \\ &&&\mathrm{if} \: 1 \leq i+k-n \leq j+l-1 \leq n \: \mathrm{and} \: i>j, \: k>l; \: \\
[X(k,l)] \cup [X(i,j)] &= [X(i+k -n, j+l -1)] \: &&\mathrm{otherwise.}
\end{aligned}
\right.
\end{equation*}

\begin{Rem}
	\label{Rem : prod Chevalley}  
	Chevalley's formula for generalized flag varieties already gives us the intersection product between a codimension $1$ Schubert class and any other Schubert class. Let us check we do recover the same formula.\\
	According to the generalization of Chevalley's formula to general $G/P$ (cf. for example \cite{Fulton}, lemma 8.1), for all $u \in W/W_P$, for all $\beta \in \Delta \setminus \Delta_P$, we have :
	\begin{equation*}
	[Y(s_{\beta})] \cup [Y(u)] = \Sigma h_{\alpha} (\omega_{\beta}) [Y(\tilde{u} s_{\alpha})],
	\end{equation*}
	where the sum goes over all postive roots $\alpha$ such that $\ell ([\tilde{u} s_\alpha]) = \ell (u) +1$, where $h_{\alpha} (\omega_{\beta}) = n_{\alpha \beta} (\beta, \beta) / (\alpha, \alpha)$ (where $n_{\alpha \beta}$ is the coefficient of $\beta$ in the expansion of $\alpha$ as a positive linear combination of positive roots), and where we write $\tilde{u}$ a minimal length representative of $u$ in $W$.
	
	We thus obtain here : \begin{equation*}
	[Y(\alpha_1)] \cup [Y(u)] = \Sigma [Y(\tilde{u} s_{\alpha})],
	\end{equation*}
	where the sum goes over the $\alpha \in \{\epsilon_1 - \epsilon_i | 2 \leq i \leq n\} \subset R^+ - R_P^+$ verifying $\ell([\tilde{u}s_\alpha]) = \ell(u)+1$.
	
	Let $\tilde{u} = w_{i,j}$ and $s_{l} = s_{\epsilon_1 -\epsilon_l}$. We observe : $\ell(w_{i,j}) = i-1+n-j$ if $i<j$, and $\ell(w_{i,j}) = n+i-j-2$ else; hence $\ell ([\tilde{u}s_l]) = \ell(u) + 1$ iff($l=i+1$ and $i<j$) or ($l=n$ and $j=i+1$) or ($l=i$ and $i>j$), from which we deduce :
	\begin{align*}
	[Y(\alpha_1)] \cup [Y(i,j)] & = [Y({i+1,j})] &\: \mathrm{if} \: j \neq i+1, \: i \neq n;\\
	[Y(\alpha_1)] \cup [Y(n,j)] &= 0;\\
	[Y(\alpha_1)] \cup[Y(i,i+1)] &= [Y(i+2,i+1)] + [Y(i+1,i)], 
	\end{align*}
	which does indeed gives us $[X(n-1,1)] \cup [X(k,p)] = [X(k-1, p)]$ if $k \neq p+1$, $k \neq 1$, $[X(n-1,1)] \cup [X(1,p)] =0$, and $[X(n-1,1)] \cup [X(k,k-1)] = [X(k-2, k-1)] + [X(k-1,k)]$.\\
\end{Rem}

\begin{Rem}
	Let $\beta \in \Delta\setminus \Delta_P$. If we use the notations of Remark \ref{Rem : prod Chevalley}, we can write here:
	\begin{equation*}
	[\mathcal{O}_{Y(s_\beta)}] \cdot [\mathcal{O}_{Y(u)}] = \sum_{\substack{\alpha \in R^+ \setminus R_P^+ , \\ \ell([\tilde{u}s_{\alpha}]) = \ell(u) +1}} h_\alpha (\omega_{\beta}) [\mathcal{O}_{Y(\tilde{u}s_{\alpha})}] - \sum_{\substack{\alpha, \gamma \in R^+ \setminus R_P^+ \\ \ell([\tilde{u}s_{\alpha}]) = \ell(u) +1 = \ell([\tilde{u}s_{\gamma}])}} \sum_{\substack{v \in R^+ \setminus R_P^+, \: \ell(v)=\ell(u)+2, \\ v \succcurlyeq \alpha \: \mathrm{et} \: v  \succcurlyeq \gamma}} h_{\alpha} (\omega_{\beta}) h_{\gamma} (\omega_{\beta}) [\mathcal{O}_{Y(v)}].
	\end{equation*} \\
\end{Rem}

\section{Three points correlators of $Fl_{1,n-1}$.}\label{sec: 3 points corr}

\subsection{Three points correlators of $\P^n$.} Let $n>0$. We compute here genus zero correlators of $\P^n$. Note that Buch-Mihalcea fully described the product of two Schubert classes in the ring $QK_s(\P^n)$ \cite{buch2011}; we could compute three points correlators using this description. We give here a different proof, by studying the geometry of Gromov-Witten varieties of $\P^n$. This result plays a key role in our derivation of correlators of $X$ presented Part \ref{subsec: 3 points cor X}. For $1 \leq i \leq n+1$, we denote by $L_i$ the Schubert variety associated with the permutation $w_{1,i}$ permuting $1$ and $i$: $$L_i := \{[x_1: \dots: x_i:0 : \dots :0] \in \P^n \mid x_1, \dots, x_i \in \C\}.$$
Let $d$ in $E(\P^n) \simeq \mathbb{N}$. We consider here the compactification of morphisms $\P^1 \rightarrow \P^n$ given by the space of quasi-maps \cite{braverman2006spaces}. A degree $d$ morphism $\P^1 \rightarrow \P^n$ is defined by $(n+1)$ homogeneous polynomials of degree $d$. We may thus parametrize degree $d$ morphisms $\P^1 \rightarrow \P^n$ by $P=\P(\C^{(d+1)(n+1)})$, where we view a point in $\C^{(d+1)(n+1)}$ as a set of $(n+1)$ degree $d$ polynomials. For a point $p$ in $P=\P(\C^{(d+1)(n+1)})$, we will denote by $$f_p: \P^1 \dashrightarrow \P^n$$ the associated map. Note that if the polynomials associated with a point $p$ in $P$ have common roots $x_i$, then $f_p$ is a rational map which is not defined on the points $x_i \in \P^1$.
The dense subvariety $U$ of $P$ parametrizing polynomials having no common roots is in bijection with $\mathrm{Hom}(\P^1, d)$. Fix distinct points $p_1$, $p_2$, $p_3$ on $\P^1$. According to the appendix of Chapter \ref{chap : stabilization correlators}, sending a point $p$ in $U$ to the point in $\mathcal{M}_{0,3}(\P^n, d)$ associated with $(f_p: \P^1 \rightarrow \P^n, \{p_1, p_2, p_3\})$ defines a morphism $\phi: U \rightarrow \mathcal{M}_{0,3}(\P^n, d)$. Since $\mathcal{M}_{0,3}(\P^n, d)$ is irreducible \cite{kim2001connectedness} and $\phi$ is an injective dominant morphism of quasi-projective normal complex varieties, $\phi$ is birational. We may hence identify a dense open subset $V$ of $\mathcal{M}_{0,3}(\P^n, d)$ with its reciprocal image by $\phi$. Finally, note that for an element $p$ in $V \subset P$, the evaluation morphism $ev_i$ assigns to the $(n+1)$ polynomials associated with $p$ the projectivization of their value at $p_i$.

\begin{proposition}\label{prop: correlators Pn}
	Let $1 \leq i_1, i_2, i_3 \leq n+1$. Let $d \in \mathbb{N}^*$.
	\begin{enumerate}[label= \roman*)]
		\item Let $g_1$, $g_2$, $g_3$ be elements in $\GL_{n+1}$, let $z$ be an element in $\P^{n}$. The variety $ev_1^{-1}(g_1 \cdot L_{i_1}) \cap ev_2^{-1}(g_2 \cdot L_{i_2}) \cap ev_3^{-1}(g_3 \cdot L_{i_3})\cap V$ is either empty or an irreducible rational variety. The same holds for the variety $ev_1^{-1}(g_1 \cdot L_{i_1}) \cap ev_2^{-1}(g_2 \cdot L_{i_2}) \cap ev_3^{-1}(z)\cap V$.
		\item For $(g_1, g_2, g_3)$ general in $\GL_{n+1}^3$, the Gromov-Witten variety $ev_1^{-1}(g_1 \cdot L_{i_1}) \cap ev_2^{-1}(g_2 \cdot L_{i_2}) \cap ev_3^{-1}(g_3 \cdot L_{i_3})$ is either empty or an irreducible rational variety. In the same way, for $(g_1,g_2)$ general in $G^2$, the general fiber of the morphism $ev_3: ev_1^{-1}(g_1L_{i_1}) \cap ev_2^{-1}(g_2L_{i_2}) \rightarrow ev_3\left( ev_1^{-1}(g_1L_{i_1}) \cap ev_2^{-1}(g_2L_{i_2})\right) \subset \P^{n-1}$ evaluating the third marked point is an irreducible rational variety.
		\item \[\langle w_{1, i_1}, w_{1,i_2}, w_{1, i_3} \rangle_d = \left\{ \begin{array}{ll} 0 & \mathrm{if} \: d=1 \: \mathrm{and} \: i_1+i_2+i_3 < n+2 \\ 1 & \mathrm{else} \end{array}\right. \]
	\end{enumerate}
\end{proposition}
\begin{proof}	\begin{enumerate}[label= \roman*)]
		\item We denote by $L:=ev_1^{-1}(g_1 \cdot L_{i_1}) \cap ev_2^{-1}(g_2 \cdot L_{i_2}) \cap ev_3^{-1}(g_3 \cdot L_{i_3})\cap V$, which we see here as a subvariety of $P= \P(\C^{(d+1)(n+1)})$. Note that, if we denote by $(P_1, \dots, P_{n+1})$ the $(n+1)$ polynomials associated to a point in $P$, $L$ is defined by: \begin{equation*}
		L= \left\{(P_1, \dots, P_{n+1}) \in V \subset (\C^{d+1})^{n+1} \mid \forall 1 \leq k \leq 3, \: [P_1(p_k): \dots: P_{n+1}(p_k)] \in g_k \cdot L_{i_k} \right\}.
		\end{equation*}
		Since $g_k \cdot L_{i_k}$ is a linear subspaces of $\P^n$, the compactification of $L$ in $P=(\C^{d+1})^{n+1}$ is defined by linear equations. Hence $L$ is the intersection of the dense open subset $V$ with a linear subspace of $P$, hence $L$ is an irreducible rational variety. 
		
		In the same way, if $L:=ev_1^{-1}(g_1 \cdot L_{i_1}) \cap ev_2^{-1}(g_2 \cdot L_{i_2}) \cap ev_3^{-1}(z) \cap V$, since $g_k \cdot L_{i_k}$ is a linear subspace of $\P^n$ and a point in $\P^n$ can be described as the zero locus of a set of linear equations, $L$ is an irreducible rational variety.
		\item  Since $\P^n$ is convex, ${\mathcal{M}_{0,3}}(\P^n, d)$ is a dense open subset of $\overline{\mathcal{M}_{0,3}}(\P^n, d)$ \cite{fulton1996notes}, hence $V$ is a dense open subset of $\overline{\mathcal{M}_{0,3}}(\P^n, d)$. According to Chapter \ref{chap : geometry W}, for $g=(g_k)$ general in $\GL_{n+1}^3$, $ev^{-1}(g \cdot (L_{i_1} \times L_{i_2} \times L_{i_3}))$ has a dense intersection with the dense open subset $V$ of ${\mathcal{M}_{0,3}}(\P^n, d)$, hence is an irreducible rational variety according to $i)$. 
		
		For $(g_1,g_2)$ in $G^2$ denote by $\W_{g_1,g_2}:= ev_1^{-1}(g_1 \cdot L_{i_1}) \cap ev_2^{-1}(g_2 \cdot L_{i_2})$. For $g=(g_1,g_2)$ general in $G^2$ the variety $\W_{g_1,g_2}$ has a dense intersection with the dense open subset $V$; hence for $z$ general in $ev_3(\W_{g_1,g_2})$, the fiber $ev_3^{-1}(z) \cap \W_{g_1,g_2}$ has a dense intersection with the open subset $V$. Hence for $(g_1,g_2)$ general in $G^2$ and $z$ general in $ev_3(\W_{g_1,g_2})$, the variety $ev_3^{-1}(z) \cap \W_{g_1,g_2}$ is an irreducible rational variety.
		\item We name $G:= \GL_{n+1}$. Since $L_i$ is the Schubert variety associated with $w_{1, i+1}$, according to (\ref{eq: correlator = chi GW variety}) we have for $(g_1, g_2, g_3)$ general in $G^3$: \begin{align*}
		\langle w_{1, i_1+1}, w_{1,i_2+1}, w_{1, i_3+1} \rangle_d = \chi(\O_{ev_1^{-1}(g_1 \cdot L_{i_1})\cap ev_2^{-1}(g_2 \cdot L_{i_2})\cap ev_3^{-1}(g_3 \cdot L_{i_3})}).
		\end{align*}
		Suppose $i_1+i_2+i_3+1 < n$ and $d=1$. Then according to Lemma \ref{lem: curve in Pn through 3 points} for $(g_1, g_2, g_3)$ general in $G^3$, the Gromov-Witten variety $ev_1^{-1}(g_1 \cdot L_{i_1})\cap ev_2^{-1}(g_2 \cdot L_{i_2})\cap ev_3^{-1}(g_3 \cdot L_{i_3})$ is empty, hence $\langle w_{1, i_1+1}, w_{1,i_2+1}, w_{1, i_3+1} \rangle_1 =0$. Now suppose $d >1$, or $d=1$ and $i_1+i_2+i_3+1 \geq n$. Then according to Lemma \ref{lem: curve in Pn through 3 points} for $(g_1, g_2, g_3)$ general in $G^3$, the Gromov-Witten variety $ev_1^{-1}(g_1 \cdot L_{i_1})\cap ev_2^{-1}(g_2 \cdot L_{i_2})\cap ev_3^{-1}(g_3 \cdot L_{i_3})$ is non empty.
		Furthermore, according to \cite{buch} Theorem 2.5. for $(g_1, g_2, g_3)$ general in $G^3$ this variety has rational singularities. Hence its sheaf Euler characteristic is equal to the sheaf Euler characteristic of its desingularization, which is also a projective rational variety; hence it is equal to $1$.
	\end{enumerate}
\end{proof}


Let $1 \leq i_1, i_2,i_3 \leq n$. 
\begin{lemme}\label{lem: curve in Pn through 3 points}
	\begin{enumerate}[label=\roman*)]
		\item For $d \geq 2$, there exists a genus zero stable map of degree $d$ joining the varieties $g_1 \cdot L_{i_1}$, $g_2 \cdot L_{i_2}$ and $g_3 \cdot L_{i_3}$ for any element $(g_1, g_2, g_3)$ in $\GL_{n+1}^3$;
		\item For $(g_1, g_2, g_3)$ general in $\GL_{n+1}^3$, there exists a line joining the varieties $g_1 \cdot L_{i_1}$, $g_2 \cdot L_{i_2}$ and $g_3 \cdot L_{i_3}$ iff $i_1 + i_2 + i_3 \geq n+2$.
	\end{enumerate}
\end{lemme}
\begin{proof}	\begin{enumerate}[label=\roman*)]
		\item We first consider the case when $d_2=2$. It is enough to show that for any three points $P_1$, $P_2$ and $P_3$ in $\P^n$ there is a genus zero stable map of degree $2$ whose image contains all points $P_i$. Consider a line joining the points $P_1$ and $P_2$; now glue this line at $P_2$ with a line joining the points $P_2$ and $P_3$. We obtain a genus zero degree $2$ stable map $\P^1 \cup \P^1 \rightarrow \P^n$ joining all three points $P_i$. 
		
		Now suppose $d \geq 2$. Consider a genus zero degree $2$ stable map $f: C \rightarrow \P^n$ joining all three points $P_i$. Let $\P^1 \rightarrow \P^n$ be a genus zero degree $d-2$ map, whose image intersects $f(C)$ at the image $f(p)$ of a non special point $p \in C$. Such a map can ba obtained by translating the image of any degree $d-2$ map by an element in $\GL_{n+1}$. Now glue $C$ and $\P^1$ together at the point $p$. We obtain a degree $d$ stable map $C \cup \P^1 \rightarrow \P^n$ joining all three points $P_i$. 
		\item Name $G:= \GL_{n+1}$. Let $g_1$ and $g_2$ be elements in $G^2$. By the functorial definition of $\P^n$, there exist vectors $u_i$ and $v_i$ in $\C^{n+1}$ such that $g_1 \cdot L_{i_1}$ parametrizes lines in $\C^{n+1}$ included in the linear subspace $<u_1, \dots, u_{i_1}> \subset \C^{n+1}$ and $g_2 \cdot L_{i_2}$ parametrizes lines in $\C^{n+1}$ included in the linear subspace $<v_1, \dots, v_{i_2}> \subset \C^{n+1}$. For $(g_1, g_2)$ general in $G^2$, the vector space generated by the vectors $u_i$ and $v_i$ is of dimension $\mathrm{min}(n,i_1+i_2+1)$. Hence for $(g_1, g_2)$ general in $G^2$ the set of points lying on a line joining $g_1 \cdot L_{i_1}$ and $g_2 \cdot L_{i_2}$ parametrizes lines in $\C^{n+1}$ included in the linear subspace $$<u_1, \dots, u_{i_1}, v_1, \dots, v_{i_2}> \simeq \C^{\mathrm{min}(n+1,i_1+i_2)} \subset \C^{n+1}.$$
		This can be rephrased more formally as: $$ev_3(ev_1^{-1}(g_1 \cdot L_{i_1})\cap ev_2^{-1}(g_2 \cdot L_{i_2})) = h \cdot L_{\mathrm{min}(n+1,i_1+i_2)},$$
		where $h$ lies in $G$ and we denote by $ev_i: \overline{\mathcal{M}_{0,3}}(\P^n, 1)$ the evaluation morphism. Hence for $(g_1,g_2,g_3)$ general in $G^3$, there exists a line joining the varieties $g_1 \cdot L_{i_1}$, $g_2 \cdot L_{i_2}$ and $g_3 \cdot L_{i_3}$ iff the intersection between $h \cdot L_{\mathrm{min}(n+1,i_1+i_2)}$ and $g_3 \cdot L_{i_3}$ is non empty for $g_3$ general in $G$-where $h$ is an element in $G$ depending only on the choice of an element $(g_1, g_2)$ general in $G^2$. This is equivalent to the following condition in the Chow ring $A_*(\P^n)$: $$[g_3 \cdot L_{i_3}] \cup [h \cdot L_{\mathrm{min}(n,i_1+i_2)}] = [L_{i_3}] \cup [L_{\mathrm{min}(n+1,i_1+i_2)}] \neq 0.$$ Finally, the equality $[L_i] \cup [L_j] = [L_{\mathrm{min}(n,i+j-n)}]$ in $A_*(\P^n)$ yields the result, where we set $L_i = \emptyset$ for $i<0$.
	\end{enumerate}
\end{proof}

\subsection{Computing correlators of $X$.}\label{sec: computing 3pts correlators} We list here properties of $X$ that will allow us to compute genus zero correlators of $X$ in Part \ref{subsec: 3 points cor X}. Recall $X$ can be described as $X \simeq \GL_n/P$, where $P \supset B \supset T$ is a parabolic subgroup of $G=\GL_n$. Note that $P$ is a parabolic subgroup with associated reductive subgroup the block diagonal matrix $\GL_1 \times \GL_{n-2} \times \GL_1$ with $\GL_{n-2}$ as the central block. Recall for $1 \leq i,j \leq n$, where $i \neq j$, we denote by $w_{i,j}$ the representant in $W/W_P$ of any permutation in $N_G(T)/T \simeq \mathfrak{S}_n$ sending $1$ to $i$ and $n$ to $j$. Let $r \geq 0$, let $1 \leq i_1, \dots, i_r, j_1, \dots, j_r \leq n$, where $i_k \neq j_k$. For an element $(g_1, \dots, g_r)$ in $G^r$, we denote by $$\mathcal{W}_{g_k} := ev_1^{-1}(g_1 \cdot X(w_{i_1,j_1})) \cap \dots \cap ev_r^{-1}(g_r \cdot X(w_{i_r,j_r}))$$ the Gromov-Witten variety parametrizing rational curves of degree $\d$ with marked points in the Schubert varieties $g_k X(w_{i_k,j_k})$ in $\overline{\mathcal{M}_{0,r+1}}(X, \d)$. 
\begin{enumerate}[label=(\alph*)]
	\item \textsc{Computing correlators by considering the image of Gromov-Witten varieties.} 
	\begin{proposition}\label{prop: (O_u, ... O_v)=chi(O_ev_r*W)}
		\begin{enumerate}[label=\roman*)]
			\item Let $(g_1, \dots, g_r)$ be an element in $G^r$ such that\begin{enumerate}[label=\Alph*)]
				\item The general fiber of the evaluation morphism $ev_{r+1}: \mathcal{W}_{g_k} := \cap_{k=1}^r ev_k^{-1}(g_k \cdot X(w_{i_k,j_k})) \rightarrow \Im (ev_{r+1}) \subset X$ is a rationally connected variety;
				\item $Tor_i^{X^{r}}(\O_{\overline{\mathcal{M}_{0,r+1}}(X, \d)}, \O_{\prod_{k=1}^r g_kX(w_{i_k,j_k})})=0$ for all $i >0$;
				\item The Gromov-Witten variety $\W_{g_k}$ has rational singularities;
				\item The projected Gromov-Witten variety $ev_{r+1}(\W_{g_k})$ has rational singularities;
			\end{enumerate}
			Then for any $\alpha$ in $K^\circ(X)$ \begin{equation*}
			\langle \O_{i_1,j_1}, \dots, \O_{i_r,j_r}, \alpha \rangle_\d = \chi \left(\alpha \cdot [\O_{ev_{r+1}(\mathcal{W}_{g_k})}] \right).\end{equation*}
			\item If for $(g_1, \dots, g_r)$ general in $G^r$: \begin{enumerate}[label=\Alph*)]
				\item  The general fiber of the evaluation morphism $ev_{r+1}: \mathcal{W}_{g_k} := \cap_{k=1}^r ev_k^{-1}(g_k \cdot X(w_{i_k,j_k})) \rightarrow \Im (ev_{r+1}) \subset X$ is a rationally connected variety;
				\item The projected Gromov-Witten variety $ev_{r+1}(\W_{g_k})$ has rational singularities
			\end{enumerate}
			then for any $\alpha$ in $K^\circ(X)$ and for $(g_1, \dots, g_r)$ general in $G^r$: \begin{equation*} \langle \O_{i_1,j_1}, \dots, \O_{i_r,j_r}, \alpha \rangle_\d 
			= \chi \left(\alpha \cdot [\O_{ev_{r+1}(\mathcal{W}_{g_k})}] \right) \end{equation*}
			\item If \begin{enumerate}[label=\Alph*)]
				\item  For $(g_1,g_2)$ general in $G^2$, the general fiber of the evaluation morphism $ev_{3}: \mathcal{W}_{g_1, g_2}=ev_1^{-1}(g_1X(i_1,j_1)) \cap ev_2^{-1}(g_2 w_0 X(i_2,j_2)) \rightarrow \Im(ev_3) \subset X$ is a rationally connected variety;
				\item The projected Gromov-Witten variety $ev_{3}(\mathcal{W}_{1,w_0})$ has rational singularities,
			\end{enumerate}
			then  \begin{equation*} \langle \O_{i_1,j_1}, \O_{i_2,j_2}, \alpha \rangle_\d 
			= \chi \left(\alpha \cdot [\O_{ev_3(\mathcal{W}_{1,w_0})}] \right) \end{equation*}
		\end{enumerate}
	\end{proposition}
	
	\begin{proof}
		\begin{enumerate}[label=\roman*)]
			\item Since $Tor_i^{X^{r}}(\O_{\overline{\mathcal{M}_{0,r+1}}(X, \d)}, \O_{\prod_{k=1}^r g_kX(w_{i_k,j_k})})=0$ for all $i >0$, we have the following identity in $K_\circ(\overline{\mathcal{M}_{0,r+1}}(X, \d))$. \begin{align}ev^* \left(\prod_{k=1}^r \O_{i_k,j_k}\right) &= ev^* [\O_{\prod_{k=1}^r g_kX(w_{i_k,j_k})}] \nonumber \\
			&= \sum_{i \geq 0} [Tor_i^{X^{r}}(\O_{\overline{\mathcal{M}_{0,r+1}}(X, \d))}, \O_{\prod_{k=1}^r g_kX(w_{i_k,j_k})})] \nonumber\\
			&= [\O_{\overline{\mathcal{M}_{0,r+1}}(X, \d) \times_{X^r} \prod_{k=1}^r g_kX(w_{i_k,j_k})}] \nonumber\\
			&= [\O_{\mathcal{W}_{g_k}}].\label{eq: Ow}\end{align}
			By definition for $\alpha$ in $K^\circ(X)$ the correlator associated with $i_k$, $j_k$ and $\alpha$ is given by \[\langle \O_{i_1,j_1}, \dots, \O_{i_r,j_r}, \alpha \rangle_\d = \chi\left(ev_1^* \O_{i_1,j_1} \dots  ev_r^*\O_{i_r,j_r} \cdot ev_{r+1}^*\alpha \cdot [\O_{\overline{\mathcal{M}_{0,r+1}}(X, \d))}] \right) \]
			Using (\ref{eq: Ow}) here above and the projection formula in $K_\circ(X)$, we obtain: \begin{align*} \langle \O_{i_1,j_1}, \dots, \O_{i_r,j_r}, \alpha \rangle_\d &= \chi\left(ev_{r+1}^* \alpha  \cdot ev_1^* \O_{i_1,j_1} \dots  ev_r^*\O_{i_r,j_r} \cdot [\O_{\overline{\mathcal{M}_{0,r+1}}(X, \d))}] \right) \\
			&= \chi \left( ev_{r+1}^* \alpha  \cdot [\O_{\mathcal{W}_{g_k}}] \right)\\
			&= \chi \left(\alpha \cdot (ev_{r+1})_*[\O_{\mathcal{W}_{g_k}}] \right). \end{align*}
			Moreover, since the projected Gromov-Witten variety $ev_{r+1}(\mathcal{W}_{g_k})$ has rational singularities,
			according to \cite{buch2011} Theorem $3.1$ we have in $K_\circ(X)$: \begin{align}\label{eq: evr* w} (ev_{r+1})_* [\O_{\mathcal{W}_{g_k}}] = [\O_{ev_{r+1}(\mathcal{W}_{g_k})}], \end{align}
			which yields the expected equality.	
			\item First, according to \cite{buch} Theorem 2.5 for $(g_1, \dots, g_r)$ general in $G^r$ the Gromov-Witten variety $\mathcal{W}_{g_k}$ has rational singularities. Furthermore, according to Sierra's homological Kleiman-Bertini theorem \cite{sierra2009general}, for $(g_1, \dots, g_r)$ general in $G^r$, we have $Tor_i^{X^{r}}(\O_{\overline{\mathcal{M}_{0,r+1}}(X, \d)}, \O_{\prod_{k=1}^r g_kX(w_{i_k,j_k})})=0$ for all $i >0$. Hypothesis $A)$ and $B)$ finally imply the expected equality according to $i)$.
			\item Since for $(g_1,g_2)$ general in $G^2$, the general fiber of the evaluation morphism $ev_{3}: \mathcal{W}_{g_k} \rightarrow \Im(ev_3) \subset X$ is a rationally connected variety there exists a dense open subset $U$ of $G^2$ such that elements in $U$ satisfy conditions $A)$ and $B)$ and $C)$ of part $i)$. Since the evaluation morphism $ev: \overline{\mathcal{M}_{0,3}}(X,\d) \rightarrow X^3$ is $G$-equivariant, $U$ is invariant by the diagonal action of $G$. Hence according to \cite{Fulton} Lemma $7.1$, there exists $(h_1,h_2)$ in $U$ such that $h_1X(w_{i_1,j_1})=X(w_{i_1,j_1})$ and $h_2X(w_{i_2,j_2})=w_0X(w_{i_2,j_2})$.  Since the projected Gromov-Witten variety $ev_{3}(\mathcal{W}_{h_1,h_2})$ has rational singularities, we obtain according to $i)$: \begin{align}
			\langle \O_{i_1,j_1},  \O_{i_2,j_2}, \alpha \rangle_\d 
			&= \chi \left(\alpha \cdot [\O_{ev_3(\mathcal{W}_{1,w_0})}] \right) \label{eq: chi}.\end{align}
		\end{enumerate}
	\end{proof}
	\item \label{item: compute 3 pts corr when Pixev is surjective} \textsc{Computing correlators of $Fl_{1,n-1}$ when the natural map $\overline{\mathcal{M}_{0,3}}(X, \d) \rightarrow \overline{\mathcal{M}_{0,3}}(\P^{n-1}, d_1) \times_{(\P^{n-1})^3} X^3$ is surjective.} Recall $X$ parametrizes lines included in hyperplanes of $\C^n$. Denote by $\pi: X \rightarrow \P^{n-1}$ the projection induced by forgetting the hyperplane. Let $\d=(d_1, d_2):=d_1 l_1 + d_2 l_2$ be an effective class of curve in $E(X)$. Note that $\pi_*(d_1, d_2)=d_1 \in E(\P^{n-1})$. The forgetful morphism $\pi$ induces a morphism $\Pi : \overline{\mathcal{M}_{0,3}}(X, \d) \rightarrow \overline{\mathcal{M}_{0,3}}(\P^{n-1}, d_1)$. Together with the evaluation morphisms $ev_i$, this induces a morphism: $$\Pi \times ev: \overline{\mathcal{M}_{0,3}}(X, \d) \rightarrow \overline{\mathcal{M}_{0,3}}(\P^{n-1}, d_1) \times_{(\P^{n-1})^3} X^3.$$ 
	\begin{proposition}
		If \begin{itemize}
			\item The morphism $\Pi \times ev: \overline{\mathcal{M}_{0,3}}(X, \d) \rightarrow \overline{\mathcal{M}_{0,3}}(\P^{n-1}, \pi_* \d) \times_{(\P^{n-1})^3} X^3$ is surjective,
			\item The projected Gromov-Witten variety $ev_3(\W_{1,w_0}) = ev_3 \left( ev_1^{-1}(X(i_1,j_1)) \cap ev_2^{-1}(w_0 X(i_2,j_2)) \right)$ has rational singularities,
		\end{itemize}
		Then  \begin{equation*} \langle \O_{i_1,j_1}, \O_{i_2,j_2}, \alpha \rangle_\d 
		= \chi \left(\alpha \cdot [\O_{ev_3^\d(\mathcal{W}_{1,w_0})}] \right). \end{equation*}
	\end{proposition}
	
	\begin{proof}
		Denote by $P_m  \supset B$ the parabolic subgroup associated with $\P^{n-1}$, i.e. $P_m$ is a parabolic subgroup satisfying $$B \subset P \subset P_m$$ with associated reductive subgroup the block diagonal matrix $\GL_1 \times \GL_{n-1}$ whose first block is $\GL_1$. Note that the image of $w_{i,j}$ by the forgetful morphism $W/W_P \rightarrow W/W_{P_m}$ is the element $w_{1,i}$ in $W/W_{P_m}$ associated with any permutation sending $1$ to $i$. Recall the Schubert variety of $\P^{n-1}$ associated with the element $w_{1,i}$ in $W/W_{P_m}$ is $L_{i} := \{[x_1: \dots:x_i:0:\dots:0]\}.$
		For an element $(g_1, g_2)$ in $G^3$, denote by \begin{align*}\overline{\mathcal{M}_{0,3}}(X, \d) \supset \W_{g_k}:= \prod_{k=1,2} ev_k^{-1}(g_k X(w_{i_k,j_k})) &\:\:\: \mathrm{and} &  \W_{g_k}':= \prod_{k=1,2} ev_k^{-1}(g_k L_{i_k}) \subset \overline{\mathcal{M}_{0,3}}(\P^{n-1},d_1)\end{align*} the Gromov-Witten variety of $X$ (respectively $\P^{n-1}$) associated with the translated Schubert varieties $g_k X(w_{i_k,j_k})$ (respectively $g_k L_{i_k}$). 
		
		Since the morphism $\Pi \times ev$ is surjective and $X$ is a two step flag variety, according to Chapter \ref{chap : stabilization correlators} (Proposition \ref{prop : geometry general fiber pr1W}) the general fiber of $\Pi \times ev$ is rational. Hence according to Chapter \ref{chap : geometry W} (Proposition \ref{prop : W is irreducible}) the general fiber of the restriction $\W_{g_1,g_2} \rightarrow \W_{g_1,g_2}' \times_{(\P^{n-1})^3} g_1X(i_1,j_1) \times g_2 X(i_2,j_2) \times X$ of $\Pi \times ev$  is rationally connected. Furthermore, according to Proposition \ref{prop: correlators Pn}, the general fiber of the morphism \[ev_3^{d_1}: \W_{g_1,g_2}' \rightarrow ev_3(\W_{g_1,g_2}') \subset \P^{n-1}\] evaluating the third marked point is an irreducible rational variety.  Note that since $X \rightarrow \P^{n-1}$ is a fibration in $\P^{n-2}$'s, it is a flat morphism. Recall from Chapter \ref{chap : geometry W} that the flat base change and composition of a rationally connected fibration is a rationally connected fibration (Theorem \ref{th: ex stability under composition} (3) and Lemma \ref{lem: flat base change perserves (P)}); hence the general fiber of the morphism $\W_{g_1,g_2}' \times_{(\P^{n-1})^3} g_1X(i_1,j_1) \times g_2 X(i_2,j_2) \times X \rightarrow \P^{n-1} \times_{(\P^{n-1})^3} g_1X(i_1,j_1) \times g_2 X(i_2,j_2) \times X =X $ is also rationally connected. We obtain the following commutative diagram of rationally connected fibrations. \[\begin{tikzcd} \W_{g_1,g_2} \arrow{drr}{ev_3^{\d}} \arrow{d}{\Pi \times ev} &&\\ \W_{g_1,g_2}'\times_{(\P^{n-1})^3} g_1X(i_1,j_1) \times g_2 X(i_2,j_2) \times X \arrow{rr} \arrow{d}{ev_3^{d_1}}  &  & X \\ \P^{n-1}
		\end{tikzcd}\]
		By composition the general fiber of the morphism $ev_3^{\d}: \W_{g_1,g_2} \rightarrow ev_3^{\d}(\W_{g_1,g_2})$ is rationally connected. Since the projected Gromov-Witten variety $ev_3^{\d}(\W_{1,w_0})$ has rational singularities, then according to Proposition \ref{prop: (O_u, ... O_v)=chi(O_ev_r*W)} for any $\alpha$ in $K_\circ(X)$  we have $\langle \O_{i_1,j_1}, \O_{i_2,j_2}, \alpha \rangle_\d 
		= \chi \left(\alpha \cdot [\O_{ev_3(\mathcal{W}_{1,w_0})}] \right)$.
	\end{proof}
	\item \label{item: symmetry} \textsc{Symmetry of correlators of $Fl_{1,n-1}$.} The isomorphism $Gr(1,n) \simeq Gr(n-1,n)$ yields an automorphism $\varphi$ of $X$, which implies the following symmetry for correlators of $X$. 
	\begin{proposition}
		\begin{equation*}
		\langle w_{i_1,j_1}, \dots , w_{i_r, j_r} \rangle_{d_1 l_1 + d_2 l_2} = \langle w_{n-j_1+1,n-i_1+1}, \dots, w_{n-j_r+1,n- i_r+1} \rangle_{d_2 l_1 + d_1 l_2}.
		\end{equation*}
	\end{proposition}
	
	\begin{proof}
		We consider here $X$ as the flag variety parametrizing lines included in hyperplanes of $\C^n$. We identify explicitly a line in $\C^n$ with a hyperplane in $\C^n$ in the following way. Consider a Hermitian inner product $\langle \cdot \mid \cdot \rangle$ on $\C^n$. For vectors $u_1$, $\dots$, $u_r$ in $\C^n$, we denote by $<u_1, \dots, u_r>$ the $\C$-linear vector space generated by the $u_i$. For any linear space $E=<u_1, \dots, u_r>$ in $\C^n$ denote by $E^\bot$ its orthogonal with respect to $\langle \cdot \mid \cdot \rangle$; $E^\bot$ may be defined as the intersection of the kernels of the $r$ linear maps sending a vector $v$ in $\C^n$ to the complex number $\langle \overline{u_i} \mid v \rangle$, where we denote by $\overline{u_i}$ the conjugate of $u_i$. Then the morphism $[L] \rightarrow [L^\bot]$ yields an identification $\P^{n-1} \simeq Gr(n-1,n)$. Furthermore, for any hyperplane $H$ in $\C^n$ consider the one dimensional linear space $H^\bot$ orthogonal to it. Since $(L^\bot)^\bot=L$, we obtain the inverse map $Gr(n-1,n) \rightarrow \P^{n-1}$.
		Now consider the natural embedding $X \hookrightarrow \P^{n-1} \times \P^{n-1}$. The automorphism $\varphi$ may be seen as the restriction to $X$ of the map $\P^{n-1} \times Gr(n-1,n) \rightarrow \P^{n-1} \times Gr(n-1,n) $ sending the projectivization of a line and a  hyperplane $([L], [H])$ onto the projectivization of the orthogonal spaces $([H^\bot],[L^\bot])$. Note that $\varphi$ exchanges classes $l_1$ and $l_2$, and sends a Schubert variety $X(w_{i,j})$ onto a translate of the Schubert variety $X(w_{n-j+1, n-i+1})$. Indeed, according to (\ref{eq : Schubert var}) the image by $\varphi$ of any Schubert variety $X(w_{i,j})$ may be described in the following way, where we consider a basis $(e_i)_{1 \leq i \leq n}$ of $\C^n$ which is orthogonal with respect to the inner product $\langle \cdot \mid \cdot \rangle$. \begin{align*}
		\varphi(X(w_{i,j})) &= \varphi \left(\left\{ (L,H) \in X \bigr\rvert L\subset <e_1, \cdots, e_i>,\: <e_1, \cdots, e_{j-1}> \subset H \right\} \right) \\
		&= \left\{ (H^\bot, L^\bot) \in X \bigr\rvert L^\bot \supset (<e_1, \cdots, e_i>)^\bot,\: (<e_1, \cdots, e_{j-1}>)^\bot \supset H^\bot \right\}\\
		&= \left\{ (L, H) \in X \bigr\rvert H \supset <e_{i+1}, \cdots, e_n>,\: <e_j, \cdots, e_{n}>)^\bot \supset L \right\}\\
		& = w_0 \cdot \left\{ (L, H) \in X \bigr\rvert H \supset <e_{1}, \cdots, e_{n-i}>,\: <e_1, \cdots, e_{n-j+1}>)^\bot \supset L \right\} \\
		&= w_0 \cdot X(w_{n-j+1,n-i+1}),
		\end{align*} where $w_0$ is the element in $G=\GL_n$ associated with a maximal length permutation of the basis vectors $e_i$. Furthermore, we have $\varphi_*(l_1) = [\varphi(X(w_{2,n}))]=[w_0 \cdot X(w_{1,n-1})]=l_2$, and in the same way $\varphi_* (l_2) = l_1$. The isomorphism $\varphi: X \rightarrow X$ induces an isomorphism $\phi: \overline{\mathcal{M}_{0,r}}(X, d_1 l_1 + d_2 l_2) \rightarrow \overline{\mathcal{M}_{0,r}}(X, \varphi_*(d_1 l_1 + d_2 l_2))= \overline{\mathcal{M}_{0,r}}(X, d_2 l_1 + d_1 l_2)$. We obtain the following commutative diagram. \[\begin{tikzcd}
		\overline{\mathcal{M}_{0,r}}(X, d_1 l_1 + d_2 l_2) \arrow{r}{\phi}\arrow{d}{\epsilon v} & \overline{\mathcal{M}_{0,r}}(X, d_2 l_1 + d_1 l_2) \arrow{d}{ev}\\ X^r \arrow{r}{\varphi^r} & X^r\end{tikzcd}\]
		Since $\varphi$ and $\phi$ are isomorphisms, for $\alpha$ in $K^\circ(X)\simeq K_\circ(X)$ the projection formula implies \begin{align*} \phi_*\left((\epsilon v_i)^*\alpha \cdot [\O_{\overline{\mathcal{M}_{0,r}}(X, d_1l_1+d_2l_2)}]\right) &= \phi_*\left((\epsilon v)_i^*\varphi^*\varphi_*\alpha \cdot [\O_{\overline{\mathcal{M}_{0,r}}(X, d_1l_1+d_2l_2)}]\right) \\ &= \phi_*\left(\phi^*ev_i^*\varphi_*\alpha \cdot [\O_{\overline{\mathcal{M}_{0,r}}(X, d_1l_1+d_2l_2)}]\right) \\ &= ev_i^* \varphi_*\alpha \cdot [\O_{\overline{\mathcal{M}_{0,r}}(X, d_2l_1+d_1l_2)}]. \end{align*} This yields the following equalities, where we denote by $\chi : K_\circ(\overline{\mathcal{M}_{0,r}}(X, \d)) \rightarrow K_\circ(\Spec(\C))$ the pushforward to the point. \begin{align*}
		\langle &w_{i_1,j_1}, \dots, w_{i_r, j_r} \rangle_{d_1 l_1 + d_2 l_2} \\ & =\chi\left((\epsilon v_1)^*\O_{w_{i_1,j_1}} \cdot \dots \cdot (\epsilon v_r)^*\O_{w_{i_r,j_r}} \cdot [\O_{\overline{\mathcal{M}_{0,r}}(X, d_1l_1+d_2l_2)}]\right)\\
		& =\chi\left(\phi_* \left(\phi^*ev_1^*\varphi_*\O_{w_{i_1,j_1}} \cdot \dots \cdot \phi^*ev_r^*\varphi_*\O_{w_{i_r,j_r}} \cdot [\O_{\overline{\mathcal{M}_{0,r}}(X, d_1l_1+d_2l_2)}]\right)\right)\\
		&= \chi\left(ev_1^*(\varphi_*\O_{w_{i_1,j_1}}) \cdot \dots \cdot ev_r^*(\varphi_*\O_{w_{i_r,j_r}})\cdot [\O_{\overline{\mathcal{M}_{0,r}}(X, d_2l_1+d_1l_2)}]\right) \\
		&= \chi\left(ev_1^*(\O_{w_{n-i_1+1,n-j_1+1}}) \cdot \dots \cdot ev_r^*(\O_{w_{n-i_r+1,n-j_r+1}}) \cdot [\O_{\overline{\mathcal{M}_{0,r}}(X, d_2l_1+d_1l_2)}]\right),
		\end{align*}
		where the last equality holds since $\varphi_*\O_{w_{i,j}} = \varphi_*[\O_{X(w_{i,j})}] = [\O_{\varphi(X(w_{i,j}))}] = [\O_{w_0 \cdot X(w_{n-j+1,n-i+1})}]=\O_{n-j+1,n-i+1}.$
		To sum up, correlators of $X$ exhibit the following symmetry. \begin{equation*}
		\langle w_{i_1,j_1}, \dots , w_{i_r, j_r} \rangle_{d_1 l_1 + d_2 l_2} = \langle w_{n-j_1+1,n-i_1+1}, \dots, w_{n-j_r+1,n- i_r+1} \rangle_{d_2 l_1 + d_1 l_2}.
		\end{equation*}
	\end{proof}
\end{enumerate}

\begin{Rem} Let $G$ be a semi-simple linear algebraic group, let $P$ be a parabolic subgroup of $G$ and $H$ be a subgroup of $G$ acting on $G/P$ by left multiplication. Then point $(a)$ holds for $H$-equivariant correlators of $G/P$. If $G=\GL_n$, then point $(c)$ implies that $H$-equivariant correlators of the flag variety $\GL_n/P$ are equal to the corresponding correlators of the dual flag variety. The symmetry observed here comes from the fact that $Fl_{1,n-1}$ is self dual.
\end{Rem}
\subsection{Three points correlators of $X$.}\label{subsec: 3 points cor X} We use here the results of the preceding parts to compute all genus zero three points correlators of $X=\GL_n/P$. Let $1 \leq i_1, i_2, i_3, j_1, j_2, j_3 \leq n$, where $i_k \neq j_k$. Recall for $1 \leq i,j \leq n$, $i \neq j$, we denote by $w_{i,j}$ the element in $W^P$ representing any permutation sending $1$ to $i$ and $n$ to $j$. For a point $x$ in $\P^{n-1} \times \P^{n-1}$, we denote by $$(L_x, H_x) \in \C^n \times \C^n$$ vectors satisfying $x=([L_x],[H_x])$. We consider here $X$ as a subvariety cut out from $\P^{n-1} \times \P^{n-1}$ by the incidence relation "line included in hyperplane" in the following way. Consider a Hermitian inner product $\langle \cdot \mid \cdot \rangle$ on $\C^n$. Consider a basis $(e_i)$ of $\C^n$ orthogonal with respect to $\langle \cdot \mid \cdot \rangle$. Denote by $\overline{u}$ the conjugate of a complex number. The isomorphism $\P^{n-1} \simeq Gr(n-1,n)$ identifies $H_x \in \C^n$ with the kernel $H_x^\bot$ of the linear map $\C^n \ni v \rightarrow \langle \overline{H_x} \mid v \rangle \in \C$, which is a hyperplane of $\C^n$. For any vectors $L_x$, $H_x$ in $\C^n$ the associated point $([L_x],[H_x])$ lies in $X$ iff $L_x$ lies within the hyperplane $H_x^\bot \subset \C^n$, i.e. iff $\langle L_x \mid \overline{H_x} \rangle =0$. For an element $g$ in $G=\GL_n$, we denote by $g^\bot$ the adjoint of $\overline{g}$ for the inner product $\langle \cdot \mid \cdot \rangle$, i.e. $\langle gL \mid \overline{g^\bot L'} \rangle = \langle L \mid \overline{L'} \rangle$ for any $L$, $L'$ in $\C^n$. We consider here the following action of $G$. \begin{align*} G \times \P^{n-1} \times \P^{n-1}  \rightarrow &\P^{n-1} \times \P^{n-1} \\ \left(g, ([L_x], [H_x]) \right)  \rightarrow &([g \cdot L_x], [g^\bot \cdot H_x]) \end{align*} The action of $G$ on $\P^{n-1} \times \P^{n-1}$ we just described induces by restriction the natural action of $G$ by left multilication on $X \simeq G/P$. For $1 \leq i \leq n$, we set \[L_i :=\{[x_1: \dots :x_i:0 \dots, 0]\} \subset \P^{n-1}.\]
Recall we denote by $w_0$ the permutation of $1$ $\dots$ $n$ of largest length. Note that we have \begin{align*}w_0L_i &=\{[0: \dots; 0: x_{n-i+1}: \dots:x_n]\} \subset \P^{n-1}\\
L_i^\bot &= \{[0: \dots; 0: x_{i+1}: \dots:x_n]\} = w_0L_{n-i}.\end{align*} Recall for $1\leq i,j \leq n$ and $i \neq j$, $w_{i,j}$ denotes the element in $W^P$ representing any permutation sending $1$ to $i$ and $n$ to $j$. Using (\ref{eq : X in PnxPn}) Schubert varieties of $X$ can be described in the following way. \begin{align} X(w_{i,j})&= \left\{ x=([L_x],[H_x]) \in X \mid L_x \subset <e_1, \cdots, e_i>,\: H_x \subset <e_j, \cdots, e_{n}> \right\} \nonumber \\
&= (L_i \times w_0 L_{n-j+1}) \cap X. \label{eq: 3 pts correlators--Sch var} \end{align} Let $\d$ be a degree in $E(X)$. For $(g_1,g_2)$ in $G^2$, we denote by $$X_{g_1,g_2}^\d := ev_1^{-1}(g_1 \cdot X(w_{i_1,j_1})) \cap ev_2^{-1}(g_2 \cdot X(w_{i_2,j_2})) \subset \overline{\mathcal{M}_{0,3}}(X, \d)$$ the Gromov-Witten variety parametrizing genus zero degree $\d$ stable maps sending their marked points into the translated Schubert varieties $g_k \cdot X(w_{i_k,j_k})$.

\begin{enumerate}[label=(\alph*)]
	\item \textsc{When $\d=l_2$.} Denote by $\pi_1 : X \subset \P^{n-1} \times \P^{n-1} \rightarrow \P^{n-1}$ (respectively $\pi_2$) the first (respectively second) projection. Note that since $l_2$ is a generating element in the Chow ring $A_*(X)$, curves of class $l_2$ are irreducible. Furthermore, since ${\pi_1}_* l_2 =0 \in A_*(\P^{n-1})$, a curve $C$ of class $l_2$ in $X$ projects to the point, and since ${\pi_2}_* l_2 =l_2 \in A_*(\P^{n-1})$, a curve $C$ of class $l_2$  projects to a line in $\P^{n-1}$. Hence curves of class $l_2$ in $X$ are curves of the form $\{x\} \times L$, where $x$ is a point in $\P^{n-1}$ and $L$ is a line in $\P^{n-1}$. Since lines are rational, any such curves are rational. In particular we observe that there is either a unique rational curve of class $l_2$ joining two distinct points in $X$, or none.
	
	We denote by ${ev_3}_{|g_1,g_2}$ the restriction of the morphism evaluating the third marked point to the Gromov-Witten variety $X_{g_1,g_2}^{l_2}$. Denote by $U \subset X^3$ the dense open subset parametrizing distinct triplets of points in $X^2$. 
	\begin{lemme}\label{lem: 3pts corr l2}
		\begin{enumerate}[label=\roman*)]
			\item For $(g_1,g_2)$ general in $G^2$, the variety $X_{g_1,g_2}^{l_2}$ has a dense intersection with $(ev_1 \times ev_3)^{-1}(U)$.
			\item For $(g_1,g_2)$ general in $G^2$, the general fiber of ${ev_3}_{|g_1,g_2}$ is a rational variety.
			\item Suppose $j_1+j_2 \leq  n+2$. Then there exists an element $h$ in $G$ such that $ev_3(X_{1,w_0}^{l_2})= h \cdot X(w_{i_1+i_2-n,1})$ if $i_1+i_2 > n+1$, and $ev_3(X_{1,w_0}^{l_2})= h \cdot X(w_{1,2})$ if $i_1+i_2 = n+1$, and is empty else. 
		\end{enumerate}
	\end{lemme}
	\begin{proof}
		We set $X_{g_1,g_2}:=X_{g_1,g_2}^{l_2}$.
		\begin{enumerate}[label=\roman*)]
			\item Note that for $(g_1,g_2)$ general in $G^2$, the intersection of $X_{g_1,g_2}^{l_2}$ with $(ev_1 \times ev_3)^{-1}(U)$ is a dense open subset of $X_{g_1,g_2}^{l_2}$. Indeed, since $X$ is convex, the subset ${\mathcal{M}_{0,3}}(X, l_2) \subset \overline{\mathcal{M}_{0,3}}(X, l_2)$ parametrizing degree $l_2$ stable maps $\P^1 \rightarrow X$ is a dense open subset of $\overline{\mathcal{M}_{0,3}}(X, l_2)$ \cite{fulton1996notes}, whose image is contained in $U$. Hence $(ev_1 \times ev_3)^{-1}(U)$ is a dense open subset of $\overline{\mathcal{M}_{0,3}}(X, l_2)$. Then for $(g_1,g_2)$ general in $G^2$ the variety $X_{g_1,g_2}^{l_2}$ has a dense intersection with $(ev_1 \times ev_3)^{-1}(U)$.
			
			\item Let $(g_1,g_2)$ be an element in $G^2$, let $z=(z_1,z_2)$ be a point in $ev_3(X_{g_1,g_2})$. According to (\ref{eq: 3 pts correlators--Sch var}), we have: \begin{align*} p_2\left(g_1X(w_{i_1,j_1})\times g_2X(w_{i_2,j_2}) \right) &= g_1^\bot \cdot p_2(X(w_{i_1,j_1})) \times g_2^\bot \cdot p_2(X(w_{i_2,j_2})) \subset \P^{n-1} \times \P^{n-1} \\
			&= (g_1^\bot \cdot w_0L_{n-j_1+1}) \times  (g_2^\bot \cdot w_0 L_{n-j_2+1}) \\
			&=L_{g_1} \times L_{g_2},\end{align*} where we denote by $L_{g_1}:= g_1^\bot \cdot w_0L_{n-j_1+1}$ and $L_{g_2}:= g_2^\bot \cdot w_0 L_{n-j_2+1}$. The point $z_1$ in $\P^{n-1}$ defines by duality a codimension one linear subspace $H_1$ in $\P^{n-1} \simeq Gr(n-1,n)$. We consider the set of points in $(L_{g_1} \times L_{g_2}) \cap H_1$ distinct from $z$ that lie in a line containing $z_2=[z^1: \dots:z^n]$; this can be seen as the intersection of the image of the following map with $(L_{g_1} \times L_{g_2})\cap H_1$. \begin{align*}
			&&\mathbb{A}^1 \times L_{g_1} \setminus \{z_2\} &\rightarrow \P^{n-1} \times \P^{n-1} \\
			f: &&(t,x=[x^1: \dots:x^n]) & \rightarrow (x, [x^1+tz^1: \dots: x^n+tz^n]).
			\end{align*}
			Note that since $L_{g_2}$ and $H_1$ are described by linear equations, $\Im f \cap (L_{g_1} \times L_{g_2})$ is described by linear equations on the $x^i$ and $t$, and hence is a rational variety. We now consider the morphism ${ev_3}_{| g_1,g_2}^{-1}(z) \rightarrow X^2 \subset (\P^{n-1} \times \P^{n-1})^2$ evaluating the first two marked points. Note that this morphism factors through ${ev_3}_{| g_1,g_2}^{-1}(z) \rightarrow (\{z_1\} \times \P^{n-1})^2 \simeq (\P^{n-1})^2$. We denote by $f_U$ the restriction of this morphism to the subset ${ev_3}_{| g_1,g_2}^{-1}(z)\cap (ev_1 \times ev_2)^{-1}(U)$. By definition, ${ev_3}_{| g_1,g_2}^{-1}(z)\cap (ev_1 \times ev_2)^{-1}(U)$ parametrizes degree $l_2$ stable maps $(h:\P^1 \rightarrow \{z_1\} \times \P^{n-1} \simeq \P^{n-1}, \{p_1, p_2,p_3\})$ satisfying the following conditions: \begin{itemize} \item $h$ sends its first two marked points into the projection $\pi_1(g_k X(w_{i_k,j_k}))$ of the Schubert varieties considered here, i.e. $h(p_1) \in L_{g_1}$, $h(p_2) \in L_{g_2}$;
				\item $h$ sends its third marked point into $z$, i.e. $h(p_3)=z_2$;
				\item $h(p_1) \neq z_2$ and $h(p_2) \neq z_2$;
				\item The image of $h$ is contained in $X$, i.e. $h(p_1) \in H_1$ and $h(p_2) \in H_1$.
			\end{itemize} Since a degree $l_2$ map $(\P^1 \rightarrow \P^{n-1}, \{p_1,p_2,p_3\})$ is uniquely determined by the data of three distinct points in $\P^{n-1}$ joined by a line, points in the fiber ${ev_3}_{| g_1,g_2}^{-1}(z)\cap (ev_1 \times ev_2)^{-1}(U)$ are in one-to-one correspondence with distinct points in $\Im f \cap (L_{g_1} \times L_{g_2}) \cap H_1$. Hence $f_U$ is an injective morphism of quasi-projective complex varieties, dominating the rational variety $\Im f \cap (L_{g_1} \times L_{g_2}) \cap H_1$.  Hence ${ev_3}_{| g_1,g_2}^{-1}(z)\cap (ev_1 \times ev_3)^{-1}(U)$ is a rational variety.
			
			Finally, since according to $i)$ for $(g_1,g_2)$ general in $G^2$, $X_{g_1,g_2}$ has a dense intersection with $(ev_1 \times ev_3)^{-1}(U)$, for $(g_1,g_2)$ general in $G^2$, the general fiber ${ev_3}_{| g_1,g_2}^{-1}(z)$ has a dense intersection with $(ev_1 \times ev_3)^{-1}(U)$, and hence is also a rational variety.
			\item $ev_3(X_{1,w_0})$ parametrizes elements $(L,H)$ in $\P^{n-1} \times \P^{n-1}$ such that: \begin{itemize}[label={--}]
				\item  There exists a line in $\P^{n-1}$ joining $H$ with elements $L_1 \subset <e_{j_1}, \dots, e_n>$ and $L_2 \subset <e_1, \dots, e_{n-j_2+1}>$;
				\item $L \subset <e_1, \dots, e_{i_1}> \cap <e_{n-i_2+1}, \dots, e_n>$;
				\item $L \subset H^\bot$.
			\end{itemize}
			First note that the subset of elements $H$ satisfying the first condition is the vector subspace generated by $<e_1, \dots, e_{n-j_2+1}, e_{j_1}, \dots, e_n> \simeq \C^{\min(n,2n-j_1-j_2+2)}$.
			Note that the second condition is impossible iff $i_1 <n-i_2+1$, and else describes elements $L$ in $<e_{n-i_2+1}, \dots, e_{i_1}> \simeq \C^{i_1+i_2-n+1}.$ Hence $ev_3(X_{1,w_0})$ is the set of elements $(L,H)$ in $\P^{n-1} \times \P^{n-1}$ such that \begin{itemize}[label={--}]
				\item  $H \subset <e_1, \dots, e_{n-j_2+1}, e_{j_1}, \dots, e_n>$;
				\item $L \subset <e_{n-i_2+1}, \dots, e_{i_1}>$;
				\item $L \subset H^\bot$.
			\end{itemize}
			Note that for $j_1+j_2 \leq n+2$, the linear subspace $<e_1, \dots, e_{n-j_2+1}, e_{j_1}, \dots, e_n>$ is $\C^n$, and hence the first condition is allways satsified. Then \[ev_3(X_{1,w_0}) = \{([0: \dots: 0:x_{n-i_2+1} \dots x_{i_1}:0 \dots 0], [y_1 \dots y_n]) \in X\}.\]
			Translation by a permutation $h$ of the basis vectors yields \begin{align*}
			ev_3(X_{1,w_0}) = h \cdot \{([x_1 \dots x_{i_1+i_2-n}:0: \dots : 0], [y_1 \dots y_n]) \in X\},
			\end{align*}
			which is equal to $h \cdot X(w_{i_1+i_2-n,1})$ if $i_1+i_2-n \neq 1$, and is equal to $h \cdot X(w_{1,2})$ if $i_1 +i_2=n+1$.
		\end{enumerate}
	\end{proof}
	
	Suppose $j_1+j_2 \leq n+2$ and $i_1+i_2 \geq n+1$. According to Lemma \ref{lem: 3pts corr l2} $ii)$ the general fiber of the map $ev_{3}: {X}_{g_1,g_2}^{l_2} \rightarrow X$ induced by evaluating the third point is a rational variety. Furthermore, according to Lemma \ref{lem: 3pts corr l2} $iii)$ its image $ev_3(\mathcal{X}_{1,w_0})$ is a translate of a Schubert variety by an element $h$ in $G$, and hence has rational singularities according to Theorem \ref{th: schubert var have rational sing}. According to Part \ref{sec: computing 3pts correlators} \ref{item: compute 3 pts corr when Pixev is surjective}, we have for $i_1+i_2 > n+1$: \begin{align}
	\langle w_{i_1,j_1}, w_{i_2, j_2}, \I_{i_3,j_3} \rangle_{l_2} &= \chi \left([\O_{h \cdot X(w_{i_1+i_2-n,1})}] \cdot \I_{i_3,j_3} \right) \nonumber \\
	&= \chi \left([\O_{X(w_{i_1+i_2-n,1})}] \cdot \I_{i_3,j_3} \right). \label{eq: corr l2}
	\end{align}
	and if $i_1+i_2=n+1$ \begin{align}
	\langle w_{i_1,j_1}, w_{i_2, j_2}, \I_{i_3,j_3} \rangle_{l_2}
	&= \chi \left([\O_{X(w_{1,2})}] \cdot \I_{i_3,j_3} \right) \nonumber \\
	&= \delta_{i_3,1} \delta_{j_3,2} \label{eq: corr l2 i1+i2=n+1}
	\end{align}
	Furthermore, according to Lemma \ref{lem: 3pts corr l2} $iii)$ if $j_1+j_2 \leq n+2$ and $i_1+i_2 < n+1$, then \begin{align}\label{eq: corr l2 empty}
	\langle w_{i_1,j_1}, w_{i_2, j_2}, \I_{i_3,j_3} \rangle_{l_2} =0 \end{align}
	\item \textsc{When $\d=l_1+l_2$.}   An irreducible rational curve of degree $(d_1l_1+d_2l_2)$ on $X$ may be parametrized as \[\P^1 \ni [u:v] \rightarrow \left([P_1(u:v)], [P_2(u:v)]\right) \in X, \] where $P_1$ and $P_2$ are homogeneous polynomials of degree $d_1$ and $d_2$ respectively. Furthermore, for an element $(g_1, g_2, g_3)$ general in $G^3$, we denote by $$\mathcal{M}_{g_k, w_{i_k,j_k}}^\d := ev_1^{-1}(g_1 \cdot X(w_{i_1, j_1})) \cap ev_2^{-1}(g_2 \cdot X(w_{i_2, j_2})) \cap ev_3^{-1}(g_3 \cdot X(w_{i_3,j_3}))$$ the Gromov-Witten variety associated with the $i_k$, $j_k$, where $ev_k: \overline{\mathcal{M}_{0,3}}(X, \d) \rightarrow X$ denotes evaluation at the $k$-th marked point. 
	For a point $(x,y)$ general in $X^2$ we denote by $C_{x,y}$ the line joining the points $[L_x]$ and $[L_y]$ in $\P^{n-1}$, and by $C_{x,y}'$ the line joining the points $[H_x]$ and $[H_y]$ in $\P^{n-1}$. Recall the varieties $h_i$ are the classes of the Schubert varieties of codimension $1$ of $X$. Note that the class $[h_1]$ (respectively $[h_2]$) in the Chow ring $A_*X$ is the pullbacks of the Cartier divisor $c_1(\O_{\P^{n-1}}(-1))$ by the first (respectively second) projection $\P^{n-1} \times \P^{n-1} \rightarrow \P^{n-1}$. 
	\begin{lemme}\label{lem: 3 pts correlators l1+l2}
		\begin{enumerate}[label=\roman*)]
			\item For a point $(x,y)$ in $X^2$ one of the two following properties is verified:
			\begin{itemize} \item There exists a unique connected rational curve of class $(l_1+l_2)$ joining the points $x$ and $y$. This curve is the intersection of the variety $C_{x,y} \times C_{x,y}' \subset \P^{n-1} \times \P^{n-1}$ with $X$. \item There are infinitely many rational curves of class $(l_1+l_2)$ joining the points $x$ and $y$. These curves cover the intersection of the variety $C_{x,y} \times C_{x,y}'$ with $X$. \end{itemize}
			\item For a point $(x, y)$ general in $X^2$, the rational curve in $X$ of class $(l_1+l_2)$ joining the points $x$ and $y$ is unique and irreducible. 	
			\item For $(g_1, g_2, g_3)$ general in $G^3$, the Gromov-Witten variety $\mathcal{M}_{g_k, w_{i_k,j_k}}^{l_1+l_2}$ is birational to its image in $X^3$ by the evaluation morphism.
			\item Let $(g_1, g_2)$ be a general element in $G^2$. Then  $$ev_3\left(ev_1^{-1}(g_1 \cdot h_i) \cap ev_2^{-1}(g_2 \cdot X(w_{i_2,j_2}))\right) = X.$$
		\end{enumerate}
	\end{lemme}
	\begin{proof}
		\begin{enumerate}[label=\roman*)]
			\item  We consider $X$ as a divisor of class $[h_1]+[h_2]$ in $A^1(\P^{n-1} \times \P^{n-1})$. Let $p_1=(x_1,y_1)$ and $p_2=(x_2,y_2)$ be points on $X$. If $x_1=x_2$, we can construct infinitely many genus zero stable map of class $(l_1+l_2)$ whose image contains the points $p_i$ by glueing together the degree $l_2$ map $\P^1 \rightarrow X$ joining $p_1$ and $p_2$ with any map $\P^1 \rightarrow X$ of degree $l_1$ intersecting this curve at another point than the $p_i$. The same holds if $y_1=y_2$. Now suppose $x_1 \neq x_2$ and $y_1 \neq y_2$. Denote by $L_1$ a line in $\P^{n-1}$ joining the points $x_1$ and $x_2$, and by $L_2$ a line in $\P^{n-1}$ joining the points $y_1$ and $y_2$.  Note that any rational curve of class $(l_1+l_2)$ in $X$ joining the points $p_i$ projects to $L_k$, hence any such rational curve is contained in the intersection $\left(L_1 \times L_2\right) \cap X$. Since the projective variety $\P^{n-1} \times \P^{n-1}$ is smooth, the intersection of $X$ with $L_1 \times L_2$ has dimension at least $1$. If the variety $(L_1 \times L_2) \cap X$ has dimension $2$, then any parametrization $\P^1 \rightarrow L_1$ and $\P^1 \rightarrow L_2$ sending the points $[1:0]$ and $[0:1]$ in $\P^1$ onto the points $p_i$ yield a degree $(l_1+l_2)$ map $\P^1 \rightarrow X$ joining the points $p_i$. We thus obtain infinitely many irreducible rational curves of class $(l_1+l_2)$ joining the points $p_i$. Indeed, since the classes $l_i$ generate $A_1(X)$, the map $\P^1 \rightarrow X$ cannot be a multiple covering onto its image, which is then a curve of class $(l_1+l_2)$. Now suppose the variety $(L_1 \times L_2) \cap X$ has dimension $1$. Note that we have the following equality in the Chow ring $A_*(\P^{n-1} \times \P^{n-1})$:  \begin{align*} [C'] &= [X] \cup [L_1 \times L_2] \\ &= ([h_1] + [h_2]) \cup [L_1] \boxtimes [L_2]\\ 
			&= ([h_1]+[h_2]) \cup (l_1+l_2)\\&=l_1+l_2 \end{align*} where we also denote by $l_i$ the pushforward to $\P^{n-1} \times \P^{n-1}$ of the classes $l_i \in H_*(X, \Z) \simeq  A_*(X)$.  Hence $C'$ is a curve of class $(l_1+l_2)$ in $X$. We now only have to consider the case where $C'$ has several irreducible components. Since $C'$ is a curve of class $(l_1+l_2)$ contained in $L_1 \times L_2$, $C'$ can be described as $C'= (L_1 \times \{y\}) \cup (\{x\} \times L_2)$ where $x$ and $y$ are points in $\P^{n-1}$. Note that since $x_1 \neq x_2$ and $y_1 \neq y_2$, both points $p_i$ cannot belong to the same irreducible component. For example suppose $p_1$ lies on $C_1$ and $p_2$ lies on $C_2$. Then $C'=(L_1 \times \{y_1\}) \cup (\{x_2\}\times L_2)$. Since $L_1$ joins $x_1$ and $x_2$, we obtain that $(x_2,y_1)$ belongs to $L_1 \times \{y_1\}$. In the same way, $(x_2,y_1)$ belongs to $\{x_2\} \times L_2$; hence the two irreducible components of $C'$ come together at $(x_2,y_1)$. We obtain a genus zero degree $(l_1+l_2)$ stable map $\P^1 \cup \P^1 \simeq (L_1 \times \{y_1\}) \cup (\{x_2\}\times L_2) \rightarrow X$ whose image contains the points $p_i$.
			\item According to here above, the morphism of irreducible projective varieties $$ev: \overline{\mathcal{M}_{0,2}}(X, l_1+l_2) \rightarrow X^2$$ is dominant, and hence is surjective. Hence for $(x,y)$ general in $X^2$, the general fiber $ev^{-1}(x,y)$ has dimension:\begin{align*}
			\dim ev^{-1}(x,y) & = \dim \overline{\mathcal{M}_{0,2}}(X, l_1+l_2) - 2 \dim X\\
			&= \dim X + \int_{l_1 + l_2} c_1(T_X) -1 - 2 \dim X\\
			&= \int_{l_1+l_2} ((n-1)[h_1] + (n-1)[h_2]) - \dim X -1 \\
			&= 0.
			\end{align*}
			The general fiber $ev^{-1}(x,y)$ thus contains finitely many points; hence according to $i)$ it contains only one point. The connected rational curve of class $(l_1+l_2)$ joining two points of $X$ in general position thus is unique. Finally, according to here above such a curve is irreducible.
			\item According to $ii)$, there exists a dense open subset $U \subset X^2$ such that there exists a unique irreducible rational curve of class $(l_1+l_2)$ joining any two points in $U$. Since the evaluation map $ev_1 \times ev_2: \overline{\mathcal{M}_{0,3}}(X, l_1+l_2) \rightarrow X^2$ is surjective, according to Chapter \ref{chap : geometry W} Lemma \ref{lem : dense intersection with open for homogeneous spaces} for $(g_1, g_2)$ general in $G^2$ the variety $ev_1^{-1}(g_1 \cdot X(w_{i_1,j_1}))\cap ev_2^{-1}(g_2 \cdot X(w_{i_2,j_2}))$ has a dense intersection with the dense open subset $(ev_1 \times ev_2)^{-1}(U)$. Now denote by $V$ the dense open subset of $U \times X$ containing distinct points $x$, $y$ and $z$ in $U\times X$. Note that the restriction of the evaluation morphism \[ev_{g_1,g_2}: \left(ev_1^{-1}(g_1 \cdot X(w_{i_1,j_1}))\cap ev_2^{-1}(g_2 \cdot X(w_{i_2,j_2}))\right) \cap ev^{-1}(V) \rightarrow X^3\] is an injective map. Indeed, if we specify two points $x$ and $y$ in $U$, by definition there is a unique irreducible curve $C$ of class $(l_1+l_2)$ joining $x$ and $y$. Consider a third point $z$ lying on $C$, distinct from $x$ and $y$. There exists a unique parametrization $\P^1 \rightarrow C$ sending the points $[1:0]$, $[0:1]$ and $[1:1]$ onto the points $x$, $y$ and $z$; hence the data of three distinct points $x$, $y$ and $z$ in $U \times X$ defines a unique point in $\overline{\mathcal{M}_{0,3}}(X, l_1+l_2)$.  The restriction of the evaluation morphism to $\mathcal{M}_{g_k, w_{i_k,j_k}}^{l_1+l_2}\cap ev^{-1}(V)={ev_{g_1,g_2}}^{-1}(g_3 \cdot X(w_{i_3,j_3}))$ thus is an injective morphism of quasi-projective complex varieties, hence is birational onto its image. Finally, note that for $(g_1,g_2,g_3)$ general in $G^3$ according to Chapter \ref{chap : geometry W} Lemma \ref{lem : dense intersection with open for homogeneous spaces} the variety $\mathcal{M}_{g_k, w_{i_k,j_k}}^{l_1+l_2}$ has a dense open intersection with $ev^{-1}(V)$, and hence is birational onto its image by the evaluation morphism. 
			\item Consider points $p_1$ and $p_2$ respectively in $g_2 \cdot X(w_{i_2,j_2})$ and $X$. According to $i)$ there exists a connected rational curve $C$ in $X$ of class $(l_1+l_2)$ joining $p_1$ and $p_2$. Denote by $i : X \hookrightarrow \P^{n-1} \times \P^{n-1}$ the natural embedding. We have the following equality in $A_*(\P^{n-1} \times \P^{n-1})$: \begin{equation*}
			[h_i] \cup [C] = [h_i] \cup (l_1+l_2)= [pt],
			\end{equation*}
			hence the intersection between $h_i$ and $C$ is non empty. 		
		\end{enumerate}	
	\end{proof}
	
	According to Lemma \ref{lem: 3 pts correlators l1+l2} $iii)$ for $(g_1, g_2)$ general in $G^2$, the Gromov-Witten variety $\mathcal{W}_{g_1,g_2} := ev_1^{-1}(g_1 \cdot h_1) \cap ev_2^{-1}(g_2X(w_{i_2,j_2}))$  is birational onto its image in $X^{3}$ via the evaluation map. Then the map $ev_{3}: \mathcal{W}_{g_1,g_2} \rightarrow X$ induced by evaluating the third point is a degree $1$ map of projective complex varieties. Furthermore, according to Lemma \ref{lem: 3 pts correlators l1+l2} $iv)$ its image $ev_3(\mathcal{W}_{w_0})$ is the smooth variety $X$. According to Part \ref{sec: computing 3pts correlators} \ref{item: compute 3 pts corr when Pixev is surjective}, we have: \begin{align*}
	\langle w_{i_1,j_1}, w_{i_2, j_2}, \I_{i_3,j_3} \rangle_{l_1+l_2} &= \chi \left([\O_{X}] \cdot \I_{i_3,j_3} \right) \\
	&= \chi \left([\O_{X(w_{n,1})}] \cdot \I_{i_3,j_3} \right) \\
	&= \delta_{n, i_3} \delta_{1, j_3}.
	\end{align*}
	
	\item \textsc{When $\d= d_1 l_1 + d_2 l_2$, where $0 \leq d_1 \leq d_2$ and $d_2 \geq 2$.} Recall $X$ is the variety parametrizing lines included in hyperplanes of $\C^n$. A point in $X$ corresponds to the inclusion of a line in a hyperplane of $\C^n$. Denote by $\pi: X \rightarrow \P^{n-1}$ the projection induced by forgetting the hyperplane. Let $\d=(d_1, d_2):=d_1 l_1 + d_2 l_2 \in E(X)$. Note that $\pi_*(d_1, d_2)=d_1 \in E(\P^{n-1})$. The forgetful morphism $\pi$ induces a morphism $\Pi : \overline{\mathcal{M}_{0,3}}(X, \d) \rightarrow \overline{\mathcal{M}_{0,3}}(\P^{n-1}, d_1)$. Together with the evaluation morphisms $ev_i$, this induces a morphism: $$\Pi \times ev: \overline{\mathcal{M}_{0,3}}(X, \d) \rightarrow \overline{\mathcal{M}_{0,3}}(\P^{n-1}, d_1) \times_{(\P^{n-1})^3} X^3.$$ We begin by studying when $\Pi \times ev$ is surjective, before computing correlators of $X$.
	
	\begin{proposition}\label{prop: when Pixev is surjective}
		\begin{enumerate}[label= \roman*)]
			\item \label{item: PIxev surj for (0,d2)} If $\d=(0,d_2)$, where $d_2 \geq 2$, then $\Pi \times ev$ is surjective;
			\item \label{item: PIxev surj for (1,d2)} Suppose $\d=(1, d_2)$, where $d_2 \geq 2$. The morphism $\Pi \times ev$ is surjective;
			\item \label{item: PIxev surj for (2,d2)} Suppose $\d=(d_1,d_2)$, where $d_1 \geq 2$ and $d_2 \geq d_1$. Then $\Pi \times ev$ is surjective.
			\item \label{item: ev surj} Suppose $\d=(d_1,d_2)$, where $d_1 \geq 2$ and $d_2 \geq 2$. Then $ev^\d: \overline{\mathcal{M}_{0,3}}(X, \d) \rightarrow X^3$ is surjective.
			\item \label{item: ev3 for (1,2)} Suppose $\d=(1, d_2)$, where $d_2 \geq 2$. Then $ev_3(X_{1,w_0}^{l_1+d_2l_2}) = h \cdot X(w_{i_1+i_2,1})$, where $h$ is an element in $G$.
		\end{enumerate}
	\end{proposition}
	
	\begin{lemme}\label{surjection boundary implies surjection}
		Let $f: Y \rightarrow Z$ be a morphism of projective irreducible varieties. Let $Z_0$ be a subvariety of $Z$ of codimension $1$ in $Z$. Suppose the image of $f$ contains $Z_0$. Furthermore, suppose the inverse image by $f$ of $Z_0$ is a strict subvariety of $Y$, i.e. suppose $f^{-1}(Z_0) \subsetneq Y$. Then $f$ is surjective.
	\end{lemme}
	\begin{proof}
		By contradiction, suppose $f$ is not surjective. Then $f(Y)$ has dimension smaller than $\dim f(f^{-1}(Z_0))$. Indeed, we have $\dim f(f^{-1}(Z_0)) = \dim Z_0 = \dim Z-1$, hence $\dim f(f^{-1}(Z_0)) \geq \dim f(Y)$ since $f(Y)$ is a strict closed subvariety of the irreducible variety $Z$.
		Furthermore, note that the irreducible projective variety $f(Y)$ contains the projective variety $f(f^{-1}(Z_0))$. Hence $f(f^{-1}(Z_0))$ is a projective subvariety of the irreducible variety $f(Y)$ of dimension larger than $f(Y)$; hence $f(Y) = f(f^{-1}(Z_0))$ and $f(Y)$ is contained in $Z_0$. Thus $Y=f^{-1}(Z_0)$, which contradicts the fact that $ f^{-1}(Z_0)$ is a strict closed subvariety of $Y$.
	\end{proof}
	
	\begin{proof}[Proof of Proposition \ref{prop: when Pixev is surjective}]
		We consider here $X$ as a subvariety of $\P^{n-1} \times \P^{n-1}$, where $X$ is cut out from $\P^{n-1} \times \P^{n-1}$ by the relation "line included in hyperplane" in the following way. Consider a point $(x,y)$ in $P^{n-1} \times \P^{n-1}$. By the functorial definition of $\P^{n-1}$, $x$ defines a line $L_x \subset \C^{n}$ and, if we identify $\P^{n-1}$ with its dual, $y$ defines an hyperplane $H_y \subset \C^n$. Then $(x,y)$ lies in $X$ iff $L_x \subset H_y$. We call \textit{curve} a projective variety of dimension $1$. We will denote by $pi_1$ (respectively $\pi_2$) the first (respectively second) projection $X \to \P^{n-1}$. 
		\begin{enumerate}[label= \roman*)]
			\item Consider a point $p$ on $\P^{n-1}$; denote by $L_p \subset \C^{n}$ the line defined by $p$. Let $(p,y_i)$ be three points on $X$, where the point $y_i$ in $\P^{n-1}$ correspond to a hyperplane $H_i$ in $\C^n$ containing the line $L_p$. Our goal here is to construct a genus zero stable map $C \rightarrow X$ of degree $d_2 l_2$ whose three marked points are the points $(p,y_i)$. We first construct a curve of degree $l_2$ joining the points $(p,y_1)$ and $(p,y_2)$, before constructing a genus zero stable map of degree $2 l_2$, and finally a genus zero stable map of degree $d_2 l_2$ joining all three points $(p,y_i)$. 
			
			Suppose the points $y_1$ and $y_2$ are distinct. Denote by $C$ a line in $\P^{n-1}$ joining the points $y_1$ and $y_2$; $C$ can be described as the subset of points in $\P^{n-1}$ associated with hyperplanes $u H_1 + v H_2$, where $[u:v]$ runs over $\P^1$. Since the line $L_p$ is included in $H_1$ and $H_2$, for all $[u:v]$ in $\P^1$ it is also included in $u H_1 + v H_2$. Hence the curve $C_0:=\{p\} \times C$ in $\P^{n-1} \times \P^{n-1}$ induced by $C$ is a curve of class $l_2$ lying in $X$ joining the points $(p,y_1)$ and $(p,y_2)$. In the same way, if the points $y_2$ and $y_3$ are distinct, we may construct a rational curve $C' \subset X$ of class $l_2$ joining the points $(p,y_2)$ and $(p,y_3)$. 		
			Consider parametrizations $f_i: \P^1 \rightarrow C_i$, where we fix points $q_1$, $q_2$, $q_3$ and $q_4$ on $\P^1$ such that $f_1$ sends $q_1$ and $q_2$ to $(p,y_1)$ and $(p,y_2)$, and $f_2$ sends $q_3$ and $q_4$ to $(p,y_2)$ and $(p,y_3)$. Consider the tree of $\P^1$'s $\P^1 \cup \P^1$ with $3$ marked points obtained by glueing $(\P^1, \{q_1, q_2\})$ with $(\P^1, \{q_3, q_4\})$ by identifying $q_2$ and $q_3$. The morphisms $f_i$ induce a genus zero, degree $2l_2$ stable map $(f:\P^1 \cup \P^1 \rightarrow X, \{q_1, q_2, q_4\})$ joining all three points $(p,y_i)$. Finally, notice that if two points $y_1$ and $y_2$ come together, a degree $2l_2$ stable map $\P^1 \cup \P^1 \rightarrow X$ collapsing an irreducible component $\P^1$ to the point $(p,y_1)$ joins all three points $(p,y_i)$.
			
			Now suppose $d_2 \geq 2$. Consider the degree $2l_2$ stable map $(f:\P^1 \cup \P^1 \rightarrow X, \{q_1, q_2, q_4\})$ described here above.  Now add an irreducible component $\P^1$ to $\P^1 \cup \P^1$, which we glue to $\P^1 \cup \P^1$ at a non marked point, and send this irreducible component to $X$ via a degree $(d_2-2)l_2$ map. We obtain a degree $d_2 l_2$ stable map $(\P^1 \cup \P^1 \cup \P^1 \rightarrow X, \{q_1, q_2, q_4\})$ joining all three points $(p,y_i)$.	
			\item Let $L$ be a line in $\P^{n-1}$, let $p_1$, $p_2$, $p_3$ be three points on $L$. Denote by $(p_i,y_i)$ three points on $X$ satisfying $\pi(p_i,y_i) =p_i$, i.e. the point $y_i$ in $\P^{n-1}$ corresponds to a hyperplane in $\C^n$ containing the line $L_i$ associated with the point $p_i$. Our goal here is to prove that there exists a genus zero stable map of degree $l_1+d_2l_2$ projecting to $L$ and whose image contains the points $(p_i,y_i)$. Consider a genus zero stable map $(\P^1 \rightarrow C \hookrightarrow \P^{n-1}, \{n_1, n_2\})$ of class $d_2l_2$ in $\P^{n-1}$ sending its marked points $n_i$ onto the points $y_1$ and $y_2$. If $d_2 \geq 2$, we also consider the case where $\P^1 \rightarrow X$ is a genus zero stable map of class $d_2 l_2$ in $\P^{n-1}$ sending its marked points onto the three points $y_i$. Since $\P^{n-1} \times \P^{n-1}$ is a smooth variety and the intersection $(L \times C) \cap X$ is non empty, $(L \times C) \cap X$ is of dimension at least $1$. If it has dimension $2$, then $L \times C$ is contained in $X$. Then identify the line $L$ with $\P^1$; by considering the product of the closed immersion $L \hookrightarrow \P^{n-1}$ with the map $\P^1 \rightarrow \P^{n-1}$, we obtain a genus zero stable map $\P^1 \rightarrow X$ of degree $(l_1 + d_2 l_2)$ whose image contains all points $(p_i,y_i)$. 
			
			Now consider the case where the dimension of the intersection $(L \times C) \cap X$ is $1$. Denote by $h_1$ (respectively $h_2$) the pullbacks of the Cartier divisor $c_1(\O_{\P^{n-1}}(-1))$ by the first (respectively second) projection. Note that we have the following equality in $A_*(\P^{n-1} \times \P^{n-1})$: $[X] = h_1 + h_2$. Furthermore, denote by $\delta$ the degree of $C$; note that $\P^1 \rightarrow C$ is a degree $d_2/\delta$ covering. If the variety $(L \times C) \cap X$ has dimension $1$, it is a curve $C'$ in $X$ satisfying the following equality in the Chow ring $A_*(\P^{n-1} \times \P^{n-1})$: \begin{align*}[C'] &= [L \times C] \cup [X]\\ &= (l_1 \boxtimes \delta l_2) \cup (h_1 + h_2) \\ &= l_1 + \delta l_2,\end{align*}
			where we also denote by $l_i$ the pushforward to $\P^{n-1} \times \P^{n-1}$ of the classes $l_i \in H_*(X, \Z) \simeq  A_*(X)$. Hence $C'$ is a curve of class $(l_1 + \delta l_2)$ in $X$. Note that, since the points $(p_i, y_i)$ belong to both $L \times C$ and $X$, they lie on $C'$. Furthermore, the image of $C'$ by the first projection $\pi_1: \P^{n-1} \times \P^{n-1} \rightarrow \P^{n-1}$ is $L$. Indeed, $C'$ is a curve of class $(l_1 + \delta l_2)$ contained in $L \times C$; hence its image by the projection $\pi_1$ is a curve of class $(\pi_1)_* (l_1+\delta l_2)= 1 \in A_*(\P^{n-1})$ contained in the irreducible degree~one curve $L \subset \P^{n-1}$. Hence $C'$ surjects onto $L$. 
			
			We now only have to prove that $C'$ can be descibed as the image of a stable map. Note that the curve $C'$ is either the union of two curves $C_i$ satisfying $(\pi_2)_*[C_1]+ (\pi_2)_*[C_1]=\delta=\pi_2(C')$, or is an irreducible curve of class $(l_1 + \delta l_2)$. First consider the case where $C'$ is the union of two curves $C_i$, where the first curve $C_1$ surjects onto $L$. Then $C_2$ is obtained as $\{Q_2\} \times C$, where $Q_2$ is a point on $\P^{n-1}$. Indeed, $C_2$ is a curve satisfying $(\pi_2)_* [C_2]=\pi_2(C_2)$ projecting to the irreducible degree $\delta$ curve $C$, hence the restriction ${\pi_2}_{| C_2}$ of the projection map to $C_2$ is a degree $1$ map onto $C$, hence is an isomorphism. In particular $C_2$ is a curve of class $\delta l_2$. In the same way, $C_1$ is a curve of class $l_1$ projecting to $L$, hence is obtained as $L \times \{Q_1\}$, where $Q_1$ is a point in $\P^{n-1}$. Then either the points $(p_i, y_i)$ all lie on the irreducible rational curve $C_2$-which yields a rational curve of class $\delta l_2$ containg all points $(p_i,y_i)$- or $C_1$ contains one or two point $(p_i,y_i)$. We will consider here the case where $C_1$ contains one point $(p_i,y_i)$, the same proof holds when $C_1$ contains two points. If $C_1=L \times \{Q_1\}$ contains a point $(p_i,yi)$, for example the point $(p_1, y_1)$ then $Q_1=y_1$ and both $(p_1, y_1)$ and $(p_2, y_1)$ lie on $C_1$. We obtain a genus zero degree $l_1$ stable map $f_1:=(\P^1 \simeq L \hookrightarrow X, \{(p_1, y_1), (p_2, y_1)\})$. Denote by $f_2:= (C_2' \rightarrow X, \{c_1, c_2, c_3\})$ a genus zero stable map of class $d_2 l_2$ sending its marked points $c_i$ onto the points $(p_2, y_1)$ and $(p_2, y_2)$ and $(p_2, y_3)$; such a map exists according to $i)$. Now glue $\P^1$ and $C_2'$ together by identifying the points $(p_2, y_1)$ and $c_1$; we obtain a genus zero stable map $\P^1 \cup C_2' \rightarrow X$ of class $(l_1 + \delta l_2)$ joining all points $(p_i,y_i)$.  
			
			Finally consider the case where $C'$ is irreducible. Denote by $\pi_1: \P^{n-1} \times \P^{n-1} \rightarrow \P^{n-1}$ the first projection. Since $C'$ is irreducible and $(\pi_1)_*[C'] = 1=[\pi_1(C')]$ in $A_*(\P^{n-1})$, the projection $(\pi_1)_{| C'}$ of the curve $C'$ to $\P^{n-1}$ is a degree $1$ map; hence $C'$ is birational to its image $L$ in $\P^{n-1}$, hence $C'$ is an irreducible rational curve.
			\item We proceed by iteration on $(d_1,d_2)$. Let us name $D(d_1-1,1)$ the boundary locus of $\overline{\mathcal{M}_{0,3}}(\P^{n-1},d_1)$ which is naturally isomorphic to $\overline{\mathcal{M}_{0,3}}(\P^{n-1}, d_1-1) \times_{\P^{n-1}} \overline{\mathcal{M}_{0,2}}(\P^{n-1},1)$, and by $D(\d-(l_1+l_2), l_1+l_2)$ the boundary locus of $\overline{\mathcal{M}_{0,3}}(X, \d-(l_1+l_2))$ ismorphic to $\overline{\mathcal{M}_{0,3}}(X, \d-(l_1+l_2)) \times_X \overline{\mathcal{M}_{0,2}}(X, l_1+l_2)$. Note that, since $D(\d-(l_1+l_2), l_1+l_2)$ is a strict closed subvariety of the irredudible projective variety $\overline{\mathcal{M}_{0,3}}(X, \d)$, according to Lemma \ref{surjection boundary implies surjection} it is enough to prove that $D(\d-(l_1+l_2), l_1+l_2)$ surjects into $D(d_1-1,1)$. Consider an element $p$ in $D(d_1-1,1)$. Denote by $(f_p : C_p \rightarrow \P^{n-1}, \{p_1, p_2, p_3\})$ and $(h_p: \P^1 \rightarrow \P^{n-1}, \{q_1,q_2\})$ the genus zero stable maps associated with $p$; $f_p$ is a degree $(d_1-1)$ map and $h_p$ is a degree $1$ map defining a line $h_p(\P^1)$ intersecting the curve $f_p(C)$ at $h_p(q_1)=f_p(p_3)=x_3$. The induction hypothesis, or if $d_1=2$ Lemma \ref{lem: 3 pts correlators l1+l2}, ensure that there is a genus zero degree $(\d-(l_1+l_2))$ stable map $(f_p': C_p' \rightarrow X, \{p_1, p_2, p_3\})$ such that \begin{itemize}[label={--}] \item The image by $\Pi$ of $f_p'$ is $f_p$; \item The points $f_p'(p_1)$ and $f_p'(p_2)$ in $X$ are projected to the points $f_p(p_i)$ in $\P^{n-1}$ by the forgetful map $X \rightarrow \P^{n-1}$.\end{itemize} If $d_1>2$, this is a direct consequence of the induction hypothesis. If $d_1=2$ note that  according to Lemma \ref{lem: 3 pts correlators l1+l2} there exists a genus zero degree $(l_1+l_2)$ stable map projecting to $f_p$. We can then choose marked points $p_i$ projecting to the points $f_p(p_i)$. Denote by $p_3'$ a non singular point in $f_p'^{-1}(p_3)$. Denote by $(x_3,y_3)$ a point in $X$ such that $y_3$ is contained in $f_p'(f_p^{-1}(p_3))$. According to Lemma \ref{lem: 3 pts correlators l1+l2} there exists a genus zero degree $(l_1+l_2)$ stable map $(h_q': C_q \rightarrow X, \{q_1', q_2'\})$  sending the point $q_1'$ to the point $(x_3,y_3)$ and the point $q_2'$ to a point of the form $(h_p(q_2), y')$  in $X$. Now glue the genus zero curves $C_p'$ and $C_q$ together by identifying the point $p_3'$ on $C_p'$ with the point $q_1'$ on $C_q$. We obtain a genus zero degree $\d$ stable map $(C_p' \cup C_q \rightarrow X, \{p_1,p_2, q_1', q_2'\})$ whose projection by $\Pi$ is $p$.		
			\item Let $p_1$, $p_2$, $p_3$ be three points on $X$. According to $ii)$, there exist genus zero degree $(l_1+l_2)$ stable maps $f_1:=(C_1 \rightarrow X, \{x_1, x_2\})$ and $f_2:=(C_2 \rightarrow X, \{y_1, y_2\})$ sending the marked points $x_i$ onto the points $p_1$ and $p_2$ and the marked points $y_i$ onto the points $p_2$ and $p_3$. Glue $C_1$ and $C_2$ together by identifying $x_2$ and $y_1$, and furthermore glue $C_2$ with $\P^1$ at a non special point $q$ of $C_2$. Consider a degree $\d-2(l_1+l_2)$ map $\P^1 \rightarrow X$ sending $q$ onto $f_2(q)$. We obtain a genus zero degree $\d$ stable map $(C_1 \cup C_2 \cup \P^1 \rightarrow X, \{x_1, x_2, y_2\})$ sending its marked points onto the points $p_i$. 
			\item Let $\W := X_{1,w_0}^{l_1+d_2l_2}$ be the Gromov-Witten variety of $X$ associated with the $i_k$, $j_k$, and $\W' := \overline{\mathcal{M}_{0,3}}(\P^{n-1},1) \times_{(\P^{n-1})^2}  (L_{i_1} \times w_0 L_{i_2})$ be the Gromov-Witten variety of $\P^{n-1}$ associated  with the $i_k$. We consider the map $\W' \times_{(\P^{n-1})^3} X^3 \rightarrow X$ induced by the third projection $X^3 \rightarrow X$. We  obtain the following commutative diagram. \[\begin{tikzcd}\W \arrow{d} \arrow{dr}{ev_3^{\d}} & \\
			\W' \times_{(\P^{n-1})^2} X(w_{i_1,j_1})\times X(w_{i_2,j_2}) \times X  \arrow{r} \arrow{dr}[swap]{ev_3^{d_1}} &\pi^{-1}(ev_3^{d_1}(\W'))\subset X \arrow{d}{\pi} \\
			& ev_3^{d_1}(\W') \subset \P^{n-1}
			\end{tikzcd}\]  Since according to \ref{item: PIxev surj for (1,d2)} the map $\Pi \times ev$ is surjective, according to Chapter \ref{chap : geometry W} the restriction $\W \rightarrow \W' \times_{(\P^{n-1})^2} X(w_{i_1,j_1})\times X(w_{i_2,j_2}) \times X$ of $\Pi \times ev$ to $\W$ is also surjective. Furthermore note that since $\pi: X \rightarrow \P^{n-1}$ is surjective, the projection $\W' \times_{(\P^{n-1})^3} X(w_{i_1,j_1})\times X(w_{i_2,j_2}) \rightarrow \pi^{-1}(ev_3^{d_1}(\W'))$ is also surjective. Hence by composition $ev_3^\d: \W \rightarrow \pi^{-1}(ev_3^{d_1}(\W')) \cap \left(X(w_{i_3,j_3})\right)$ is also surjective. Note that $ev_3^{d_1}(\W')$ is the locus of points $p$ in $\P^{n-1}$ such that \begin{itemize}[label ={--} ]
				\item $p$ lies on a line $L \subset \P^{n-1}$;
				\item a point in $L_{i_1} =\{[x_1, \dots, x_{i_1}:0: \dots:0]\}$ lies on $L$;
				\item a point in $w_0L_{i_2}=\{[0:\dots:0:x_{n-i_2+1} :\dots: x_n]\}$ lies on $L$.
			\end{itemize}
			$p$ satisfies these three conditions iff $p$ lies in $\{[x_1 \dots x_{i_1}:0 \dots0: x_{n-i_2+1} \dots x_n]\} = h \cdot L_{\min(n,i_1+i_2)}$, where $h$ is a permutation in $\mathfrak{S}_n$. We obtain \begin{align*}
			ev_3(\W) &= \pi^{-1}(h \cdot L_{\min(n,i_1+i_2)}) \\
			&= h \cdot \pi^{-1}(L_{\min(n,i_1+i_2)}) \\
			&= h \cdot \{([x_1: \dots : x_{\min(n,i_1+i_2)}: 0 \dots :0], [y_1: \dots : y_n]) \in X\}\\
			&= h \cdot X(w_{\min(n,i_1+i_2),1}).\end{align*}
		\end{enumerate}
	\end{proof}
	
	Suppose $d_2 \geq 2$. According to Proposition \ref{prop: when Pixev is surjective} \ref{item: ev3 for (1,2)} the projected Gromov-Witten variety $ev_3(X_{1,w_0}^{l_1+d_2l_2})$ is the translate of a Schubert variety and hence according to Theorem \ref{th: schubert var have rational sing} has rational singularities. Furthermore there exists an element $h$ in $G$ such that $ev_3(X_{1,w_0}^{l_1+d_2l_2}) = h \cdot X(w_{i_1+i_2,1})$. Part \ref{sec: computing 3pts correlators} \ref{item: compute 3 pts corr when Pixev is surjective} then implies the following identity.
	\begin{align} \forall\:\: d_2 \geq 2, \:\: \langle w_{i_1,j_1}, w_{i_2, j_2}, \I_{i_3, j_3} \rangle_{l_1+ d_2 l_2} &=  \chi \left( [\O_{ev_3^{l_1+d_2 l_2}({X}_{1,w_0})}] \cdot \I_{i_3,j_3} \right) \nonumber\\
	&= \chi \left( [\O_{ h \cdot X(w_{\min(n,i_1+i_2),1})}] \cdot \I_{i_3,j_3} \right)\nonumber\\
	&= \chi \left( \O_{\min(n,i_1+i_2),1} \cdot \I_{i_3,j_3} \right). \label{eq: corr l1+l2}\end{align} 
	
	Consider the partial order $\geq $ on $E(X)$ defined by: $$\d \geq \d' \:\mathrm{iff} \: \d-\d' \in E(X).$$  Note that classes $\d=(d_1, d_2)$ and $\d'=(d_1', d_2')$ in $E(X)$ satisfy $\d \geq \d'$ iff $d_1 \geq d_1'$ and $d_2 \geq d_2'$. According to Proposition \ref{prop: when Pixev is surjective} $iv)$ the image of $ev_1^{-1}(X(w_{i_1,j_1}) \cap ev_2^{-1}(w_0X(w_{i_2,j_2}))$ by the morphism $ev_3^{\d}$ evaluating the third marked point is $X$ for $\d \geq 2(l_1+l_2)$. Since $X$ is smooth, Proposition \ref{prop: when Pixev is surjective} $iii)$ and Part \ref{sec: computing 3pts correlators} \ref{item: compute 3 pts corr when Pixev is surjective} yield the following implication. \begin{align}
	\d \geq (2l_1 + 2l_2) \:\: \Rightarrow && \langle w_{i_1,j_1}, w_{i_2, j_2}, \I_{i_3, j_3} \rangle_\d = \chi \left( [\O_{X}] \cdot \I_{i_3,j_3} \right).
	\end{align}
	
	\item \textsc{When $\d=d_1l_1+d_2l_2$, where $0 \leq d_2 \leq d_1$.} Then according to Part \ref{sec: computing 3pts correlators} \ref{item: symmetry} we can dedude degree $\d$ correlators from the preceding points using the symmetry between corelators of degree $d_1l_1 + d_2l_2$ and correlators of degree $d_2l_1+d_1l_2$.
\end{enumerate}
\section{Degree of $\O_{h_i} \star \O_v$.}\label{sec: degree}

\subsection{Definitions.} Consider a degree $\d$ in the semi-group $E(X) \simeq \mathbb{N}^2$ of effective classes of curves in $H_2(X, \mathbb{Z})\simeq \Z^2$. For $r \geq 0$, denote by $\overline{\mathcal{M}_{0,r}}(X, \d)$ the variety parametrizing degree $\d$ genus zero stable maps to $X$.
The \textit{evaluation morphism} $ev_i : \overline{\mathcal{M}_{0,r}}(X, \d) \rightarrow X$ assigns to a map the image of its $i$-th marked point. For $X$ a homogeneous variety, the evaluation morphism is flat.
Recall for $u$ in $W^P$, we denote by $[\O_u]:=[\O_{X(u)}]$ the associated Schubert class in the Grothendieck group $K_\circ(X)$ of coherent sheaves. Since $X$ is non singular, an element $\O_u$ in $K_\circ(X)$ is assimilated with an element in the Grothendieck group $K^\circ(X) \simeq K_\circ(X)$ of locally free sheaves on $X$. Let $u_1$, $\dots$, $u_r$ be elements in the Weyl group $W$; the \textit{correlator} $\langle u_1, \dots , u_r\rangle_{\d}$ associated with the elements $u_i$ is defined by:
\begin{equation*}
\langle u_1, \dots , u_r\rangle_{\d} := \chi \left(ev_1^*(\O_{u_1}) \cdot \: \dots \: \cdot ev_r^*(\O_{u_r}) \cdot [\mathcal{O}_{\overline{\mathcal{M}_{0,r}}(X, d)}]\right).
\end{equation*}
Note that for $X$ a homogeneous variety, this correlator is equal to the sheaf Euler characteristic of the Gromov-Witten variety associated with the $u_i$ (cf. for example \cite{buch2011}, Part 4.1); i.e. for $(g_1, \dots, g_r)$ general in $G^r$, we have: \begin{equation}\label{eq: correlator = chi GW variety}
\langle u_1, \dots , u_r\rangle_{\d} := \chi \left(\O_{ev_1^{-1}(g_1 \cdot X(u_1)) \cap \dots \cap ev_r^{-1}(g_r \cdot X(u_r))}\right).
\end{equation}
We consider the basis $(\O_u)_{u \in W^P}$ of $K(X):=K^\circ(X) \simeq K_\circ(X)$. Let $t = \sum_{\alpha \in W^P} t^\alpha \O_\alpha$. Classes $l_1 := [X(w_{2,n})]$ and $l_2:=[X(w_{1,n-1})]$ of one dimensional Schubert varieties form a nef basis of $H_2(X, \mathbb{Z})$. Denote by $Q_1$ and $Q_2$ the Novikov variables associated with $l_1$ and $l_2$. For a degree $\d =(d_1, d_2) := d_1 l_1 + d_2 l_2$ in $E(X)$, we write $Q^\d := Q_1^{d_1} Q_2^{d_2}$.
Introduce the Gromov-Witten potential 
\begin{equation*} 
\mathcal{G}(t) := \sum_{\d \in E(X)} \sum_{r=0}^{+ \infty} \frac{Q^\d}{r!} \langle \underbrace{t, \dots, t}_r \rangle_{\d},
\end{equation*}
We define a \textit{deformed metric} by $(\O_\alpha, \O_\beta) = \delta^\alpha \delta^\beta \mathcal{G}(0)$, where $\delta^\alpha=\delta/\delta t^\alpha$ . 
The \textit{genus zero big quantum product} $\O_\alpha \bullet \O_\beta$ of two basis elements of $K^\circ(X)$ is defined by:
\begin{equation*}
(\O_\alpha \bullet \O_\beta, \O_\gamma) := \delta^\alpha \delta^\beta \delta^\gamma \mathcal{G}(t).
\end{equation*}

It is then extended by bilinearity over $K^\circ(X) \otimes \mathbb{Q}[[Q,t]]$ and defines the \textit{big quantum K-ring} $(K^\circ(X)\otimes \mathbb{Q}[[Q,t]], \bullet)$. This ring is associative and commutative \cite{lee2004quantum}, and we recover the ring $K^\circ(X)$ by taking the classical limit $Q \rightarrow 0$.
The restriction of the big quantum K-ring to $t=0$ defines the \textit{small} quantum K-ring $(QK_{s}(X), \star)$. Note that for $X$ a homogeneous variety, the product of two basis elements in $QK_s(X)$ is a polynomial in $Q$ \cite{anderson2018quantum}. 
The pairing defined by $$\G(\O_\alpha, \O_\beta) =  \G_{\alpha, \beta} = (\O_\alpha, \O_\beta)$$ is extended to a $\Q[[Q,t]]$-bilinear pairing $\mathcal{G}$ on $K(X) \otimes \Q[[Q,t]]$.  Denote by $\mathcal{G}^{\alpha, \beta}$ the inverse matrix. By definition, we have:  \begin{align*}
\O_u \star \O_v = \sum_{\gamma, w \in W^P}  \delta^u \delta^v \delta^\gamma \mathcal{G}(0) \: \mathcal{G}^{\gamma, w} \O_w.
\end{align*}


\subsection{A combinatorial lemma.}
We prove here a purely combinatorial lemma that will imply that the product of two elements in $K(X)$ is a polynomial in the Novikov variable $Q$.
\begin{lemme}
	\label{lem : prod binom}
	Let $n$, $m$ and $N$ be positive integers.  \begin{equation*}
	\sum_{\substack{(k,p) \in \llbracket 0; n \rrbracket \times \llbracket 0; m \rrbracket, \\ k+p=N}} \binom{n}{k} \binom{m}{p} = \binom{n+m}{N}.
	\end{equation*}
	Hence for $0 \leq N \leq n,m$ :
	\begin{equation*}
	\sum_{k=0}^N \binom{n}{k} \binom{m}{N-k} = \binom{m+n}{N},
	\end{equation*}
	and for $N \geq n$, $N \leq m$ :
	\begin{equation*}
	\sum_{k=0}^n \binom{n}{k} \binom{m}{N-k} = \binom{m+n}{N}.
	\end{equation*}	
\end{lemme}

\begin{proof}
	$(x+y)^n (x+y)^m = (x+y)^{n+m}.$
\end{proof}

\begin{lemme}
	\label{lem : 0 1}
	Let $n$ and $k$ be positive integers, with $n \geq 2$.
	\begin{equation*}
	\sum_{k=0}^n \binom{n}{k} (-1)^k k = 0.
	\end{equation*}
\end{lemme}

\begin{proof}
	Let $f : x \rightarrow (x-1)^n = \sum_{k=0}^{n} \binom{n}{k} (-1)^{n-k} x^k$.
	Then $f'(1)=0$.
\end{proof}

\begin{lemme}
	\label{lem : 0 k}
	Let $(d, \delta)\in \mathbb{N}^2$, let $(d_0, \delta_0)\in \mathbb{N}^2$ such that $\delta_0 < \delta-1$ or $d_0 <d-1$. Then
	\begin{equation*}
	\sum_{(\mathbf{d}, \pmb{\delta}) = ((d_0, \delta_0); \cdots; (d_r, \delta_r))} (-1)^r =0,
	\end{equation*}
	where we consider the sum over all integers $r$ and all possible decompositions of a pair $(d, \delta)$ in $\mathbb{N}^2$ on a $(r+1)$-tuple $ ((d_0, \delta_0); \cdots ; (d_r, \delta_r))$ satisfying $\sum_{i=0}^{r} d_i =d$, $\sum_{i=0}^{r} \delta_i =\delta$, and $d_i >0$ or $ \delta_i >0$ for all $i>0$.
\end{lemme}

\begin{proof}
	Let $0 \leq d_0 < d$ and $0 \leq \delta_0 < \delta -1$. 
	We set $S = \sum_{(\mathbf{d}, \pmb{\delta}) = ((d_0, \delta_0); \cdots; (d_r, \delta_r))} (-1)^r $.
	Let $r_1$ be the elements $d_i$ other than $d_0$ that satisfy $d_i \neq 0$. Note that $1 \leq r_1 \leq d-d_0$.  For a given $r_1$ the number $r$  varies from $r_1$ to $\delta - \delta_0 + r_1$. Denote by $r_2$ the number of elements $\delta_i \neq 0$ where $i \neq 0$; for given $r_1$ and $r$, $r_2$ varies from $\mathrm{max}(1,r-r_1)$ to $\mathrm{min}(\delta-\delta_0,r)$.	
	For given $r_1$, $r$ and $r_2$, we have $\binom{r}{r_1}$ possible choice of an $i$ satisfying $d_i \neq 0$, $\binom{r_1}{r_2-r+r_1}$ possible choice of an $i$ satisfying $\delta_i \neq 0$, $\binom{d-d_0-1}{r_1-1}$ possible  choice of the $d_i$, and $\binom{\delta-\delta_0-1}{r_2-1}$ choice for the integers $\delta_i$. Hence:
	\begin{equation*}
	S = \sum_{r_1=1}^{d-d_0} \sum_{r=r_1}^{\delta-\delta_0+r_1} \sum_{r_2 = \mathrm{max}(1,r-r_1)}^{\mathrm{min}(\delta-\delta_0,r)} (-1)^r \binom{r}{r_1} \binom{d-d_0-1}{r_1-1}  \binom{r_1}{r_2-r+r_1} \binom{\delta-\delta_0-1}{r_2-1}.
	\end{equation*}
	
	Let us now consider $S(r_1,r) = \sum_{r_2 = \mathrm{max}(1,r-r_1)}^{\mathrm{min}(\delta-\delta_0,r)} (-1)^r \binom{r}{r_1} \binom{d-d_0-1}{r_1-1}  \binom{r_1}{r_2-r+r_1} \binom{\delta-\delta_0-1}{r_2-1}$; alors $$S = \sum_{r_1=1}^{d-d_0} \sum_{r=r_1}^{\delta-\delta_0+r_1} S(r_1,r).$$
	
	\textbf{If $r \geq \delta - \delta_0$.} Then according to Lemma \ref{lem : prod binom}, if we set $k=r_2-1$, $n= \delta - \delta_0 -1$, $m = r$, $N=r-1$, we obtain:
	\begin{equation*}
	S(r,r) = (-1)^r  \binom{\delta - \delta_0 -1 +r}{r-1} \binom{d-d_0-1}{r_1-1}.
	\end{equation*}
	
	Let $r > r_1$. Set $k=r_2-r+r_1$, $n=r_1$, $m = \delta - \delta_0 -1$, $N= \delta - \delta_0 - r +r_1$. Note that $\delta - \delta_0 -r + r_1 \leq r_1$ and $\delta - \delta_0 -r +r_1 \leq \delta-\delta_0-1$. Then according to Lemma \ref{lem : prod binom} we have
	\begin{equation*}
	S(r_1,r) = (-1)^r \binom{r}{r_1} \binom{\delta - \delta_0 -1 +r_1}{\delta - \delta_0 -r +r_1} \binom{d-d_0-1}{r_1-1}.
	\end{equation*}
	
	\textbf{If $r \leq \delta - \delta_0$.} Set $k=r_2-1$, $n= \delta - \delta_0 -1$, $m = r$, $N=r-1$. Then according to Lemma \ref{lem : prod binom} we have
	\begin{equation*}
	S(r,r) = (-1)^r  \binom{\delta - \delta_0 -1 +r}{r-1} \binom{d-d_0-1}{r_1-1}.
	\end{equation*}
	
	Let $r_1 <r$. Set $k=r_2-r+r_1$, $N=r_1$, $n = \delta - \delta_0 + r_1 -r$, $m= r-1$. Note that $\delta - \delta_0 -r + r_1 \geq r_1$ et $r_1 \leq r-1$. Then according to Lemma \ref{lem : prod binom} we have
	\begin{align*}
	S(r_1,r) &= (-1)^r \binom{r}{r_1} \binom{\delta - \delta_0 -1 +r_1}{r_1} \frac{r_1 ! (\delta - \delta_0 -1)!}{(r-1)!(\delta - \delta_0 + r_1 -r)!}  \binom{d-d_0-1}{r_1-1}; \\
	& = (-1)^r \binom{r}{r_1} \binom{\delta - \delta_0 -1 +r_1}{r-1} \binom{d-d_0-1}{r_1-1};\\
	& = (-1)^r \binom{r}{r_1} \binom{\delta - \delta_0 -1 +r_1}{\delta - \delta_0-r+r_1} \binom{d-d_0-1}{r_1-1}.
	\end{align*}
	Note that $\delta-\delta_0 \geq 2$. According to Lemma  \ref{lem : 0 1},  $\sum_{r=r_1+1}^{\delta - \delta_0 +r_1} (-1) ^r r \binom{\delta - \delta_0 }{r-r_1} = (-1)^{r_1+1}$, which yields
	\begin{align*}
	\sum_{r=r_1+1}^{\delta - \delta_0 +r_1} S(r_1,r) & = \sum_{r=r_1+1}^{\delta - \delta_0 +r_1} (-1)^r r \frac{(\delta - \delta_0 + r_1 -1)!}{r_1! (r-r_1)!(\delta - \delta_0 + r_1 -r)!}  \binom{d-d_0-1}{r_1-1};\\
	& =   (-1)^{r_1+1} \binom{\delta - \delta_0 + r_1 -1}{r_1-1}  \binom{d-d_0-1}{r_1-1}.
	\end{align*}
	Hence $\sum_{r=r_1}^{\delta-\delta_0 +r_1} S(r_1,r)=0$, hence $S=0$.\\
	
	
	Finally, since this problem is symmetric relatively to $d$ and $\delta$, the same result holds for $d_0<d-1$ and $\delta_0 < \delta$. We now only have left to consider the case whare $d_0=d$ and $\delta_0 < \delta -1$, or $\delta_0 = \delta$ and $d_0 < d-1$. For example suppose $\delta_0 =d$. Then, since $\delta_0 < \delta -1$
	\begin{equation*}
	S = \sum_{r=1}^{\delta-\delta_0} (-1)^r \binom{\delta - \delta_0 -1}{r-1} = -(1-1)^{\delta -\delta_0-1} =0.
	\end{equation*}
\end{proof}

\subsection{Geometric interpretation of elements in $\O_{h_i} \star \O_{k,p}$ for $k \neq p+1$.}\label{subsec: geometric interpretation} From now on we fix elements $u$ and $v$ in $W^P$ such that $X(u) = h_i$ and $v=w_{k,p}$, where $k \neq p+1$. Recall $h_1 := X(w_{n-1,1})$ and $h_2 := X(w_{n,2})$. We fix a degree $\d \in E(X) \setminus \{0\}$. We {suppose $\d \geq l_1,l_2$}, i.e. we suppose $\d = d_1 l_1 + d_2 l_2$ where $d_1 \geq 1$ or $d_2 \geq 1$. 

Let $\d = \d_0 + \sum_{1 \leq i \leq k} \d_i$ be a decomposition of $\d$ into a sum of degrees $\d_i$. We denote by $M_{\d_0, \dots, \d_k}$ the following scheme 
\[ M_{\d_0, \dots, \d_k} = \overline{\mathcal{M}_{0,3}}(X, \d_0) \times_X \overline{\mathcal{M}_{0,2}}(X, \d_1) \times_X\dots \times_X \overline{\mathcal{M}_{0,2}}(X, \d_k).\]
Note that $ M_{\d_0, \dots, \d_k}$ is in bijection with the boundary stratum of $\overline{\mathcal{M}_{0,3}}(X, \d)$ parametrizing genus zero stable maps $(f: C_0 \cup C_1 \dots \cup C_k \rightarrow X, \{p_1,p_2,p_3\})$ such that \begin{itemize}[label={--}]
	\item The curve $ \cup_{0 \leq i \leq k} C_i$ is a union of $k$ genus zero quasi-stable curves $C_i$ meeting in a point.
	\item The first two marked points $p_j$ lie on $C_0$ and $p_3$ lies on $C_k$
	\item For all $0 \leq i \leq k$, the map $f_{| C_i}$ represents the degree $\d_i$.
\end{itemize}
Furthermore, note that $ M_{\d_0, \dots, \d_k}$ is equipped with evaluation maps $ev_i$, where $1 \leq i \leq 3$. We call \textit{boundary Gromov-Witten variety associated with $u$ and $v$} the inverse image \[M_{\d_0, \dots, \d_k}(u,v) := ev_1^{-1}(X(u)) \cap ev_2^{-1}(w_0X(v)) \subset M_{\d_0, \dots, \d_k}.\] Note that $M_{\d_0, \dots, \d_k}(u,v)$ parametrizes stable maps  $(f: C_0 \cup C_1 \dots \cup C_k \rightarrow X, \{p_1,p_2,p_3\})$ satisfying the three conditions here above and sending the first two marked points $p_j$ within the varieties $X(u)$ and $w_0X(v)$ respectively. Finally, we denote by $\Gamma_{\d_0, \dots, \d_k}(u,v)$ the projection of the boundary Gromov-Witten variety $ M_{\d_0, \dots, \d_k}(u,v)$ by the third evaluation morphism, i.e. $$\Gamma_{\d_0, \dots, \d_k}(u,v) := ev_3(M_{\d_0, \dots, \d_k}(u,v)) \subset X.$$
Note that $\Gamma_{\d_0, \dots, \d_k}(u,v)$ is the set of points $x$ in $X$ such that there exists a stable map $(\varphi : C \to X, \{p_1, p_2, p_3\})$ associated with an element in $M_{\d_0, \dots, \d_k}$ such that the image $\varphi (p_1)$ of the first marked point belongs to the Schubert variety $X(u)$, the image $\varphi(p_2)$ of the second marked point belongs to the opposite Schubert variety $w_0 X(v)$ and the image of the third marked point $\varphi(p_3)$ is the point $x$. 
\begin{lemme}\label{lem: ev(Gk)=ev(ev-1Gk-1)}
	\begin{enumerate}[label=\roman*)]
		\item Let $f: M \rightarrow X$ and $g_i: M' \rightarrow X$ be morphisms of schemes, where $i=1,2$. Let $Y$ be a subvariety of $X$. Denote by $M \times_X M'$ the fiber product defined by $f$ and $g_1$, and by $g: M \times_X M' \rightarrow X$ the morphism induced by $g_2$. Note that the morphim $f$ induces a map $f': M \times_X Y \to X$, and the morphisms $g_i$ induce maps $g_i': M' \times_X Y \to X$. Then \[g(M \times_X \times M' \times_X Y) = g_2'\left( {g_1'}^{-1}(f'(M \times_X Y))\right) . \]
		\item Let $k \geq 1$. Set $Y := ev_3(M_{\d_0, \dots, \d_{k-1}}(u,v))$.
		Then $ev_3(M_{\d_0, \dots, \d_k}(u,v))= ev_2^{\d_k}(ev_1^{-1}(Y)).$
	\end{enumerate}
\end{lemme}
\begin{proof}	\begin{enumerate}[label=\roman*)]
		\item This is a direct consequence of $M \times_X \times M' \times_X Y = (M \times_X Y) \times_X M'.$
		\item Apply $i)$ to $M:= M_{\d_0, \dots, \d_{k-1}}$, $M'=\overline{\mathcal{M}_{0,2}}(X, \d_k)$ and $Y=X(u)$.
	\end{enumerate}
\end{proof}
\begin{lemme}\label{lem: 2 points corr}
	Let $\d$ be a degree in $E(X)$. Then $ev_2^{\d} (ev_1^{-1}(X(u)))$ is a Schubert variety, and for any $\alpha$ in $K(X)$, $\langle \O_u, \alpha \rangle_{\d} = \chi(\alpha \cdot [\O_{ev_2^{\d} (ev_1^{-1}(X(u)))}])$.
\end{lemme}
\begin{proof}
	We consider the Gromov-Witten variety $ev_1^{-1}(X(u)) \subset \overline{\mathcal{M}_{0,2}}(X, \d)$. Let us consider the morphism $ev_2^{\d}: \overline{\mathcal{M}_{0,2}}(X,\d) \supset ev_1^{-1}(X(u)) \rightarrow X$ evaluating the second marked point. According to \cite{buch2011} Proposition 3.2., the general fiber of $ev_1$ is an irreducible rational variety and the image $ev_2^{\d}(ev_1^{-1}(X(u)))$ is a Schubert variety, and hence has rational singularities. Note that since $ev_1$ and $ev_2$ are $G$-equivariant, for any element $g_1$ in $G$ the  variety $ev_2^\d(ev_1^{-1}(g_1 \cdot X(u))) =g_1 \cdot ev_2^{\d}(ev_1^{-1}(X(u)))$ is also a Schubert variety. Finally Part \ref{sec: computing 3pts correlators} $(a)$ yields the result.
\end{proof}
\begin{proposition}\label{prop: sum corr=image Gamma}
	
	\begin{enumerate}
		Let $u$ be an element in $W^P$ such that $X(u)$ is a codimension $1$ Schubert variety $h_i$, and $v=w_{k,p}$ be an element in $W^P$ satisfying $k \neq p+1$. Let $\d = \sum_{k=1}^k \d_i$ be an element in $E(X)$, where $k>0$.
		\item The projected Gromov-Witten variety $\Gamma_{\d_0, \dots, \d_k}(u,v) := ev_3^{\d_0 + \dots + \d_r}(M_{\d_0, \dots, \d_k}(u,v))$ is the translate of a Schubert variety of $X$.
		\item \[\sum_{\alpha_i \in W^P}  \langle \O_u, \O_v, \I_{\alpha_1} \rangle_{\d_0} \langle \O_{\alpha_1}, \I_{\alpha_1} \rangle_{\d_1}  \dots  \langle \O_{\alpha_k}, \I_{w} \rangle_{\d_k}  = \chi([\O_{\Gamma_{\d_0, \dots, \d_k}(u,v)}], \I_w).\]
	\end{enumerate}
\end{proposition}
\begin{proof}\begin{enumerate}
		\item  By iteration on $k$. First note that for $\d_0 >0$ according to Lemma \ref{lem: 3 pts correlators l1+l2} and Proposition \ref{prop: when Pixev is surjective} there exists an element $h$ in $G$ such that $ev_3^{\d_0}(M_{\d_0}(u,v)) = ev_3^{\d_0}(\overline{\mathcal{M}_{0,3}}(X, \d_0) \times_{X^2}(X(u) \times w_0 X(v))$ is given by the translation by $h$ of a Schubert variety. For $\d_0 =0$, since $k \neq p+1$, according to Proposition \ref{prop: degenerating} $ev_3^{0}(\overline{\mathcal{M}_{0,3}}(X, \d_0) \times_{X^2}(X(u) \times w_0 X(v))=X(u) \cap w_0 X(v)$ is the translation by ane element $h$ in $G$ of a Schubert variety. Now suppose the induction hypothesis verified for all $i<k$. Then the projected Gromov-Witten variety $ev_3(M_{\d_0, \dots, \d_{k-1}}(u,v))$ is the translation by an element $h$ in $G$ of a Schubert variety $X(w)$. According to Lemma \ref{lem: ev(Gk)=ev(ev-1Gk-1)} we have \[ev_ 3(M_{\d_0, \dots, \d_{k}}(u,v)) = ev_2(ev_1^{-1}(h \cdot X(w))) = h \cdot ev_2(ev_1^{-1}(X(w))), \] where we consider the maps $ev_i : \overline{\mathcal{M}_{0,2}}(X, \d_k) \rightarrow X$, $i=1,2$. Finally according to \cite{buch} Proposition 3.2 the variety $ev_2(ev_1^{-1}(X(w)))$ is a Schubert variety.
		\item By iteration on $k$. According to Lemma \ref{lem: 3 pts correlators l1+l2} and Proposition \ref{prop: when Pixev is surjective} since $\d_0 >l_1,l_2$ we have \[\langle \O_u, \O_v, \I_{\alpha_1} \rangle_{\d_0} = \chi\left( \O_\beta \cdot \I_{\alpha_1} \right),\] where $h$ is an element in $G$ and $\beta$ is an element in $W^P$ such that $\Gamma_{\d_0}(u,v)=h \cdot X(\beta)$. Finally notice that $\chi([\O_{h \cdot X(\beta)}], \I_{\alpha_1})= \chi(\O_\beta, \I_{\alpha_1})$. 
		
		Let us now suppose the induction hypothesis satisfied for $i = k-1$. According to $i)$ there exists $h$ in $G$ and $\beta$ in $W^P$ such that $\Gamma_{\d_0, \dots, \d_k} (u,v) = h \cdot X(\beta)$. Hence the induction hypothesis yields \[\sum_{\alpha_i \in W^P, 1 \leq i \leq k-1}  \langle \O_u, \O_v, \I_{\alpha_1} \rangle_{\d_0} \langle \O_{\alpha_1}, \I_{\alpha_1} \rangle_{\d_1}  \dots  \langle \O_{\alpha_{k-1}}, \I_{\alpha_{k-1}} \rangle_{\d_k} = \chi([\O_{h \cdot X(\beta)}], \I_{\alpha_{k-1}}) = \delta_{\beta, \alpha_{k-1}},\]
		where the symbol $\delta_{\beta, \alpha_{k-1}}$ is one if $\beta = \alpha_{k-1}$, $0$ else.
		We now only have to compute $\sum_{\alpha_{k-1} \in W^P} \delta_{\beta, \alpha_{k-1}} \langle \O_{w}, \I_w \rangle_{\d_k}$. Let $u'$ be the element in $W^P$ such that $ev_2^{\d_k}(ev_1^{-1}(X(\beta))=X(u')$. We obtain  \begin{align*} \sum_{\alpha_{k-1} \in W^P} \delta_{\beta, \alpha_{k-1}} \langle \O_{w}, \I_w \rangle_{\d_k} &=\langle \O_{\beta}, \I_\beta \rangle_{\d_k}  \end{align*}
		where $\O_\beta=[\O_{ev_2^{\d_k}(ev_1^{-1}( X(\beta)))}] = [\O_{h \cdot ev_2^{\d_k}(ev_1^{-1}( X(\beta)))}]=
		[\O_{ev_2^{\d_k}(ev_1^{-1}(\Gamma_{\d_0, \dots, \d_{k-1}}(u,v)))}] =[\O_{\Gamma_{\d_0, \dots, \d_k}(u,v)}] $ according to Lemma \ref{lem: ev(Gk)=ev(ev-1Gk-1)}. 
	\end{enumerate}
\end{proof}
Hence for $v=w_{k,p}$, where $k \neq p+1$, the product $h_i \star \O_v$ can be rewritten as  \begin{align*} 
\O_{h_i} \star \O_v&= \O_{h_i} \cdot \O_v + + Q_1 \left( [\O_{\Gamma_{l_1}(h_i,v)}] - [\O_{\Gamma_{0,l_1}(h_i,v)}]\right)\\
&+ Q_2 \left( [\O_{\Gamma_{l_2}(h_i,v)}] -  \O_{\Gamma_{0,l_2}(h_i,v)}]\right)\\
&+Q_1 Q_2 \left( [\O_{\Gamma_{l_1+l_2}(h_i,v)}] - [\O_{\Gamma_{l_1,l_2}(h_i,v)}] -[\O_{\Gamma_{l_2,l_1}(h_i,v)}] - [\O_{\Gamma_{0,l_1+l_2}(h_i,v)}]\right. \\ & \:\:\:\:\:\:\: \:\:\:\:\:\:\: \:\:\:\:\:\:\: \:\:\:\:\:\:\: \:\:\:\:\:\:\:\left. +[\O_{\Gamma_{0,l_1,l_2}(h_i,v)}]+ [\O_{\Gamma_{0,l_2,l_1}(h_i,v)}] \right)\\
\end{align*} 

\subsection{Computing the degree of $\O_{h_i} \star \O_v$.} Let $1 \leq  i_2, j_2 \leq n$ where $i_2 \neq j_2$. We note $v=w_{i_2,j_2}$. We compute here the degree of $\O_{h_i} \star \O_v := \O_{h_i} \star \O_{i_2,j_2}$. More precisely, we prove the following result.

\begin{proposition}
	For any element $v$ in $W^P$, the product $\O_{h_i} \star \O_{i_2,j_2}$ only includes terms of degree $Q_1$, $Q_2$, and $Q_1 Q_2$, i.e.
	\begin{align*}
	\O_{h_i} \star \O_{i_2,j_2}= \O_{h_i} \cdot \O_v + Q_1 P_{l_1}(i,v) + Q_2 P_{l_2}(i,v) + Q_1 Q_2 P_{l_1+l_2}(i,v),
	\end{align*}
	where for $\d$ in $E(X)$, $P_\d(i,v)$ is an element in $K(X)$.
\end{proposition}

Our proof occurs in two times. We begin by rewriting the quantum product between two classes $\O_{h_i}$ and $\O_v$ in $K(X)$. We then note that most of these terms cancel together.

\textsc{Rewriting the quantum product.} Denote by $g$ the bilinear pairing on $K(X)$ induced by the sheaf Euler characteristic; we write $g_{\alpha, \beta} := \chi(\O_\alpha \cdot \O_\beta)$. We have \begin{align*} \G_{\alpha, \beta} &= \sum_{\d \in E(X)} Q^\d  \langle \O_\alpha, \O_\beta \rangle_{\d}\\ & = g_{\alpha, \beta} + \sum_{\d \in E(X)\setminus \{0\}} Q^\d  \langle  \O_\alpha, O_\beta \rangle_{\d}.\end{align*} The identity $(g+A)^{-1}=(g(\id + g^{-1} A))^{-1} = \sum (-1)^k (g^{-1}A)^k g^{-1}$ yields the following equality: \begin{align*}
&{\mathcal{G}^{\alpha, \beta}} \\
&= \sum_{\d \in E(X)} Q^\d \sum_{k=0}^{+ \infty} (-1)^k \sum_{\substack{\d_i \in E(X) \setminus \{0\}\\ \d_1+ \dots + \d_k = \d}} \:\: \sum_{\alpha_i \in W^P} g^{\alpha, \alpha_1} \langle \O_{\alpha_1}, O_{\alpha_2} \rangle_{\d_1} g^{\alpha_2, \alpha_3} \dots g^{\alpha_{2k-2}, \alpha_{2k-1}} \langle \O_{\alpha_{2k-1}}, O_{\alpha_{2k}} \rangle_{\d_k} g^{\alpha_{2k}, \beta}\\
&= g^{\alpha, \beta} + \sum_{k=0}^{+ \infty} (-1)^k \sum_{\substack{\d_i \in E(X) \setminus \{0\}}}  Q^{\d_1 + \dots + \d_k} \sum_{\alpha_i \in W^P} g^{\alpha, \alpha_1} \langle \O_{\alpha_1}, \I_{\alpha_2} \rangle_{\d_1}  \dots  \langle \O_{\alpha_k}, \I_{v} \rangle_{\d_k} 
\end{align*}
since $\I_\alpha$ is the element in $K(X)$ dual to $\O_\alpha$ for the pairing $g$.
This implies for any elements $u$ and $v$ in $W^P$ \begin{align*} \chi &\left( (\O_u \star  \O_v) \cdot \I_w \right)\\
&= \sum_{\alpha \in W^P} \sum_{\d \in E(X)} Q^{\d} \langle \O_u, \O_v, \O_\alpha \rangle_{\d} \left( g^{\alpha, v} + 
\sum_{k=0}^{+ \infty} (-1)^k \sum_{\substack{\d_i \in E(X) \setminus \{0\}}}  Q^{\d_1 + \dots + \d_k} \sum_{\alpha_i \in W^P} g^{\alpha, \alpha_1} \langle \O_{\alpha_1}, \I_{\alpha_2} \rangle_{\d_1}  \dots  \langle \O_{\alpha_k}, \I_{w} \rangle_{\d_k} \right)\\
&= \sum_{k=0}^{+ \infty} (-1)^k \sum_{\substack{\d_i \in E(X) \setminus \{0\}, \: \d_0 \in E(X)}}  Q^{\d_0 + \d_1 + \dots + \d_k} \sum_{\alpha_i \in W^P}  \langle \O_u, \O_v, \I_{\alpha_1} \rangle_{\d_0} \langle \O_{\alpha_1}, \I_{\alpha_1} \rangle_{\d_1}  \dots  \langle \O_{\alpha_k}, \I_{w} \rangle_{\d_k} \\
\end{align*} 

\textsc{Cancellation between different terms of $\O_{h_i} \star \O_v$.} 
\begin{lemme}\label{lem: image of Ev-1(Xu)}
	Let $1 \leq i,j \leq n$, $i \neq j$.
	\begin{enumerate}[label=\roman*)]
		\item For any $1 \leq s,t \leq n$, $s \neq t$, $\langle \O_{i,j}, \I_{s,t} \rangle_{l_1} = \delta_{s,n} \delta_{t,j} (1- \delta_{j,n}) + \delta_{j,n} \delta_{s,n-1} \delta_{t,n}$.
		\item For any $1 \leq s,t \leq n$, $s \neq t$, $\langle \O_{i,j}, \I_{s,t} \rangle_{l_2} = \delta_{s,i} \delta_{t,1} (1- \delta_{i,1}) + \delta_{i,1} \delta_{t,2} \delta_{s,1}$.
		\item For any $1 \leq i_2, j_2 \leq n$, $i_2 \neq j_2$, \begin{align*} \sum_{1 \leq i_1,j_1 \leq n, \: i_1 \neq j_1} \langle \O_{i,j}, \I_{i_1,j_1} \rangle_{l_1}  &\langle \O_{i_1,j_1}, \I_{i_2,j_2} \rangle_{l_2}\\ &= (1- \delta_{j,n})\delta_{i_2,n} \delta_{j_2,1} + \delta_{j,n} \delta_{i_2,n-1} \delta_{j_2,1}. \end{align*}
		\item For any $1 \leq i_2, j_2 \leq n$, $i_2 \neq j_2$, \begin{align*} \sum_{1 \leq i_1,j_1 \leq n, \: i_1 \neq j_1} \langle \O_{i,j}, \I_{i_1,j_1} \rangle_{l_2}  &\langle \O_{i_1,j_1}, \I_{i_2,j_2} \rangle_{l_1}\\ &= (1- \delta_{i,1})\delta_{i_2,n} \delta_{j_2,1} + \delta_{i,1} \delta_{i_2,n} \delta_{j_2,2}. \end{align*}
		\item For any $1 \leq i_2, j_2 \leq n$, $i_2 \neq j_2$, \begin{align*} \sum_{u,v \in W^P} \langle \O_{i,j}, \I_{u} \rangle_{l_1}  \langle \O_{u}, \I_{v} \rangle_{l_2} \langle \O_v, \I_{i_2,j_2} \rangle_{l_1}  = \delta_{i_2,n} \delta_{j_2,1}. \end{align*}
		\item  Let $1 \leq i_1,j_1, i_2, j_2 \leq n$, $i_k \neq j_k$, such that $j_1 + j_2 \leq n+2$. Then for any $1 \leq i,j \leq n$, $i \neq j$, we have \[\sum_{w \in W^P} \langle \O_{w_{i_1,j_1}}, \O_{w_{i_2,j_2}}, \I_w \rangle_{l_2} \langle \O_w, \I_{i,j} \rangle_{l_2} = \langle \O_u, \O_v, \I_{i,j} \rangle_{l_2}.\]
	\end{enumerate}
\end{lemme}
\begin{proof}
	\begin{enumerate}[label=\roman*)]
		\item According to (\ref{eq: corr l2}) we have \[\langle \O_{w_{i,j}}, \I_w \rangle_{l_1} = \chi \left( [\O_{ev_2^{l_1}(ev_1^{-1}(X(w_{i,j})))}] , \I_w\right). \]
		Note that $ev_2^{l_1}(ev_1^{-1}(X(w_{i,j}))$ parametrizes points $(x,y)$ in $X$ such that \begin{itemize}[label={--}]
			\item $y = [0: \dots :0:x_{n-j+1} \dots x_n]$;
			\item  $x$ lies on a line $L \subset \P^{n-1}$;
			\item  $L_{i}=\{[x_1 \dots x_i:0 \dots 0]\}$ intersects $L$.
		\end{itemize}
		Since there is a line joining any two points in $\P^{n-1}$, $ev_2^{l_1}(ev_1^{-1}(X(w_{i,j}))$ is the set of points $(x,y)$ in $X$ such that $y = [0: \dots :0:x_{n-j+1} \dots x_n]$; this is the Schubert variety $X(w_{n,j})$ if $j \neq n$ and $X(w_{n-1,n})$ if $j=n$.
		\item By symmetry, cf. Part \ref{sec: computing 3pts correlators} \ref{item: symmetry}.
		\item This is obtained by multiplying the expressions found in $i)$ and $ii)$.
		\item This is obtained by multiplying the expressions found in $i)$ and $ii)$.
		\item This is obtained by multiplying the expressions found in $i)$ and $iii)$.
		\item This obtained by multiplying the expression in $i)$ with the expression from (\ref{eq: corr l2}).
	\end{enumerate}
\end{proof}

Recall from Subsection \ref{subsec: geometric interpretation} the following definition. For classes $\d_0$, $\dots$, $\d_k$ in $E(X)$, for elements $u$ and $v$ in $W^P$, we set \[\Gamma_{\d_0, \dots, \d_k}(u,v) := ev_3 \left( ev_1^{-1}(X(u)) \cap ev_2^{-1}(w_0X(v)) \right) \subset X\]
the projected boundary Gromov-Witten variety associated with the classes $\d_i$ and the Schubert varieties $X(u)$ and $w_0 X(v)$. In particular when $v=w_0$ satisfies $X(w_0)=X$, we name \[\Gamma_{\d_0, \dots, \d_k}(u) :=\Gamma_{\d_0, \dots, \d_k}(w_0,u) := ev_3\left(ev_2^{-1}(X(u))\right),\]
and when $u$ satisfies $X(u)=h_i$, we name \[\Gamma_{\d_0, \dots, \d_k}(i,v) := ev_3\left(ev_1^{-1}(h_i) \cap ev_2^{-1}(X(v))\right).\]
\begin{lemme}\label{lem: cancellation G^ab}
	\begin{enumerate}[label=\roman*)]
		\item Let $\d=\sum_{i=1}^k \d_i$ be a degree in $E(X)$.  \[  \sum_{\alpha_i \in W^P}   \langle \O_{u}, \I_{\alpha_2} \rangle_{\d_1}  \dots  \langle \O_{\alpha_k}, \I_{\beta} \rangle_{\d_k} = \chi \left([\O_{\Gamma_{\d_1, \dots, \d_k}(u)}] , \I_\beta \right).\]
		\item $\Gamma_{l_1, l_2, l_1}(u) = X = \Gamma_{l_2,l_1,l_2} (u)=\Gamma_{l_1+l_2,l_1}(u)=\Gamma_{l_1, l_1+l_2}(u)$.
		\item $\Gamma_{dl_1}(u)$ and $\Gamma_{dl_2}(u)$ do not depend on the choice of $d \geq 2$.
		\item Consider a degree $\d=d_1l_1+d_2l_2$ in $E(X)$ such that $d_2 \geq 2$. Then \[\sum_{k=0}^{+ \infty} (-1)^k \sum_{\substack{\d_i \in E(X) \setminus \{0\}\\  \sum_{1 \leq i \leq k} \d_i=\d}}  \sum_{\alpha_i \in W^P}   \langle \O_{\alpha_1}, \I_{\alpha_1} \rangle_{\d_1}  \dots  \langle \O_{\beta}, \I_{\beta} \rangle_{\d_k} = 0.\]
	\end{enumerate}
\end{lemme}
\begin{proof}	\begin{enumerate}[label=\roman*)]
		\item This is a direct consequence of the equality $\langle [\O_X], \alpha, \beta \rangle_{\d} = \langle \alpha, \beta \rangle_{\d}$ for all $\alpha$, $\beta$ in $K(X)$ and of Proposition \ref{prop: sum corr=image Gamma}.
		\item Cf. Lemma \ref{lem: image of Ev-1(Xu)}.
		\item Cf. Lemma \ref{lem: image of Ev-1(Xu)}.
		\item By applying $i)$, $ii)$ and $iii)$ we obtain \begin{align*}\sum_{k=0}^{+ \infty} (-1)^k &\sum_{\substack{\d_i \in E(X) \setminus \{0\}\\  \sum_{1 \leq i \leq k} \d_i=\d}}  \sum_{\alpha_i \in W^P}   \langle \O_{\alpha_1}, \I_{\alpha_2} \rangle_{\d_1}  \dots  \langle \O_{\alpha_k}, \I_{\beta} \rangle_{\d_k} \\
		&= \sum_{k=0}^{+ \infty} (-1)^k \sum_{\substack{\d_i \in E(X) \setminus \{0\}\\  \sum_{1 \leq i \leq k} \d_i=\d}} \chi \left([\O_{\Gamma_{\d_1, \dots, \d_k}(u)}] , \I_\beta \right)\\
		&=   \chi \left([\O_{X}] , \I_\beta \right) \sum_{k=0}^{+ \infty} (-1)^k  \# \left\{ \d_i \in E(X) \setminus \{0\}, \:  \sum_{1 \leq i \leq k} \d_i=\d \right\} \\ &+ \left( \chi ([\O_{\Gamma_{d_1 l_1, d_2l_2}}], \I_w) -  \chi \left([\O_{X}] , \I_\beta \right) \right)   \sum_{k=0}^{+ \infty} (-1)^{k}  \# \left\{(d_1^i, d_2^j) \in \mathbb{N}^k, \: d_1 = \sum_{1 \leq i \leq k} d_1^i, \: d_2 = \sum_{1 \leq i \leq k} d_2^i \right\} \\ & + \left( \chi ([\O_{\Gamma_{d_2 l_2, d_1l_1}}], \I_w) -  \chi \left([\O_{X}] , \I_\beta \right) \right)   \sum_{k=0}^{+ \infty} (-1)^{k}  \#\left\{(d_1^i, d_2^j) \in \mathbb{N}^k, \: d_1 = \sum_{1 \leq i \leq k} d_1^i, \: d_2 = \sum_{1 \leq i \leq k} d_2^i\right\}  \end{align*}
		Finally, since $d_1 \geq 2$ or $d_2 \geq 2$, Lemma \ref{lem : 0 k} yields the annulation of the three components of the sum. The annulation of the first component comes from applying the lemma to $d_0=0=\delta_0$ and $\d_i = (d_i, \delta_i)$. The annulation of the second and third part of the sum comes from applying the lemma to $d_0=0=\delta_0$ and $d_i=d_1^i$ and $\delta_i=0$  if $d_1 \geq 2$, and to $d_0=0=\delta_0$ and $d_i=0$ and $\delta_i=d_2^i$  if $d_2 \geq 2$.
	\end{enumerate}
\end{proof}

We denote by \begin{align*} P_{\d}&(i, v) &\\
&:= \sum_{k=0}^{+ \infty} (-1)^k \sum_{\substack{\d_i \in E(X) \setminus \{0\}, \d_0 \in E(X)\\ \d_0 + \sum_{1 \leq i \leq k} \d_i=\d}}  Q^{\d_0 + \d_1 + \dots + \d_k} \sum_{\alpha_i \in W^P}  \langle \O_{h_i}, \O_v, \I_{\alpha_1} \rangle_{\d_0} \langle \O_{\alpha_1}, \I_{\alpha_1} \rangle_{\d_1}  \dots  \langle \O_{\alpha_k}, \I_{w} \rangle_{\d_k} \O_w & \in K(X) \otimes \Q_{\d}[Q] \end{align*}
the term of $\O_{h_i} \star \O_v$ of degree $\d$. 
Let $\d$ be a degree in $E(X)$ such that $\d \geq 2 l_1$ or $\d \geq  2 l_2$. According to Lemma \ref{lem: cancellation G^ab} most of the elements in $P_{\d}(u, v)$ cancel. More precisely, for a given $\d_0$ satisfying $\d_0 \leq \d - 2l_1$ or $\d_0 \leq \d -2 l_2$, all terms associated with decompositions $\d = \d_0 + \sum_{1 \leq i \leq k} \d_i$ cancel together. Suppose $d_1 \geq 2$ and $d_2 \geq 2$. The only terms left are then the ones associated with $\d_0 = \d$, $\d_0 = \d-l_1$, $\d_0= \d-l_2$ and $\d_0 = \d - l_1-l_2$. Using Proposition \ref{prop: sum corr=image Gamma} we obtain \begin{align*} P_{\d}(u, v) 
&= Q^{\d} \left(   [\O_{ev_3(\Gamma_{\d}(u,v))}]   -  [\O_{ev_3(\Gamma_{\d-l_1,l_1}(u,v))}] - [\O_{ev_3(\Gamma_{\d-l_2,l_2}(u,v))}] \right. \\ & \:\: \left. - [\O_{ev_3(\Gamma_{d-l_1-l_2, l_1 + l_2}(u,v))}] + [\O_{ev_3(\Gamma_{d-l_1-l_2, l_1, l_2}(u,v))}] + [\O_{ev_3(\Gamma_{d-l_1-l_2, l_2, l_1}(u,v))}]
\right), \end{align*}
which is equal to $0$ since according to Lemma \ref{lem: cancellation G^ab}  for $d_1 \geq 2$ and $d_2 \geq 2$ the projected varieties are equal to $X$. Now suppose $d_2 \geq 2$ and $d_1=0,1$. The only terms left in $P_{\d}(i, v)$ are those associated with $\d_0 = \d$ and $\d_0 = \d-l_2$. We obtain \begin{align*} P_{\d}(i, v) 
&= Q^{\d} \left(   [\O_{ev_3(\Gamma_{\d}(i,v))}]   -  [\O_{ev_3(\Gamma_{\d-l_2,l_2}(i,v))}]
\right), \end{align*} which is equal to zero since $1 + j_2 \leq n+2$ according to Lemma \ref{lem: 3pts corr l2}. Furthermore by symmetry if $d_1 \geq 2$, since $n+i_2 \geq n+2$ we have $P_{\d}(i, v)$ is equal to $0$.

Hence for any element $v$ in $W/W_P$ all terms of degree $\d \geq l_1$ and of degree $\d \geq l_2$ cancel in the product $\O_{h_i} \star \O_{i_2,j_2}$, which finally yields \begin{align*}
\O_{h_i} \star \O_{i_2,j_2}= \O_{h_i} \cdot \O_v + Q_1 P_{l_1}(i,v) + Q_2 P_{l_2}(i,v) + Q_1 Q_2 P_{l_1+l_2}(i,v).
\end{align*}

\section{A Chevalley formula in $QK_s(X)$}\label{sec: Chevalley formula}
Recall $h_1:=X(w_{n-1,1})$ and $h_2 := X(w_{n,2})$ are the two Schubert varieties whose classes generate $A^1(X) \simeq \Z^2$. We name $\O_{h_i} := [\O_{h_i}]$ the two Schubert classes associated. Let $1 \leq k,p \leq n$, where $k \neq p$. This part is dedicated to proving the following expressions. 
\begin{proposition}\label{prop: chevalley formula QK(X)}
	\begin{align}\label{eq: Oh1 star Ok,p}
	\O_{h_1} \star \O_{k,p} &=\left \{ \begin{array}{lr} Q_1\O_{n-1,n} +Q_1Q_2 \left( [\O_X] - \O_{h_1} \right) & \mathrm{if} \: k=1, p=n\\
	Q_1 \O_{n,p} & \mathrm{if} \:k=1, \: p<n\\
	\O_{1,2} + Q_1 \left( [\O_X] - \O_{h_1} \right)  & \mathrm{if} \: k=2, p=1\\
	\O_{k-1,p} & \mathrm{if} \: k>1, \: k \neq p+1\\
	\O_{p-1,p} + \O_{p,p+1} - \O_{p-1,p+1} & \mathrm{if} \: 1< p<n-1, \: k=p+1\\
	\end{array}\right.
	\end{align}
	\begin{align*}
	\O_{h_2} \star \O_{k,p} &=\left \{ \begin{array}{lr} Q_2\O_{1,2} +Q_1Q_2 \left( [\O_X] - \O_{h_2} \right) & \mathrm{if} \: k=1, p=n\\
	Q_2 \O_{k,1} & \mathrm{if} \:k>1, \: p=n\\
	\O_{n-1,n} + Q_2 \left( [\O_X] - \O_{h_2} \right)  & \mathrm{if} \: k=n, p=n-1\\
	\O_{k,p+1} & \mathrm{if} \: p<n \: k \neq p+1\\
	\O_{p,p+1} + \O_{p-1,p} - \O_{p-1,p+1} & \mathrm{if} \: 1< k<n-1, \: k=p+1\\
	\end{array}\right.
	\end{align*}
\end{proposition}

\subsection{Computing terms of degree $l_2$.}\label{sec: compute terms l2 P(hi)}
\textsc{Computing $\O_{h_1} \star \O_v$.}
According to (\ref{eq: corr l2 empty}) if $n-1+k<n+1$, i.e. if $k=1$, we have \[\langle \O_{h_1}, \O_{k,p}, \I_{i,j} \rangle_{l_2} =0\]
According to (\ref{eq: corr l2 i1+i2=n+1}) if $n-1+k=n+1$, i.e. if $k=2$ we have \[\langle \O_{h_1}, \O_{k,p}, \I_{i,j} \rangle_{l_2} =\delta_{i,1} \delta_{j,2}.\]
According to (\ref{eq: corr l2}) if $k >2$ we have \[\langle \O_{h_1}, \O_{k,p}, \I_{i,j} \rangle_{l_2} =\delta_{i,k-1} \delta_{j,1}.\]
Furthermore, the product in $K$-theory is given by $$\O_u \cdot \O_v = \sum_{w \in W^P} \chi(\O_u \cdot \O_v \cdot \I_w) \O_w = \sum_{w \in W^P} \langle \O_u, \O_v, \I_w \rangle_{0} \O_w.$$ Hence Proposition \ref{prop : produit dans K(X)} directly yields the invariants $\langle \O_{h_1}, \O_{k,p}, \I_{i,j} \rangle_0$. \begin{align*}
\langle \O_{h_1}, \O_{k,p}, \I_{i,j} \rangle_0 &=0   \: \: &\mathrm{if} \: k=1\\ \langle \O_{h_1}, \O_{k,p}, \I_{i,j} \rangle_0 &=\delta_{i,k-1} \delta_{j,p} & \mathrm{if} \: k \neq p+1, \: k \neq 1\\
\langle \O_{h_1}, \O_{k,p}, \I_{i,j} \rangle_0 &=\delta_{i,k-2} \delta_{j,p}+\delta_{i,k-1} \delta_{j,p+1}-\delta_{i,k-2} \delta_{j,p+1} & \mathrm{if} \: k = p+1.
\end{align*}
Recall from Lemma \ref{lem: image of Ev-1(Xu)} that for any $1 \leq s,t \leq n$, $s \neq t$, $\langle \O_{i,j}, \I_{s,t} \rangle_{l_2} = \delta_{s,i} \delta_{t,1} (1- \delta_{i,1}) + \delta_{i,1}  \delta_{s,1} \delta_{t,2}$. We deduce \[
\sum_{w \in W^P} \langle \O_{h_1}, \O_{k,p}, \I_{w} \rangle_{0} \langle \O_w, \I_{w_{i,j}} \rangle_{l_2} = \left\{ \begin{array}{l r} 0 \:& \mathrm{if} \: k=1\\  \delta_{i,1} \delta_{j,2} \: &\mathrm{if} \: k=2 \\\delta_{i,k-1} \delta_{j,1} \: &\mathrm{if} \: k>2    \end{array} \right. 
\]
Moreover recall that the term of degree $l_i$ in $\O_u \star \O_v$ is given by \begin{align*}
P_{l_i}(u,v) = \sum_{w \in W^P} \langle \O_{u}, \O_{v}, \I_{w} \rangle_{l_i} \O_w - \sum_{w, w' \in W^P} \langle \O_{u}, \O_{v}, \I_{w} \rangle_{0} \langle \O_w, \I_{w'} \rangle_{l_i} \O_{w'}.
\end{align*}
We obtain for all $v$ in $W^P$ \begin{equation}\label{eq: H1 star l2}
P_{l_2}(h_1,v) =0.
\end{equation}
\textsc{Computing $\O_{h_2} \star \O_v$.}
According to (\ref{eq: corr l2 i1+i2=n+1}) if $n+k=n+1$, i.e. if $k=1$ we have \[\langle \O_{h_2}, \O_{k,p}, \I_{i,j} \rangle_{l_2} =\delta_{i,1} \delta_{j,2}.\]
According to (\ref{eq: corr l2}) if $k >1$ we have \[\langle \O_{h_2}, \O_{k,p}, \I_{i,j} \rangle_{l_2} =\delta_{i,k} \delta_{j,1}.\]
We now obtain the invariants $\langle \O_{h_1}, \O_{k,p}, \I_{i,j} \rangle_0$ from Proposition \ref{prop : produit dans K(X)}. \begin{align*}
\langle \O_{h_2}, \O_{k,p}, \I_{i,j} \rangle_0 &=0   \: \: &\mathrm{if} \: p=n\\ \langle \O_{h_2}, \O_{k,p}, \I_{i,j} \rangle_0 &=\delta_{i,k} \delta_{j,p+1} & \mathrm{if} \: k \neq p+1, \: k \neq 1\\
\langle \O_{h_2}, \O_{k,p}, \I_{i,j} \rangle_0 &=\delta_{i,p} \delta_{j,p+1}+\delta_{i,p+1} \delta_{j,p+2}-\delta_{i,p} \delta_{j,p+2} & \mathrm{if} \: k = p+1.
\end{align*}
Recall from Lemma \ref{lem: image of Ev-1(Xu)} that for any $1 \leq s,t \leq n$, $s \neq t$, $\langle \O_{i,j}, \I_{s,t} \rangle_{l_2} = \delta_{s,i} \delta_{t,1} (1- \delta_{i,1}) + \delta_{i,1}  \delta_{s,1} \delta_{t,2}$. We deduce for $p<n$ \[
\sum_{w \in W^P,\: 1 \leq i,j \leq n, \: i\neq j} \langle \O_{h_2}, \O_{k,p}, \I_{w} \rangle_{0} \langle \O_w, \I_{w_{i,j}} \rangle_{l_2} = \left\{ \begin{array}{l r}  \delta_{i,1} \delta_{j,2} \: &\mathrm{if} \: k=1 \\\delta_{i,k} \delta_{j,1} \: &\mathrm{if} \: k>1 \: \mathrm{and} \: k \neq p+1 \\\delta_{i,p+1} \delta_{j,1} \: &\mathrm{if} \: 2 < k <n \: \mathrm{and} \: k = p+1 \\\delta_{i,2} \delta_{j,1} \: &\mathrm{if} \: k=2 \: \mathrm{and} \: p=1 \\\delta_{i,n-1} \delta_{j,1} \: &\mathrm{if} \: k=n \: \mathrm{and} \: p=n-1 \end{array} \right. 
\]
We obtain  for $2 \leq k,p \leq n$, $k \neq p$ \begin{align}\label{eq: H2 star l2}
\left\{ \begin{array}{lr} P_{l_2}(h_2,w_{1,p}) = \delta_{p,n} Q_2 \O_{1,2} \\
P_{l_2}(h_2,w_{k,n}) = Q_2 \O_{k,1} \:\:\:\:\:\:\:\:\:\:\:\:\:\:\:\:\:\:\:\:\:& \mathrm{if} \: k>1\\
P_{l_2}(h_2,w_{n,n-1}) = Q_2 (\O_{n,1}-\O_{n-1,1}) = Q_2 ([O_X] -\O_{h_1})\\
P_{l_2}(h_2,w_{k,p}) =0 \:\:\:\:\:\:\:\:\:\:\:\:\:\:\:\:\:\:\:\:\: &\mathrm{if} \: (1 <k <n \: \mathrm{and} \: p<n) \: \mathrm{or} \: (k=n \: \mathrm{and} \: p<n-1) \end{array}\right.
\end{align}
\subsection{Computing terms of degree $l_1$.} Recall from Part \ref{sec: computing 3pts correlators} \ref{item: symmetry} that correlators of degree $l_1$ associated with $h_1$ and $w_{k,p}$ are equal to correlators of degree $l_2$ associated with $h_2$ and $w_{n-p+1,n-k+1}$. We deduce from (\ref{eq: H1 star l2}) \begin{equation}\label{eq: H2 star l1}
P_{l_1}(h_2,w_{k,p}) =0,
\end{equation}
and we deduce from (\ref{eq: H2 star l2})  \begin{align}\label{eq: H1 star l1}
\left\{ \begin{array}{llr} &P_{l_1}(h_1,w_{k,n}) = \delta_{k,1} Q_1 \O_{n-1,n}& \\
&P_{l_1}(h_1,w_{1,p}) = Q_1 \O_{n,p} & \mathrm{if} \: p<n\\
&P_{l_1}(h_1,w_{2,1}) = Q_1 (\O_{n,1}-\O_{n,2}) = Q_1 ([O_X] -\O_{h_2})&\\
&P_{l_1}(h_1,w_{k,p}) =0 & \mathrm{if} \: (2 <k  \: \mathrm{and} \: p<n) \: \mathrm{or} \: (k=2 \: \mathrm{and} \: p>1)  \end{array}\right.
\end{align}

\subsection{Computing terms of degree $l_1+l_2$.}
According to (\ref{eq: corr l1+l2}), for all $1 \leq i,j \leq n$, $i \neq j$ \[ \langle \O_{h_1}, \O_{k,p}, \I_{i,j} \rangle_{l_1+l_2} = \delta_{\min(n,k+n-1),i}\delta_{n-\min(2n+1-p,n)+1,j} = \delta_{n,i} \delta_{1,j}.\]
According to (\ref{eq: corr l2 i1+i2=n+1}) if $n+k=n+1$, i.e. if $k=1$ we have \[\langle \O_{h_2}, \O_{k,p}, \I_{i,j} \rangle_{l_2} =\delta_{i,1} \delta_{j,2}.\]
According to (\ref{eq: corr l2}) if $k >1$ we have \[\langle \O_{h_2}, \O_{k,p}, \I_{i,j} \rangle_{l_2} =\delta_{i,k} \delta_{j,1}.\]
By symmetry we have \begin{align*}
&\langle \O_{h_1}, \O_{k,p}, \I_{i,j} \rangle_{l_1} =\delta_{j,n} \delta_{i,n-1} & \mathrm{if} \: p=n\\
&\langle \O_{h_1}, \O_{k,p}, \I_{i,j} \rangle_{l_1} =\delta_{j,p} \delta_{i,n} & \mathrm{if} \: p<n
\end{align*}
Recall from Lemma \ref{lem: image of Ev-1(Xu)} that for any $1 \leq s,t \leq n$, $s \neq t$, $\langle \O_{i,j}, \I_{s,t} \rangle_{l_2} = \delta_{s,i} \delta_{t,1} (1- \delta_{i,1}) + \delta_{i,1}  \delta_{s,1} \delta_{t,2}$. We deduce \[
\sum_{w \in W^P} \langle \O_{h_1}, \O_{k,p}, \I_{w} \rangle_{l_1} \langle \O_w, \I_{w_{i,j}} \rangle_{l_2} = \left\{ \begin{array}{l r}   \delta_{i,n-1} \delta_{j,1} \: &\mathrm{if} \: p=n \\\delta_{i,n} \delta_{j,1} \: &\mathrm{if} \: p<n   \end{array} \right. 
\]
If $k=1$, we have \[\langle \O_{h_1}, \O_{k,p}, \I_{i,j} \rangle_{l_2} =0\]
If $k=2$ we have \[\langle \O_{h_1}, \O_{k,p}, \I_{i,j} \rangle_{l_2} =\delta_{i,1} \delta_{j,2}.\]
If $k >2$ we have \[\langle \O_{h_1}, \O_{k,p}, \I_{i,j} \rangle_{l_2} =\delta_{i,k-1} \delta_{j,1}.\]
Recall from Lemma \ref{lem: image of Ev-1(Xu)} that for any $1 \leq s,t \leq n$, $s \neq t$, $\langle \O_{i,j}, \I_{s,t} \rangle_{l_1} = \delta_{s,n} \delta_{t,j} (1- \delta_{j,n}) + \delta_{j,n} \delta_{s,n-1} \delta_{t,n}$. We deduce \[
\sum_{w \in W^P} \langle \O_{h_1}, \O_{k,p}, \I_{w} \rangle_{l_2} \langle \O_w, \I_{w_{i,j}} \rangle_{l_1} = \left\{ \begin{array}{l r} 0 \:& \mathrm{if} \: k=1\\  \delta_{i,n} \delta_{j,2} \: &\mathrm{if} \: k=2 \\\delta_{i,n} \delta_{j,1} \: &\mathrm{if} \: k>2    \end{array} \right. 
\]
According to Proposition \ref{prop : produit dans K(X)} we have \begin{align*}
\langle \O_{h_1}, \O_{k,p}, \I_{i,j} \rangle_0 &=0   \: \: &\mathrm{if} \: k=1\\ \langle \O_{h_1}, \O_{k,p}, \I_{i,j} \rangle_0 &=\delta_{i,k-1} \delta_{j,p} & \mathrm{if} \: k \neq p+1, \: k \neq 1\\
\langle \O_{h_1}, \O_{k,p}, \I_{i,j} \rangle_0 &=\delta_{i,k-2} \delta_{j,p}+\delta_{i,k-1} \delta_{j,p+1}-\delta_{i,k-2} \delta_{j,p+1} & \mathrm{if} \: k = p+1.
\end{align*}
Furthermore,  according to Lemma \ref{lem: image of Ev-1(Xu)} for any $1 \leq i_2, j_2\leq n$, $i_2 \neq j_2$,  \begin{align*} \sum_{1 \leq i_1,j_1 \leq n, \: i_1 \neq j_1} \langle \O_{i,j}, \I_{i_1,j_1} \rangle_{l_1}  \langle \O_{i_1,j_1}, \I_{i_2,j_2} \rangle_{l_2} &= (1- \delta_{j,n})\delta_{i_2,n} \delta_{j_2,1} + \delta_{j,n} \delta_{i_2,n-1} \delta_{j_2,1}. \end{align*}
Hence  \[
\sum_{w, w' \in W^P} \langle \O_{h_1}, \O_{k,p}, \I_{w} \rangle_{0} \langle \O_w, \I_{w'} \rangle_{l_1} \langle \O_{w'}, \I_{i,j} \rangle_{l_2} = \left\{ \begin{array}{l r} 0 \:& \mathrm{if} \: k=1\\  \delta_{i,n-1} \delta_{j,1} \: &\mathrm{if} \: p=n, \: k > 1 \\\delta_{i,n} \delta_{j,1} \: &\mathrm{if} \: k>1 \: \mathrm{and} \: p <n   \end{array} \right. 
\]
Moreover  according to Lemma \ref{lem: image of Ev-1(Xu)} for any $1 \leq i_2, j_2\leq n$, $i_2 \neq j_2$, \begin{align*} \sum_{1 \leq i_1,j_1 \leq n, \: i_1 \neq j_1} \langle \O_{i,j}, \I_{i_1,j_1} \rangle_{l_2}  \langle \O_{i_1,j_1}, \I_{i_2,j_2} \rangle_{l_1} = (1- \delta_{i,1})\delta_{i_2,n} \delta_{j_2,1} + \delta_{i,1} \delta_{i_2,n} \delta_{j_2,2}. \end{align*} 
Hence  \[
\sum_{w, w' \in W^P,\: 1 \leq i,j \leq n, \: i\neq j} \langle \O_{h_1}, \O_{k,p}, \I_{w} \rangle_{0} \langle \O_w, \I_{w'} \rangle_{l_2} \langle \O_{w'}, \I_{i,j} \rangle_{l_1} = \left\{ \begin{array}{l r} 0 \:& \mathrm{if} \: k=1\\  \delta_{i,n} \delta_{j,2} \: &\mathrm{if} \: k=2 \\\delta_{i,n} \delta_{j,1} \: &\mathrm{if} \: k>2   \end{array} \right. 
\]
Finally according to Lemma \ref{lem: 2 points corr} and Lemma \ref{lem: 3 pts correlators l1+l2} we have $\langle \O_u, \I_w \rangle_{l_1+l_2} = \chi \left( [\O_X] \cdot \I_w \right)$ hence \begin{align*} \sum_{u \in W^P} \langle \O_{h_1}, \O_{k,p}, \I_{u} \rangle_{0}  \langle \O_{u}, \I_{i,j} \rangle_{l_1+l_2} = \left\{ \begin{array}{l r} 0 \:&\mathrm{if} \: k=1\\ \delta_{i,n} \delta_{j,1} & \mathrm{else} \end{array} \right. \end{align*} 

We are now ready to compute $P_{l_1+l_2}(h_1,W_{k,p})$. Moreover recall that the term of degree $(l_1+l_2)$ in $\O_{h_1} \star \O_v$ is given by \begin{align*}
P_{l_1+l_2}&(h_1,v) \\
&=\sum_{w \in W^P} \langle \O_{h_1}, \O_{v}, \I_{w} \rangle_{l_1+l_2} \O_w - \sum_{w, w' \in W^P} \langle \O_{h_1}, \O_{v}, \I_{w} \rangle_{l_1} \langle \O_w, \I_{w'} \rangle_{l_2} \O_{w'} - \sum_{w, w' \in W^P} \langle \O_{h_1}, \O_{v}, \I_{w} \rangle_{l2} \langle \O_w, \I_{w'} \rangle_{l_1} \O_{w'} \\ &+ \sum_{w, w', w'' \in W^P} \langle \O_{h_1}, \O_{v}, \I_{w} \rangle_{0} \langle \O_w, \I_{w'} \rangle_{l_1} \langle \O_{w'}, \I_{w''} \rangle_{l_2}  \O_{w''} + \sum_{w, w', w'' \in W^P} \langle \O_{h_1}, \O_{v}, \I_{w} \rangle_{0} \langle \O_w, \I_{w'} \rangle_{l_2} \langle \O_{w'}, \I_{w''} \rangle_{l_1}  \O_{w''} \\ &- \sum_{w, w' \in W^P} \langle \O_{h_1}, \O_{v}, \I_{w} \rangle_{0} \langle \O_w, \I_{w'} \rangle_{l_1+l_2} \O_{w'} .
\end{align*} 
\begin{enumerate}[label=\alph*)]
	\item \textbf{If $v=w_{k,p}$, where $k \neq 1$.} Note that then the second and fourth terms of the sum are equal, and the third and fifth terms are equal, and the first and last terms are equal. We obtain \begin{align*}
	P_{l_1+l_2}(h_1, w_{k,p})&= Q_1 Q_2\left( \O_{n,1} -\O_{n,1}- \O_{n,1} + \O_{n,1} + \O_{n,1}-\O_{n,1} \right) \\
	&= 0	& \mathrm{if} \: k \neq 1.
	\end{align*}
	\item \textbf{If $v=w_{1,p}$, where $p<n$.} Note that then all terms of the sum are equal to zero, except from the first term $\sum_{w \in W^P} \langle \O_{h_1}, \O_{v}, \I_{w} \rangle_{l_1+l_2} \O_w$, which is non zero iff $w=w_{n,1}$, and $\sum_{w, w' \in W^P} \langle \O_{h_1}, \O_{v}, \I_{w} \rangle_{l_1} \langle \O_w, \I_{w'} \rangle_{l_2} \O_{w'}$, which is also non zero iff $w=w_{n,1}$. We obtain \begin{align*}
	&P_{l_1+l_2}(h_1, w_{1,p})= 0 &\mathrm{if} \: p<n
	\end{align*}
	\item \textbf{if $v=w_{1,n}$.} Note that then $\O_v = [\O_{X(w_{1,n})}] = [\O_{point}]$ in $K(X)$. Then all terms of the sum are zero, except from the first term $\sum_{w \in W^P} \langle \O_{h_1}, \O_{v}, \I_{w} \rangle_{l_1+l_2} \O_w$, which is non zero iff $w=w_{n,1}$, and $\sum_{w, w' \in W^P} \langle \O_{h_1}, \O_{v}, \I_{w} \rangle_{l_1} \langle \O_w, \I_{w'} \rangle_{l_2} \O_{w'}$, which is non zero iff $w=w_{n-1,1}$. We obtain \[P_{l_1+l_2}(h_1, w_{1,n})= Q_1 Q_2 (\O_{n,1} -\O_{n-1,1})= Q_1Q_2 \left( [\O_X] - \O_{h_1}\right).\]
\end{enumerate}
To sum up, we have \begin{align}\label{eq: H1 star l1+l2}
P_{l_1+l_2}(h_1, w_{k,p})= \left\{ \begin{array}{lr} 0	& \mathrm{if} \: (k,p) \neq (1,n)\\
Q_1Q_2 \left( [\O_X] - \O_{h_1}\right) & \mathrm{if} \: (k,p) = (1,n) \end{array}\right.
\end{align}
Finally, by symmetry we deduce \begin{align}\label{eq: H2 star k,p l1+l2}
P_{l_1+l_2}(h_2, w_{k,p})= \left\{ \begin{array}{lr} 0	& \mathrm{if} \: (k,p) \neq (1,n)\\
Q_1Q_2 \left( [\O_X] - \O_{h_2}\right) & \mathrm{if} \: (k,p) = (1,n) \end{array}\right.
\end{align}
\subsection{Computing the Chevalley formula.}[Proof of Proposition \ref{prop: chevalley formula QK(X)}] According to Part \ref{sec: degree} the only terms appearing in $\O_{h_i} \star \O_v$ are those of degree $l_i$ and of degree $(l_1+l_2)$. We now only have to sum the classical part $\O_{h_i} \cdot \O_v$ computed Proposition \ref{prop : produit dans K(X)} with the part of degree $l_1$ given by (\ref{eq: H1 star l1}) and (\ref{eq: H2 star l1}), the part of degree $l_2$ given by (\ref{eq: H1 star l2}) and (\ref{eq: H2 star l2}), and finally the part of degree $(l_1+l_2)$ given by (\ref{eq: H1 star l1+l2}) and (\ref{eq: H2 star k,p l1+l2}). 
\subsection{Comparison with small quantum cohomology.}
Fulton and Woodward provide a Chevalley formula in in the small quantum cohomology ring of any generalized flag variety \cite{Fulton}.	
In the case of the incidence variety, this formula yields
\begin{equation*}
\left\{
\begin{aligned}
h_1 \star_S [X(i,j)] &= h_1 \cdot [X(i,j)] + q_1 [X] \delta_{i,2} \delta_{j,1}   &\: \mathrm{if} \: i \neq 1;\\
h_1 \star_S [X(1,j)] &=  q_1 [X(n,j)]  &\: \mathrm{if} \:  j <n;\\
h_1 \star_S [X(1,n)]&=   q_1 [X(n-1,j)]  + q_1 q_2 [X] &\: \mathrm{if} \: j=n,
\end{aligned}
\right.
\end{equation*}
and
\begin{equation*}
\left\{
\begin{aligned}
h_2 \star_S [X(i,j)] &= h_2 \cdot [X(i,j)] + q_2 [X] \delta_{n,i} \delta_{n-1,j} &\mathrm{si} \: j\neq n;\\
h_2 \star_S [X(i,n)] &=  q_2 [X(i,1)] &\mathrm{si} \: i >1;\\
&= q_2 [X(1,2)] + q_1 q_2 [X] &\mathrm{si} \: i=1,
\end{aligned}
\right.
\end{equation*}
where the product $\star_S$ is the product in the small quantum cohomology ring of $X$.
If we compare the product in small quantum cohomology to the product in small quantum K theory, we observe (in blue the terms which are present in small quantum K-theory but not in small quantum cohomology):
\begin{align*}
\O_{h_1} \star \O_{k,p} &=\left \{ \begin{array}{lr} Q_1\O_{n-1,n} +Q_1Q_2 \left( [\O_X] \textcolor{blue}{- \O_{h_1}} \right) & \mathrm{if} \: k=1, p=n\\
Q_1 \O_{n,p} & \mathrm{if} \:k=1, \: p<n\\
\O_{1,2} + Q_1 \left( [\O_X] \textcolor{blue}{- \O_{h_1}} \right)  & \mathrm{if} \: k=2, p=1\\
\O_{k-1,p} & \mathrm{if} \: k>1, \: k \neq p+1\\
\O_{p-1,p} + \O_{p,p+1} \textcolor{blue}{- \O_{p-1,p+1}} & \mathrm{if} \: 1< p<n-1, \: k=p+1\\
\end{array}\right.
\end{align*}


%
\subsection{Generalization to other adjoint varieties ?} In \cite{chaput2011} Chaput-Perrin interpret Fulton and Woodward's Chevalley formula in quantum cohomology in terms of roots systems. Note that their construction can be generalized to describe the Chevally formula in quantum $K$-theory described in Proposition \ref{prop: chevalley formula QK(X)}. One might wonder if this construction holds for other types of adjoint varieties.

\section{On Littlewood-Richardson coefficients in $QK_s(X)$.}\label{sec: LR coefficients in QK(X)}
Recall that for any elements $u$ and $v$ in $W^P$ there is a unique expression 
\[\O_u \star \O_v = \sum_{w \in W^P} P_{u, v}^w(Q_1, Q_2) \O_w\]
where $P_{u,v}^w(Q_1, Q_2)$ is a polynomial in the Novikov variables $Q_1$, $Q_2$. Note that taking the limit $Q_1, Q_2 \rightarrow 0$ of $P_{u,v}^w(Q_1, Q_2)$ yields the Littlewood-Richardson coefficients in $K(X)$.
We provide here an algorithm computing all coefficients $P_{u,v}^w(Q_1, Q_2)$, and provide a closed formula that matches our computations for small values of $n$. Furthermore, we prove that the signs of the coefficients of the polynomials $P_{u,v}^w(Q_1, Q_2)$ satisfy a positivity rule.
\subsection{An algorithm to compute Littlewood-Richardson coefficients in $QK_s(X)$.}\label{subsec: algo}
Recall from Part \ref{sec: Chevalley formula}
\begin{align*}
\O_{h_1} \star \O_{k,p} &=\left \{ \begin{array}{lr} Q_1\O_{n-1,n} +Q_1Q_2 \left( [\O_X] - \O_{h_1} \right) & \mathrm{if} \: k=1, p=n\\
Q_1 \O_{n,p} & \mathrm{if} \:k=1, \: p<n\\
\O_{1,2} + Q_1 \left( [\O_X] - \O_{h_1} \right)  & \mathrm{if} \: k=2, p=1\\
\O_{k-1,p} & \mathrm{if} \: k>1, \: k \neq p+1\\
\O_{p-1,p} + \O_{p,p+1} - \O_{p-1,p+1} & \mathrm{if} \: 1< p<n-1, \: k=p+1\\
\end{array}\right.
\end{align*}
\begin{align*}
\O_{h_2} \star \O_{k,p} &=\left \{ \begin{array}{lr} Q_2\O_{1,2} +Q_1Q_2 \left( [\O_X] - \O_{h_2} \right) & \mathrm{if} \: k=1, p=n\\
Q_2 \O_{k,1} & \mathrm{if} \:k>1, \: p=n\\
\O_{n-1,n} + Q_2 \left( [\O_X] - \O_{h_2} \right)  & \mathrm{if} \: k=n, p=n-1\\
\O_{k,p+1} & \mathrm{if} \: p<n \: k \neq p+1\\
\O_{p,p+1} + \O_{p-1,p} - \O_{p-1,p+1} & \mathrm{if} \: 1< k<n-1, \: k=p+1\\
\end{array}\right.
\end{align*}
This Chevalley formula allows us to write down an algorithm computing the product between any two Schubert classes in $QK(X)$ (cf. appendix C for a python code implementing this algorithm); we indeed can compute for all $1 \leq k,p \leq n$, where $k \neq p$, $\forall u \in W^P$, the product $\mathcal{O}_{w_{k,p}} \star \mathcal{O}_u$ by following the here under steps : \begin{itemize}
	\item For all $1<k \leq n$, compute $\mathcal{O}_{w_{k,1}} \star  \mathcal{O}_u$ iteratively by noticing $\mathcal{O}_{w_{k,1}} = \mathcal{O}_{h_1} \star \mathcal{O}_{w_{k+1,1}}$ and initializing with $\mathcal{O}_{w_{n,1}} \star  \mathcal{O}_u = [\mathcal{O}_{X}] \star  \mathcal{O}_u =\mathcal{O}_u$;
	\item for all $1\leq p < k \leq n$, compute $\mathcal{O}_{w_{k,p}} \star  \mathcal{O}_u$ iteratively by noticing $\mathcal{O}_{w_{k,p}} = \mathcal{O}_{h_2} \star \mathcal{O}_{w_{k,p-1}}$;
	\item compute $\mathcal{O}_{w_{1,2}} \star  \mathcal{O}_u$ by noticing $\mathcal{O}_{w_{1,2}} = \mathcal{O}_{h_1} \star \mathcal{O}_{w_{2,1}} -Q_1([\O_X] - \O_{h_1})$;
	\item for all $1 \leq p <n$, compute $\mathcal{O}_{w_{p,p+1}} \star  \mathcal{O}_u$ iteratively by noticing $\mathcal{O}_{w_{p,p+1}} = \mathcal{O}_{h_1} \star \mathcal{O}_{w_{p+1,p}} + (\mathcal{O}_{h_2} -id) \star \mathcal{O}_{w_{p-1,p}}$;
	\item for all $1\leq k < p \leq n$, where $k \neq p-1$, compute $\mathcal{O}_{w_{k,p}} \star  \mathcal{O}_u$ iteratively by noticing $\mathcal{O}_{w_{k,p}} = \mathcal{O}_{h_1} \star \mathcal{O}_{w_{k+1,p}}$.
\end{itemize}

\subsection{Positivity in $QK_s(X)$.}\label{subsec: sign LR coeff}
Let $X \simeq G/P$ be a generalized flag variety. For $u$, $v$ in $W^P$ decompose the quantum product $\mathcal{O}_u \star \mathcal{O}_v = \sum_{w \in W^P, \beta \in E(X)} N_{u,v}^{w, \beta} Q^\beta \mathcal{O}_w$. For any elements $u$ and $v$ in $W^P$, we call the product $\O_u \star \O_v$ \textit{positive} if for all $w \in W^P$, $\beta \in E(X)$, we have : \begin{equation}\label{eq: rule positivity}
(-1)^{\mathrm{codim}(X(w)) - \mathrm{codim}(X(u)) - \mathrm{codim}(X(v)) + \int_\beta c_1(T_X)} N_{u,v}^{w, \beta} \geq 0.
\end{equation}
Note that this is equivalent to  \begin{equation*}
(-1)^{\mathrm{dim}(X(v)) - \mathrm{codim}(X(u)) - \mathrm{dim}(X(w)) + \int_\beta c_1(T_X)} N_{u,v}^{w, \beta} \geq 0.
\end{equation*}
Buch-Chaput-Mihalcea-Perrin proved that the product of any two Schubert classes is positive for $X$ a cominuscule variety \cite{buch2016chevalley}. Denote by $\mathcal{O}_u \cdot \mathcal{O}_v = \sum_{w \in W^P} N_{u,v}^w \mathcal{O}_w$ the decomposition of the product of two Schubert classes in $K(X)$. Note that for $\beta =0$ (\ref{eq: rule positivity}) yields \begin{equation*}
(-1)^{\mathrm{codim}(X(w)) - \mathrm{codim}(X(u)) - \mathrm{codim}(X(v))} N_{u,v}^w \geq 0.
\end{equation*}
Brion proved that, for $X$ any generalized flag variety, the product of any two Schubert classes satisfies this positivity rule. 

Let us go back now to our setting, i.e. let $X=Fl_{1,n-1}$. 
\begin{proposition}\label{prop: positivity Ou star Ov}
	Let $u$, $v$ be elements in $W^P$. The product $\O_u \star \O_v$ is positive.
\end{proposition}
Let $u=w_{k,p}$ in $W^P$. We will prove here that the product of $\O_u$ with any Schubert class is positive by induction on $(k,p)$, running over the steps of the algorithm described here above. 
We will  use the following formulas. According to (\ref{eq : Schubert var embedded i<j}) \begin{equation*} 
X(i,j) =\left\{ \left( [x_1: \dots :x_i:0: \dots :0], [0: \dots 0: y_j : \dots : y_n] \right) \in \mathbb{P}^{n-1} \times \mathbb{P}^{n-1} \right\},
\end{equation*}
we have \[\dim X(k,p) = \left\{\begin{array}{lr} k+n-p-1 & \mathrm{if} \: k<p \\ k+n-p-2 & \mathrm{if} \: k>p \end{array}  \right.  =  (k+n-p-2)  + \Gamma_{p-k,0}\]
where we denote by $\Gamma_{n,0}$ the symbol of value $1$ if $n \geq 0$, $0$ else. Furthermore note that since $X$ is a degree $(1,1)$ hypersurface of $\P^{n-1} \times \P^{n-1}$ the adjunction formula yields for any $d_1, d_2 \geq 0$  \[\int_{d_1l_1 + d_2 l_2} c_1(T_X) = d_1 (n-1) + d_2 (n-1)\]
Let us our now begin our induction.
\begin{itemize}
	\item \textbf{Let $1 < k \leq n$.} Then $\O_{k,1} = \O_{h_1} \star \O_{k+1,1}$. Suppose the product $\O_{k+1,1} \star \O_v$ is positive for all $v \in W^P$. Let $v$ be an element in $W^P$.  According to Lemma \ref{lem: Ohi star Ov positive} $i)$ $\O_{h_1} \star \O_v$ is positive,  Hence  according to Lemma \ref{lem: Ohi star Ov positive} $iii)$ $\O_{k,1} \star \O_v$ is positive. Indeed, for $k>1$ we have $\dim X(k,1)= k+n-3 = \dim X(k+1,1)-1$.
	\item \textbf{Let $1\leq p < k \leq n$.} Suppose the product $\O_{k,p-1} \star \O_v$ is positive for all $v \in W^P$. Note that $\mathcal{O}_{{k,p}} = \mathcal{O}_{h_2} \star \mathcal{O}_{{k,p-1}}$.  Let $v$ be an element in $W^P$. According to Lemma \ref{lem: Ohi star Ov positive} $i)$ $\O_{h_2} \star \O_v$ is positive,  Hence  according to Lemma \ref{lem: Ohi star Ov positive} $iii)$ $\O_{k,p} \star \O_v$ is positive. Indeed, for $1 \leq p <k \leq n$ we have $p-1 <k$ and $\dim X(k,p)= k+n-p-1 = \dim X(k,p-1)-1$.
	\item \textbf{ For $k=1$ and $p=2$.} Let $v$ be an element in $W^P$. Note that $\mathcal{O}_{w_{1,2}} = \mathcal{O}_{h_1} \star \mathcal{O}_{w_{2,1}} -Q_1([\O_X] - \O_{h_1})$. Since $\dim X(1,2)= n-2 = \dim X(2,1)-1$, according to Lemma \ref{lem: Ohi star Ov positive} $ii)$ and $iii)$ the term $mathcal{O}_{h_1} \star \mathcal{O}_{w_{2,1}} \star \O_v$ has the expected sign. Furthermore, since $ \dim X -1-(n-1)=2n-4-n+1=\dim X(1,2)$ , according to Lemma \ref{lem: Ohi star Ov positive} $iv)$, the terms $-Q_1 [\O_X] \star \O_v$ and $Q_1 \O_{h_1} \star \O_v$ have the expected sign.
	Hence the product $\O_{k,p} \star \O_v$ is positive for all $w \in W^P$.
	\item \textbf{Let $1 \leq p <n$.} Let $v$ be an element in $W^P$. Suppose the product $\O_{p+1,p} \star \O_v$ is positive.  Note that $\mathcal{O}_{w_{p,p+1}} = \mathcal{O}_{h_1} \star \mathcal{O}_{w_{p+1,p}} + (\mathcal{O}_{h_2} -id) \star \mathcal{O}_{w_{p-1,p}}$. Note that $\dim X(p,p+1)= n-2 = \dim X(p+1,p)-1$. Hence, since $\O_{p+1,p} \star \O_v$ is positive $-\O_{p+1,p} \star \O_v$ has the same sign as the expected sign for $\O_{p,p+1} \star \O_v$. Furthermore, according to Lemma \ref{lem: Ohi star Ov positive} both $\O_{h_i} \star \O_{p,p+1} \star \O_v$ have the same sign as the expected sign for $\O_{p,p+1} \star \O_v$. Hence $\O_{p,p+1} \star \O_v$ is positive.
	\item \textbf{Let $1\leq k < p \leq n$, where $k \neq p-1$.} Let $v$ be an element in $W^P$. Suppose the product $\O_{k+1,p} \star \O_v$ is positive.  Note that  $\mathcal{O}_{w_{k,p}} = \mathcal{O}_{h_1} \star \mathcal{O}_{w_{k+1,p}}$.Since $k<p-1$, $\dim X(k,p) = k+n-p-1 = \dim X(k+1,p)$. Hence according to Lemma \ref{lem: Ohi star Ov positive} $iii)$ the product $\O_{k,p} \star \O_v$ is positive.
\end{itemize}
Finally by induction $\O_{k,p} \star \O_v$ is positive for all  $v$ in $W^P$.
\begin{lemme}\label{lem: Ohi star Ov positive}
	\begin{enumerate}[label=\roman*)]
		\item For any element $v$ in $W^P$ the product $\O_{h_1} \star \O_{v}$ is positive. 
		\item For any element $v$ in $W^P$ the product $\O_{h_2} \star \O_{v}$ is positive. 
		\item 	Let $u$ be an element in $W^P$. Suppose there exists an element $u'$ in $W^P$ such that $\O_u = \O_{h_i} \star \O_{u'}$, and $\O_{u'} \star \O_v$ is positive for all $v$ in $W^P$. Furthermore suppose $\dim X(u) = \dim X(u')-1$. Then for any element $v$ in $W^P$ the product $\O_u \star \O_v$ is positive.
		\item  Let $u$ be an element in $W^P$. Suppose there exists an element $u'$ in $W^P$ such that $\O_u = Q_j  \O_{u'}$, where $j \in \{1,2\}$, and $\O_{u'} \star \O_v$ is positive for all $v$ in $W^P$. Furthermore suppose $\dim X(u) = \dim X(u')-(n-1)$. Then for any element $v$ in $W^P$ the product $\O_u \star \O_v$ is positive.
	\end{enumerate}
	\item Let $1 \leq k,p \leq n$. 
\end{lemme}
\begin{proof}
	\begin{enumerate}[label=\roman*)]
		\item 	We compute explicitly $(-1)^{\mathrm{dim}(X(v)) - \mathrm{codim}(h_i) - \mathrm{dim}(X(w)) + \int_\beta c_1(T_X)} N_{u,v}^{w, \beta}$ using (\ref{eq: Oh1 star Ok,p}). Note that according to \cite{brion2002positivity} $N_{u,v}^{w, 0}$ has the expected sign, hence we only have to consider terms $N_{u,v}^{w, \beta}$ with $\beta >0$. 
		\begin{enumerate}
			\item Recall we have $\O_{h_1} \star \O_{1,n} = Q_1\O_{n-1,n} +Q_1Q_2 \left( [\O_X] - \O_{h_1} \right)$. We have \[(-1)^{dim X(n-1,n)-\dim X(1,n) - \codim h_1 +(n-1)} = (-1)^{(n-1)-0 -1 + (n-2)+(n-1)} = (-1)^{2n-4} \geq 0 .\]
			\[(-1)^{\dim X- \dim X(1,n) -\codim h_1 +2(n-1)} = (-1)^{2n-3-1} \geq 0 .\]
			Hence $\O_{h_1} \star \O_{1,n}$ is positive.
			\item Let $1<p<n$. Recall we have $\O_{h_1} \star \O_{1,p} = Q_1 \O_{n,p}$. We have \[(-1)^{dim X(1,p)-\dim X(n,p) -\codim h_i +(n-1)} = (-1)^{(n-p-1) -(2n-p-2)-1+(n-1)} \geq 0 .\]
			\item  Recall we have $\O_{h_1} \star \O_{2,1} = \O_{1,2} + Q_1 \left( [\O_X] - \O_{h_1} \right) $. We have \[(-1)^{dim X-\dim X(2,1) -\codim h_i +(n-1)} = (-1)^{(2n-3) -(n-2)-1+(n-1)} \geq 0 .\]
		\end{enumerate}
		\item Note that by symmetry if there is a term in $Q_2 \O_{k,p}$ in $\O_{h_2} \star \O_{i,j}$, we have a term in $Q_1 \O_{n-p+1,n-k+1}$ in $\O_{h_1} \star \O_{n-i+1,n-j+1}$, and if there is a term in $Q_1 Q_2 \O_{k,p}$ in $\O_{h_2} \star \O_{i,j}$, we have a term in $Q_1Q_2 \O_{n-p+1,n-k+1}$ in $\O_{h_1} \star \O_{n-i+1,n-j+1}$. Since \begin{align*}\dim X(n-j+1, n-i+1) &= n-j+1+n-(n-i+1) -1 + \Gamma_{n-i+1-(n-j+1),0} \\
		&=n+i-j-1 + \Gamma_{j-i,0}\\
		&= \dim X(i,j), \end{align*}
		part $i)$ yields the positivity of $\O_{h_2} \star \O_v$ for all $v$ in $W^P$.
		\item Let $v$ be an element in $W^P$. Consider the decomposition $\mathcal{O}_{u'} \star \mathcal{O}_v = \sum_{w \in W^P, \beta \in E(X)} N_{u',v}^{w, \beta} Q^\beta \mathcal{O}_w$  of $\O_{u'} \star \O_v$ in the Schubert basis. Since the quantum product $\star$ is associative we have \begin{align*}\O_u \star \O_v &= \O_{h_i} \star \mathcal{O}_{u'} \star \mathcal{O}_v = \sum_{w \in W^P, \beta_1 \in E(X)} N_{u',v}^{w, \beta} Q^{\beta_1} \O_{h_i} \star \mathcal{O}_w \\
		&=  \sum_{w,\: w' \in W^P,\: \beta_1, \beta_2 \in E(X)} N_{u',v}^{w, \beta_1} Q^{\beta_1} N_{h_i,w}^{w', \beta_2} Q^{\beta_2} \mathcal{O}_{w'} \end{align*}
		hence $N_{u,v}^{w', \beta} =\sum_{w \in W^P} \sum_{\substack{\beta_1, \: \beta_2 \in E(X),\\ \: \beta = \beta_1 + \beta_2}} N_{u',v}^{w, \beta_1} N_{h_i,w}^{w', \beta_2}$. Since $\dim X(u) = \dim X(u')-1$ and the products $\O_{u'} \star \O_{v}$ and $\O_{h_i} \star \O_{w}$ are positive we have \begin{align*}   (-1)^{\dim X(u) - \dim X(w') - \codim X(v) + \int_\beta c_1(T_X)} &N_{u,v}^{w', \beta}\\ &=  \sum_{w \in W^P} (-1)^{\dim X(u') -1 - \dim X(w')- \codim X(v) + \int_\beta c_1(T_X)} N_{u',v}^{w, \beta_1} N_{h_i,w}^{w', \beta_2}\\
		&= \sum_{w \in W^P} (-1)^{\dim X(u') - \dim X(w) - \codim X(v) + \int_{\beta_1} c_1(T_X)} N_{u',v}^{w, \beta_1} \\
		& \:\:\:\:\:\:\:\:\:\:\:\:\:\:\:\:\:\:\times (-1)^{\dim X(w) - \dim X(w') -\codim h_i + \int_{\beta_2} c_1(T_X)} N_{h_i,w}^{w', \beta_2} \\
		&\geq 0
		\end{align*}
		\item Let $v$ be an element in $W^P$. Consider the decomposition $\mathcal{O}_{u'} \star \mathcal{O}_v = \sum_{w \in W^P, \beta \in E(X)} N_{u',v}^{w, \beta} Q^\beta \mathcal{O}_w$  of $\O_{u'} \star \O_v$ in the Schubert basis. We have \begin{align*}\O_u \star \O_v &= Q_j \mathcal{O}_{u'} \star \mathcal{O}_v = \sum_{w \in W^P, \beta \in E(X)} N_{u',v}^{w, \beta+l_j} Q^{\beta}  \mathcal{O}_w \\
		&=  \sum_{w \in W^P,\: \beta \in E(X)} N_{u',v}^{w, \beta + l_j} Q^{\beta} 
		\end{align*}
		hence $N_{u,v}^{w', \beta} =\sum_{w \in W^P} \sum_{\substack{\beta \in E(X), \: \beta \geq l_j}} N_{u',v}^{w, \beta-l_j}$. Since $\dim X(u) = \dim X(u')-(n-1)$ and the product $\O_{u'} \star \O_{v}$ is positive we have \begin{align*}   (-1)^{\dim X(u) - \dim X(w') - \codim X(v) +  \int_\beta c_1(T_X)}& N_{u,v}^{w', \beta}\\ &=  \sum_{w \in W^P} (-1)^{\dim X(u') -(n-1) - \dim X(w')- \codim X(v) + \int_{\beta} c_1(T_X)} N_{u',v}^{w, \beta-l_j} \\
		&= \sum_{w \in W^P} (-1)^{\dim X(u') - \dim X(w) - \codim X(v) + \int_{\beta - l_j} c_1(T_X)} N_{u',v}^{w, \beta - l_j} \\
		&\geq 0
		\end{align*}
	\end{enumerate}
\end{proof}

\subsection{A conjectural formula for Littlewood-Richardson coefficients in $QK_s(X)$.}\label{subsec: conjecture LR}
We computed for various values of $n$ Littlewood-Richardson coefficients in $QK_s(X)$ using the code in Appendix B. The results allowed us to guess and write down the following formula. This formula has been verified for small values of $n$. 
Let us now introduce some specific notations we need to express our result. First, we define "degree operators" $d_i$ on $(W^P)^3$ by : \begin{align*}
\forall 1 \leq i,j,k,p,s,t \leq n, \: i \neq j, \: k \neq p, \: s \neq t, \:\:\:\: &d_1(w_{i,j}, w_{k,p}, w_{s,t}) = 1 - \lfloor \frac{i+k-s}{n} \rfloor;\\
&d_2(w_{i,j}, w_{k,p}, w_{s,t}) = \lfloor \frac{j+p-t}{n} \rfloor,
\end{align*}
where $\lfloor x \rfloor$ is the integer part of $x$.
Furthermore, let us define maps $t_i : (W^P)^2 \rightarrow W^P$ by : \begin{align*}
\forall 1 \leq i,j,k \leq n, \: i \neq j, \: k \neq p,  \:\:\:\: & t_0(w_{i,j},w_{k,p})=w_{(i+k-1) mod(n) +1, (j+p-2) mod(n) +1}\\
&t_1(w_{i,j},w_{k,p})=w_{(i+k-2) mod(n) +1, (j+p-2) mod(n) +1}\\
&t_2(w_{i,j},w_{k,p})=w_{(i+k-1) mod(n) +1, (j+p-1) mod(n) +1}\\
&t_3(w_{i,j},w_{k,p})=w_{(i+k-2) mod(n) +1, (j+p-1) mod(n) +1}.
\end{align*}
Finally, we define an operator $\Delta$ on $(W^P)^3$ by setting : \begin{align*}
\forall u, \: v, \: w \: \in W^P, \: \Delta(u,v,w)&=1 \: &\mathrm{if} \: (-1)^{\ell(w_0w) - \ell(w_0u) - \ell(w_0 v) + \int_{d_1(u,v,w) l_1 + d_2(u,v,w) l_2} c_1(T_X)} \geq 0 \\
&=0 &\mathrm{if} \: (-1)^{\ell(w_0w) - \ell(w_0u) - \ell(w_0 v) + \int_{d_1(u,v,w) l_1 + d_2(u,v,w) l_2} c_1(T_X)} < 0
\end{align*}
\begin{conjecture}
	The product between two Schubert classes in $QK_s(X)$ is given by the following formula :
	\begin{align*}
	\forall \: u, \: v \in W^P, \: \: \mathcal{O}_u \star \mathcal{O}_v =& \:\:\:\:\:\:\Delta(u,v,t_0(u,v)) & q_1^{d_1(u,v,t_0(u,v))}  q_2^{d_2(u,v,t_0(u,v))} \mathcal{O}_{t_0(u,v)} \\ &+ (1-\Delta(u,v,t_1(u,v)))\bigg(& q_1^{d_1(u,v,t_1(u,v))} q_2^{d_2(u,v,t_1(u,v))} \mathcal{O}_{t_1(u,v)}\:\: \\ &&+q_1^{d_1(u,v,t_2(u,v))} q_2^{d_2(u,v,t_2(u,v))} \mathcal{O}_{t_2(u,v)} \:\:\\ & & -q_1^{d_1(u,v,t_3(u,v))} q_2^{d_2(u,v,t_3(u,v))} \mathcal{O}_{t_3(u,v)} \bigg)
	\end{align*}
\end{conjecture}

\section*{Appendix B}

%
%
%
Here under is the python code used to compute the product between any two Schubert classes in $QK(Fl(1,n-1)$; these computations helped us write down the Littlewood-Richardson formula we show in chapter $4$.

\begin{lstlisting}
# Computing product in the small quantum K-theory QKs of incidence varieties Fl(1,n-1;n)
# Date : 18/04/2019
# Author : Sybille Rosset
import math
from sympy import *
from sympy import Matrix
import numpy

# Parametrizing n
n=5 #i.e. we look at lines in hyperplanes of C^n

## This code computes the product in small quantum K theory of Fl(1,n-1;n) between 
#Schubert varieties and writes it down within a file named "Output.txt"

#--------------------------------------------------
#Indexing Schubert varieties X(i,j) by a 1D array !
#--------------------------------------------------
def coeffij(k,n): #Function sending  0<k<=n(n-1)) to the corresponding pair (i,j)
i=math.ceil(k/(n-1))
if (k-(i-1)*(n-1) < i):
j=k-(i-1)*(n-1)
else:
j=k+1-(i-1)*(n-1)
return([i,j])

def invcoeffij(i,j,n): #Function sending  (i,j) to the corresponding 0<k<=n(n-1))
if(j>i):
t=(i-1)*(n-1)+j-1
else:
t=(i-1)*(n-1)+j
return(t-1)

#for k in range(1,n*(n-1)):
#[i,j]=coeffij(k)
#t=invcoeffij(i,j)
#print(i,j,k,t)
print(invcoeffij(1,2,3))

#--------------------------------------------------------
#Defining the action of O_h1 and O_h2 on QKs             !
#--------------------------------------------------------
q1 = Symbol('q1')
q2 = Symbol('q2')

# !!.I.!! Defining the action of O_h1
def h1(n):
h1 = zeros(n*(n-1))
# defining Oh1*O1,n
h1[invcoeffij(n-1,n,n),invcoeffij(1,n,n)]=q1
h1[invcoeffij(n,1,n),invcoeffij(1,n,n)]=q1*q2
h1[invcoeffij(n-1,1,n),invcoeffij(1,n,n)]=-q1*q2
#defining Oh1*O1,p
for p in range(2,n):
h1[invcoeffij(n,p,n),invcoeffij(1,p,n)]=q1
#defining Oh1*O2,1
h1[invcoeffij(1,2,n),invcoeffij(2,1,n)]=1
h1[invcoeffij(n,1,n),invcoeffij(2,1,n)]=q1
h1[invcoeffij(n,2,n),invcoeffij(2,1,n)]=-q1
#defining Oh1*Op+1,p
for p in range(2,n):
h1[invcoeffij(p-1,p,n),invcoeffij(p+1,p,n)]=1
h1[invcoeffij(p,p+1,n),invcoeffij(p+1,p,n)]=1
h1[invcoeffij(p-1,p+1,n),invcoeffij(p+1,p,n)]=-1
#defining Oh1*Ok,p
for k in range(2,n+1):
for p in range(1,n+1):
#print(k,p)
if(k!=p+1 and k!=p):
h1[invcoeffij(k-1,p,n),invcoeffij(k,p,n)]=1
return(h1)

# !!.II.!! Defining the action of O_h2 
def h2(n):       
h2 = zeros(n*(n-1))
# defining Oh2*O1,n
h2[invcoeffij(1,2,n),invcoeffij(1,n,n)]=q2
h2[invcoeffij(n,1,n),invcoeffij(1,n,n)]=q1*q2
h2[invcoeffij(n,2,n),invcoeffij(1,n,n)]=-q1*q2
#defining Oh2*Ok,n
for k in range(2,n):
h2[invcoeffij(k,1,n),invcoeffij(k,n,n)]=q2
#defining Oh2*On,n-1
h2[invcoeffij(n-1,n,n),invcoeffij(n,n-1,n)]=1
h2[invcoeffij(n,1,n),invcoeffij(n,n-1,n)]=q2
h2[invcoeffij(n-1,1,n),invcoeffij(n,n-1,n)]=-q2
#defining Oh2*Op+1,p
for p in range(1,n-1):
h2[invcoeffij(p,p+1,n),invcoeffij(p+1,p,n)]=1
h2[invcoeffij(p+1,p+2,n),invcoeffij(p+1,p,n)]=1
h2[invcoeffij(p,p+2,n),invcoeffij(p+1,p,n)]=-1
#defining Oh2*Ok,p
for p in range(1,n):
for k in range(1,n+1):
if(k!=p+1 and k!=p):
h2[invcoeffij(k,p+1,n),invcoeffij(k,p,n)]=1
return(h2)

#--------------------------------------------    
#Defining the action of O_k,p on any O_i,j  !
#--------------------------------------------
def O(n):
O=[None]*n*(n-1)           #we create a list O, where we will stock O_k,p as the invcoeffij(k,p)-th element in O
O[invcoeffij(n,1,n)]=eye(n*(n-1))  #initializing OX as the identity matrix
a = h1(n)
b=h2(n)
J=b-eye(n*(n-1))
#We first define O_k,1
for k in range(1,n-1):
O[invcoeffij(n-k,1,n)]=a*O[invcoeffij(n-k+1,1,n)]
#We then define O_k,p, where k>p
for k in range(2,n+1):
for p in range(2,k):
#print(k,p)
O[invcoeffij(k,p,n)]=b*O[invcoeffij(k,p-1,n)]
#We now define O_p,p+1
O[invcoeffij(1,2,n)]=a*O[invcoeffij(2,1,n)]+q1*J
for p in range(2,n):
#print("p",p)
O[invcoeffij(p,p+1,n)]=a*O[invcoeffij(p+1,p,n)]+J*O[invcoeffij(p-1,p,n)]
#We finally define O_k,p for k<p and k!=p-1
for p in range(3,n+1):
for k in range(2,p):
print(k,p)
O[invcoeffij(p-k,p,n)]=a*O[invcoeffij(p-k+1,p,n)]
return(O)


#---------------------------------------------------------
# MAIN TASK : computing O_i,j*O_k,p                      !
#---------------------------------------------------------
#Computing QKs(Fl(1,n-1))
L=O(n)

#Writing down small quantum K theory product in a text file named "Output.txt"
Text=[]
for k in range(0,n*(n-1)):
A=L[k]
[l,p]= coeffij(k+1,n)
for r in range(0,n*(n-1)):
[i,j]=coeffij(r+1,n)
R=''
R=R+"O_"+str(l)+','+str(p)+" "+"*O_"+str(i)+','+str(j)+"= "
if(l+i<n+1 and j+p<n+1 and ((i>j and l>p) or (i+l<j+p and i>j and l<p) or (i+l<j+p and i<j and l>p))):
for s in range(0,n*(n-1)):
if(A[s,r]==0):
R=R
else:
[u,v]=coeffij(s+1,n)
R=R+'+('+str(A[s,r])+')*'+'O_'+str(u)+","+str(v)+" "
#print(R)
Text.append(R+';')

#print(Text)               
numpy.savetxt("Output.txt",Text,delimiter=";", fmt='%s')
print("Ok")


\end{lstlisting}

\bibliography{bibliographie}
\bibliographystyle{alpha} 
\end{document}